\documentclass{amsart}
\usepackage{graphicx} 
\usepackage{amsmath}
\usepackage{amsthm}
\usepackage{amssymb}
\usepackage{comment}
\usepackage{nameref}
\usepackage{cite}
\usepackage{xcolor}
\usepackage{dsfont}
\usepackage{appendix}
\usepackage{comment}
\numberwithin{equation}{section}

\newtheorem{theo}{Theorem}[section]
\newtheorem{lemm}[theo]{Lemma}
\newtheorem{coro}[theo]{Corollary}
\newtheorem{prop}[theo]{Proposition}
\newtheorem{assum}[theo]{Assumption}
\newtheorem{property}{Property}
\newtheorem{rema}[theo]{Remark}

\newcommand{\vs}{v_*}
\newcommand{\ga}{\gamma}
\newcommand{\al}{\alpha}

\newcommand{\rn}{\rho_\nu}
\newcommand{\rhn}{\rho^\nu}
\newcommand{\vn}{v^\nu}
\newcommand{\hn}{h^\nu}
\newcommand{\un}{u^\nu}
\newcommand{\s}{\sigma}

\newcommand{\de}{\delta}
\newcommand{\la}{\lambda}
\newcommand{\La}{\Lambda}
\newcommand{\e}{\epsilon}

\newcommand{\td}{\tilde}
\newcommand{\um}{U_-}
\newcommand{\up}{U_+}

\newcommand{\urn}{U^{R,\nu}}

\newcommand{\ru}{U^{R}}
\newcommand{\rv}{v^{R}}
\newcommand{\rh}{h^{R}}

\newcommand{\usn}{\td U^{\nu}}
\newcommand{\vsn}{\td v^{\nu}}
\newcommand{\hsn}{\td h^{\nu}}
\newcommand{\su}{\td U}
\newcommand{\sv}{\td v}
\newcommand{\sh}{\td h}
\newcommand{\sn}{X^\nu}
\newcommand{\sdn}{\dot{X}^\nu}
\newcommand{\sd}{\dot{X}}
\newcommand{\las}{\td \la}

\newcommand{\unn}{U^{P,\nu}}
\newcommand{\nnu}{U^{P}}
\newcommand{\nv}{v^{P}}
\newcommand{\nh}{h^{P}}
\newcommand{\cz}{c_0}
\newcommand{\co}{c_1}
\newcommand{\cor}{c_1}
\newcommand{\cop}{c_1}
\newcommand{\cf}{c_2}
\newcommand{\cfi}{c_3}
\newcommand{\cs}{c_*}
\newcommand{\cpr}{c_R}
\newcommand{\cg}{c_\gamma}

\newcommand{\ub}{\bar{U}_{\nu}}
\newcommand{\ust}{U_*}
\newcommand{\p}{\partial}
\newcommand{\sse}{\mathcal{S}}
\newcommand{\rse}{\mathcal{R}}
\newcommand{\npse}{\mathcal{NP}}
\newcommand{\be}{\beta}

\newcommand{\bu}{\bar{U}}
\newcommand{\bv}{\bar{v}}
\newcommand{\bh}{\bar{h}}

\newcommand{\bvn}{\bar{v}_\nu}
\newcommand{\bhn}{\bar{h}_\nu}
\newcommand{\bun}{\bar{u}_\nu}
\newcommand{\brn}{\bar{\rho}_\nu}

\newcommand{\inb}{\int_{y_b}^\infty}
\newcommand{\dx}{\, dx}

\newcommand{\mi}{\mathcal{I}}
\newcommand{\my}{\mathcal{Y}}
\newcommand{\md}{\mathcal{D}}
\newcommand{\mg}{\mathcal{G}}
\newcommand{\mcp}{\mathcal{P}}
\newcommand{\mj}{\mathcal{J}}
\newcommand{\ms}{\mathcal{S}}
\newcommand{\mr}{\mathcal{R}}
\newcommand{\mnp}{\mathcal{NP}}
\newcommand{\ma}{\mathcal{A}}
\newcommand{\ml}{\mathcal{L}}

\newcommand{\me}{\mathbf{e}}
\newcommand{\tg}{\tilde{G}}
\newcommand{\ts}{\phi_\be}
\newcommand{\pd}{p(v)-p(\bv)}
\newcommand{\bsi}{\bar{\s}_{\be}}
\newcommand{\bfc}{\tilde{C}}
\newcommand{\ty}{\tilde{y}}

\newcommand{\mef}{\me_2}
\newcommand{\mev}{\me_3}
\newcommand{\mei}{\me^{0,1}}
\newcommand{\meit}{\me^{0,2}}

\newcommand{\lv}{\underline{v}}
\newcommand{\lu}{\underline{U}}
\newcommand{\om}{\Omega}

\newcommand{\pdu}{\{\pd > \de_1\}}
\newcommand{\pdl}{\{\pd < - \de_1\}}
\newcommand{\pds}{p(v) - p(\lv_s)}
\newcommand{\pdb}{p(v) - p(\lv_b)}

\newcommand{\uji}{\underline{\mj_\be}}
\newcommand{\ji}{\mj_\be}
\newcommand{\xoi}{x_0^\be}
\newcommand{\cde}{C}

\newcommand{\tc}{\mathbf{C}}

\newcommand{\bc}{\mathbf{C}_1}

\newcommand{\cbet}{\mathbf{c}}

\newcommand{\rne}{\rho^\nu}
\newcommand{\une}{u^\nu}
\newcommand{\udU}{\underline{U}}
\newcommand{\udr}{\underline{\rho}}
\newcommand{\udu}{\underline{u}}

\newcommand{\ef}{\textbf{x}}
\newcommand{\eft}{\textbf{y}}

\title[Weak-BV Stability and Inviscid Limits for Outflow]{Weak-BV Stability and Inviscid Limits for Barotropic Flows with Outflow Boundary Conditions}

\thanks{\textbf{Acknowledgment.} A. Vasseur was partially supported by the NSF grants DMS 2219434, 2306852, and 1840314.} 

\author[J. Meng]{Jiayun Meng}
\address{Department of Mathematics,
The University of Texas at Austin.}
\email{jiayun@utexas.edu}
\author[A.F. Vasseur]{Alexis F. Vasseur}
\address{Department of Mathematics,
The University of Texas at Austin.}
\email{vasseur@math.utexas.edu}

\begin{document}

\begin{abstract}
    Boundary value problems for systems of conservation laws, including inviscid compressible flows, arise naturally in physical models such as shallow water flow in a canal, where fluid may leave the domain through the boundary. In contrast to the Cauchy problem, the inviscid limit is highly sensitive to the underlying viscous model: different viscous regularizations may select different effective boundary conditions for the limiting system.

    In this paper, we show that small BV solutions of the one-dimensional compressible barotropic Euler equations arise as inviscid limits of the barotropic Navier-Stokes system for the outflow boundary problem. We work with the physically meaningful effective boundary condition selected by the Navier-Stokes system for the limiting Euler solutions. Moreover, going beyond smooth flows or a fixed class of Riemann solutions, our stability results apply to small BV solutions exhibiting general wave patterns and complicated configurations generated by interactions with the outflow boundary, including reflected shocks.
\end{abstract}

\maketitle

\tableofcontents
\section{Introduction}
Boundary value problems for systems of conservation laws arise naturally in physical models, including inviscid compressible fluid mechanics. A typical example is shallow water flow in a horizontal canal with a flat bed. The evolution of the water depth and the depth-averaged horizontal velocity is described by the shallow water equations, which form a $2\times 2$ system of conservation laws. In this paper, we study the one-dimensional compressible barotropic Euler equations, which include the shallow water equations as a special case, in the outflow setting, where fluid leaves the domain through the boundary.

In the presence of a boundary, the selection of physically admissible solutions to the Euler equations is considerably more delicate than in the Cauchy problem. The entropy inequality in the interior does not, by itself, determine a physically meaningful boundary condition. Such a condition must be compatible with both the characteristic structure of the system and the entropy inequality. A natural approach is to view the Euler equations as the vanishing viscosity limit of a given viscous model and identify the effective boundary condition selected by the associated boundary layers.

The selection of admissible solutions as inviscid limits of viscous approximations is a classical approach known as the vanishing viscosity method. A major breakthrough in this direction was the celebrated work of Bianchini and Bressan \cite{vanishingnondegenerate} in 2005, who proved that, for the Cauchy problem on the whole line, small BV solutions of strictly hyperbolic systems of conservation laws are vanishing viscosity limits of approximations with an artificial nondegenerate viscosity. For the Cauchy problem, different viscous models tend to select the same entropy solution In the presence of a boundary, however, they may generate different boundary layers and thus select different effective boundary conditions for the limiting system.

For boundary value problems, the vanishing viscosity limit is highly sensitive to the underlying viscous model, as discussed in \cite{JLe99,AMS24}. In particular, the effective boundary condition for the limiting Euler system selected by the physical Navier-Stokes system differs from that selected by a viscous system with artificial nondegenerate viscosity. In this paper, going beyond the framework developed in \cite{vanishingnondegenerate}, we work with inviscid limits of the Navier-Stokes system and thereby study small BV solutions of the Euler system with physically meaningful boundary conditions.

In addition to treating a physical viscous model, our inviscid limit results, Theorems~\ref{main thm vanishing viscosity sol} and~\ref{main thm inviscid limit}, extend beyond smooth flows and fixed classes of Riemann solutions: they apply to small BV solutions with complex wave interactions and general wave patterns. The initial boundary value problem introduces further difficulties beyond those arising in the Cauchy problem. In particular, our results account for the complicated configurations generated by interactions with the outflow boundary, including reflected shocks.

\section{Precise statement of the results}
The Euler system, introduced more than 250 years ago in \cite{Euler}, is one of the earliest systems of partial differential equations ever written. In this paper, we consider the one-dimensional barotropic Euler system on the half line
\begin{equation}\label{system E rho, rho u}
    \begin{cases}
        \rho_t + (\rho u)_x = 0 \, , \quad x \in \mathbb{R}_+ \, , \; t>0 \, ,\\
        (\rho u)_t + (\rho u^2)_x + p(\rho)_x = 0 \, , \quad x \in \mathbb{R}_+  \, , \; t>0 \, ,
    \end{cases}
\end{equation}
where $\rho$ denotes the density, $u$ denotes the velocity, and the pressure law is 
\begin{align*}
    p(\rho) = \rho^\ga \, ,
\end{align*}
where $\gamma>1$ is the adiabatic exponent. We write the unknown vector as $U=(\rho, \rho u)$ and impose the following entropy inequality 
\begin{equation}\label{entropy inequality}
    \p_t \eta(\rho, \rho u) + \p_x \bigl(q(\rho, \rho u)\bigr) \leq 0 
\end{equation}
where $\eta$, the mathematical entropy, is in this case the physical energy, and $q$ is the corresponding entropy flux. They are given by
\begin{align*}
    \eta(\rho, \rho u) = \frac{\rho |u|^2}{2} + \frac{\rho^\ga}{\ga-1} \, , \quad
    q(\rho, \rho u) = \bigl(\eta(\rho, \rho u) + \rho^\ga \bigr) u \, .
\end{align*}
The underlying physical viscous system of the Euler system is the Navier-Stokes system
\begin{equation}\label{system rho, rho u}
    \begin{cases}
        \rhn_t + (\rhn\un)_x = 0 \, , \quad t>0 \, , \; x \in \mathbb{R}_+ \, , \\
        (\rhn \un)_t + \bigl(\rhn (\un)^2\bigr)_x + p(\rhn)_x = \nu\bigl(\mu(\rhn)\un_x\bigr)_x \, , \quad t>0 \, , \; x \in \mathbb{R}_+ \, .
    \end{cases}
\end{equation}
where $\nu>0$ is the viscosity parameter and $\mu$ is a viscosity coefficient satisfying the following assumption. 
\begin{assum}\label{property viscosity}
    Let $\ga>1$. We assume $\mu$ to be differentiable and that there exist constants $C_1,C_2>0$ such that
    \begin{align}
        C_1 \rho^\al \leq \mu(\rho) \leq C_2 \rho^\al \, , \quad \textup{for } \rho\geq 1 \, , \label{property viscosity alpha} \\
        C_1 \rho^{\widetilde{\al}} \leq \mu(\rho) \leq C_2 \rho^{\widetilde{\al}} \, , \quad \textup{for } \rho< 1 \, , \label{property viscosity wide alpha} 
    \end{align}
    for some 
    \begin{align*}
        &\al = \ga-1 \, , \quad \textup{if } \ga\leq 5/3 \, , \quad
        \ga-1\leq \al\leq \ga \, , \quad \textup{otherwise,} \\
        &\widetilde{\alpha} = \alpha \, , \quad \textup{if } \al < 1/2 \, , \quad
        0<\widetilde{\alpha}<1/2 \, , \quad \textup{otherwise.}
    \end{align*}
\end{assum}
This assumption, though it requires polynomial growth, covers a broad class of viscosity functions. This growth condition is consistent with the shallow water case, for which $\gamma = 2$ and the corresponding polynomial growth exponent is $\alpha = 1$. Note that the stronger restriction on the degeneracy of the viscosity near vacuum, namely \eqref{property viscosity wide alpha}, is needed only for the existence of the Navier-Stokes solutions.

In most cases, for example \cite{Lions,FNP,BJ}, one can construct Navier-Stokes solutions satisfying the entropy inequality
\begin{equation}\label{entropy inequality NS}
    \p_t \eta(\rhn, \rhn \un) + \p_x \bigl(q(\rhn, \rhn \un)-\nu \mu(\rhn) \un \p_x \un \bigr) + \nu \mu(\rhn) |\p_x \un|^2 \leq 0 \, .
\end{equation}
Such inequality reduces formally to \eqref{entropy inequality} in the limit.

\medskip
In this paper, we consider the outflow boundary problem, in which the fluid leaves the domain through the boundary. We first supplement the Navier-Stokes system \eqref{system rho, rho u} with the outflow boundary condition
\begin{align}\label{outflow boundary condition}
    u(0,t) = u_B < 0 \, .
\end{align}
The sound speed is 
\begin{align*}
    c(\rho) = \sqrt{p'(\rho)} = \sqrt{\ga \rho^{\ga-1}} \, .
\end{align*}
According to the relation between $c(\rho)$ and $|u|$, the state space is divided into the supersonic, transonic, and subsonic regions:
\begin{align*}
    \Omega_{super} = \{U: |u|>c(\rho)\} \, , \quad
    \Omega_{tran} = \{U: |u|=c(\rho) \} \, , \quad
    \Omega_{sub} = \{U: |u|<c(\rho) \} \, .
\end{align*}
The eigenvalues of the flux Jacobian are
\begin{align*}
    \la_1(U) = u - c(\rho) \, , \quad \la_2(U) = u + c(\rho) \, .
\end{align*}
In this paper, we focus on the most delicate and interesting region of the state space for small BV solutions of the Euler system: the subsonic region. In the subsonic region, the two characteristic speeds have opposite signs, so exactly one boundary condition is prescribed for the Euler system. The effective boundary condition for the Euler system selected by the Navier-Stokes system is precisely the condition given in \eqref{outflow boundary condition}.

\medskip

We prescribe a right far-field state $U_* = (\rho_*, \rho_* u_*)$ in
\begin{align}\label{def subsonic}
    \Omega_{sub}^{-} = \Omega_{sub} \cap \{U=(\rho, \rho u): u<0\} \, ,
\end{align} 
with $u_*$ close to, but not necessarily equal to, $u_B$.

We consider the following set of Euler initial values
\begin{equation}\label{def Euler initial value}
    \begin{aligned}
        E^0 = \{U^0 = (\rho^0, \rho^0 u^0) \in L^1_{\mathrm{loc}} \cap L^\infty (\mathbb{R}_+) \, \bigl| \; \inf (\rho^0)>0 \, , \\
        \int_{\mathbb{R}_+} \bigl(\rho^0 |u^0-u_*| + |\rho^0 - \rho_*| \bigr) \dx < \infty \}
    \end{aligned} \, .
\end{equation}
As in \cite{compensated,NStoE}, we consider slightly regularized initial values for the Navier-Stokes system. For any $U^0 \in E^0$, we say that $U_0^\nu = (\rhn_0, \rhn_0 \un_0)$ is an adapted family of initial values associated with $U^0$ for the Navier-Stokes system \eqref{system rho, rho u} if $U_0^\nu$ is smooth and satisfies
\begin{align}
    u_0^\nu(0) = u_B \, ,
\end{align}
and
\begin{equation}\label{NS initial value}
    \begin{aligned}
        \lim_{\nu\rightarrow 0} \Bigl( \int_{\mathbb{R}_+} |U_0^\nu - U^0| \dx + 
        \int_{\mathbb{R}_+} \rhn_0 |\un_0 - u^0|^2 \dx + \int_{\mathbb{R}_+} |\rhn_0 - \rho^0|^\ga \dx \\
        + \nu^2 \int_{\mathbb{R}_+} |\p_x \phi(\rhn_0)|^2 \dx \Bigr) = 0 \, ,
    \end{aligned}
\end{equation}
where we choose $\phi$ such that 
\begin{align}\label{def of phi}
    \phi'(\rho) = \frac{\mu(\rho)}{\rho^{3/2}} \, .
\end{align}

We will show that for any $U^0 \in E^0$, one can construct an adapted family of initial values for the Navier-Stokes system.

\begin{lemm}\label{lemma existence of initial values NS}
For any $U^0 \in E^0$, where $E^0$ is defined in \eqref{def Euler initial value}, there exists an adapted family of initial values $U_0^\nu$ associated with $U^0$ for the Navier-Stokes system \eqref{system rho, rho u}.
\end{lemm}

Our first result shows that, under the outflow boundary condition \eqref{outflow boundary condition}, every small BV solution of the Euler system is a vanishing viscosity solution.

\begin{theo}\label{main thm vanishing viscosity sol}
Assume that the viscosity function satisfies Assumption \ref{property viscosity}. For any $U_* = (\rho_*, \rho_* u_*)$ in $\Omega_{sub}^-$ as defined in \eqref{def subsonic}, there exists $\e>0$ such that the following is true. 

Let $U^0 = (\rho^0, \rho^0 u^0) \in E^0$, where $E^0$ is defined in \eqref{def Euler initial value}, and $u_B<0$. Assume that
\begin{align*}
    \|U^0\|_{BV(\mathbb{R}_+)} + |u_* - u_B|\leq \e \, .
\end{align*}
Let $U_0^\nu$ be an adapted family of initial values associated with $U^0$. We denote by $U^\nu$ the solution of the Navier-Stokes system \eqref{system rho, rho u} with initial value $U_0^\nu$ and outflow boundary condition \eqref{outflow boundary condition}. 

Then $U^\nu$ converges in $L^1_{\mathrm{loc}}(\mathbb{R}_+\times\mathbb{R}_+)$ to the small BV solution of the Euler system \eqref{system E rho, rho u} with initial value $U^0$ and outflow boundary condition \eqref{outflow boundary condition}.
\end{theo}

Under Assumption~\ref{property viscosity} on the viscosity, \eqref{def Euler initial value} on the Euler initial data, and the conditions on the associated adapted family of initial data for the Navier-Stokes system, the existence and uniqueness of solutions to the Navier-Stokes system \eqref{system rho, rho u} follow from \cite{nspresse}. More precisely, it is shown in \cite{nspresse} that for every $T>0$, there exist constants $\kappa_2^\nu(T)>\kappa_1^\nu(T)>0$ such that
\begin{align*}
    \rho^\nu-\rho_* \in L^\infty(0,T;H^1(\mathbb{R_+})) \, , \quad 
    u^\nu-u_* \in L^\infty(0,T;H^1(\mathbb{R_+})) \cap L^2(0,T;H^2(\mathbb{R_+})) \, .
\end{align*}
and
\begin{align*}
    \kappa_1^\nu(T) \leq \rho^\nu(x,t) \leq \kappa_2^\nu(T) \, , \; \textup{for any } (x,t)\in \mathbb{R}_+ \times [0,T] \, .
\end{align*}
This is the only point in the paper at which we impose the stronger restriction \eqref{property viscosity wide alpha} on the degeneracy of the viscosity near vacuum. Under the weaker assumption that $\mu(\rho)$ is comparable to $\rho^\al$ for all $\rho>0$, an existence and uniqueness theory for Navier-Stokes solutions on the whole space was established in \cite{existence}. However, an extension of this result to the boundary setting is not yet established.

Let $U^0 \in E^0$, and let $U_0^\nu$ be an adapted family of initial values associated with $U^0$ for the Navier-Stokes system \eqref{system rho, rho u}. We denote by $U^\nu$ the Navier-Stokes solution to \eqref{system rho, rho u} with initial value $U_0^\nu$. We call any subsequential limit of $U^\nu$ an inviscid limit associated with the Euler initial value $U^0$. We will show the existence of an inviscid limit.

\begin{lemm}\label{lemma existence of inviscid limit}
    For any $U^0 \in E^0$, there exists an inviscid limit associated with $U^0$.
\end{lemm}

Our second result shows that, under the outflow boundary condition \eqref{outflow boundary condition}, if the initial values for the inviscid limits converge to a small BV initial value for the Euler system, then the inviscid limits converge to the small BV solution of the Euler system. In the presence of the outflow boundary condition \eqref{outflow boundary condition}, we show that small BV solutions are stable under perturbations of the initial values among vanishing viscosity solutions.
\begin{theo}\label{main thm inviscid limit}
Assume that the viscosity function satisfies Assumption \ref{property viscosity}. For any $U_* = (\rho_*, \rho_* u_*)$ in $\Omega_{sub}^-$ as defined in \eqref{def subsonic}, there exists $\e>0$ such that the following is true. 

Let $U^0 = (\rho^0, \rho^0 u^0) \in E^0$, where $E^0$ is defined in \eqref{def Euler initial value}, and $u_B<0$. Assume that
\begin{align*}
    \|U^0\|_{BV(\mathbb{R}_+)} + |u_* - u_B| \leq \e \, .
\end{align*}
Let $U_n^0 = (\rho_n^0, \rho_n^0 u_n^0) \in E^0$ be a sequence of initial values such that
\begin{align*}
    \lim_{n\rightarrow \infty} \Bigl[\int_{\mathbb{R}_+}|U_n^0(x)-U^0(x)| \dx + \int_{\mathbb{R}_+} \rho_n^0 |u_n^0 - u^0|^2 \dx + \int_{\mathbb{R}_+} |\rho_n^0 - \rho^0|^\ga \dx \Bigr] = 0 \, .
\end{align*}
We denote by $U_n$ the inviscid limit of the Navier-Stokes system \eqref{system rho, rho u} associated with the initial value $U_n^0$.

Then $U_n$ converges in $L^1_{\mathrm{loc}}(\mathbb{R}_+\times\mathbb{R}_+)$ to the small BV solution of the Euler system \eqref{system E rho, rho u} with initial value $U^0$ and outflow boundary condition \eqref{outflow boundary condition}.
\end{theo}

\section{Stability of small BV solutions}\label{sec history stability of small bv}
We begin with the history in the whole space. In 1965, Glimm \cite{Glimm} established the global existence of solutions to one-dimensional $n\times n$ systems of conservation laws for initial values with sufficiently small BV norm using the random choice method. Following this pioneering work, the stability theory for small BV solutions has been extensively developed within an $L^1$-based framework. A central tool is the front tracking method, which constructs approximate solutions consisting of finitely many elementary waves that preserve the small BV structure and can be compared with weak solutions.

The existence theory based on front tracking was developed in \cite{bressanbook,Dafermos2,HR}, while the $L^1$-stability of front tracking approximations was established in \cite{BLY,BCP,LY,BC}. For general $n\times n$ systems, uniqueness of small BV solutions was first proved using front tracking under additional assumptions, such as the Tame Oscillation Condition \cite{BL,BG} or the Bounded Variation Condition \cite{BLe}. These technical conditions were first removed for $2\times2$ systems in \cite{CKV}, where the $a$-contraction method yielded a stronger Weak/Strong stability principle applicable to bounded weak entropy solutions satisfying the Strong Trace Property. More recently, these conditions were also removed for general $n\times n$ systems in \cite{BGu,BDL} using front tracking method. However, these latter results do not establish a Weak/Strong stability principle analogous to that in \cite{CKV}: within the front tracking framework, the weak solutions under consideration are still required to have small total variation.

Next, we discuss the development of the vanishing viscosity method. Bianchini and Bressan \cite{vanishingnondegenerate} showed in 2005 that small BV solutions arise as vanishing viscosity limits for a general $n\times n$ system with nondegenerate artificial viscosity. Their proof relies on the small BV structure. As a result, it is difficult to extend the result to physical systems with degenerate viscosity, such as the Navier-Stokes system, where uniform BV estimates with respect to the viscosity parameter are no longer available (see \cite{LYu} for estimates with fixed viscosity).

For physical systems with degenerate viscosity, Chen and Perepelitsa \cite{compensated} proved in 2010 the existence of vanishing viscosity solutions to the one-dimensional Euler system with $\mu=1$, using the compensated compactness method \cite{Tartar}. The convergence was established for initial data with merely bounded energy. However, for small BV initial data, this result does not establish whether the limit is itself a small BV solution or whether it is unique.

For small BV solutions in physical systems with degenerate viscosity, no such result was available until the work of Chen, Kang, and Vasseur \cite{NStoE} in 2024, using the $a$-contraction with shifts method.

\subsection{$a$-contraction with shifts}\label{subsec a-contraction with shifts}
The $a$-contraction with shifts method is a recently developed $L^2$-based approach to stability, aimed at studying a broader class of weak solutions. This method builds upon the relative entropy framework introduced by Dafermos and DiPerna \cite{Dafermos,DiPerna}, originally for studying the stability of Lipschitz solutions. In particular, the method considers the following relative entropy
\begin{align}\label{def of relative entropy}
    \eta(U|\bar{U}) = \eta(U) - \eta(\bar{U}) - \nabla \eta(\bar{U}) (U-\bar{U}) \, .
\end{align}

The fundamental obstacle to extending the relative entropy framework to solutions with discontinuities is that shocks are not necessarily contractive in $L^2$. This phenomenon already appears for Burgers' equation: an initial perturbation of $L^2$-size $\delta$ may change the shock speed by an amount of order $\delta$, producing, through the displacement of the shock, an $L^2$-perturbation of order $\sqrt{\delta t}$ at time $t$.

In the scalar case, contraction can be recovered by introducing a shift whose speed depends on the perturbation \cite{Leger}. For systems, however, an asymmetry in the strength of compression on the two sides of a shock may lead to a failure of contraction \cite{SV}. To compensate for this asymmetry and restore contraction, a weight $a$ was introduced in \cite{weightaintroduction}. The resulting weighted relative entropy functional, together with dynamically chosen shifts, forms the framework of $a$-contraction with shifts. This method is not restricted to shocks, but also applies to other elementary waves, including rarefactions and contact discontinuities, and therefore provides a powerful tool for studying small BV stability. The following are the recent developments.

At the inviscid level, Chen, Krupa, and Vasseur \cite{CKV} established the stability of small BV solutions to general $2\times2$ hyperbolic systems within the class of $L^\infty$-bounded weak entropy solutions satisfying the Strong Trace Property. This qualitative result was later strengthened by Chen, Faile, and Krupa \cite{CKF} to the quantitative H\"older estimate
\begin{align}
    \|W(\cdot,t)-U(\cdot, t)\|_{L^2} \leq C\|W(\cdot, 0)-U(\cdot, 0)\|_{L^2}^{1/2} \, ,
\end{align}
where $U$ denotes the small BV solution and $W$ the weak solution. The exponent $1/2$ is sharp, as suggested by the Burgers shock example earlier.

For the vanishing viscosity method, Chen, Kang, and Vasseur \cite{NStoE} proved in 2024 the stability of small BV solutions to the Euler system \eqref{system E rho, rho u} within the inviscid limits of the Navier-Stokes system \eqref{system rho, rho u} with degenerate viscosity, in a setting where no $\nu$-uniform BV estimates are available. Using the method of $a$-contraction with shifts, they work with Navier-Stokes solutions for which no $\nu$-uniform $L^\infty$ bound is known to exist. The key difficulty is to establish stability uniformly in $\nu$. This requires proving contraction for arbitrarily large perturbations, which is already challenging even for viscous shock layers \cite{MR4195742,invent,twoshocks}.

The core idea in \cite{NStoE} is to construct approximate solutions $\bu_\nu$ that serve as a bridge between the limiting small BV solution and the Navier-Stokes solutions $U^\nu$. Starting from step functions approximating the small BV initial data, a front tracking algorithm is used to construct globally defined approximate solutions at the Euler level that retain the small BV structure and consist of only finitely many elementary waves at each time. Although the algorithm is carried out at the Euler level, the approximate solutions $\bu_\nu$ are also designed to be comparable to the Navier-Stokes solutions $U^\nu$. The elementary waves are replaced by their viscous counterparts to incorporate the viscous structure. Moreover, perturbation dependent shifts are imposed on the shock fronts to accommodate the $L^2$-stability framework.

The construction of approximate solutions by front tracking originated in the existence and $L^1$-stability theory for inviscid systems. In their study of viscous systems, however, Bianchini and Bressan \cite{vanishingnondegenerate} adopted a different approach that does not rely on front tracking approximations. In \cite{NStoE}, Euler-level front tracking approximations of small BV solutions are, for the first time, adapted to incorporate viscous structure, thereby allowing them to be compared with Navier-Stokes solutions.

\subsection{With boundaries}\label{subsec with boundaries}
In the subsonic region, the Euler system admits three types of boundary configurations, depending on the motion of the fluid through the boundary: inflow, outflow, and impermeable wall, corresponding respectively to fluid entering the domain, leaving the domain, or not crossing the boundary.

With boundaries, there is an additional challenge: one must show that the limit of the approximate solutions still satisfies the prescribed boundary condition. For an impermeable wall, there is a shortcut. A solution satisfying the impermeable wall condition admits a special extension to a solution on the whole space. Therefore, the stability of small BV solutions satisfying the impermeable wall condition can be reduced to the corresponding whole space stability problem, both within the class of bounded weak entropy Euler solutions satisfying the Strong Trace Property, as established in \cite{CKV,CKF}, and within the class of inviscid limits associated with the Navier-Stokes system, as established in \cite{NStoE}. This reduction was carried out in \cite{impermeable}.

For the subsonic inflow and outflow cases, no such extension is available. Instead, the stability of small BV Euler solutions within the class of bounded weak entropy solutions satisfying the Strong Trace Property was established in \cite{prepweakbv} by deriving a spatially uniform bound on the total variation in time of the front tracking approximations. In particular, the notion of elements was developed to identify and track how constituent wave elements combine and separate through successive interactions. In the same spirit, although we work with a different algorithm in this paper, we show that the limit of the approximate solutions satisfies the prescribed boundary condition.

Moreover, the approximate solutions constructed in this paper are adapted to the viscous structure so that they can be compared with Navier-Stokes solutions. We substantially modify the front tracking method of Chen, Kang, and Vasseur \cite{NStoE} to remove the nondegeneracy assumption imposed there on the viscosity coefficient $\mu$ at low density. This improvement also applies to the whole space setting considered in \cite{NStoE}. In particular, the inviscid limit results of \cite{NStoE} remain valid when $\mu(\rho)$ is comparable to $\rho^\al$. 

To place the viscosity on the variable associated with the nonlinear flux, we first formulate the Navier-Stokes system in Lagrangian coordinates using the Bresch-Desjardins entropy \cite{BDentropy}. When constructing the algorithm in Lagrangian coordinates, the presence of the boundary creates an additional difficulty. In these coordinates, the boundary moves at a speed determined by the Navier-Stokes solution $U^\nu$. In the absence of a known $\nu$-uniform $L^\infty$ bound on $U^\nu$, the boundary may move pathologically fast, possibly even at supersonic speeds. By contrast, the front tracking approximation $\bu_\nu$ retains the small BV structure, so its constituent waves propagate at subsonic speeds. The boundary may therefore overtake nearby waves, leading to repeated violations of the boundary condition a large number of times. Whenever such a violation occurs, we reset the front tracking approximate solution to restore the boundary condition. Each resetting may produce an increase. Without uniform control on the total number of resettings, it is difficult to bound the total accumulated increase. To overcome this difficulty, we develop a special resetting procedure such that the increase at each resetting is compensated by the boundary contribution accumulated between successive resetting times.

\subsubsection{Other applications}
Beyond the study of small BV stability, the $a$-contraction method has also advanced the study of time-asymptotic stability. We mention several recent results. The time-asymptotic stability of general Riemann solutions for the full Euler system and the associated Navier-Stokes-Fourier system was established in \cite{KVWtime,SVintermediate}. A three-dimensional stability result for contact discontinuities in the same setting was obtained in \cite{KVW}. The time-asymptotic stability of compositions of extremal planar waves for a general $n\times n$ multidimensional viscous system was established in \cite{meng}. 

For boundary value problems, time-asymptotic stability of elementary waves has also been studied for the Navier-Stokes system. In particular, the time-asymptotic behavior of viscous shocks for the Navier-Stokes system in all subsonic boundary configurations was established in \cite{asymptoticNSinflow,asymptoticNSoutflow}, while the time-asymptotic behavior of a superposition of a boundary layer, a rarefaction wave, and a shock for the inflow problem with subsonic boundary values was established in \cite{asymptoticNSboundarylayer}.

\section{Uniform bounds}
In this section, we first introduce Lagrangian coordinates and then discuss the uniform bounds on the entropy and dissipation obtained in both the Eulerian and Lagrangian frameworks. Finally, we prove Lemma~\ref{lemma existence of initial values NS} and Lemma~\ref{lemma existence of inviscid limit}.

Following \cite{MR4195742, invent, twoshocks, NStoE}, we first perform the relative entropy method on the Navier-Stokes system \eqref{system rho, rho u} written in Lagrangian coordinates, using the Bresch-Desjardins entropy \cite{BDentropy}, so that the viscosity acts on the variable with nonlinear flux. More precisely, We first introduce the mass Lagrangian coordinate 
\begin{align}\label{def of y^nu}
    y^\nu(x,t) = \int_{0}^x \rhn(z,t) \, dz - \int_0^t \rhn (0,s) \un(0,s) \, ds \, .
\end{align}
Using $(y^\nu,t)$ as the corresponding Lagrangian space-time coordinates, we can rewrite the Navier-Stokes system \eqref{system rho, rho u} in terms of the specific volume $\vn = 1/\rhn$ and the velocity $\un$. In these coordinates, the Navier-Stokes system takes the following form
\begin{equation}\label{system v,u}
    \begin{cases}
        \vn_t - \un_{y^\nu} = 0 \, , \\
        \un_t + p_L(\vn)_{y^\nu} = \nu \bigl( \ga \frac{\mu_L(\vn)}{\vn}\un_{y^\nu} \bigr)_{y^\nu} \, ,
    \end{cases}
\end{equation}
where $p_L(v) = v^{-\ga}$ for some $\ga>1$, and $\mu_L(v) = \frac{\mu(\frac{1}{v})}{\ga}$. The corresponding initial value satisfies
\begin{align*}
    \lim_{x\rightarrow \infty} (\vn_0,\un_0) = \lim_{x\rightarrow \infty} (\frac{1}{\rho^\nu_0},\un_0) = (\frac{1}{\rho_*}, u_*) =: (v_*, u_*)\, ,
\end{align*}
and therefore has infinite total mass. Moreover, it contains no vacuum. Since the absence of vacuum is preserved by the Navier-Stokes system, the formulations \eqref{system v,u} and \eqref{system rho, rho u} are equivalent. 

The boundary in Lagrangian coordinates is given by
\begin{align}\label{def of y_b^nu}
    y_b^\nu(t) = - \int_0^t \rhn (0,s) \un(0,s) \, ds \, ,
\end{align}
with speed
\begin{align}\label{property of y_b^nu'}
    (y_b^\nu)'(t) = -\rhn(0,t) u_B = -\frac{u_B}{\vn(y_b^\nu(t),t)} \geq 0\, .
\end{align}
The outflow boundary condition \eqref{outflow boundary condition} becomes
\begin{align}\label{outflow boundary condition L}
    u^\nu(y_b^\nu(t),t) = u_B < 0 \, .
\end{align}
The subsonic region in Lagrangian coordinates is
\begin{align}\label{def subsonic L}
    (\om_{sub}^-)^L = \{(v,u):\la_1^L(v)<-\frac{u}{v}<\la_2^L(v)\} \, \cap \, \{u<0\} \, ,
\end{align}
where the eigenvalues in Lagrangian coordinates are
\begin{align*}
    \la_1^L\bigl((v,h)\bigr) = - \ga^{1/2} v^{-\frac{\ga+1}{2}} \, , \quad \la_2^L\bigl((v,h)\bigr) = \ga^{1/2} v^{-\frac{\ga+1}{2}} \, .
\end{align*}

We consider the effective velocity $\hn = \un + \nu \mu_L(\vn) (\vn)^{\ga} p(\vn)_{y^\nu}$,
which is associated with the Bresch-Desjardins entropy \cite{BDentropy}. We rewrite \eqref{system v,u} in the following BD form
\begin{equation*}
    \begin{cases}
        \vn_t - \hn_{y^\nu}  = -\nu \bigl( \mu_L(\vn) (\vn)^{\ga} p(\vn)_{y^\nu}  \bigr)_{y^\nu}  \, , \\
        \hn_t + p(\vn)_{y^\nu}  = 0 \, .
    \end{cases}
\end{equation*}
The associated entropy is $\eta^L((v,h))=Q(v) + \frac{h^2}{2}$, where $Q(v) = \frac{v^{-\ga+1}}{\ga-1}$. We see that the viscosity now acts on $v$ rather than on $u$, which plays an important role in Section~\ref{section shock}, where we estimate the terms associated with shocks.

We distinguish between the Eulerian and Lagrangian frameworks by attaching the superscript $L$ to quantities expressed in Lagrangian coordinates. However, from Section~\ref{sec mft} through Section~\ref{sec global in time estimates}, we work exclusively with the Navier-Stokes system in Lagrangian coordinates and use the BD entropy. For convenience, throughout these sections we omit the superscript $L$ and write $x$ in place of $y^\nu$. More precisely, we write
\begin{equation}\label{NS system BD}
    \begin{cases}
        \vn_t - \hn_{x}  = -\nu \bigl( \mu(\vn) (\vn)^{\ga} p(\vn)_x  \bigr)_x  \, , \\
        \hn_t + p(\vn)_x  = 0 \, ,
    \end{cases}
\end{equation}
where
\begin{align}\label{def of h}
    \hn = \un + \nu \mu(\vn) (\vn)^{\ga} p(\vn)_x \, .
\end{align}
The relative entropy is
\begin{align*}
    \eta(U^\nu|\bar{U}_\nu) = \eta\bigl((\vn,\hn)|(\bv_\nu,\bh_\nu)\bigr) = Q(\vn|\bv_\nu) + \frac{|\hn-\bh_\nu|^2}{2} \, .
\end{align*}
Corresponding to Assumption~\ref{property viscosity}, the viscosity coefficient $\mu$ satisfies the following assumption.
\begin{assum}\label{property viscosity L}
    Let $\ga>1$. We assume that $\mu$ can be decomposed into two differentiable functions, $\mu=\mu_1+\mu_2$, such that for some constants $C_1, C_2>0$, the following holds. We have
    \begin{align*}
        C_1 (\vn)^{-\al} \leq \mu_1(\vn) \leq C_2 (\vn)^{-\al} \, , \quad \textup{for all } \vn \, , 
    \end{align*}
    for some 
    \begin{align*}
        \al = \ga-1 \, , \quad \textup{if } \ga\leq 5/3 \, , \quad
        \ga-1\leq \al\leq \ga \, , \quad \textup{otherwise.}
    \end{align*}
    If $\al < 1/2$, then $\mu_2 = 0$. Otherwise, for some $0<\widetilde{\al}<1/2$,
    \begin{align*}
        0 \leq \mu_2(\vn) \leq C_2 (\vn)^{-\widetilde{\al}} \, , \quad \textup{for } \vn > 2 v_* \, , \quad \mu_2(\vn) = 0 \, , \quad \textup{for } \vn \leq 2 v_* \, ,
    \end{align*}
\end{assum}

\subsection{Uniform bounds}\label{sub sec uniform bounds}
We consider a viscous perturbation of a hyperbolic system of conservation laws
\begin{align*}
    \p_t U + \p_x A(U) = \p_x\bigl(B(U)\p_x U\bigr) \, .
\end{align*}
We define the relative flux
\begin{align}
    A(U|V) = A(U) - A(V) - \nabla A(V)(U-V) \, ,
\end{align}
and the flux of the relative entropy 
\begin{align}
    G(U;V) = G(U) - G(V) - \nabla\eta(V)(A(U)-A(V)) \, ,
\end{align}
where $G$ denotes the entropy flux associated with $\eta$. More precisely, for any $1 \leq i \leq 2$, $\p_i G(U) = \sum_{k=1}^2 \p_k \eta(U) \, \p_i A_k(U)$,  for any $1 \leq i \leq 2$.

By the relative entropy method, for any functions $U, V$, we have
\begin{align*}
    -\bigl(\nabla\eta(U)-\nabla\eta(V)\bigr)\p_x A(U) + \nabla^2 \eta(V) (U-V) \p_x A(V) = -\p_x G(U;V)  - \p_x \nabla\eta(V) A(U|V) \, .
\end{align*}
Recall the definition of relative entropy in \eqref{def of relative entropy}. We have
\begin{equation}\label{derivative of relative entropy}
\begin{aligned}
    \p_t \eta(U|V) &= \bigl(\nabla\eta(U) - \nabla \eta(V)\bigr)\p_t U - \nabla^2 \eta(V)(U-V) \p_t V \\
    &= -\bigl(\nabla\eta(U) - \nabla \eta(V)\bigr)\p_x A(U) + \nabla^2 \eta(V)(U-V)\p_x A(V) \\
    &\quad\, - \nabla^2 \eta(V)(U-V)\bigl( \p_t V + \p_x A(V) \bigr) + \bigl(\nabla\eta(U) - \nabla \eta(V)\bigr) \p_x\bigl(B(U)\p_x U\bigr)\\
    & = -\p_x G(U;V)  - \p_x \nabla\eta(V) A(U|V) - \nabla^2 \eta(V)(U-V)\bigl( \p_t V + \p_x A(V) \bigr) \\
    &\quad \, + \bigl(\nabla\eta(U) - \nabla \eta(V)\bigr) \p_x\bigl(B(U)\p_x U\bigr) \, .
\end{aligned} 
\end{equation}
We note that
\begin{align}\label{relative entropy three variables}
    \eta(U_0^\nu|U_*) = \eta(U_0^\nu|U^0) + \eta(U^0|U_*) + \bigl(\nabla \eta(U^0) - \nabla \eta(U_*)\bigr) (U_0^\nu-U^0) \, .
\end{align}
We first work in the Eulerian framework. We have
\begin{align*}
    \eta(U) = \frac{\rho u^2}{2} + \frac{\rho^\ga}{\ga-1} \, , \;
    \nabla \eta(U) = \begin{pmatrix}
        \frac{\ga}{\ga-1} \rho^{\ga-1} - \frac{u^2}{2} \\ u
    \end{pmatrix} \, , \;
    \eta(U|\udU) = \frac{\rho(u-\udu)^2}{2} + \frac{p(\rho|\udr)}{\ga-1} \, ,
\end{align*}
where $p(\rho|\udr) = \rho^\ga - \bar{\rho}^\ga - \ga \udr^{\ga-1}(\rho - \udr)$. We also have that
\begin{align*}
    A(U|\udU) = \begin{pmatrix}
        0 \\ \rho(u-\udu)^2 + p(\rho|\udr)
    \end{pmatrix} , \;
    G(U;\udU) = \frac{\rho u (u-\udu)^2}{2} + \frac{p(\rho|\udr)u}{\ga-1} + \bigl(p(\rho) - p(\udr)\bigr)(u-\udu) \, .
\end{align*}

By \eqref{def Euler initial value} and \eqref{NS initial value}, there exists a constant $C>0$, independent of $\nu$, such that
\begin{align}\label{relative entropy initial bounded}
    \int_0^\infty \eta(U_0^\nu|U_*) \, dx \leq C \, .
\end{align}

\begin{lemm}\label{lemma uniform bounds eulerian}
Let $U_*=(\rho_*, \rho_* u_*)\in \Omega_{sub}^-$ as defined in \eqref{def subsonic} and $u_B<0$. Assume that 
\begin{align*}
    \int_0^\infty \eta(U_0^\nu|U_*) \, dx \leq C \, , 
\end{align*}
for some constant $C>0$ independent of $\nu$. Then for any $T>0$, there exists a constant $C_T>0$, depending only on $T$, such that
\begin{align*}
    \sup_{0\leq t\leq T} \Bigl[\int_0^\infty \eta(U^\nu|U_*) \dx \Bigr] + 
    \nu\int_0^T\int_0^\infty \mu(\rne)|\une_x|^2 \dx
    \leq C_T \, .
\end{align*}
\end{lemm}
\begin{proof}
We work in the Eulerian framework. We choose a smooth monotone function $\udU = (\udr , \udr \udu) : (0,\infty) \rightarrow \mathbb{R}^2$ such that 
\begin{align}\label{uniform bounds lemma test function}
    \lim_{x\rightarrow \infty} \udU(x) = U_* \, , \quad \udu(0) = u_B \, , \quad \textup{supp}(\udU) \subset (1,2) \, .
\end{align}
By \eqref{derivative of relative entropy}, we have
\begin{align*}
    \frac{d}{dt} \int_0^\infty \eta(U|\udU) \dx = A_1 + I_2 + I_3 + I_4 + A_5 \, ,
\end{align*}
where
\begin{align*}
    A_1 &= - G(U;\udU)\bigl|_{x=0}^\infty \, , \\
    I_2 &= - \int_0^\infty \nabla^2\eta(\udU) \p_x \udU A(U|\udU) \dx \, , \\
    I_3 &= - \int_0^\infty \nabla^2\eta(\udU) (U-\udU) \nabla A(\udU) \p_x \udU \dx \, , \\
    I_4 &= -\int_0^\infty \bigl(\nabla \eta(U) - \nabla \eta(\udU) \bigr)_x \begin{pmatrix}
        0 \\ \nu \mu(\rho) u_x
    \end{pmatrix}  \dx \, , \\
    A_5 &= \bigl(\nabla \eta(U) - \nabla \eta(\udU) \bigr) \begin{pmatrix}
        0 \\ \nu \mu(\rho) u_x 
    \end{pmatrix} \bigl|_{x=0}^\infty \, .
\end{align*}
As $u|_{x=0} = u_B < 0$, we get
\begin{align*}
    A_1 = G(U;\udU) \bigl|_{x=0} = \frac{p(\rho|\udr) u}{\ga-1} \bigl|_{x=0} \leq 0 \, , \quad A_5 = 0 \, .
\end{align*}
Since $|A(U|\udU)|\leq C \eta(U|\udU)$, we have
\begin{align*}
    |I_2| \leq C \int_0^\infty \eta(U|\udU) \dx \, .
\end{align*}
By \eqref{uniform bounds lemma test function}, we get
\begin{align*}
    |I_3| \leq \int_1^2  |U-\udU| \dx \, .
\end{align*}
Note that $\ga>1$. Thus, we have
\begin{equation}\label{uniform lemma L^1 norm U - udU}
    \begin{aligned}
        \int_1^2  |U-\udU| \dx &\leq \int_1^2  |\rho-\udr| \dx + \int_1^2  |\rho (u-\udu)| \dx + \int_1^2  |( \rho - \udr )\udu)| \dx \\
        &\leq C \int_1^2 \bigl(1+p(\rho|\udr)\bigr) \dx + \int_1^2 \rho \dx + \int_1^2 \rho |u-\udu|^2 \dx \\
        &\leq C + C \int_1^2 \eta(U|\udU) \dx \, ,
    \end{aligned}
\end{equation}
which implies
\begin{align*}
    |I_3| \leq C + C \int_1^2 \eta(U|\udU) \dx \, .
\end{align*}
By Young's inequality and $\ga-1 \leq \al \leq \ga$, we get
\begin{align*}
    I_4 &= -\nu \int_0^\infty (u-\udu)_x \mu(\rho) u_x \dx \\
    &= - \nu \int_0^\infty \mu(\rho)|u_x|^2 \dx + \nu \int_1^2 \mu(\rho) \udu_x u_x \dx \\
    &\leq - \frac{\nu}{2} \int_0^\infty \mu(\rho)|u_x|^2 \dx + C \nu \int_1^2 \mu(\rho) \dx \\
    &\leq - \frac{\nu}{2} \int_0^\infty \mu(\rho)|u_x|^2 \dx + C \nu \int_1^2 \bigl(1+p(\rho|\udr)\bigr) \dx \\
    &\leq - \frac{\nu}{2} \int_0^\infty \mu(\rho)|u_x|^2 \dx + C\nu + C\nu \int_1^2 \eta(U|\udU) \dx \, .
\end{align*}
In all, we conclude
\begin{align*}
    \frac{d}{dt} \int_0^\infty \eta(U|\udU) \dx + \frac{\nu}{2} \int_0^\infty \mu(\rho)|u_x|^2 \dx \leq C \int_0^\infty \eta(U|\udU) \dx + C \, .
\end{align*}
By Gronwall's inequality, it follows that 
\begin{align*}
    \sup_{0\leq t\leq T} \Bigl[\int_0^\infty \eta(U|\udU) \dx \Bigr] \leq C_T \, ,
\end{align*}
and therefore,
\begin{align*}
    \nu\int_0^T\int_0^\infty \mu(\rne)|\une_x|^2 \dx
    \leq C_T \, .
\end{align*}
By \eqref{relative entropy three variables}, we have
\begin{align}
    \Bigl|\eta(U|U_*) - \eta(U|\udU)\Bigr| \leq C \mathds{1}_{(1,2)} + C \mathds{1}_{(1,2)} \bigl|U-\udU\bigr| \, .
\end{align}
By \eqref{uniform lemma L^1 norm U - udU}, we get
\begin{align*}
    \int_0^\infty \eta(U|U_*) \dx \leq C + C \int_0^\infty\eta(U|\udU) \dx \, .
\end{align*}
Hence, we show 
\begin{align*}
    \sup_{0\leq t\leq T} \Bigl[\int_0^\infty \eta(U|U_*) \dx \Bigr] \leq C_T \, .
\end{align*}
\end{proof}

We next work in the Lagrangian framework using the BD entropy. In Lagrangian coordinates, we have
\begin{equation}\label{useful computation}
\begin{aligned}
        &\nabla\eta^L(U) = 
    \begin{pmatrix}
        -p(v) \\ h
    \end{pmatrix} \, , \;
    \nabla^2\eta^L(U) = 
    \begin{pmatrix}
        -p'(v) & 0 \\ 0 & 1 
    \end{pmatrix} \, , \;
    \nabla A(U) = 
    \begin{pmatrix}
        0 & -1 \\ p'(v) & 0 
    \end{pmatrix} \, , \\
    &A(U|\bu) = 
    \begin{pmatrix}
        0 \\ p(v|\bv)
    \end{pmatrix} \, , \;
    G(U;\bu) = \bigl(p(v)-p(\bv)\bigr)(h-\bh) \, .
\end{aligned}
\end{equation}
By \eqref{relative entropy initial bounded} and \eqref{NS initial value}, there exists a constant $C>0$, independent of $\nu$, such that
\begin{align*}
    \int_{y_b^\nu}^\infty \eta^L(U_0^\nu|U_*) \, d y^\nu &= \int_0^\infty \frac{p(\rhn_0|\rho_*)}{\ga-1} + \frac{\bigl|\sqrt{\rhn_0}(\un_0-u_*)+ \nu \ga \p_x \phi(\rhn_0)\bigr|^2}{2} \dx \\
    &\leq C \int_{\mathbb{R}_+} \eta(U_0^\nu|U_*) \dx + C \nu^2 \int_{\mathbb{R}_+} |\p_x \phi(\rhn_0)|^2 \dx \\
    &\leq C \, .
\end{align*}

\begin{lemm}\label{lemma uniform bounds lagrangian}
Let $U_*^L = (v_*, u_*) \in (\Omega_{sub}^-)^L$ defined in \eqref{def subsonic L} and $u_B<0$. There exists $\e_0>0$ such that the following is true.

\noindent
Assume that $|u_*-u_B|\leq \e_0$, and
\begin{align}\label{lemma uniform bound lagrangian assumption}
    \int_{y_b^\nu}^\infty \eta^L(U_0^\nu|U_*) \, d y^\nu \leq C \, , 
\end{align}
for some constant $C>0$ independent of $\nu$. Then for any $T>0$, there exists a constant $C_T>0$, depending only on $T$, such that
\begin{align*}
    \sup_{0\leq t\leq T} \Bigl[\int_{y_b^\nu}^\infty \eta^L(U^\nu|U_*^L) \, d y^\nu \Bigr] + \nu \int_0^T \int_{y_b^\nu}^\infty \mu(\vn) (\vn)^\ga |p(\vn)_x|^2 \, d y^\nu \, dt \\
    + \int_0^T (y_b^\nu)'(t) \, dt \leq C_T \, .
\end{align*}
\end{lemm}
\begin{proof}
By \eqref{derivative of relative entropy} and \eqref{useful computation}, we have
\begin{align*}
    \frac{d}{dt} \int_{y_b^\nu}^\infty \eta^{L}(U^\nu|U_*^L) \, dy^\nu = A_1 + A_2 + I_3 + A_4 \, ,
\end{align*}
where
\begin{align*}
    A_1 &= - (y_b^\nu)'(t) \eta^L(U^\nu|U_*^L) \bigl|_{y=y_b^\nu} \, , \\
    A_2 &= \bigl( p(\vn) - p(v_*) \bigr) (h^\nu-u_*) |_{y=y_b^\nu} \, , \\
    I_3 &= -\nu \int_{y_b^\nu}^\infty \mu(\vn) (\vn)^\ga |p(\vn)_x|^2 \, d y^\nu \, , \\
    A_4 &= - \nu \bigl( p(\vn) - p(v_*) \bigr) \mu(\vn) (\vn)^\ga p(\vn)_x |_{y=y_b^\nu} \, .
\end{align*}
By \eqref{def of h} and \eqref{outflow boundary condition L}, we get
\begin{align*}
    A_2 + A_4 \leq \bigl( p(\vn) - p(v_*) \bigr) (u_B - u_*) \leq C \e_0 + C \e_0 \vn(y_b^\nu, t)^{-\ga} \mathds{1}_{\{\vn(y_b^\nu, t) \leq \frac{v_*}{2}\}} \, .
\end{align*}
By \eqref{property of y_b^nu'} and \eqref{outflow boundary condition L}, it follows that
\begin{align*}
    A_1 \leq 0 \, , \quad |A_1| \geq \frac{-u_B}{\vn(y_b^\nu, t)} Q\bigl(\vn(y_b^\nu, t) | v_*\bigr) 
    \geq C \vn(y_b^\nu, t)^{-\ga} \mathds{1}_{\{\vn(y_b^\nu, t) \leq \frac{v_*}{2}\}} \, .
\end{align*}
In all, we obtain
\begin{align*}
    A_1+A_2+A_4 + C \vn(y_b^\nu, t)^{-\ga} \mathds{1}_{\{\vn(y_b^\nu, t) \leq \frac{v_*}{2}\}} \leq C \, .
\end{align*}
Together with \eqref{property of y_b^nu'} and $\ga>1$, we get
\begin{align*}
    \frac{d}{dt} \Bigl(\int_{y_b^\nu}^\infty \eta^L(U^\nu|U_*^L) \, d y^\nu + \nu \int_{y_b^\nu}^\infty \mu(\vn) (\vn)^\ga |p(\vn)_x|^2 \, d y^\nu \Bigr) + (y_b^\nu)'(t) \leq C \, .
\end{align*}
By \eqref{lemma uniform bound lagrangian assumption}, we get 
\begin{align*}
    \sup_{0\leq t\leq T} \Bigl[\int_{y_b^\nu}^\infty \eta^L(U^\nu|U_*^L) \, dy^\nu \Bigr] + \nu \int_0^T \int_{y_b^\nu}^\infty \mu(\vn) (\vn)^\ga |p(\vn)_x|^2 \, dy^\nu \, dt \\
    + \int_0^T (y_b^\nu)'(t) \, dt \leq C_T \, .
\end{align*}
\end{proof}

We obtain the following remark. We use \eqref{lemma uniform bounds D v} and \eqref{remark lemma uniform bound on integral of Q} to control the terms localized by non-physical waves and the boundary terms in Section~\ref{sec NP and boundary}.
\begin{rema}\label{remark uniform bound on dissipation lagrangian}
Lemma \ref{lemma uniform bounds lagrangian} implies
\begin{align}\label{lemma uniform bounds D v}
    \nu\int_0^T \md_v(U) \, dt \leq C_T \, ,
\end{align}
where
\begin{equation}\label{def of D_v}
    \mathcal{D}_{v}(U) = \|\bigl(v^{-\frac{\ga+\al}{2}}\bigr)_x\|_{L^2}^2 \, .
\end{equation}
Since
\begin{align*}
    v^{-\ga+1} \mathds{1}_{v\leq \frac{v_*}{2}}\leq C Q(v|v_*) \, .
\end{align*}
Lemma \ref{lemma uniform bounds lagrangian} implies
\begin{align}\label{remark lemma uniform bound on integral of Q}
    \sup_{0\leq t\leq T} \inb v^{-\ga+1} \mathds{1}_{v \leq \frac{v_*}{2}}\, dy \leq C_T \, .
\end{align}

\end{rema}

\subsection{Proof of Lemma \ref{lemma existence of initial values NS} and Lemma \ref{lemma existence of inviscid limit}}\label{subsec proof of lemmas existence}
We first show Lemma~\ref{lemma existence of initial values NS}, the existence of an adapted family of initial values for the Navier-Stokes system. We extend $U^0 = (\rho^0, \rho^0 u^0)\in E$ to $\mathbb{R}$ by setting
\begin{align*}
    \widetilde{U}^0(x) = \begin{cases}
        U^0(x) \, , \quad \textup{if } x\in \mathbb{R_+} \, , \\
        U^0(-x) \, , \quad \textup{if } x\in \mathbb{R_-} \, .
    \end{cases}
\end{align*}
Let $\psi: \mathbb{R} \rightarrow\mathbb{R}$ be a smooth function such that
\begin{align*}
    0\leq \psi\leq 1 \, , \quad \int_{\mathbb{R}} \psi \, dx = 1 \, , \quad \textup{supp}(\psi) \subset [-1,1] \, .
\end{align*}
Then we mollify $\widetilde{U}^0=(\widetilde{\rho}^0,\widetilde{\rho}^0 \widetilde{u}^0)$ as in Section 5.1 of \cite{NStoE} by defining
\begin{align*}
    \widetilde{\rho}_0^\nu = \widetilde{\rho}^0 * \psi_\nu \, , \quad \widetilde{u}_0^\nu = \widetilde{u}^0 * \psi_\nu \, ,
\end{align*}
where $\psi_\nu(x) = \frac{1}{\sqrt{\nu}}\psi(\frac{x}{\sqrt{\nu}})$. We then restrict these functions to the half line $\mathbb{R}_+$ and further modify the velocity near the boundary so that it satisfies the outflow boundary condition \eqref{outflow boundary condition}. More precisely, we let $\phi: \mathbb{R}_+ \rightarrow\mathbb{R}$ be a smooth function such that
\begin{align*}
    \phi = \begin{cases}
        1 \, , \; \textup{on } [0,1] \, , \\
        0 \, , \; \textup{on } [2,\infty) \, ,
    \end{cases} \quad
    0\leq \phi \leq 1 \, .
\end{align*}
We let
\begin{align*}
    u_0^\nu(x) = \phi_\nu u_B + (1-\phi_\nu) \widetilde{u}_0^\nu \, ,
\end{align*}
where $\phi_\nu = \phi(\frac{x}{\nu})$. It is easy to see that $u_0^\nu(0) = u_B$ and as $\nu \rightarrow 0$,
\begin{align*}
    \|u_0^\nu - \widetilde{u}_0^\nu\|_{L^1(\mathbb{R}_+)} + \|u_0^\nu - \widetilde{u}_0^\nu\|_{L^2(\mathbb{R}_+)} \rightarrow 0 \, .
\end{align*}
Hence, together with the proof in \cite{NStoE}, $(\widetilde{\rho}_0^\nu, u_0^\nu)$ is an adapted family of initial values for the Navier-Stokes associated with $U^0$.

We prove Lemma~\ref{lemma existence of inviscid limit} by following the argument in \cite{NStoE}. Let $U^0 \in E^0$ and $U^\nu$ be Navier-Stokes solution with initial values being an adapted family associated with $U^0$. Let $T>0$. Lemma~\ref{lemma uniform bounds eulerian} implies that there exists a constant $C>0$ such that
\begin{align*}
    \lim_{\nu\rightarrow 0} \sup_{0\leq t\leq T} \int_{\mathbb{R}_+} \eta(U^\nu | U_*) \dx \leq C \, .
\end{align*}
By the strict convexity of $\eta$, the family $U^\nu$ is uniformly integrable and therefore weakly precompact in $L^1_{\mathrm{loc}}(\mathbb{R}_+\times\mathbb{R}_+)$. Hence, there exists a subsequence of $U^\nu$ that converges weakly in $L^1_{\mathrm{loc}}(\mathbb{R}_+\times\mathbb{R}_+)$. We call any such limit $U$ an inviscid limit of the Navier-Stokes system associated with the Euler initial value $U^0$.

\section{Main idea of the proof}\label{sec main idea of the proof}
In this section, we outline the main ideas of the proof. In Section~\ref{subsec stabiltiy L}, we first present the stability result, Theorem~\ref{thm v,h}, established through the $a$-contraction method in the Lagrangian framework. We then discuss the results in the Eulerian framework in Section~\ref{subsec stability E}, including a convergence result for inviscid limits established via compensated compactness in Proposition~\ref{prop E compensated compactness}, and the Eulerian counterpart of Theorem~\ref{thm v,h}, given in Proposition~\ref{prop E approximate sol}. In Section~\ref{subsec proof of main thm}, we use these two propositions to prove Theorem~\ref{main thm vanishing viscosity sol} and Theorem~\ref{main thm inviscid limit}. Finally, Section~\ref{subsec idea of proof of stability L} is devoted to explaining the main ideas behind the proof of Theorem~\ref{thm v,h}, which is the principal task of the remainder of the paper.

\subsection{Stability for the system with BD entropy in Lagrangian coordinates}\label{subsec stabiltiy L}
We first establish the stability results for system \eqref{NS system BD}.
\begin{theo}\label{thm v,h}
Assume that the viscosity function satisfies Assumption \ref{property viscosity L}. For any $U_*^L = (v_*,u_*)$ in $(\Omega_{sub}^-)^L$ as defined in \eqref{def subsonic L}, there exist $C, \e_0>0$ such that the following is true.

Let $U^0 = (v^0, h^0) \in  L^\infty(\mathbb{R}_+)$ and $u_B<0$. Assume that 
\begin{align*}
    \int_{\mathbb{R}_+} \eta^L\bigl(U^0(x)|U_*^L\bigl) \, dy^\nu + \|U^0-U_*^L\|_{L^1(\mathbb{R}_+)} < \infty \, , \\
    \|U^0\|_{BV(\mathbb{R}_+)} + |u_* - u_B|\leq \e \, , \quad \textup{for some } 0 \leq \e < \e_0 \, ,
\end{align*}
We consider a sequence of initial values $U_0^\nu = (v_0^\nu, h_0^\nu)$ such that 
\begin{align}\label{thm lagrangian assumption initial value}
    \lim_{\nu \rightarrow 0} \int_{\mathbb{R}_+} \eta^L(U_0^\nu|U^0) \, d y^\nu = 0 \, .
\end{align}
We denote by $U^\nu = (v^\nu, h^\nu) \in L^\infty\bigl((y_b^\nu,\infty)\times \mathbb{R}_+\bigr)$ the solution of \eqref{NS system BD} with initial value $U_0^\nu = (v_0^\nu, h_0^\nu)$ and outflow boundary condition \eqref{outflow boundary condition L}.

Then for any $\nu$ small enough, there exists a function $\ub = (\bar{v}_\nu, \bar{h}_\nu)\in L^\infty\bigl((y_b^\nu,\infty)\times \mathbb{R}_+\bigr)$ such that
\begin{align}\label{thm lagrangian small bv}
    \sup_{t>0} \, \bigl(\|\ub(\cdot, t)\|_{BV((y_b^\nu,\infty))} + \|\ub(\cdot, t) - U_*^L\|_{L^\infty((y_b^\nu,\infty))} \bigr)\leq  C \e \, , 
\end{align}
and, up to a subsequence, 
\begin{align}\label{thm lagrangian reparametrization approximate}
    \bar{U}^\nu (\mathbf{x},t) := \ub(\mathbf{x}+y_b^\nu,t) \, , 
\end{align}
converges in $L^1_{\mathrm{loc}}(\mathbb{R}_+\times \mathbb{R}_+)$, and its subsequential limit satisfies the outflow boundary condition \eqref{outflow boundary condition}.

Moreover, for all $T>0$, we have
\begin{align}
    \lim_{\nu\rightarrow 0 }\sup_{t\in [0,T]} \int_{y_b^\nu}^{\infty} \eta(U^\nu(x,t)|\ub(x,t)) \, dy^\nu = 0 \, , \label{thm lagrangian uni in time} \\
    \lim_{\nu\rightarrow 0 } \int_0^T G_\nu(t) \, dt = 0 \, , \label{thm lagrangian decomposition}
\end{align}
where 
\begin{align*}
    G_\nu(t) &= 
    \nu \int_{y_b^\nu}^\infty \mu(\vn) (\vn)^{\ga} \Bigr|\bigl(p(\vn)-p(\bv_\nu)\bigr)_{y^\nu} \Bigr|^2 \, d y^\nu + |{{y}_b^\nu}'(t)| \eta(U^\nu|\ub)\big|_{x=y_b^\nu(t)} \\
    &\quad + \sum_{\be} \int_{y_b^\nu}^\infty |(V_{\be}^\nu)_{y^\nu}| p(\vn|\bv_\nu) \, d y^\nu \, .
\end{align*}
Here $(\bu_\nu)_x$ can be written as a finite sum $V^\nu_{N} + \sum_{\be} V_{\be}^\nu$ with
\begin{align*}
    \lim_{\nu \rightarrow 0} \|V^\nu_{N}\|_{L^\infty((y_b^\nu,\infty)\times \mathbb{R}_+)} = 0 \, .
\end{align*}

\end{theo}
The main part of the paper, in particular Sections~\ref{sec mft} through~\ref{sec global in time estimates}, is devoted to proving this theorem. The function $\bu_\nu$ is an approximate solution to the Navier-Stokes solution $U^\nu$, constructed by the modified front tracking method introduced in Section~\ref{sec mft}. Each approximate solution consists of finitely many physical and non-physical waves. In the theorem above, we denote the wave profile associated with a physical wave $\be$ by $V_\be^\nu$, and the sum of all non-physical waves by $V_N^\nu$.

\subsection{Stability in Eulerian coordinates}\label{subsec stability E}
We first state the subsequential convergence of the Navier-Stokes solutions to a solution of the Euler system, obtained using the compensated compactness framework. This method was first developed in \cite{compensated} and later adapted to density-dependent viscosity in Section~13 of \cite{NStoE}.

In our setting, the presence of the boundary does not affect the compensated compactness argument in \cite{NStoE} in any essential way. Indeed, the entropy condition is local, and the convergence is considered in $L^1_{\mathrm{loc}}(\mathbb{R}_+\times\mathbb{R}_+)$, so the argument can be localized away from the boundary. We also note that Assumption~\ref{property viscosity} relaxes the nondegeneracy condition on the viscosity coefficient $\mu(\rho)$ near vacuum imposed in \cite{NStoE}. Nevertheless, this assumption still provides all the viscosity bounds used in the compensated compactness argument of \cite{NStoE}. Therefore, with only minor modifications to account for the half line domain, the argument of \cite{NStoE} yields the following proposition.

\begin{prop}\label{prop E compensated compactness}
Assume that the viscosity function satisfies Assumption \ref{property viscosity}.
Let $U_* = (\rho_*,\rho_* u_*) \in \mathbb{R}_+ \times \mathbb{R}$, $U^0 = (\rho^0, \rho^0 u^0) \in  E^0$, and $U_0^\nu$ be an adapted family of initial values associated with $U^0$, satisfying \eqref{NS initial value}. We denote by $U^\nu$ the solution of \eqref{system rho, rho u} with initial value $U_0^\nu$ and outflow boundary condition \eqref{outflow boundary condition}.

\noindent
Then, up to a subsequence, 
\begin{align*}
    \rhn \rightarrow \rho \, , \quad \rhn \un \rightarrow \rho u \, , \quad \textup{in } L^1_{\mathrm{loc}}(\mathbb{R}_+\times\mathbb{R}_+) \, ,
\end{align*}
where $(\rho, \rho u)$ is a solution to the Euler system \eqref{system rho, rho u} and satisfies the entropy inequality \eqref{entropy inequality}.
\end{prop}

We now reformulate Theorem~\ref{thm v,h} in Eulerian coordinates. The proof is given in Section~\ref{sec from Lagrange to Euler}, following the same idea as in \cite{NStoE}, with the additional task of showing that the limit satisfies the outflow boundary condition.
\begin{prop}\label{prop E approximate sol}
Assume that the viscosity function satisfies Assumption \ref{property viscosity}.
For any $U_* = (\rho_*,\rho_* u_*)$ in $\Omega_{sub}^-$ as defined in \eqref{def subsonic}, there exist $C, \e_0$ such that the following is true.

Let $U^0 = (\rho^0, \rho^0 u^0) \in  E^0$, where $E^0$ is defined in \eqref{def Euler initial value}. Assume that  
\begin{align*}
    \|U^0\|_{BV(\mathbb{R}_+)} \leq \e \, , \quad \textup{for some } 0 \leq \e < \e_0 \, .
\end{align*}
Let $U_0^\nu$ be an adapted family of initial values associated with $U^0$, satisfying \eqref{NS initial value}.
We denote by $U^\nu$ the solution of \eqref{system rho, rho u} with initial value $U_0^\nu$ and outflow boundary condition \eqref{outflow boundary condition}.

\noindent
Then there exists a family of functions $\ub$ such that
\begin{align}\label{thm eularian small bv}
    \sup_{t>0} \, \bigl(\|\ub(\cdot, t)\|_{BV(\mathbb{R}_+)} + \|\ub(\cdot, t) - U_*\|_{L^\infty(\mathbb{R}_+)} \bigr)\leq  C \e \, , 
\end{align}
and any weak-* $L^\infty$ subsequential limit satisfies the outflow boundary condition \eqref{outflow boundary condition}. Moreover, $U^\nu - \ub$ converges to $0$ in the sense of distributions as $\nu \rightarrow 0$.
\end{prop}
Note that, compared with Theorem~\ref{thm v,h}, we no longer obtain strong convergence because of the lack of control on extremely small and large values of the density. Nevertheless, the approximate solution $\ub$ constructed by the front tracking algorithm is comparable to the Navier-Stokes solution $U^\nu$ in the sense of distributions, and any subsequential limit in the weak-* topology of $L^\infty$ satisfies the uniform small BV bound \eqref{thm eularian small bv} and the outflow boundary condition \eqref{outflow boundary condition}.

\subsection{Proof of Theorem \ref{main thm vanishing viscosity sol} and Theorem \ref{main thm inviscid limit}}\label{subsec proof of main thm}

The proof of Theorem~\ref{main thm vanishing viscosity sol} is straightforward. Up to a subsequence, Proposition~\ref{prop E compensated compactness} implies that the limit is an entropy solution of the Euler system, while Proposition~\ref{prop E approximate sol} shows that the limit is small BV and satisfies the outflow boundary condition \eqref{outflow boundary condition}. Together with the uniqueness of small BV solutions proved in \cite{prepweakbv}, this completes the proof.

We note that, in establishing that the limit is an entropy solution of the Euler system, one could in principle work with either the front tracking approximate solutions or the Navier-Stokes solutions. Strong convergence of the approximate solutions would be sufficient to show that the limit is an entropy solution. However, to obtain contraction within the relative entropy framework, we impose perturbation-dependent shifts on the shocks. Since we do not have uniform control with respect to $\nu$ on the Navier-Stokes solutions, these shifts lack sufficient control and may produce oscillations in time, preventing us from obtaining strong convergence of the approximate solutions. Therefore, we instead apply the compensated compactness framework directly to the Navier-Stokes solutions $U^\nu$, as in Proposition~\ref{prop E compensated compactness}.

Following the same argument as in Section 5.3 of \cite{NStoE}, together with Theorem~\ref{main thm vanishing viscosity sol}, we prove Theorem~\ref{main thm inviscid limit}. 

\subsection{Main idea of the proof of Theorem \ref{thm v,h}}\label{subsec idea of proof of stability L}
The main task of this paper is to prove Theorem~\ref{thm v,h}, and Sections~\ref{sec mft} through~\ref{sec global in time estimates} are devoted to this purpose.

\subsubsection{Front tracking algorithm}
As discussed at the end of Section \ref{subsec a-contraction with shifts}, the front tracking approximate solutions consist of finitely many elementary waves. While the construction is carried out at the Euler level, each elementary wave is replaced by its viscous counterpart to incorporate the viscous structure. More precisely, each wave is smoothed and has a transition zone. The first adaptation required by the viscous structure is therefore to prevent these transition zones from overlapping. To this end, we stop the clock as soon as the transition zones of two waves meet or the transition zone of a wave reaches the boundary, before the corresponding wave interaction or boundary interaction occurs. We then carry out the algorithm at the Euler level to generate the outgoing waves, place their viscous counterparts in such a way that their transition zones remain separated from one another and from the boundary, and restart the clock.

As in the standard front tracking algorithm, we introduce non-physical waves to keep the number of waves finite, but with a slight modification: we also convert a physical wave into a non-physical wave once its strength drops below a certain threshold. This modification is introduced because, at the viscous level, we need a lower bound on the strengths of viscous shocks to prevent their profiles from spreading too widely in space. 

Following the classical front tracking construction, we also discretize large rarefactions so that their spreading effect can be approximated by many small rarefactions. At the viscous level, we choose the transition zones of all rarefaction waves to have the same scale. We note that the front tracking algorithm is highly sensitive to spatial placement. If there are more outgoing waves than incoming waves, then placing the additional outgoing waves while preventing overlap between transition zones becomes a difficulty. During the evolution, this situation arises when two shocks of the same family interact and generate a large rarefaction in the other family. To accommodate the generation of such waves and the resulting discretized rarefactions, we introduce partition and passing events.

As discussed in Section~\ref{subsec with boundaries}, the pathological motion of the boundary, which may occur at supersonic speeds, motivates the introduction of a special resetting event called merging. This event preserves the outflow boundary condition while ensuring that the boundary contribution accumulated between successive resettings compensates for the increase produced at the resetting times.

Throughout Section~\ref{sec mft}, we develop the notion of elements to identify and track the constituent elements of the waves. This notion was first introduced in \cite{prepweakbv} to show that the limit of the front tracking approximate solutions preserves the boundary condition, and it plays the same role here. In addition, the notion of elements allows us to develop the procedures of partition, passing, and merging, and to obtain sharper estimates on the number of interactions involving waves of large strength.

\subsubsection{Energy estimates}
In Section~\ref{sec energy estimate between stopping times}, we derive the estimates between stopping times, and Sections~\ref{sec rarefaction} through~\ref{section shock} are then devoted to controlling the terms associated with rarefactions, shocks, non-physical waves, the boundary, and viscosity.

Compared with \cite{NStoE}, we are able to control the terms associated with rarefactions without using viscosity by choosing the scales of their transition zones to be compatible with their strengths. This improvement allows us to relax the assumption on the viscosity coefficient $\mu$. To control the boundary terms and the terms associated with non-physical waves, we derive a $\nu$-dependent bound on $\|p(v)\|_{L^\infty}$. Thanks to the smallness of the non-physical waves and the fact that all waves are kept away from the boundary, these terms carry sufficient smallness in $\nu$.

The terms associated with shocks are controlled in essentially the same way as in \cite{NStoE}. We truncate the specific volume of the Navier-Stokes solution near that of the front tracking approximate solution. Inside the truncation, the difference localized by each viscous shock is controlled by a Poincar\'e-type inequality with the help of the viscosity and the shift associated with the shock. Outside the truncation, we apply the fundamental theorem of calculus. The main change here concerns the Poincar\'e inequality, since we work on a domain with a boundary. However, because the transition zones of the viscous shocks are kept away from the boundary, the Poincar\'e inequality can be reduced to the corresponding whole space estimate. The analysis of the viscous terms is the same.

In Section~\ref{sec estimates at stopping time}, we evaluate the increase at stopping times associated with different types of events. Then, in Section~\ref{sec global in time estimates}, we combine the energy estimates between stopping times with those at stopping times to derive the global-in-time estimates. There, it is worth noting that we estimate the total increase caused by each type of event throughout the evolution. In particular, interactions are divided into different classes according to the strengths of the waves involved in order to obtain sharper estimates. The total increase caused by passings and partitions is controlled by the total number of rarefactions that are discretized throughout the evolution. We also discuss how the boundary contribution between mergings controls the increase at the mergings.

\subsubsection{Summary}
The front tracking algorithm developed in this paper contains several new ideas specifically designed to ensure compatibility with the viscous system, including passing, partition, and the reduction of physical waves to non-physical waves. On the other hand, closing the global-in-time energy estimates imposes additional constraints on the construction, such as the compatibility between the support scales of rarefactions and their strengths. We therefore modify the scales associated with different types of waves and carefully balance the well-definedness of the front tracking approximate solution with the analytic requirements of the energy estimates. In addition, we develop the idea of merging specifically to preserve the boundary condition in the presence of the potentially pathological behavior of the boundary caused by the absence of uniform $L^\infty$ control on the Navier-Stokes solutions.

\section{Construction of the approximate solution}\label{sec mft}
In this section, we present the construction of the front tracking approximate solution $\bu_\nu$. In Section~\ref{subsec thresholds}, we introduce the thresholds used in the algorithm, specify the relations they must satisfy for both the construction of the approximate solutions and the energy estimates, and provide a concrete choice of these thresholds. In Section~\ref{subsec event}, we follow \cite{prepweakbv} and introduce elements to identify and track the constituent waves and to describe interaction, passing, partition, and merging events. In addition, we assign a generation order to each element, which is needed to obtain quantitative estimates on the number of waves and interactions at the viscous level. We then present the algorithm in Section~\ref{subsec mft}. Finally, in Section~\ref{sub number of waves and elements}, we show that the construction is well defined and derive the required quantitative estimates through an argument based on elements rather than waves.

\subsection{Riemann solutions}
For any two sufficiently close constant states $U_L$ and $U_R$, the Riemann problem admits a self-similar entropy solution consisting of at most three
constant states, which connect $U_L$ and $U_R$ through shock or rarefaction
waves.

Moreover, let $U^R=(v^R,u^R)$ be a constant state with $u^R$ sufficiently close to a prescribed value $u^L$. Then, for each $i=1,2$, there exists $v_i^L$ close to $v^R$ such that the boundary Riemann problem with left state $U_i^L=(v_i^L,u^L)$ and right state $U^R$ admits a self-similar entropy solution consisting of a single $i$-shock or $i$-rarefaction.

More precisely, for each characteristic family $i = 1,2$, there exists a $C^2$ wave curve 
\begin{align*}
    \s_i \mapsto T_i(\s_i)(\cdot) \, ,
\end{align*}
such that
\begin{align*}
    U_R = T_2(\s_2) \circ T_1(\s_1)(U_L) \, , \quad U^R = T_i(\s_i)(U^L_i) \, .
\end{align*}
We parametrize the curve $T_i$ by the difference in the velocity component $u$, namely
\begin{align}\label{def of wave strength}
    \s_i = u_+ - u_- \, ,
\end{align}
where $U_+ = T_i(\s_i)(U_-)$.
We note that the parameter $\s_i$ is negative (respectively, positive) if the map $T_i(\s_i)$ connects two states by a shock (respectively, by a rarefaction). We define $\s_i$ to be the strength of the $i$-wave associated with $T_i(\s_i)$.

According to Lemma 7.2 in \cite{bressanbook}, we have the following estimates.

\begin{lemm}\label{lemma interaction wave strength estimates}
There exists a constant $C>0$, depending only on $U_*$, such that the following estimates hold.

\noindent
Assume two physical waves of strengths $\s$ and $\widetilde{\s}$ interact, and let $\s'$ denote the strength of the outgoing $1$-wave and $\s''$ the strength of the outgoing $2$-wave.
\begin{itemize}
    \item If $\s,\widetilde \s$ are 1-waves, then we have
    \begin{align*}
        |\s' - (\s + \widetilde \s)| + |\s''| \leq C |\s \widetilde \s| (|\s| + |\widetilde \s|) \, .
    \end{align*}
    \item If $\s,\widetilde \s$ are 2-waves, then we have
    \begin{align*}
        |\s'' - (\s + \widetilde \s)| + |\s'| \leq C |\s \widetilde \s| (|\s| + |\widetilde \s|) \, .
    \end{align*}
    \item If $\s$ is a 1-wave and $\widetilde \s$ is a 2-wave, then we have
    \begin{align*}
        |\s' - \s| + |\s'' - \widetilde \s| \leq C |\s \widetilde \s|  \, .
    \end{align*}
\end{itemize}
\end{lemm}

Moreover, the types of the outgoing waves are known in the following situations. The proof can be found in \cite{smoller, NStoE}.

\begin{lemm}\label{lemma SS and SR same family}
Assume that an incoming $i$-wave interacts with an incoming $j$-wave. Then the Riemann solution satisfies the following properties:
\begin{itemize}
    \item If $i=j$ and both incoming waves are shocks, then the outgoing wave of the same family is a shock, while the outgoing wave of the other family is a rarefaction.
    \item Two rarefactions in the same family do not interact.
    \item If $i=j$ and the incoming waves are of different types, that is, one is a shock and the other is a rarefaction, then the outgoing wave in the other family is a shock.
    \item If $i\neq j$, then, for each $k=1,2$, the outgoing $k$-wave is of the same type as the incoming $k$-wave.
\end{itemize}
\end{lemm}

\subsection{Thresholds}\label{subsec thresholds}
As in the standard front tracking algorithm, we introduce non-physical waves to ensure that the approximate solution is well defined. Like in \cite{NStoE}, we require the strengths of physical waves to be bounded from below. However, unlike in \cite{NStoE}, where this is achieved by introducing additional non-physical waves, we simply allow a physical wave to be labeled as a non-physical wave once its strength falls below 
\begin{align}\label{def of threshold P and NP}
    \rn = \nu^{\cs} 
\end{align}
for some $0<\cs<1$, and then retain its non-physical classification throughout the subsequent evolution. We define the strength of a non-physical wave as the jump in the velocity component $u$ across the wave as in \eqref{def of wave strength}.

To approximate the spreading effect of a large rarefaction, we discretize it into small waves with fixed support. The corresponding threshold for partitioning is given by $\rn^{\cpr}$, for some $0<\cpr<1$.

For simplicity, we construct both rarefaction waves and non-physical waves to be compactly supported, with supports of size of order 
\begin{align}\label{def of threshold support of R and NP}
    \rn^{\co}
\end{align}
for some $0<\co<1$. On the other hand,
the viscous shock associated with \eqref{NS system BD} spreads out in space. We identify the main transition zone of the shock as an interval of size $\rn^{\cz}$ for some $0<\cz<1$.

For the front tracking algorithm, we work with the following scales:
\begin{align*}
    \rn^{\cfi} > \rn^{\cf} > \rn^{\co} > \rn^{\cz} \, ,
\end{align*}
Here, $\rn^{\cfi}$ is the scale relevant to the merging procedure, while $\rn^{\cf}$ is the scale used in the construction of the approximate initial data. The corresponding exponents satisfy $0 < \cfi < \cf < \co < \cz < 1$.

To incorporate the initial value into the construction, we introduce two additional quantities that determine the merging scale and the construction of the approximate initial value
\begin{align}\label{threshold mft e}
    \mef = \max\{\mei , \rn^{\cf}\} \, , \quad \mev = \max\{\meit,\mef^{1/2} ,\rn^{\cfi}\}
\end{align}
where
\begin{align*}
    \mei = \int_{\mathbb{R}_+} \eta(U_0^\nu|U^0) \dx \, , \; 
    \meit = \|U^0-U_*\|_{L^1((\mef^{-1/2},\infty))} + \|U^0-U_*\|_{L^2((\mef^{-1/2},\infty))}^2 \, .
\end{align*}
In the following lemma, we collect all the conditions that the thresholds must satisfy in order to derive the convergence in Section~\ref{sec global in time estimates}, and show that there exist specific choices of the thresholds for which all these conditions are satisfied.

\begin{lemm}\label{lemma threshold}
    There exist 
    \begin{align}\label{threshold mft}
        0<\cs,\cpr<1 \, , \quad 0<\cfi<\cf<\co<\cz<1
    \end{align}
    such that all the following conditions are satisfied. We require that
    \begin{align}\label{rt mft initial values}
        \cfi < \co - \cpr \, , \quad \frac{\cf}{2} < \co - \cpr \, ,
    \end{align}
    \begin{align}\label{rt mft optimal control on number}
        \cfi<\cf < \cpr \, ,
    \end{align}
    \begin{align}\label{rt mft well defined}
        \co<\cz - 2\cpr \, ,
    \end{align}
    \begin{align}\label{rt transition S}
        (1+\co) \cs < (1+\cz) \cs < 1 \, ,
    \end{align}
    \begin{align}\label{rt transition R}
        \co < 3 \cpr \, ,
    \end{align}
    \begin{align}\label{rt mft total strength NP}
        (1-\cpr) \,\cs > \frac{1+\cg}{2} \, ,
    \end{align}
    \begin{align}\label{rt transition R NP}
        2 \co \cs < 1 \, ,
    \end{align}
    \begin{align}\label{rt mft total number}
        2(2\co+\cpr)\cs < 1 \, , \quad \co+\cpr<1 \, , \quad (\co-\cpr) \cs < 1 \, ,
    \end{align}
    \begin{align}\label{rt uni in time between stopping times}
        \cfi\cs < \frac{\mathbf{c}}{2} \, ,
    \end{align}
    \begin{align}\label{rt uni in time at stopping times}
        \co> \frac{8}{3} \cpr + 2\cfi \, , \quad \cpr + \cfi < 1 \, , \quad 
        \frac{\co}{2} + \cpr < 1-\cpr \, ,
    \end{align}
    where
    \begin{align}\label{def of c gamma}
        \cg=\frac{\ga}{2\ga+\al-1} \, ,
    \end{align}
    \begin{align}\label{def of bold c}
    \cbet = \min\{ \frac{1-\cg}{4\cg} , \, (2-\cop-\cpr)\cs , \, (3\cpr-\cor) \cs , \, (1-\cpr) \cs, \, \cpr \cs, \frac{1}{4}, 1-2\cpr\cs \} \, .
    \end{align}
\end{lemm}
\begin{proof}
We prove the existence by giving a specific choice of the thresholds for which all the conditions are satisfied.

First, we observe that, since $\ga-1\leq \al \leq \ga$ and $\ga>1$,
\begin{align}\label{property of c gamma}
    0 < \frac{\ga}{3\ga-1} \leq \cg \leq \frac{\ga}{3\ga -2} < 1 \, .
\end{align}
Let us fix $\cs$ as
\begin{align*}
    \cs = \frac{1}{2}(\frac{1+\cg}{2} + 1) \, .
\end{align*}
By \eqref{property of c gamma}, we have
\begin{align*}
    \frac{1+\cg}{2}<\cs<1 \, .
\end{align*}
Then by taking $\cpr, \cz$ small enough, \eqref{rt mft total strength NP}, \eqref{rt mft total number}, \eqref{rt transition S}, \eqref{rt transition R NP}, as well as the second and third inequalities in \eqref{rt uni in time at stopping times}, are all satisfied, and 
\begin{align*}
    \cbet = \min\{(3\cpr-\co)\cs \, ,\; \cpr\cs\} \, .
\end{align*}
We choose
\begin{align*}
    \cz = 5 \cpr \, , \quad \co = \frac{17}{6}\cpr  \, , \quad 
    \cf = \frac{1}{12} \cpr \, , \quad \cfi = \frac{1}{24}\cpr \, .
\end{align*}
It is straightforward to check that all conditions required are satisfied.
\end{proof}

\begin{rema}
We require \eqref{rt mft initial values}, \eqref{rt mft optimal control on number}, and \eqref{rt mft well defined} for the construction of the modified front tracking method. In particular, \eqref{rt mft initial values} ensures that the construction of the approximate initial data and the merging procedure work. To obtain the optimal control on the number of waves and mergings, we additionally impose \eqref{rt mft optimal control on number}. Finally, for the global well-definedness of the approximate solution, we need \eqref{rt mft well defined}.

The energy estimates between stopping times suggest the requirements \eqref{rt transition S}, \eqref{rt transition R}, \eqref{rt mft total strength NP}, and \eqref{rt mft total number}. Since the strengths of shocks are bounded from below by $\rn$, the scale of their main transition regions, $\rn^{\cz}$, should satisfy \eqref{rt transition S}. To control the terms associated with rarefactions in Section~\ref{sec rarefaction}, we require their transition zones to be sufficiently large compared with their strengths, in the sense of \eqref{rt transition R}. To handle the terms associated with non-physical waves and the boundary in Section~\ref{sec NP and boundary}, we need the total strength of the non-physical waves to be sufficiently small, as suggested by \eqref{rt mft total strength NP}. To control the viscous terms, we require the transition zones associated with rarefactions and nonphysical waves to satisfy \eqref{rt transition R NP}. Finally, to estimate the terms localized by the derivatives of the weights and by the shocks, we also need the number of waves to be controlled by the first two inequalities in \eqref{rt mft total number} and the derivatives of the rarefactions to be sufficiently small, as ensured by the last inequality in \eqref{rt mft total number}.

Finally, the global-in-time uniform estimates for the relative entropy suggest the requirements \eqref{rt uni in time between stopping times} and \eqref{rt uni in time at stopping times}. In particular, the former gives a suitable estimate between stopping times, while the latter keeps the total increase caused by all events at all stopping times under control.

\end{rema}

\subsection{Event}\label{subsec event}
We introduce all events in this section, including interactions, passings, partitions, mergings, and overtakings, and in particular define the corresponding outgoing waves. In Section~\ref{subsec smoothing}, we define the speeds of these outgoing waves. While rarefaction waves move approximately with their usual speeds, shock waves are assigned speeds that include a shift, for which no uniform $L^\infty$ bound with respect to $\nu$ is available. Since the speed of physical waves in the approximate solution are not uniformly controlled from above, to prevent interactions between two non-physical waves, we let non-physical waves remain stationary and assign them a new family label, namely $1.5$-waves.  

We specify each wave by an associated ordered set, which serves as its label and encodes all information about its formation history. Each element $b$ of an ordered set is a natural number, and we assign to it two functions: a tracked strength $\s_b(t)\in\mathbb{R}$ and a generation order $k_b(t)\in\mathbb{N}_*$. The tracked strengths are chosen so that, for each wave, the sum of the tracked strengths of all elements in its associated ordered set equals the strength of the wave. We note that, by construction, at any time $t$, each element can appear in the associated ordered set of at most one wave. Moreover, the elements in each associated ordered set are distinct.

The generation order will be used later to count the total number of waves. The elements in the ordered set associated with a shock, or with a non-physical wave converted from a shock, always have the same generation order. By contrast, the elements in the ordered set associated with a rarefaction, or with a non-physical wave converted from a rarefaction, may have different generation orders.

We keep track of a quantity 
\begin{align}\label{def of N_e}
    N_e(t) \, , 
\end{align}
which counts the total number of elements ever created across all the waves up to time $t$. Initially, $N_e$ is equal to the number of elements constructed at time $0$. The value of $N_e$ is then updated whenever an event occurs.

\medskip

We first introduce interaction events. Suppose that, at time $t_*$, an $i$-wave $\be$ of strength $\s$, associated with the ordered set $\{b_1,\ldots,b_m\}$, interacts with a $j$-wave $\widetilde{\be}$ of strength $\widetilde{\s}$, associated with the ordered set $\{b_{m+1},\ldots,b_n\}$, where $m<n$, $m,n\in\mathbb{N}_+$, and $b_1,\ldots,b_n$ are distinct natural numbers.

We denote by $k_{\min}$ and $\widetilde{k}_{\min}$ the minimal generation orders among the elements of the ordered sets associated with the two incoming waves, respectively. Namely,
\begin{align*}
    k_{\min} = \min\{k_{b_1}(t_*-),\ldots,k_{b_m}(t_*-)\} \, ,
    \quad
    \widetilde{k}_{\min} = \min\{k_{b_{m+1}}(t_*-),\ldots,k_{b_n}(t_*-)\} \, .
\end{align*}

\subsubsection{Interaction of two waves in different families}\label{subsub interaction two physical in different families}
We distinguish the cases according to whether a non-physical wave is involved.

\noindent
\textbf{Case 1: both waves are physical.} We solve the associated Riemann problem and obtain an outgoing $i$-wave $\be'$ with strength $\s'$ and an outgoing $j$-wave $\widetilde{\be}'$ with strength $\widetilde{\s}'$.

We let $\be'$ inherit the associated ordered set of $\be$, $\widetilde{\be}'$ inherit the associated ordered set of $\widetilde{\be}$, keep the generation order of each $b_l$, and define the tracked strengths after the interaction by
\begin{equation*}
    \s_{b_l}(t_*+) =
    \begin{cases}
        \frac{\s_{b_l}(t_*-)}{\s} \s ' \, , \quad \textup{for } 1\leq l \leq m \, , \\
        \frac{\s_{b_l}(t_*-)}{\widetilde{\s}} \widetilde{\s} ' \, , \quad \textup{for } m+1\leq l \leq n \, .
    \end{cases}
\end{equation*}
Note that $\s_{b_l}$ is defined so that
\begin{align*}
    \s' = \sum_{l=1}^m \s_{b_l}(t_*+) \, , \quad 
    \s = \sum_{l=1}^m \s_{b_l}(t_*-) \, , \quad
    \widetilde{\s} ' = \sum_{l=m+1}^n \s_{b_l}(t_*+) \, , \quad
    \widetilde{\s} = \sum_{l=m+1}^n \s_{b_l}(t_*-) \, .
\end{align*}

\noindent
\textbf{Case 2: one of the incoming waves is non-physical.}
We let the physical wave pass through the non-physical wave while keeping its strength unchanged. 

More precisely, we first consider the case $i=1$ and $j=1.5$. In this case, $\be$ is a physical $1$-wave, $\widetilde{\be}$ is a non-physical $1.5$-wave, and $\widetilde{\be}$ lies to the left of $\be$. Let $U_L$ be the left state of $\widetilde{\be}$, let $U_M$ be the right state of $\widetilde{\be}$, which is also the left state of $\be$, and let $U_R$ be the right state of $\be$.

We define the outgoing $1$-wave $\be'$ as the wave connecting $U_L$ to $\widetilde U_M$, with strength $\s$. That is,
\begin{align*}
    \widetilde U_M = T_1(\s)(U_L) \, .
\end{align*}
The outgoing non-physical wave $\widetilde{\be}'$ is then defined as the wave connecting $\widetilde U_M$ to $U_R$. We denote the strength $\widetilde{\be}'$ by $\widetilde{\s}'$.

Similarly as in case 1, we let $\be'$ inherit the associated ordered set of $\be$, $\widetilde{\be}'$ inherit the associated ordered set of $\widetilde{\be}$, keep the generation order of each $b_l$, and define the tracked strengths after the interaction by
\begin{equation*}
    \s_{b_l}(t_*+) =
    \begin{cases}
        \s_{b_l}(t_*-) \, , \quad \textup{for } 1\leq l \leq m \, , \\
        \frac{\s_{b_l}(t_*-)}{\widetilde{\s}} \widetilde{\s} ' \, , \quad \textup{for } m+1\leq l \leq n \, .
    \end{cases}
\end{equation*}
Note that $\s_{b_l}$ is defined so that
\begin{align*}
    \s' = \s = \sum_{l=1}^m \s_{b_l}(t_*+) = \sum_{l=1}^m \s_{b_l}(t_*-)\, , \quad 
    \widetilde{\s} ' = \sum_{l=m+1}^n \s_{b_l}(t_*+) \, , \quad
    \widetilde{\s} = \sum_{l=m+1}^n \s_{b_l}(t_*-) \, .
\end{align*}

The case $i=2$ and $j=1.5$ is defined analogously.

\medskip

Using an argument similar to the proof of Lemma 7.2 in \cite{bressanbook}, we obtain the following estimate for the interaction between a physical wave and a non-physical wave.
\begin{lemm}\label{lemma interaction wave strength estimates NP}
There exists a constant $C>0$, depending only on $U_*$, such that the following estimates hold.

\noindent
Assume a physical wave of strength $\s$ interacts with a non-physical wave of strength $\widetilde{\s}$, and let $\s'$ denote the strength of the outgoing physical wave and $\widetilde{\s} '$ the strength of the outgoing non-physical wave. Then
\begin{align*}
    \s'=\s \, , \quad |\widetilde{\s} ' - \widetilde \s| \leq C |\s \widetilde \s| \, .
\end{align*}
\end{lemm}

\subsubsection{Interaction of two physical waves in the same family}\label{subsub interaction two physical in same family}
We solve the associated Riemann problem and obtain an outgoing $i$-wave $\be'$ with strength $\s'$ and an outgoing $i'$-wave ${\be}''$ with strength ${\s}''$, where $i'\in \{1,2\}$ and $i'\neq i$.

\noindent
\textbf{Case 1: both incoming waves are shocks.}
Lemma~\ref{lemma interaction wave strength estimates} implies that $\be'$ is also a shock. We let $\be'$ inherit the associated ordered sets of the two incoming waves $\be$ and $\widetilde{\be}$. Namely,
\begin{align*}
    \be' = \{b_1, ..., b_{m}, {b}_{m+1}, ..., b_n\} \, .
\end{align*}
The ordering of the elements does not matter.
For each element $b_l$ in $\be'$, we define its generation order to be
\begin{align*}
    \min\{k_{\min},\widetilde{k}_{\min}\} \, ,
\end{align*}
that is, the minimum generation order among all elements of $\be$ and $\widetilde{\be}$, and define its tracked strength after the interaction by
\begin{align*}
    \s_{b_l}(t_*+) = \frac{\s_{b_l}(t_*-)}{\s+\widetilde{\s}} \s ' \, .
\end{align*}
Note that $\s_{b_l}$ is defined so that
\begin{align*}
    \s' = \sum_{l=1}^n \s_{b_l}(t_*+) \, , \quad 
    \s = \sum_{l=1}^m \s_{b_l}(t_*-) \, , \quad
    \widetilde{\s} = \sum_{l=m+1}^n \s_{b_l}(t_*-) \, .
\end{align*}
Lemma~\ref{lemma SS and SR same family} implies that $\be''$ is a rarefaction. Recall that $N_e(t_* -)$ denotes the total number of elements created before time $t_*$. We define the ordered set associated with $\be''$ by
\begin{align*}
    \be''=\{N_e(t_*-)+1, ..., N_e(t_*-)+p\} \, ,
\end{align*}
where 
\begin{align*}
    p = \left\lceil \frac{\s''}{\rn^{\cpr}} \right\rceil,
\end{align*}
with $\lceil s \rceil$ denotes the smallest integer greater than or equal to $s$. We define the generation order of each element in $\be''$ to be
\begin{align*}
    \max\{k_{\min},\widetilde{k}_{\min}\}+1 \, ,
\end{align*}
and assign to each element the tracked strength
\begin{align*}
    \s_{N(t_*-)+l}(t_*+) = \frac{\s''}{p} \, .
\end{align*}
In particular, we have
\begin{align}\label{newly generated R strength}
    \frac{\rn^{\cpr}}{2} \leq \s_{N(t_*-)+l}(t_*+)  \leq \rn^{\cpr} \, .
\end{align}
We update
\begin{align*}
    N_e(t_*+) = N_e(t_*-)+p \, .
\end{align*}

\noindent
\textbf{Case 2: one of the incoming waves is a rarefaction.}
Without loss of generality, assume that $\be$ and $\be'$ are of the same type. We then let $\be'$ inherit the associated ordered set of $\be$, keep the generation order of each $b_l$, and define the tracked strengths after the interaction by
\begin{equation*}
    \s_{b_l}(t_*+) = \frac{\s_{b_l}(t_*-)}{\s} \s ' \, , \quad \textup{for } 1\leq l \leq m \, .
\end{equation*}
Note that $\s_{b_l}$ is defined so that
\begin{align*}
    \s' = \sum_{l=1}^m \s_{b_l}(t_*+) \, , \quad 
    \s = \sum_{l=1}^m \s_{b_l}(t_*-) \, .
\end{align*}
Lemma~\ref{lemma SS and SR same family} implies that $\be''$ is a shock. We define the ordered set associated with $\be''$ by
\begin{align*}
    \be''=\{N_e(t_*-)+1\} \, .
\end{align*}
We define the generation order of its element to be
\begin{align*}
    \widetilde{k}_{\min}+1 \, ,
\end{align*}
and assign to it the tracked strength
\begin{align*}
    \s_{N_e(t_*-)+1}(t_*+) = \s '' \, .
\end{align*}
We update
\begin{align*}
    N_e(t_*+) = N(t_*-)+1 \, .
\end{align*}

\medskip
\medskip
Next, we introduce passing and partition events. These events are used to determine the placement of newly generated partitioned rarefactions. In particular, when Case~1 in Section~\ref{subsub interaction two physical in same family} occurs, we aim to partition the outgoing wave $\be''$ into $p$ rarefactions, each associated with an ordered set consisting of a single element whose initial tracked wave strength is of order $\rn^{\cpr}$. 

Therefore, rather than placing $\be''$ as a single outgoing $i'$-wave, we need a procedure that creates enough space to place $p$ outgoing $i'$-waves. The passing and partition events introduced below provide the framework for this procedure. The full placement construction will be described in Section \ref{subsub j stopping time}.

\subsubsection{Passing}\label{subsub passing}
Suppose that, at time $t_*$, an $i$-rarefaction $\be$ of strength $\s$, associated with the ordered set $\{b_1,\ldots,b_m\}$, is adjacent to another $i$-rarefaction $\widetilde{\be}$ of strength $\widetilde{\s}$, associated with the ordered set $\{b_{m+1}\}$, where $2\leq m\in\mathbb{N}_+$, and $b_1,\ldots,b_{m+1}$ are natural numbers. If $i=1$, then $\be$ lies to the right of $\widetilde{\be}$; if $i=2$, then $\be$ lies to the left of $\widetilde{\be}$.

Let $U_L$ be the left state of the leftmost incoming rarefaction, and let $U_R$ be the right state of the rightmost incoming rarefaction. We define the outgoing waves to be two $i$-rarefactions,
\begin{align*}
    \widetilde{\be}' = \{b_1\} \, , \quad \be' = \{b_2, ..., b_{m+1}\} \, ,
\end{align*}
connecting $U_L$ to $U_R$. We keep the generation order and the tracked strength of each element unchanged. Thus, the strength of $\widetilde{\be}'$ is $\s_{b_1}(t_*-)$, while the strength of $\be'$ is $\s+\widetilde{\s}-\s_{b_1}(t_*-)$. The waves are arranged so that $\widetilde{\be}'$ lies on the same side of $\be'$ as $\be$ lies relative to $\widetilde{\be}$.

In this way, the rarefaction whose associated ordered set contains more than one element is moved to the left when $i=1$, and to the right when $i=2$. We remove $b_1$ from this rarefaction and incorporate $b_{m+1}$ into it in order to obtain the estimate in Lemma~\ref{lemma L + q + omega}, as shown in Section~\ref{subsubsec decay of the weight}.

\subsubsection{Partition}\label{subsub partition}
Suppose that, at time $t_*$, there is an $i$-rarefaction $\be$ of strength $\s$, associated with the ordered set $\{b_1,\ldots,b_m\}$, for some $m\geq 2$, where $b_1,\ldots,b_{m}$ are natural numbers. We denote the left and right states of $\be$ by $U_L$ and $U_R$, respectively. We partition $\be$ into two $i$-rarefactions,
\begin{align*}
    \widetilde{\be}' = \{b_1\},
    \qquad
    \be' = \{b_2,\ldots,b_m\},
\end{align*}
which together connect $U_L$ to $U_R$. The outgoing rarefactions are arranged so that $\be'$ lies to the left of $\widetilde{\be}'$ when $i=1$, and to the right of $\widetilde{\be}'$ when $i=2$. The generation order and tracked strength of each element are kept unchanged. Thus, the strength of $\widetilde{\be}'$ is $\s_{b_1}(t_*-)$, while the strength of $\be'$ is $\s-\s_{b_1}(t_*-)$.

\medskip
\medskip
Finally, we discuss the case in which a wave approaches the boundary. We distinguish between physical and non-physical waves approaching the boundary.
In the former case, the event must be handled carefully through a merging procedure, since a physical wave may approach the boundary multiple times. 
In the latter case, we simply allow the boundary to overtake the non-physical wave, because the total error generated by non-physical waves can be well controlled.

Here we introduce the framework for merging physical waves. The precise choice of which waves are merged, as well as where the resulting waves are placed, will be described in Section \ref{subsub j stopping time}. For non-physical waves, the overtaking procedure will also be described there. We note that an overtaking event has no outgoing waves.

\subsubsection{Merging}\label{subsub merging}
Intuitively, we evaluate the information near the boundary and obtain the net change. We then shrink all incoming waves whose strengths have the same sign as this net change accordingly. Finally, we keep only those waves that remain physical after the shrinking. 

Suppose that, at time $t_*$, we merge a collection of adjacent waves near the boundary. Let $U_L=(v_L,u_L)$ denote the left state of the leftmost wave, namely the wave closest to the boundary, and let $U_R=(v_R,u_R)$ denote the right state of the rightmost wave. We solve the boundary Riemann problem with right state $U_R$ and boundary condition $u=u_L$. This boundary Riemann problem yields a single $2$-wave of strength $\s'=u_R-u_L$.

We denote by $\s$ the sum of the strengths of the incoming waves whose strengths have the same sign as $\s'$. In particular, if $\s'>0$, then $\s>0$ is the sum of the strengths of all incoming rarefactions and incoming non-physical waves with positive strength. If $\s'<0$, then $\s<0$ is the sum of the strengths of all incoming shocks and incoming non-physical waves with negative strength. These waves may belong to different classes: they may be $1$-waves, $2$-waves, or non-physical $1.5$-waves. Since wave strength is defined as the jump in the velocity component, we directly obtain
\begin{align*}
    |u_R-u_L| = |\s'| \leq |\s| \, .
\end{align*}
Ordered from left to right, we denote by
\begin{align*}
    \be^{(1)}=\{b_1^{(1)},\ldots,b_{m_1}^{(1)}\} \, , \; \ldots \, , \;
    \be^{(M)}=\{b_1^{(M)},\ldots,b_{m_M}^{(M)}\}
\end{align*}
the incoming physical waves whose strengths $\s^{(i)}$ have the same sign as $\s'$ and for which
\begin{align}\label{selection of outgoing waves based on strength merging}
    {\s^{(i)}}' := \frac{\s'}{\s} \s^{(i)} \geq \rn \, .
\end{align}
We note that
\begin{align*}
    \sum_{i=1}^M |\s^{(i)}| \leq |\s| \, , \quad \sum_{i=1}^M |{\s^{(i)}}'| \leq |\s'| \, .
\end{align*}

\noindent
\textbf{Case 1: $\s'>0$.}
We define the outgoing waves to be $M$ adjacent $2$-rarefactions, ordered from left to right,
\begin{align*}
    {\be^{(1)}}'=\{b_1^{(1)},\ldots,b_{m_1}^{(1)}\},
    \quad \ldots,\quad
    {\be^{(M)}}'=\{b_1^{(M)},\ldots,b_{m_M}^{(M)}\},
\end{align*}
where the strength of ${\be^{(i)}}'$ is ${\s^{(i)}}'$, and the right state of ${\be^{(M)}}'$ is $U_R$. 

For each element $b_l^{(i)}$, we keep its generation order unchanged and assign to it the tracked strength
\begin{align*}
    \s_{b_l^{(i)}}(t_*+) = \frac{\s'}{\s}\,\s_{b_l^{(i)}}(t_*-) \, .
\end{align*}

\noindent
\textbf{Case 2: $\s'<0$.}
We define the outgoing wave to be the $2$-shock 
\begin{align*}
    \be' = \{b_1^{(1)} ,\ldots,b_{m_1}^{(1)}, \ldots,
    b_1^{(M)},\ldots,b_{m_M}^{(M)}\} \, ,
\end{align*}
where the strength of $\be'$ is $\sum_{i=1}^M {\s^{(i)}}'$, and the right state of $\be'$ is $U_R$.

For each element $b_l^{(i)}$ in $\be'$, we define its generation order to be the minimum generation order among all elements of $\be^{(1)},\ldots,\be^{(M)}$, and assign to it the tracked strength
\begin{align*}
    \s_{b_l^{(i)}}(t_*+) = \frac{\s'}{\s}\,\s_{b_l^{(i)}}(t_*-) \, .
\end{align*}

\smallskip
Note that, in both cases, the outgoing waves are physical $2$-waves.

\subsubsection{Properties of events}
In the constructions of outgoing waves for interactions, passing, and partition, we regard all outgoing waves as physical for clarity in the definition. It may happen that the strength of an outgoing wave falls below $\rn$. In such a case, we label it as non-physical and keep it stationary, while keeping its associated ordered set and the generation order of each element exactly as defined in the construction.

\smallskip

From the construction, we have the following three properties.
\begin{property}\label{prop associated ordered set S}
Once an associated ordered set becomes the associated ordered set of a shock, it can only remain unchanged, gain elements, or disappear.
\end{property}
\begin{property}\label{prop introduction of new elements}
    Only interactions between two waves in the same family can introduce new elements. Moreover, these new elements constitute the associated ordered set of the outgoing wave in the other family.
\end{property}
\begin{property}\label{prop non-increasing passing, partition, merging}
    Passing, partition, and merging do not increase the tracked strength of any element.
\end{property}
The generation order of an element that belongs both to the associated ordered set of an incoming wave and to the associated ordered set of an outgoing wave is not changed by an event, except when the event is an interaction between two shocks in the same family, or a merging that produces a $2$-shock. In these two cases, the generation order of every element in the associated ordered set of the outgoing shock is redefined as the minimum of the generation orders of all such elements before the event. We therefore obtain the following two properties.
\begin{property}\label{prop generation order non-increasing}
At any event, the generation order of an element is non-increasing.
\end{property}
\begin{property}\label{prop generation order S same}
All elements in the ordered set associated with a given shock have the same generation order.
\end{property}

By the construction in Section~\ref{subsub interaction two physical in same family} and Property~\ref{prop introduction of new elements}, we have the following properties. In particular, the second one follows from \eqref{newly generated R strength}.

\begin{property}\label{prop new elements order}
    The newly introduced elements have the same generation order, which is one greater than the minimum generation order among the elements in the associated ordered set of one of the incoming waves. 
    
    If the event is an interaction between a shock and a rarefaction in the same family, then this incoming wave is the one whose type is different from that of the outgoing wave in the same family. If the event is an interaction between two shocks in the same family, then this incoming wave is the shock whose associated ordered set contains elements of the larger generation order.
\end{property}

\begin{property}\label{prop newly generated R strength}
If the associated ordered set of an outgoing rarefaction wave contains $p>1$ newly generated elements, then these elements have the same tracked strength, which lies between $\frac{\rn^{\cpr}}{2}$ and $\rn^{\cpr}$.
\end{property}

We also observe that no event increases the cardinality of the ordered set associated with a rarefaction. More precisely, we have the following property.

\begin{property}\label{prop R cardinality}
For every event for which exactly one of the incoming waves is a rarefaction whose associated ordered set has cardinality $p\geq 1$, at most one outgoing wave is a rarefaction whose associated ordered set has cardinality greater than $1$, and the cardinality of that set is at most $p$.
\end{property}

We observe that when a non-physical wave interacts with another wave, the latter must be physical, since non-physical waves are stationary. In such an interaction, as defined in Case 2 of Section \ref{subsub interaction two physical in different families}, the physical wave passes through the non-physical wave, leaving the ordered set associated with the non-physical wave unchanged. Moreover, the incoming waves in passing and partition events are rarefactions, whereas the outgoing waves produced by a merging event contain no non-physical waves. Finally, an overtaking event has no outgoing waves. Combining these observations, we obtain the following property.
\begin{property}\label{prop ordered set NP unchanged}
    Once an element belongs to the ordered set associated with a non-physical wave, it either disappears or continues to belong to the ordered set associated with a non-physical wave. Once an ordered set becomes the associated ordered set of a non-physical wave, it is kept unchanged throughout the subsequent evolution, until the non-physical wave with this associated ordered set disappears.
\end{property}

\subsection{Smoothing of elementary waves}\label{subsec smoothing}
To construct a $\nu$-approximate solution associated with \eqref{NS system BD}, we smooth the elementary waves on scales compatible with both the Navier-Stokes dynamics, governed by the viscosity parameter $\nu$, and the Euler dynamics in the limit as $\nu\to0$.

More precisely, physical shocks are regularized on a scale determined by the viscosity parameter $\nu$. Rarefactions and non-physical waves are smoothed on scales chosen so that the modified front tracking algorithm is well-defined and the energy estimates remain controllable over time.

We denote the strength of the wave by $\s$, its family by $i$, and its left and right states by $U_-$ and $U_+$, respectively.

\smallskip

\noindent
\textbf{Viscous shocks.}
We denote by
\begin{align*}
    \td U (x- \td \la t; \um, \up, \s, i) = \bigl(\td v (x- \td \la t; \um, \up, \s, i), \td h (x- \td \la t; \um, \up, \s, i)\bigr)
\end{align*}
the viscous $i$-shock with Rankine-Hugoniot speed $\td \la$, associated with \eqref{NS system BD} for viscosity parameter $\nu=1$.
We have
\begin{equation}\label{equations - unshifted viscous shock}
    \begin{cases}
        -\td \la \td v_x - \td h_x = -\bigl(\mu(\td v) \td v^\gamma p(\td v)_x\bigr)_x \, , \\
        -\td \la \td h_x + p(\td v)_x = 0 \, ,
    \end{cases}
    \;
    \underset{x\rightarrow - \infty}{\textup{lim}} \td U (x) = \um \, , \quad 
    \underset{x\rightarrow \infty}{\textup{lim}} \td U (x) = \up \, .
\end{equation}
We write
\begin{align}\label{def of shock nu}
    \usn (x- \td \la t; \um, \up, \s, i) = \td U (\frac{x-\td \la t}{\nu}; \um, \up, \s, i) \, .
\end{align}
for the viscous $i$-shock associated with \eqref{NS system BD} for viscosity parameter $\nu$.

\noindent
\textbf{Viscous Rarefaction.}
We define the smoothened $i$-rarefaction profile associated to \eqref{NS system BD} by 
\begin{align}\label{def of rarefaction nu}
    \urn(x;\um, \up, \s, i) = \mathbf{R}_{\um}^\be(\s\phi(\frac{x}{\rn^{\cor}})) \, ,
\end{align}
where $\mathbf{R}_{\um}^\be$ denotes the $i$-integral curve passing through $\um$, $\rn^{\co}$ is the support scale for rarefactions defined in \eqref{def of threshold support of R and NP}, and $\phi$ is a smooth monotone function such that
\begin{align}\label{def test function smoothen waves}
    \phi(x) = 0 \, , \; \textup{ for } x\leq -1 \, , \; \quad \phi(x) = 1 \, , \; \textup{ for } x\geq 1 \, .
\end{align}
Note that the tangent vector of the rarefaction curve $\mathbf{R}_{\um}^\be$ is the right eigenvector $\mathbf{r}_i$ of the $i$-family, that is,
\begin{align*}
    \frac{d}{d \s} \mathbf{R}_{\um}^\be (\cdot) = \mathbf{r}_i \bigl(\mathbf{R}_{\um}^\be (\cdot)\bigr) \, ,
\end{align*}
Moreover, \eqref{def test function smoothen waves} implies
\begin{align*}
    \underset{x\rightarrow - \infty}{\textup{lim}} \urn (x) = \um \, , \quad 
    \underset{x\rightarrow \infty}{\textup{lim}} \urn (x) = \up \, .
\end{align*}

\noindent
\textbf{Smooth non-physical wave.}
We define a non-physical shock by
\begin{align}\label{def of NP nu}
    \unn(x;\um,\up,\s) = \um + (\up-\um) \, \phi(\frac{x}{\rn^{\cop}}) \, ,
\end{align}
where $\phi$ is defined in \eqref{def test function smoothen waves} and $\rn^{\co}$ is the support scale for non-physical waves defined in \eqref{def of threshold support of R and NP}.

\subsection{Modified front tracking algorithm}\label{subsec mft}

In this subsection, we focus on introducing the algorithm and postpone the proof of well-definedness, as well as the estimates on the total number of waves and events, to Section~\ref{sub number of waves and elements}.

For the moment, we assume that the total variation of the approximate solution is controlled by the total variation at time $0$, up to a multiplicative constant. This assumption will be justified in Lemma \ref{lemma L+kQ} in Section~\ref{sub number of waves and elements}.

\subsubsection{Propagation speeds of physical and non-physical waves}
Consider a wave $\be$. 
We denote its strength by $\s$, its characteristic family by $i$, the states it connects by $U_{-}$ and $U_{+}$, and its initial location by $x_\be$. 

If $\be$ is a shock, then we let it propagate with a shift $\sn$ defined in \eqref{ODE of Xi} (with $\sn(0)=0$). In particular, we consider
\begin{align}\label{def shifted viscous shock}
    {\usn}(x,t) = \usn(x-x_{\be}-{\las} t - \sn(t); U_{-}, U_{+}, \s , i) \, ,
\end{align}
where $\las$ denotes the Rankine-Hugoniot speed of the shock and $\td U^\nu$ is defined in (\ref{def of shock nu}).
By \eqref{equations - unshifted viscous shock} and \eqref{def of shock nu}, we have
\begin{equation}\label{equations - viscous shock}
    \begin{cases}
        {\vsn}_t + \sdn {\vsn}_x - \hsn_x = -\nu\bigl(\mu(\vsn) (\vsn)^\gamma p(\vsn)_x\bigr)_x \, , \\
        \hsn_t + \sdn \hsn_x + p(\vsn)_x = 0 \, ,
    \end{cases}
\end{equation}
and 
\begin{equation}\label{equations - viscous shock 2}
    \begin{cases}
        -{\las} {\vsn}_x - \hsn_x = -\nu\bigl(\mu(\vsn) (\vsn)^\gamma p(\vsn)_x\bigr)_x \, , \\
        -{\las} {\hsn}_x  + p(\vsn)_x = 0 \, .
    \end{cases}
\end{equation}
If $\be$ is a rarefaction, then we let it propagate with speed $\la_{i}(U_{-})$. More precisely, we consider
\begin{align}\label{def shifted rarefaction}
    \urn(x,t) =
    \urn(x-x_{\be}-\la_{i}(U_{-}) \, t\, ;\, U_{-},\, U_{+},\, \s ,\, i) \, .
\end{align}
If $\be$ is a non-physical wave, then we keep it stationary and consider
\begin{align}\label{def non-physical wave}
    \unn(x) = \unn(x-x_\be; U_{-}, U_{+}, \s) \, ,
\end{align}
where $\unn$ is defined in \eqref{def of NP nu}.

We denote the transition zone of a wave $\be$ by
\begin{align}\label{def of transition zone}
    I_{\be} = \{x: |x-y_{\be}|\leq \rn^{\co}\} \, ,
\end{align}
where $y_\be$ denotes the location of the wave. In particular, if $\be$ is a viscous shock, we define its main transition region by
\begin{align}\label{def of main transition zone S}
    \widetilde{I}_{\be}:=\{x: |x-y_{\be}|\leq \rn^{\cz}\} \subset I_\be \, .
\end{align}

\subsubsection{Initial step}\label{subsub initial step}
At time $t=0$, we approximate the small BV initial value $U^0$ by a step function $\udU_\nu^0$ consisting of $1/\mef$ steps, each of width $\mef^{1/2}$, where $\mef$ is defined in \eqref{threshold mft e}.
On each interval, we define $\udU_\nu^0$ as the average of $U^0$ over that interval. Outside the approximation region, we set $\udU_\nu^0=U_*$, where $U_*$ is the right far-field limit of $U^0$. We get the following estimate
\begin{equation}\label{initial value intermediate step 1}
    \begin{aligned}
        &\|U^0 - \udU_\nu^0\|_{L^1(\mathbb{R}_+)} + \|U^0 - \udU_\nu^0\|_{L^2(\mathbb{R}_+)}^2 \\
        &\leq \mef^{1/2} \textup{TV}(U^0) + \|U^0 - U_*\|_{L^1((\mef^{-1/2},\infty))} + \|U^0 - U_*\|_{L^2((\mef^{-1/2},\infty))}^2 \\
        &\leq 2\mev \, ,
    \end{aligned}
\end{equation}
since
\begin{align*}
    \int_0^{\mef^{-1/2}} |U^0 - \udU_\nu^0| \, dx
    &\leq \sum_{i=1}^{\mef^{-1}} \int_{(i-1)\mef^{1/2}}^{i\mef^{1/2}} |U^0 - \udU_\nu^0| \, dx \\
    &\leq \sum_{i=1}^{\mef^{-1}} \mef^{1/2} \textup{TV}(U^0;[(i-1)\mef^{1/2},i\mef^{1/2})) \\
    &\leq \mef^{1/2} \textup{TV}(U^0) \, ,
\end{align*}
where $\textup{TV}(U^0;I)$ denotes the total variation of $U^0$ on the interval $I$.

We solve the associated Riemann problems locally in the interior. At the boundary, we solve the corresponding boundary Riemann problem with the outflow boundary condition \eqref{outflow boundary condition}, which yields a $2$-wave. Locally, we place the outgoing waves of a Riemann problem at mutual distances of $3\rn^{\co}$, and partition any outgoing rarefaction whose strength $\s$ exceeds $\rn^{\cpr}$ into $p$ rarefactions of equal strength, where
\begin{align*}
    p = \left\lceil \frac{\s}{\rn^{\cpr}} \right\rceil,
\end{align*}
We place the resulting rarefactions at mutual distances of $3\rn^{\co}$. By \eqref{rt mft initial values}, we can ensure that the distance between any two adjacent waves, as well as the distance between any wave and the boundary, is at least $3\rn^{\co}$.

We denote such an approximation of the initial value $U^0$ by $\bu_\nu^0$. By \eqref{threshold mft}, \eqref{threshold mft e}, \eqref{rt mft initial values}, and \eqref{rt mft optimal control on number}, such a construction is possible, and we have 
\begin{align*}
    \|\udU_\nu^0 - \bu_\nu^0\|_{L^1(\mathbb{R}_+)} + \|\udU_\nu^0 - \bu_\nu^0\|_{L^2(\mathbb{R}_+)}^2 \lesssim \rn^{\co}(\mef^{-1}+\rn^{-\cpr})
    \ll \rn^{\cfi} \leq \mev \, .
\end{align*}
Together with \eqref{initial value intermediate step 1}, we get
\begin{align}\label{initial value intermediate step 2}
    \|U^0 - \bu_\nu^0\|_{L^1(\mathbb{R}_+)} + \|U^0 - \bu_\nu^0\|_{L^2(\mathbb{R}_+)}^2 \leq C\mev \, .
\end{align}
Since
\begin{align*}
    \eta(U^\nu_0|\bu_0^\nu) &= \eta(U^\nu_0|U^0) + \eta(U^0) - \eta(\bu_0^\nu) 
    \\ &\qquad + \bigl( \nabla \eta(U^0) - \nabla \eta(\bu_0^\nu) \bigr)(U^\nu_0 - U^0)
    - \nabla \eta(\bu_0^\nu)(U^0 - \bu_0^\nu) \, ,
\end{align*}
we have
\begin{align*}
    \int_{\mathbb{R}_+} \eta(U^\nu_0|\bu_0^\nu) \dx &\leq \int_{\mathbb{R}_+} \eta(U^\nu_0|U^0) \dx + C \|U^0-\bu_0^\nu\|_{L^1(\mathbb{R}_+)} \\
    &\quad+ C \int_{\mathbb{R}_+} |U^0-\bu_0^\nu| |U_0^\nu-U^0| \dx\, ,
\end{align*}
where
\begin{align*}
    &\int_{\mathbb{R}_+} |U^0-\bu_0^\nu| \, |U_0^\nu-U^0| \dx \\
    &\leq \int_{\mathbb{R}_+} |U^0-\bu_0^\nu|\,|v_0^\nu-v^0| \dx + \int_{\mathbb{R}_+} |U^0-\bu_0^\nu|\,|h_0^\nu-h^0| \dx \\
    &\leq C \|U^0-\bu_0^\nu\|_{L^1(\mathbb{R}_+)} + C \int_{\mathbb{R}_+} Q(v_0^\nu|v^0) \dx + \|U^0-\bu_0^\nu\|_{L^2(\mathbb{R}_+)}^2 + \int_{\mathbb{R}_+} |h_0^\nu-h^0|^2 \dx \, ,
\end{align*}
since
\begin{align*}
    |v_0^\nu-v^0| \leq C \mathds{1}_{\{v_0^\nu\leq 2v_*\}} + C Q(v_0^\nu|v^0)\mathds{1}_{\{v_0^\nu> 2v_*\}} \, .
\end{align*}
Recall \eqref{threshold mft e}.
By \eqref{initial value intermediate step 2}, we conclude
\begin{align}\label{decay rate of difference between initial and app initial}
    \int \eta(U^\nu_0|\bu_0^\nu) \dx \leq \mathbf{C}_0\mev \, ,
\end{align}
for some $\mathbf{C}_0>0$.

Recall $N_e$ defined in \eqref{def of N_e}.
We let $N_e(0)$ be the total number of waves in the approximate solution $\bu_0^\nu$, and from left to right, we assign to these waves the associated ordered sets
\begin{align}\label{def initial associated ordered set}
    \{1\} \, , \; ... \, , \; \{N_e(0)\} \, .
\end{align}
We assign generation order $1$ to all elements. By \eqref{rt mft optimal control on number}, we have
\begin{align}\label{control of N(0)}
    N_e \lesssim \rn^{-\cpr} \, .
\end{align}

We let $\sse$, $\rse$, and $\npse$ denote the sets of shocks, rarefactions, and non-physical waves, respectively, in the approximate solution. These sets are constant between stopping times and may change only when a stopping time is crossed.

For $t$ before the first stopping time, we define the approximate solution by
\begin{align}\label{def approximate sol at time 0}
    \ub (x,t) = \ust - \int_{x}^\infty \sum_{\sse} \p_x {\usn}(y,t) + \sum_{\rse} \p_x \urn(y,t) + \sum_{\npse} \p_x \unn(y) \, dy \, .
\end{align}

\subsubsection{$j$-th stopping time $t_j$}\label{subsub j stopping time}
We denote by $d(t)$ the minimum of the distances between any two distinct waves, and between the leftmost wave and the boundary. We stop the clock for the $j$-th time when
\begin{align*}
    d(t) = d_j \, ,
\end{align*}
where, for each $j\in\mathbb{N}_+$, we define
\begin{align}\label{def of dj}
    d_j = 3\rn^{\co} - 5 (m_j + 1 )\rn^{\cz}  \, ,
\end{align}
where
\begin{align*}
    m_j = |\{1\leq i\leq j-1: \textup{$i$-th stopping time is triggered by an interaction}\}| \, .
\end{align*}
From the construction of the initial step in Section \ref{subsub initial step}, we have
\begin{align*}
    d(0) \geq 3\rn^{\co} \, .
\end{align*}
We can continue to construct the approximate solution by the modified front tracking algorithm introduced below as long as $d_j>0$. At the end of this introduction, we prove in \eqref{dj infinity} that
\begin{align*}
    \lim_{j\to\infty} d_j > 2\rn^{\co} > 0.
\end{align*}

Between stopping times, since each $i$-rarefaction wave with left state $U_-$ is propagated at speed $\lambda_i(U_-)$, the distance between any two adjacent rarefaction waves is increasing. Hence, each stopping time corresponds to the occurrence of an interaction, a merging event, or an overtaking event. A merging and an overtaking cannot occur simultaneously. We give priority to interactions over boundary events, namely merging and overtaking. Moreover, if multiple interactions occur at the same stopping time, we prioritize the rightmost one.

\medskip

We first discuss the placement of the outgoing waves generated by interactions and mergings, and then introduce overtakings. Intuitively, for each event, we place the outgoing waves according to the rule that $1$-waves propagate to the left and $2$-waves propagate to the right. More precisely, we require the algorithm to satisfy the following property.
\begin{property}\label{prop 1-wave to left 2-wave to right}
    Suppose that an element exists both before and after a stopping time $t$, belonging to the associated ordered set of an $i$-wave $\be$ before the stopping time and to that of an $i'$-wave $\be'$ after the stopping time. Then $y_\be(t-)\geq y_{\be'}(t+)$ if $i=i'=1$, while $y_\be(t-)\leq y_{\be'}(t+)$ if either $i=i'=2$ or $i=1$ and $i'=2$. Note that, because of the presence of the boundary, it is possible to have $i=1$ and $i'=2$, but it is impossible to have $i=2$ and $i'=1$.
\end{property}
We note that, to establish this property, it suffices to verify it for each individual event. We also remark that this property implies the following.
\begin{rema}\label{remark interaction at most two times}
Any two elements $b,\widetilde{b}$ interact at most twice throughout the evolution, where we say that $b$ and $\widetilde{b}$ interact if a wave whose associated ordered set contains $b$ interacts with another wave whose associated ordered set contains $\widetilde{b}$.
\end{rema}

\smallskip

\noindent
\textbf{Placement of outgoing waves generated by an interaction.}
Ordered from left to right, we denote the incoming $i_l$-wave by $\be_l$ and the incoming $i_r$-wave by $\be_r$. Similarly, ordered from left to right, we denote the outgoing $i_l'$-wave by $\be_l'$ and the outgoing $i_r'$-wave by $\be_r'$.
We place the outgoing waves $\be_l'$ and $\be_r'$ as follows
\begin{equation}\label{placement of outgoing waves interaction i=1}
    y_{\be_l'} = 
    \begin{cases}
        y_{\be_l}-2\rn^{\cz} \, , \quad \textup{if } \be_l \textup{ and } \be_l' \textup{ are 1-shocks,} \\
        y_{\be_l} \, , \quad \quad \textup{otherwise,}
    \end{cases}
\end{equation}
and
\begin{equation}\label{placement of outgoing waves interaction i=2}
    y_{\be_r'} = 
    \begin{cases}
        y_{\be_r}+2\rn^{\cz} \, , \quad \textup{if } \be_r \textup{ and } \be_r' \textup{ are 2-shocks,} \\
        y_{\be_r} \, , \quad \quad \textup{otherwise.}
    \end{cases}
\end{equation}
We see the placement of outgoing waves verifies Property \ref{prop 1-wave to left 2-wave to right}. We perturb the position of the outgoing shock by $2\rn^{\cz}$ when the incoming wave at the same position is a shock of the same family, to obtain the estimate in Lemma~\ref{lemma L + q + omega}.

\medskip

\noindent
\textbf{Placement of outgoing waves generated by a merging.} 
We first introduce the set of elements belonging to the associated ordered sets of physical waves that have already participated in a merging event. More precisely, for each time $t$, we define
\begin{align*}
    \mathcal{M}(t)= \bigl\{ b\in \mathbb{N}_{+} : b \textup{ belongs to the associated ordered set of a physical wave at time } t \\
    \textup{and to the associated ordered set of some wave merged during } [0,t) \bigr\} \, .
\end{align*}
At time $t$, we say that a wave has been previously merged if its associated ordered set contains an element of $\mathcal{M}(t)$. Note that, once a previously merged wave becomes non-physical, we no longer track it as a previously merged wave. Therefore, what we call a previously merged wave is, more precisely, a wave that has been previously merged and is still physical at the current time.

Since the outgoing waves generated by mergings are all $2$-waves, and since a $2$-wave never transforms into a $1$-wave, we have the following property.
\begin{property}\label{prop merging previously merged wave are 2-wave}
    All previously merged waves are $2$-waves.
\end{property}
Hence, we have
\begin{equation}\label{merging def of M}
    \begin{aligned}
        \mathcal{M}(t)= \bigl\{ b\in \mathbb{N}_{+} : b \textup{ belongs to the associated ordered set of a 2-wave at time } t \\
        \textup{and to the associated ordered set of some wave merged during } [0,t) \bigr\} \, .
    \end{aligned}
\end{equation}
We define 
\begin{align}\label{def of y_M}
    y_{\mathcal{M}}
\end{align}
to be the position of the leftmost previously merged wave, if such a wave exists, and set $y_{\mathcal M}=\infty$ otherwise.

We next specify the region in which the merging is performed. Let
\begin{equation}\label{def of tilde y merging}
    \widetilde{y}(t) =
    \begin{cases}
        y_\mathcal{M}(t) \, , \quad \textup{if } \mathcal{M}(t) \neq \emptyset \, , \textup{ and } \mev \leq y_{\mathcal{M}}(t) - y_b(t)< 3\mev \, , \\
        y_b(t) + 3 \mev \, , \quad \textup{otherwise.}
    \end{cases}
\end{equation}
We then merge all waves located in the interval $(y_b(t_j),\widetilde{y}(t_j-))$. Notice that a wave located exactly at $\widetilde{y}(t_j-)$, if any, is not merged. We also note that we always have
\begin{align}\label{lower bound of widetilde y}
    \widetilde{y} \geq \mev \, .
\end{align}

As defined in Section~\ref{subsub merging}, let
\begin{align*}
    {\be^{(1)}}' \, , \; ... \, , \; {\be^{(M)}}' \, , 
\end{align*}
denote the outgoing physical $2$-waves, constructed from incoming waves $\be^{(1)}$, ..., $\be^{(M)}$. We place the outgoing waves as follows
\begin{align}\label{def of merging rightmost}
    y_{{\be^{(M)}}'} = \max\{\tilde{y}(t_j-)-3\rn^{\co}, y_{\be^{(M)}}\} \, , 
\end{align}
and, for $1 \leq i \leq M-1$,
\begin{align}\label{def of merging mutual}
    y_{{\be^{(i+1)}}'} - y_{{\be^{(i)}}'} = \min\{y_{{\be^{(i+1)}}} - y_{{\be^{(i)}}}, 3\rn^{\co} \} \, .
\end{align}
Intuitively, we place the outgoing $2$-waves as far to the right as possible by shrinking both their mutual distances and the distance between the leftmost outgoing wave and $\tilde y$ to $3\rn^{\co}$ whenever they are larger, and keeping them unchanged otherwise. 
We have that for any $1\leq i\leq M$,
\begin{align*}
    y_{{\be^{(i)}}'} \geq y_{{\be^{(i)}}} \, .
\end{align*}
Hence, this placement verifies Property~\ref{prop 1-wave to left 2-wave to right}. We note that
\begin{align}\label{lower bound of merging after distance}
    y_{{\be^{(1)}}'} - y_b(t_j) > \widetilde{y}(t_j-) - \textup{total number of waves} \times 3\rn^{\co} \, .
\end{align}
Hence, \eqref{lower bound of widetilde y}, Lemma \ref{lemma number of waves and interactions}, and \eqref{rt mft initial values} imply that
\begin{equation}\label{merging after distance}
    \begin{aligned}
        y_{{\be^{(1)}}'} - y_b(t_j)  > \mev - \textup{total number of waves} \times 3\rn^{\co} \geq \frac{\mev}{2} \gg \rn^{\co} \, .
    \end{aligned}
\end{equation}
By \eqref{def of merging rightmost}, \eqref{def of merging mutual}, and \eqref{merging after distance}, we have the following property
\begin{property}\label{prop merging nondecrease distance}
    Merging does not decrease the mutual distance between any two waves, nor the distance between the leftmost wave and the boundary.
\end{property}

\medskip

\noindent
\textbf{Overtakings.} Assume that a non-physical wave lies at distance $d_j$ from the boundary. We remove it by redefining the approximate solution to be equal to the right state of the non-physical wave, so that the solution becomes constant across the former location of the wave. 

\smallskip

We note that the outflow condition \eqref{outflow boundary condition} of the approximate solution may be violated due to the presence of mergings and overtakings. However, we will show that the total error caused by these violations is bounded, and that the limit of the approximate solutions still satisfies the boundary condition. The details are discussed in Section~\ref{sub sec proof of prop limit of approximate solution}.

\medskip
\medskip
We have now described how to place the outgoing waves generated by events that trigger the stopping time. There are two possible scenarios: either exactly one outgoing wave is a rarefaction whose associated ordered set has cardinality greater than $1$, or no such rarefaction exists. In the latter case, we simply restart the clock. In the former case, Lemma~\ref{lemma SS and SR same family} implies that the event triggering the stopping time is an interaction between two shocks in the same family, and that this interaction generates a rarefaction in the other family with strength greater than $\rn^{\cpr}$. In this case, we initiate the following procedure and allow subsequent events to occur, restarting the clock only once the associated ordered set of every rarefaction has cardinality one. In other words, the goal is to obtain the following property.
\begin{property}\label{prop before and after stopping time one element R}
    Immediately before and after each stopping time, the associated ordered set of each rarefaction contains only one element.
\end{property}

\noindent
\textbf{Discretization of a rarefaction}
We denote by $\be=\{b_1,\ldots,b_p\}$ the outgoing rarefaction whose associated ordered set has cardinality greater than $1$, and we denote its family by $i$. If $i=1$, we let $y_{\textup{adj}}$ denote the position of the rightmost wave $\be_{\textup{adj}}$ to the left of $\be$, if such a wave exists, and the position of the boundary $y_b$ otherwise. If $i=2$, we let $y_{\textup{adj}}$ denote the position of the leftmost wave $\be_{\textup{adj}}$ to the right of $\be$, if such a wave exists, and $\infty$ otherwise.

\noindent
\textit{Case 1.}
Assume that
\begin{equation*}
    y_{{\textup{adj}}}
    \begin{cases}
          \leq y_{\be} - 2d_j \, , \quad \textup{if } i = 1 \, , \\
          \geq y_{\be} + 2d_j \, , \quad \textup{if } i = 2 \, .
    \end{cases}
\end{equation*}
Then there is enough space to place a rarefaction. We perform the partition event defined in Section~\ref{subsub partition}, and place the outgoing waves $\be' = \{b_2,\ldots,b_m\}$ and $\widetilde{\be}' = \{b_1\}$ as follows
\begin{equation*}
    y_{\widetilde{\be}'} = y_{\be} \, , \qquad 
    y_{\be'} = 
    \begin{cases}
        y_{\be} - d_j \, , \quad \textup{if } i = 1 \, , \\
        y_{\be} + d_j \, , \quad \textup{if } i = 2 \, .
    \end{cases}
\end{equation*}
Note that this placement verifies Property~\ref{prop 1-wave to left 2-wave to right}.

\smallskip
\noindent
\textit{Case 2.}
Assume that 
\begin{equation*}
    y_{{\textup{adj}}}
    \begin{cases}
          \geq y_{\be} - 2d_j \, , \quad \textup{if } i = 1 \, , \\
          \leq y_{\be} + 2d_j \, , \quad \textup{if } i = 2 \, .
    \end{cases}
\end{equation*}
Suppose $\be_{\mathrm{adj}}$ exists. We perform the passing event defined in Section~\ref{subsub passing} if $\be_{\mathrm{adj}}$ is an $i$-rarefaction. In this case, we obtain two outgoing waves, denoted by $\be'$ and $\widetilde{\be}'$, whose relative placement is the same as that of $\be_{\textup{adj}}$ and $\be$. We place the outgoing waves as follows
\begin{align*}
    y_{{\be}'} = y_{\be_{\mathrm{adj}}} \, , \quad 
    y_{\widetilde{\be}'} = y_{{\be}} \, .
\end{align*}
By the construction of passing in Section \ref{subsub passing}, Property \ref{prop 1-wave to left 2-wave to right} is satisfied.
If $\be_{\mathrm{adj}}$ is not an $i$-rarefaction, then we let $\be$ interact with $\be_{\textup{adj}}$ and place the outgoing waves according to the procedure described earlier.

Suppose that $\be_{\mathrm{adj}}$ does not exist. Then the only remaining possibility is that $i=1$ and $y_b \geq y_{\be}-2d_j$. In this case, we let the corresponding boundary event, i.e., merging, happen.

We repeat the above procedure inductively until the associated ordered set of any rarefaction contains only one element. We then restart the clock.

\medskip
\medskip
Intuitively, the discretization process of a $2$-rarefaction proceeds as follows. We let it travel to the right and discretize it whenever there is enough space; if there is not enough space, we keep moving it to the right until such space becomes available. We continue this process until the rarefaction is fully discretized or becomes a non-physical wave. The discretization process of a $1$-rarefaction is similar, except that we let it travel to the left, with one additional possibility due to the presence of the boundary. When it becomes close to the boundary, we let the boundary event, which is a merging, occur. Then the rarefaction either disappears, or a $2$-rarefaction is generated and awaits discretization. In the latter case, we repeat the discretization process described above for this newly generated $2$-rarefaction.

By Property~\ref{prop introduction of new elements} and Lemma~\ref{lemma SS and SR same family}, at each stopping time, a rarefaction whose generation triggers the discretization process can only be generated by an interaction between two shocks in the same family. We also observe that the incoming waves of each event in the discretization process include at least one rarefaction. Hence, we obtain the following lemma.
\begin{property}\label{prop creation of at most one large rarefaction}
    At each stopping time, at most one rarefaction whose associated ordered set has cardinality greater than $1$ and contains newly generated elements can be created. Moreover, such a rarefaction must be generated by the interaction that triggers the stopping time, and this interaction must be an interaction between two shocks in the same family.
\end{property}
Since the discretization process occurs at most once at each stopping time, Property \ref{prop creation of at most one large rarefaction} and Property~\ref{prop R cardinality} give the following property.
\begin{property}\label{prop events triggered by generation of big R}
Assume that the event triggering the stopping time generates a rarefaction whose associated ordered set contains $p>1$ elements. For each event in the discretization process, exactly one incoming wave and at most one outgoing wave are rarefactions whose associated ordered sets have cardinality greater than $1$. Moreover, the cardinality of each such ordered set does not exceed $p$.
\end{property}

At a stopping time, only interactions and boundary events can change the mutual distances between waves, or the distance between a wave and the boundary. By \eqref{placement of outgoing waves interaction i=1}, \eqref{placement of outgoing waves interaction i=2}, Property~\ref{prop merging nondecrease distance}, and the definition of overtaking, these distances can decrease only when two waves in the same family interact and the outgoing wave in that same family is a shock. By Property~\ref{prop creation of at most one large rarefaction}, the discretization process terminates once such an interaction occurs. Hence, we obtain the following property.
\begin{property}\label{prop minimum distance decrease by interaction}
    Only interactions can decrease the minimum distance $d$ between any two waves, and between the leftmost wave and the boundary. Moreover, at each stopping time, $d$ can decrease by at most $4\rn^{\cz}$.
\end{property}
Together with \eqref{merging after distance}, this yields the following property.
\begin{property}\label{prop at most one merging each stopping time}
    At each stopping time, at most one merging event occurs.
\end{property}

\medskip
Assume that a boundary event triggers the stopping time.
Since interactions are given priority over boundary events, it follows that, at $t_j-$, the minimum distance between two waves, or between a wave and the boundary, is attained only by the distance between the leftmost wave and the boundary. By \eqref{merging after distance}, boundary events increase the distance between the leftmost wave and the boundary. Together with Property \ref{prop minimum distance decrease by interaction}, if the $j$-th stopping time is triggered by a boundary event, then we restart the clock immediately afterward and have
\begin{align*}
    d(t_j+) > d_j = d_{j+1} \, .
\end{align*}
On the other hand, if the $j$-th stopping time is triggered by an interaction, then Property \ref{prop minimum distance decrease by interaction} implies that
\begin{align*}
    d(t_j+) \geq d_j - 4\rn^{\cz} > d_{j+1} \, .
\end{align*}
We note that Lemma \ref{lemma number of waves and interactions} and \eqref{rt mft well defined} imply that 
\begin{align}\label{dj infinity}
    d_j \geq 3\rn^{\co} - C (\rn^{-\cpr} \ln \rn^{-1})^2 \rn^{\cz} > 2\rn^{\co}\, .
\end{align}
This gives the following property.
\begin{property}\label{prop transition zone}
    The transition zones of all waves, as defined in \eqref{def of transition zone}, are pairwise disjoint and remain at a distance of at least $\rn^{\co}$ from the boundary.
\end{property}

\smallskip

For $t$ between the $j$-th and $(j+1)$-th stopping times, we define the approximate solution by
\begin{align}\label{def approximate sol at time t_j}
    \ub (x,t) = \ust - \int_{x}^\infty \sum_{\sse} \p_x {\usn}(y,t) + \sum_{\rse} \p_x \urn(y,t) + \sum_{\npse} \p_x \unn(y) \, dy \, .
\end{align}

\subsubsection{Properties of merging}
Here we derive a lower bound on the distance traveled by the boundary before a merging involving a previously merged wave can occur. This lower bound will be used in Section~\ref{subsec global-in-time estimates} to compensate for the growth of the relative entropy caused by the potentially large number of resettings, which may arise when the boundary moves faster than the $2$-waves in the approximate solution.

\smallskip

We begin by discussing the three possible scenarios for mergings at time $t_*$. Recall the position $y_{\mathcal M}$ of the leftmost previously merged wave, defined in \eqref{def of y_M}, and the right endpoint $\widetilde y$ of the merging interval, defined in \eqref{def of tilde y merging}.
The first scenario occurs when there is a previously merged wave very close to the boundary, namely when
\begin{align*}
    y_{\mathcal{M}}(t_*-) - y_b(t_*) < \mev \, .
\end{align*}
In this case, the merging region is $(y_b(t_*),y_b(t_*)+3\mev)$, and the incoming waves include a previously merged wave.

The second scenario occurs when the leftmost previously merged wave is at an intermediate distance from the boundary, that is when
\begin{align*}
    \mev \leq y_{\mathcal{M}}(t_*-) - y_b(t_*) \leq 3\mev \, .
\end{align*}
The merging region in this case is $(y_b(t_*),y_{\mathcal M}(t_*-))$, which excludes the endpoint $y_{\mathcal M}(t_*-)$. Hence, no previously merged wave is merged.

The third scenario occurs when all previously merged waves are sufficiently far from the boundary, or when no previously merged wave exists. More precisely, this is the case when
\begin{align*}
    y_{\mathcal{M}}(t_*-) - y_b(t_*) > 3\mev \, .
\end{align*}
In this case, the merging region is $(y_b(t_*),y_b(t_*)+3\mev)$, and it contains no previously merged wave.

We note that the incoming waves include a previously merged wave only in the first scenario.

\medskip

Property~\ref{prop merging previously merged wave are 2-wave} and Property ~\ref{prop 1-wave to left 2-wave to right} imply that $y_{\mathcal M}$ can decrease only when a merging occurs and that, between stopping times, $y_{\mathcal M}$ moves to the right with speed at least the minimal speed of the physical $2$-waves in the approximate solution, denoted by
\begin{align}\label{def of the minimal speed of the physical 2-wave}
    \la_{\inf}^{(2)}(\cdot) \, .
\end{align}

We denote the merging times of the first scenario by
\begin{align*}
    s_0:=0<s_1<s_2<... 
\end{align*}
For $j\in \mathbb{N}_+$, we define
\begin{align*}
    \tilde{s}_{j-1} = \inf\bigl\{s_{j-1}<s<s_{j}: y_\mathcal{M}(t) - y_b(t) \leq 3 \mev \textup{ for any } t \in (s,s_j)\bigr\} \, .
\end{align*}
By \eqref{lower bound of merging after distance}, we have 
\begin{align}\label{lower bound of y_M}
    y_\mathcal{M}(\tilde{s}_{j-1}+) - y_b(\tilde{s}_{j-1}+) > 3\mev - \textup{total number of waves} \times 3\rn^{\co} \, .
\end{align}
We also observe that any merging occurring on $(\widetilde{s}_{j-1},s_j)$ must fall under the second scenario.

Recall the definition of the set $\mathcal{M}$ in \eqref{merging def of M}. We claim that each element can enter $\mathcal{M}$ at most once during the evolution. Indeed, if an element leaves $\mathcal{M}$, then the corresponding wave either disappears or becomes non-physical. Assume that the latter case occurs. Property~\ref{prop ordered set NP unchanged} implies that this element remains in the associated ordered set of a non-physical wave. By the construction in \eqref{selection of outgoing waves based on strength merging}, if a non-physical wave is merged, then it disappears immediately. Hence, this element cannot enter $\mathcal{M}$ again. In summary, we obtain the following property.
\begin{property}\label{prop merging any element merged at most twice}
    Over the entire evolution, each element is merged at most twice in merging events in which none of the incoming waves has been previously merged.
\end{property}
Therefore, from $\widetilde{s}_{j-1}+$ to $s_{j}-$, each element can decrease $y_{\mathcal{M}}$ at most once, and by at most $3\rn^{\co}$.
Together with \eqref{lower bound of y_M}, this yields that, on $(\tilde{s}_{j-1},s_{j})$,
\begin{align*}
    y_{\mathcal{M}}(s)-y_b(\tilde{s}_{j-1}) &> 3\mev - \textup{total number of waves}\times 3\rn^{\co}\\
    &\quad -\textup{total number of elements}\times 3\rn^{\co} + \int_{\tilde{s}_{j-1}}^{s} \la_{\inf}^{(2)}(s) \, ds \, ,
\end{align*}
and
\begin{align*}
    y_{\mathcal{M}}(s_{j}-) - y_b(s_{j})<\mev \, .
\end{align*}
Hence, we get
\begin{align*}
    y_b(s_{j}) - y_b(\tilde{s}_{j-1}) &= y_b(s_j) - y_{\mathcal{M}}(s_j-) + y_{\mathcal{M}}(s_j-) - y_b(\tilde{s}_{j-1}) \\
    &> 2\mev - 2\times \textup{total number of elements}\times 3\rn^{\co} +\int_{\tilde{s}_{j-1}}^{s_{j}} \la_{\inf}^{(2)}(s)\, ds\, .
\end{align*}
Lemma \ref{lemma number of waves and interactions} and \eqref{rt mft initial values} imply
\begin{align}\label{merging lower bound of y_b}
    y_b(s_{j}) - y_b(\tilde{s}_{j-1}) - \int_{\tilde{s}_{j-1}}^{s_{j}} \la_{\inf}^{(2)}(s)\, ds > \mev \gg 3\rn^{\co} \times \textup{total number of elements}\, .
\end{align}

\subsection{Bounds on the number of elements and events}\label{sub number of waves and elements}
In this subsection, we count the total number of elements and events.

In Section~\ref{subsub mft total variation}, we first prove the well-definedness of the approximate solution and show that the tracked strength of each element is controlled by its initial value, up to a multiplicative constant. To this end, we introduce a weighted total variation $\ml$, which assigns larger weights to the strengths of elements with higher generation order, instead of using the usual sum of the total variation and the Glimm interaction potential. By proving that this weighted total variation is controlled by its initial value, we obtain sharper bounds on the strengths of elements of higher generation order. The detailed proof is given in Appendix~\ref{sec proof of lemma L+kQ}, where we simultaneously prove that the tracked strength of each element is bounded by its initial value through the construction of an equivalent tracked strength $g_b$, defined in \eqref{def of total variation g}. The construction of $g_b$ is inspired by \cite{prepweakbv}. 

Then, in Section~\ref{subsub bounds on number of waves and events}, we count the total number of elements and different types of events throughout the evolution. In standard counting arguments, such as that in \cite{bressanbook}, one counts the number of waves. Here, however, we work with elements in order to prove in Section~\ref{sub sec proof of prop limit of approximate solution} that any limit of the approximate solutions satisfies the boundary condition. Therefore, instead of introducing a convention for identifying incoming and outgoing waves as the same wave, we directly count the total number of elements throughout the evolution by introducing the notion of flagged elements.

Finally, in Section~\ref{subsub bounds on number of large waves and events}, the framework of elements and associated ordered sets developed above allows us to obtain a sharper estimate on the number of interactions for which at least one incoming wave has strength greater than or equal to a prescribed threshold.

\subsubsection{Control of the total variation and the tracked strength}\label{subsub mft total variation}
We start with a lemma collecting estimates on the change in tracked strength when different events occur. When an event occurs, we define the tracked strength after the event proportionally to its tracked strength before the event. In this way, the change in wave strength is reflected proportionally in the change of the tracked strength. The estimates for merging, partition, and passing follow directly from the definition, while the estimates for interactions follow from Lemma~\ref{lemma interaction wave strength estimates} and Lemma~\ref{lemma interaction wave strength estimates NP}. A detailed proof of the interaction estimates, in the same spirit, can be found in \cite{prepweakbv}.

\begin{lemm}\label{lemma tracked strength at an event}
Assume that an event occurs at time $t$. Then there exists a constant $C>0$, depending only on $U_*$, such that the following estimates hold.

Let $b$ be an element that belongs to the associated ordered set of one of the incoming waves and remains in the associated ordered set of one of the outgoing waves. Then
\begin{align*}
    \textup{sign}\bigl(\s_b(t+)\bigr) = \textup{sign}\bigl(\s_b(t-)\bigr) \, .
\end{align*}
If a merging occurs, then
\begin{align*}
    |\s_b(t+)| - |\s_b(t-)| \leq 0 \, .
\end{align*}
If a partition or a passing occurs, then
\begin{align*}
    \s_b(t+) = \s_b(t-) \, .
\end{align*}
Assume an interaction occurs. We denote by $\widetilde{\s}$ the strength of the other incoming wave, namely the one whose associated ordered set does not contain $b$. If the two waves interacting belong to different families, then 
\begin{equation*}
    |\s_b(t+)-\s_{b}(t-)| \leq C |\s_b(t-)||\widetilde{\s}| \, ,
\end{equation*}
If the two waves interacting belong to the same family, then
\begin{align*}
    |\s_b(t+)-\s_b(t-)| \leq C |\s_b(t-)||\widetilde{\s}|\, \bigl(|\s|+|\widetilde{\s}|\bigr) \, ,
\end{align*}
where $\s$ denotes the strength of the incoming wave whose associated ordered set contains $b$.
\end{lemm}

We associate with each element $b$ the function
\begin{align}\label{def of total variation g}
    g_b(t) = |\s_b(t)| \exp\{- C_g \sum |\s_{\tilde{b}}(s_{\tilde{b}}-)|\} \, , 
\end{align}
where the sum is taken over all pairs $(\widetilde b,s_{\widetilde b})$ such that $b$ interacts with $\widetilde b$ at time $s_{\widetilde b}\in[0,t)$. In Lemma~\ref{lemma L+kQ}, we fix the constant $C_g>0$ and show that $g_b$ is equivalent to $|\s_b|$.

We then introduce a weighted total variation at time $t$ by
\begin{align}\label{def of weighted total variation}
    \ml(t) = \sum_{k\geq 1} \ml(0)^{-k+1} \ml_k(t) \, ,
\end{align}
where
\begin{align}\label{def of k weighted total variation}
    \ml_k(t) = \sum_{\{b: \textup{ element of order } k \textup{ at time } t\}} g_{b}(t) \, ,
\end{align}
denotes the sum of the weighted strengths of all elements of generation order $k$ at time $t$. Note that at the initial time, all elements have generation order $1$, and $g_b$ coincides with $|\s_b|$. Hence, $\ml(0)$ coincides with the initial total variation.

For each element $b$, we define its generation time 
\begin{align}\label{def of generation time}
    t_0^b
\end{align}
as the earliest time at which it belongs to the associated ordered set of a wave.
We define the total initial strength of all elements generated over all time by
\begin{align*}
    \ml_{total} = \sum_{\{b:0\leq t_0^b<\infty\}} |\s_b(t_0^b)| \, .
\end{align*}
Note that $\ml_{total}$ does not depend on time.

In the following lemma, we show that the weighted total variation plays the same role as the sum of the total variation and a large constant times the interaction potential in the usual front tracking algorithm. In particular, we show that $g_b$ is equivalent to $|\s_b|$, that the weighted total variation is controlled by the weighted total variation at the initial time, and that, at each interaction, the weighted total variation decreases by an amount proportional to the product of the strengths of the two incoming waves. Moreover, we show that the tracked strength of any element is bounded, up to a constant, by its strength at the generation time, and that the total generation strength of all elements is controlled by the weighted total variation at time $0$. The proof of the following lemma is given in Appendix~\ref{sec proof of lemma L+kQ}.

\begin{lemm}\label{lemma L+kQ}
For any $U_*\in \om_{sub}^-$, there exists $\e_0, C_g, C >0$ such that the following is true. 

\noindent
Let $C_g$ be the constant defining $g_b$ in \eqref{def of total variation g}, and let $\ml(0)=\e<\e_0$. Then the approximate solution defined as in Section~\ref{subsec mft} is globally well-defined. For any $t \in \mathbb{R}_+$,
\begin{equation*}
    g_b(t)\leq |\s_b(t)|\leq 2 g_b(t) \, .
\end{equation*}
Moreover, $\ml$ is non-increasing, and if two waves of strengths $\s$ and $\widetilde{\s}$ interact at time $t_*$, then
\begin{align*}
    \Delta \ml (t_*) \leq -C|\s||\widetilde{\s}| \, .
\end{align*}
For each element $b$ generated at time $t_0^b$, we have
\begin{align*}
    |\s_b(t)|\leq 2|\s_b(t_0^b)| \, .
\end{align*}
for any $t\geq t_0^b$ before the element disappears. Finally, we have
\begin{equation*}
    \ml_{total} \leq 2 \ml(0) \, .
\end{equation*}
\end{lemm}

We record the following corollary. In particular, the first statement follows directly from the definition, since we label a physical wave as non-physical once its strength drops below $\rn$. The second statement is implied by Property~\ref{prop newly generated R strength}, Lemma~\ref{lemma L+kQ}, and Property~\ref{prop before and after stopping time one element R}. Finally, Lemma~\ref{lemma L+kQ} and Property~\ref{prop ordered set NP unchanged} imply the third statement.

\begin{coro}\label{lemma strength of waves}
There exists $C>0$ such that, between stopping times, the strengths of physical and non-physical waves satisfy the following properties:
\begin{itemize}
    \item[1)] The strength of any physical wave is bounded from below by $\rn$.
    \item[2)] The strength of each rarefaction is bounded from above by $C\rn^{\cpr}$.
    \item[3)] The strength of each non-physical wave is bounded from above by $C\rn$.
\end{itemize}    
\end{coro}

By the construction of the initial value in Section \ref{subsub initial step}, Property \ref{prop newly generated R strength}, Lemma \ref{lemma L+kQ}, and Property \ref{prop events triggered by generation of big R}, we have the following property.
\begin{property}\label{prop strength R discretization}
    Assume that the event triggering the stopping time generates a rarefaction of strength $\s''$ whose associated ordered set contains $p>1$ elements. For each event in the subsequent discretization process, the strength of any incoming rarefaction is bounded by both $Cp\rn^{\cpr}$ and $\s''$.
\end{property}

\subsubsection{Bounds on the number of elements and events}\label{subsub bounds on number of waves and events}

We count the total number of elements created throughout the evolution. Recall Property~\ref{prop introduction of new elements}. Apart from the initial time, new elements are introduced only by interactions between two waves in the same family. According to Section~\ref{subsub interaction two physical in same family}, more than one new element can be created only when two shocks in the same family interact and the strength of the outgoing rarefaction in the other family exceeds $\rn^{\cpr}$. In fact, Property~\ref{prop newly generated R strength} and Lemma~\ref{lemma L+kQ} give the following upper bound on the total number of elements created through such events throughout the evolution:
\begin{align*}
    \frac{\ml_{total}}{\rn^{\cpr}/2} \leq \rn^{-\cpr} \, ,
\end{align*}
We record this estimate as a lemma for future reference.

\begin{lemm}\label{lemma total number of R introduced big fan}
    The total number of elements created throughout the evolution by events that introduce more than one new element is bounded above by $\rn^{-\cpr}$.
\end{lemm}

The generation order of any element created throughout the evolution can also be bounded from above. Assume that, at time $t$, two waves in the same family interact and that the outgoing wave in the other family is physical. Its associated ordered set may contain one or more elements, all with the same generation order $k$, and the tracked strength of each element is bounded below by $\rn$.  Lemma~\ref{lemma L+kQ}, \eqref{def of weighted total variation}, and \eqref{def of k weighted total variation} imply that
\begin{align*}
    \rn \leq 2\ml_{k}(t+) \leq 2\ml(0)^k \, .
\end{align*}
Hence, $k$ is bounded above by $C\ln\rn^{-1}$. By Property \ref{prop introduction of new elements}, Property~\ref{prop new elements order}, and Property~\ref{prop ordered set NP unchanged}, the maximum generation order among the elements in the associated ordered set of a non-physical wave is at most one greater than the maximum generation order among the elements in the associated ordered set of a physical wave. Therefore, we obtain the following lemma. 
\begin{lemm}\label{lemma max generation order}
    The generation order of any element is bounded above by $C\ln\rn^{-1}$.
\end{lemm}

Typically, one counts the total number of waves throughout the evolution. This is simpler because, at each event, the number of outgoing waves is no greater than the number of incoming waves, except for interactions between two waves in the same family that produce more than one partitioned rarefaction in the other family, as mentioned above. Together with a bound on the total number of such rarefactions throughout the evolution, obtained by an argument similar to that in Lemma~\ref{lemma total number of R introduced big fan}, and the notion of generation order, this gives the standard counting argument for the total number of waves (see, for example, Chapter~7 of \cite{bressanbook}).

Here, however, we develop the notion of elements in order to show that any limit of the approximate solutions satisfies the outflow boundary condition \eqref{outflow boundary condition}, as will be discussed in Section~\ref{sub sec proof of prop limit of approximate solution}. Instead of introducing a convention for identifying incoming and outgoing waves as the same wave, we count the total number of elements throughout the evolution directly. Unlike in the counting of waves, when two shocks in the same family interact, the total number of elements in the associated ordered sets of the outgoing waves is greater than that of the incoming waves. To address this issue, we introduce the notion of flagged elements as follows.

\smallskip

\noindent
\textbf{Flagged elements.}
At the initial time $0$, every element is flagged. Assume that an event occurs at time $t$. We flag all newly introduced elements. If an outgoing wave is a shock or a non-physical wave whose associated ordered set contains more than one element $b_1,\ldots,b_n$ that was flagged before the event, we unflag all but one of these elements, choosing an element $b_l$ that had minimal generation order before the event.

\smallskip
We observe that a flagged element may become unflagged, whereas an unflagged element is never flagged again. We also note the following.
\begin{property}\label{prop flagged}
    The associated ordered set of a shock or a non-physical wave contains exactly one flagged element, while every element in the associated ordered set of a rarefaction is flagged.
\end{property}
By Property~\ref{prop before and after stopping time one element R}, between stopping times, each associated ordered set contains exactly one flagged element. Moreover, every element is flagged when it is created. Hence, the total number of elements that are ever flagged throughout the evolution is exactly the total number of elements created throughout the evolution. 

It suffices to count the number of flagged elements. First, by \eqref{control of N(0)}, the number of flagged elements at the initial time is bounded by $C\rn^{-\cpr}$. After the initial time, new flagged elements are introduced only when two waves in the same family interact. There are two cases: either exactly one new flagged element is introduced, or more than one new flagged element is introduced. By Lemma~\ref{lemma total number of R introduced big fan}, the total number of flagged elements introduced in the second case throughout the evolution is bounded by $\rn^{-\cpr}$.

For the first case, Property~\ref{prop flagged} and Property \ref{prop new elements order} imply that the number of flagged elements in the associated ordered sets of the outgoing waves is no greater than that in the associated ordered sets of the incoming waves, and that the generation order of the new flagged element is one greater than that of a flagged element in the associated ordered sets of the incoming waves that becomes unflagged or disappears after the interaction.

Thus, at any time, the number of flagged elements is bounded by the sum of the number of flagged elements at the initial time and the total number of flagged elements introduced in the second case, and is therefore bounded by $C\rn^{-\cpr}$. Whenever the first case occurs, one flagged element either becomes unflagged or disappears, and a new flagged element of generation order one greater is created. Together with Lemma \ref{lemma max generation order}, we conclude that the total number of elements that are ever flagged throughout the evolution is bounded above by $C\rn^{-\cpr}\ln\rn^{-1}$. This gives the first estimate in the following lemma.

Remark \ref{remark interaction at most two times} gives the second estimate in the following lemma. Moreover, each overtaking causes at least one element to disappear, so the total number of overtakings is bounded by the total number of elements, which yields the third estimate.

By \eqref{merging lower bound of y_b} and Lemma~\ref{lemma uniform bounds lagrangian}, we obtain the fourth statement in the following lemma. Property~\ref{prop merging any element merged at most twice} implies that the total number of mergings for which none of the incoming waves has been previously merged is bounded by the total number of elements. Together with \eqref{rt mft optimal control on number}, this gives the final estimate in the lemma.

\begin{lemm}\label{lemma number of waves and interactions}
There exists a constant $C>0$ such that the following is true. 

\noindent
Throughout the evolution, the total number of elements is bounded by
\begin{align*}
    C \rn^{-\cpr} \ln \rn^{-1} \, ,
\end{align*}
the total number of interactions is bounded by 
\begin{align*}
    C \bigl( \rn^{-\cpr} \ln \rn^{-1} \bigr)^2 \, ,
\end{align*}
the total number of overtakings is bounded by
\begin{align*}
    C \rn^{-\cpr} \ln \rn^{-1} \, ,
\end{align*}
the total number of mergings that merge a previously merged wave is bounded by
\begin{align*}
    C \mev^{-1} \, ,
\end{align*}
the total number of mergings is bounded by
\begin{align*}
    C \rn^{-\cpr} \ln \rn^{-1} \, .
\end{align*}
\end{lemm}

\subsubsection{Bounds on the number of interactions involving a large wave}\label{subsub bounds on number of large waves and events}
We obtain a sharper estimate for the total number of interactions throughout the evolution in which at least one incoming wave is large, namely, has strength greater than or equal to a prescribed threshold $\rn^c$, for some $0<c<\cpr$. This yields a sharper estimate for the total growth of the relative entropy caused by interactions. The details are given in Section~\ref{sec global in time estimates}.

\begin{lemm}\label{lemma number of big interactions}
There exists a constant $C>0$ such that the following is true. 

\noindent
Over all time, the total number of interactions for which at least one incoming wave has strength greater than or equal to $\rn^c$, for some $0 < c < \cpr$, is bounded by
\begin{align*}
    C \rn^{-c-\cpr} \ln \rn^{-1}   \, .
\end{align*}
\end{lemm}
\begin{proof}
We first count the total number of interactions involving a large shock. We consider the entire evolution and select the associated ordered sets of those large shocks for which none of the elements in the associated ordered set has previously belonged to the associated ordered set of another large shock. Property~\ref{prop associated ordered set S} implies that these selected associated ordered sets are distinct from one another. Together with Lemma~\ref{lemma L+kQ}, the total number of such selected associated ordered sets is bounded above by
\begin{align*}
    \frac{\ml_{total}}{\rn^c}\leq \rn^{-c} \, .
\end{align*}
Property~\ref{prop associated ordered set S} implies that the associated ordered set of any large shock contains one of these selected associated ordered sets. Together with Remark~\ref{remark interaction at most two times}, this implies that the total number of interactions involving a large shock is bounded by $\rn^{-c}$ times the total number of elements.

Then we count the total number of interactions involving a big rarefaction. By the second statement in Corollary~\ref{lemma strength of waves}, there are no large rarefactions between stopping times. Hence, interactions involving a large rarefaction can occur only during the discretization process. Property \ref{prop strength R discretization} and Lemma~\ref{lemma L+kQ} imply that the total number of interactions involving an incoming large rarefaction is bounded, up to a constant, by $\rn^{-c}$ times the total number of elements.

Together with Lemma~\ref{lemma number of waves and interactions}, we conclude the proof.
\end{proof}

\section{Weighted Relative entropy with shifts}\label{sec energy estimate between stopping times}
We first derive the energy estimate between stopping times. In this section, we apply the relative entropy method, compute the resulting terms, and group them as in \eqref{relative entropy method terms}. The terms associated with rarefaction waves are treated in Section~\ref{sec rarefaction}, while those associated with nonphysical waves and the boundary are handled in Section~\ref{sec NP and boundary}. The parabolic terms are estimated in Section~\ref{sec parabolic}. Finally, the terms associated with the weights and shocks are analyzed in Section~\ref{section shock}. We summarize the resulting energy estimate between stopping times in Proposition~\ref{proposition energy estimate}. Finally, in Section~\ref{subsec useful estimates}, we present several useful estimates, including estimates for the waves and interaction terms, and a quantitative estimate based on the subsonicity of the approximate solutions.

We work with the Navier-Stokes system \eqref{NS system BD} with viscosity parameter $\nu$. For convenience, until the beginning of Section \ref{sub sec proof of prop limit of approximate solution}, we omit the explicit dependence on $\nu$ in our notation and write
\begin{align*}
    U = U^\nu \, , \quad \bar{U} = \bar{U}_\nu \, ,\ \quad X = X^\nu \, .
\end{align*}

\subsection{Choice of $\e_0$}\label{subsec choice of e_0}
We show that there exists a choice of $\e_0>0$, depending only on $U_*$, such that Theorem~\ref{thm v,h} holds. Throughout Sections~\ref{sec mft} to \ref{sec global in time estimates}, various constants and constraints are introduced, leading to many conditions on the choice of $\e_0$. To avoid ambiguity and any potential circularity in the argument, we explain here how the constants are fixed in sequence and carefully track the dependence of each constraint. Note that throughout the process, we repeatedly impose additional upper bounds on $\e_0$, allowing its value to decrease as needed.

First, we impose a smallness requirement on $\e_0$ and fix $\La_0,C_X$ according to Lemma~\ref{lemma estimates on relative functions 2}. Note that all the choices made at this stage depend only on $U_*$. Moreover, $\La_0$ quantifies the subsonicity of the wave speeds in the front tracking approximate solution,
while $C_X$ defines the shifts imposed on the shocks through \eqref{ODE of Xi}. We then fix $C_w$, depending only on $U_*$, as in Lemma~\ref{lemma L + q + omega}, which ensures a sufficient decrease of the weight $a$, defined in \eqref{def of weight}, at interactions, passings, and partitions.

Then we choose $\de_1$, the small threshold used to truncate the specific volume $v$ of the Navier-Stokes solution in order to apply the Poincar\'e-type inequality in Section~\ref{subsection estimate inside the truncation}. We take $\de_1$ to be smaller than $1$ and than all the constants $\de_0$ in Lemma \ref{lemma relative evaluations}, Lemma~\ref{lemma smallness of localized relative entropy 1}, Lemma~\ref{lemma smallnss of Y_i}, and Lemma~\ref{lemma inside truncation}, where $C_1$ in Lemma~\ref{lemma inside truncation} is chosen as in Lemma~\ref{lemma smallnss of Y_i}, so that the truncation is sufficiently localized for the Poincar\'e-type inequality to yield the desired cancellation. All these constraints depend only on $U_*$. 

Next, we choose $\de_*$, the constant defining the weights in \eqref{def of weight a_i} and \eqref{def of weight}, sufficiently small so that \eqref{relation delta 8}, \eqref{relation delta 5}, \eqref{relation delta 1}, \eqref{relation delta 2}, \eqref{relation delta 3}, \eqref{relation delta 4}, \eqref{relation delta 6}, \eqref{relation delta 7}, and \eqref{relation delta 9} are satisfied. These conditions allow us to obtain in Section~\ref{section shock} the desired energy estimates for the terms localized by the derivatives of the shocks, as stated in Proposition~\ref{proposition shock estimate}. We note that $\de_*$ is chosen sufficiently small compared to $\de_1$, and that all these constraints depend only on $U_*,\de_1,C_X,\La_0$.

Finally, we choose $\e_0$ to be smaller than all the constants denoted by $\e_0$ in Lemma~\ref{lemma uniform bounds lagrangian}, Lemma~\ref{lemma L+kQ}, Lemma~\ref{lemma estimates on relative functions 2}, and Lemma~\ref{lemma L + q + omega}, and also smaller than $C_0\de_1$, where $C_0$ is the constant in Lemma~\ref{lemma relative evaluations}. Moreover, we choose it sufficiently small so that Lemma~\ref{lemma global in time estimates} can be applied in Section~\ref{subsec global-in-time estimates}. In addition, we impose
\begin{align}\label{weight relation between L and de*}
    \ml(0) \leq \e_0 \leq \de_*^3 \, ,
\end{align}
where $\ml$ denotes the weighted total variation at time $t$, as defined in \eqref{def of weighted total variation}. All these constraints depend only on the constants fixed above, which in turn depend only on $U_*$. Hence, $\e_0>0$ can be chosen depending only on $U_*$.

\subsection{Construction of the weights}\label{subsec construction of weights}
We introduce the weight $a$. For each viscous shock $\su_\be$, we introduce an associated weight $a_\be$ satisfying
\begin{equation}\label{def of weight a_i}
    (a_\be)_x = - \frac{p(\sv_\be)_x}{\de_*} \, .
\end{equation}
We have 
\begin{align}\label{weight equivalence}
    |(a_\be)_x| \sim \frac{|(\sv_\be)_x|}{\de_*} \sim \frac{|(\sh_\be)_x|}{\de_*} \, ,
\end{align}
Moreover, by \eqref{def shifted viscous shock}, 
\begin{align}\label{time derivative of a_i}
    (a_\be)_t = (- \dot{X}_\be - \las_\be) (a_\be)_x \, .
\end{align}
We define the weight $a$ by
\begin{align}\label{def of weight}
    a(x,t) = 1 + \frac{C_w}{\de_*}\ml(t) + \int_{\infty}^x \sum_{\be \in \mathcal{S}} \p_x a_\be(y,t) \, dy \, ,
\end{align}
where $C_w$ was fixed earlier in this subsection as the constant chosen in Lemma~\ref{lemma L + q + omega}, depending only on $U_*$. Lemma \ref{lemma L+kQ} and \eqref{weight relation between L and de*} imply that
\begin{align*}
    |a-1| \leq (C_{w}+\mathbf{C})\de_*^2 \, ,
\end{align*}
for some $\mathbf{C}>0$ depending only on $U_*$. We can take $\de_*$ small enough such that
\begin{align}\label{relation delta 8}
    (C_{w}+\mathbf{C})\de_*^2 \leq \frac{1}{4}
\end{align}
and get
\begin{align}\label{weight a close to 1}
    |a-1| \leq \frac{1}{4}  \, .
\end{align}

\subsection{Relative entropy method}
We rewrite \eqref{NS system BD} as a viscous perturbation of a hyperbolic system of conservation laws
\begin{align*}
    \p_t U + \p_x A(U) = 
    \begin{pmatrix}
        -\nu \bigl( \mu(v) v^\ga p(v)_x \bigr)_x \\
        0
    \end{pmatrix} \, ,
\end{align*}
where
\begin{align*}
    U = 
    \begin{pmatrix}
        v \\ h
    \end{pmatrix} \, , \quad
    A(U) = 
    \begin{pmatrix}
        -h \\ p(v)
    \end{pmatrix} \, .
\end{align*}
Recall \eqref{derivative of relative entropy}. Let us consider the case $V=\bu$. Since non-physical waves are stationary, we get
\begin{align*}
    \p_t \bu + \p_x A(\bu) &= \sum_{\be \in \mathcal{S}} \bigl(\p_t \su_\be + \nabla A(\bu)\p_x \su_\be\bigr) + \sum_{\be \in \mathcal{R}} \bigl(\p_t \ru_\be + \nabla A(\bu)\p_x \ru_\be\bigr) \\
    &\quad \, + \sum_{\be \in \mathcal{NP}}  \nabla A(\bu)\p_x \nnu_\be \, .
\end{align*}
By \eqref{equations - viscous shock}, we have
\begin{align*}
    \p_t \su_\be + \sd_\be \p_x \su_\be + \nabla A(\su_\be)\p_x \su_\be = 
    \begin{pmatrix}
        - \nu \bigl(\mu(\sv_\be) (\sv_\be)^{\ga} p(\sv_\be)_x \bigr)_x \\ 0
    \end{pmatrix}  
    \, .
\end{align*}
By \eqref{def shifted rarefaction}, we get
\begin{align*}
    \p_t \ru_\be = - \la_\be(\ru_{\be,-}) \p_x \ru_\be \, ,
\end{align*}
where $i$ is the characteristic family of $\ru_\be$ and $\ru_{\be,-}$ is the left state of $\ru_\be$.
Therefore, we get
\begin{align*}
    \p_t \bu + \p_x A(\bu) &= 
    \sum_{\be \in \mathcal{S}} \Bigl[- \sd_\be \p_x \su_\be + \bigl(\nabla A(\bu) - \nabla A(\su_\be) \bigr) \p_x \su_\be + 
    \begin{pmatrix}
        - \nu \mu(v) v^{\ga} p(v)_x \\ 0
    \end{pmatrix}  \Bigr] \\
    &\quad \, + \sum_{\be \in \mathcal{R}} \Bigl[ \bigl(- \la_\be(\ru_{\be,-}) + \nabla A(\ru_\be) \bigr)\p_x \ru_\be + \bigl( \nabla A(\bu) - \nabla A(\ru_\be) \bigr)\p_x \ru_\be
    \Bigr] \\
    &\quad \, + \sum_{\be \in \mathcal{NP}} \nabla A(\bu) \p_x \nnu_\be \, .
\end{align*}
We substitute this equality into \eqref{derivative of relative entropy}. Together with \eqref{useful computation}, we get
\begin{align*}
    \p_t \eta(U|\bu) &= - \p_x \Bigl( \bigl(p(v)-p(\bv)\bigr)(h-\bh) \Bigr) \\
    &\quad \, - \sum_{\be \in \mathcal{S}} \p_x \sh_\be \, p(v|\bv) - \sum_{\be \in\mathcal{R}} \nabla^2 \eta(\bu) \p_x \ru_\be A(U|\bu) - \sum_{\be \in \mathcal{NP}} \p_x \nh_\be \, p(v|\bv)  \\
    &\quad \, +  \nabla^2 \eta(\bu) (U-\bu) \sum_{\be \in \mathcal{S}} \sd_\be \p_x \su_\be
    - (h-\bh)\sum_{\be \in \mathcal{S}} \bigl( p'(\bv) - p'(\sv_\be) \bigr) (\sv_\be)_x \\
    &\quad \, - \nu p'(\bv)(v-\bv) \sum_{\be \in \mathcal{S}} \bigl(\mu(\sv_\be) (\sv_\be)^{\ga} p(\sv_\be)_x \bigr)_x \\
    &\quad \, - \nabla^2 \eta(\bu) (U-\bu) \sum_{\be \in \mathcal{R}} \bigl(- \la_\be(\ru_{\be,-}) + \nabla A(\ru_\be) \bigr)\p_x \ru_\be \\
    &\quad - (h-\bh)\sum_{\be \in \mathcal{R}} \bigl( p'(\bv) - p'(\rv_\be) \bigr) (\rv_\be)_x \\
    &\quad \, - p'(\bv)(v-\bv) \sum_{\be \in \mathcal{NP}} \p_x \nh_\be - (h-\bh)p'(\bv) \sum_{\be \in \mathcal{NP}} \p_x \nv_\be \\
    &\quad \, + \nu \bigl(p(v)-p(\bv)\bigr) \bigl( \mu(v) v^{\ga} p(v)_x \bigr)_x \, .
\end{align*}
By \eqref{equations - viscous shock 2}, we have
\begin{align*}
    - \nu p'(\bv)(v-\bv) \sum_{\be \in \mathcal{S}} \bigl(\mu(\sv_\be)(\sv_\be)^{\ga} p(\sv_\be)_x \bigr)_x &= p(v|\bv)\sum_{\be \in \mathcal{S}} ({\las}_\be \p_x\sv_\be + \p_x\sh_\be )\\
    &\quad \, - \nu\bigl(p(v)-p(\bv)\bigr)\sum_{\be \in \mathcal{S}}
    \bigl(\mu(\sv_\be)(\sv_\be)^{\ga} p(\sv_\be)_x \bigr)_x \, .
\end{align*}
Hence, we have
\begin{align*}
    \p_t \eta(U|\bu) = Y + H_1 + H_2 + H_3 + P \, ,
\end{align*}
where
\begin{align*}
    Y &= \nabla^2 \eta(\bu) (U-\bu) \sum_{\be \in \mathcal{S}} \sd_\be \p_x \su_\be \, , \\
    H_1 &= - \p_x \Bigl( \bigl(p(v)-p(\bv)\bigr)(h-\bh) \Bigr) + p(v|\bv)\sum_{\be \in \mathcal{S}} {\las}_\be \p_x\sv_\be - (h-\bh)\sum_{\be \in \mathcal{S}} \bigl( p'(\bv) - p'(\sv_\be) \bigr) (\sv_\be)_x \, , \\
    H_2 &= - \sum_{\be \in\mathcal{R}} \nabla^2 \eta(\bu) \p_x \ru_\be A(U|\bu) - (h-\bh)\sum_{\be \in \mathcal{R}} \bigl( p'(\bv) - p'(\rv_\be) \bigr) (\rv_\be)_x \\
    &\quad - \nabla^2 \eta(\bu) (U-\bu) \sum_{\be \in \mathcal{R}} \bigl(- \la_\be(\ru_{\be,-}) + \nabla A(\ru_\be) \bigr)\p_x \ru_\be \, , \\
    H_3 &= - \bigl(p'(v) - p'(\bv) \bigr) \sum_{\be \in \mathcal{NP}} \p_x \nh_\be - (h-\bh)p'(\bv) \sum_{\be \in \mathcal{NP}} \p_x \nv_\be \, , \\
    P &= \nu \bigl(p(v)-p(\bv)\bigr) \Bigl(\mu(v) v^{\ga} p(v)_x - \sum_{\be \in \mathcal{S}}\mu(\sv_\be)(\sv_\be)^{\ga} p(\sv_\be)_x \Bigr)_x \, .
\end{align*}
Together with \eqref{time derivative of a_i}, we obtain the following. The boundary terms $\mathcal{B}$ arise from the first terms of $H_1$ and $P$ through integration by parts. Together with \eqref{def of h} and Property \ref{prop transition zone}, we get
\begin{equation}\label{relative entropy method inequality}
    \begin{aligned}
        \frac{d}{dt} \int_{y_b(t)}^\infty a \eta(U|\bu) \, dx =& - y_b'(t) a \eta(U|\bu)|_{x=y_b(t)} + \mathcal{B} (U) + \sum_{\be \in\mathcal{S}} \sd_\be(t) Y_\be  \\
        & + \mathcal{H}_1 (U) + \mathcal{H}_2 (U) + \mathcal{H}_3 (U)
        + \mathcal{P} (U) - \mathcal{J}_{good}  (U) \, ,
    \end{aligned}
\end{equation}
where 
\begin{equation}\label{relative entropy method terms}
\begin{aligned}
    &\mathcal{B} (U) = \Bigl[ a \bigl(p(v)-p(\bv)\bigr)(u-\bar{u}) \Bigr]\Big|_{x=y_b(t)} \, ,\\
    &Y_\be (U) = - \inb (a_\be)_x \eta(U|\bu) \dx + \inb a \nabla^2 \eta(\bu) (U-\bu) \p_x \su_\be \dx \, ,\\
    &\mathcal{H}_1 (U) = \sum_{\be \in \mathcal{S}} \Bigl[\inb (a_\be)_x \bigl(p(v)-p(\bv)\bigr) (h-\bh) \dx + \las_\be \inb a (\sv_\be)_x p(v|\bv) \dx\Bigr] \\
    &\quad\quad\quad\quad - \sum_{\be \in \mathcal{S}} \inb a (h-\bh) \bigl( p'(\bv) - p'(\sv_\be) \bigr) (\sv_\be)_x \dx \, , \\
    &\mathcal{H}_2 (U) = - \sum_{\be \in\mathcal{R}} \inb a \nabla^2 \eta(\bu) \p_x \ru_\be A(U|\bu) \dx - \sum_{\be \in \mathcal{R}} \inb a (h-\bh) \bigl( p'(\bv) - p'(\rv_\be) \bigr) (\rv_\be)_x \dx \\
    &\quad\quad\quad\quad - \sum_{\be \in \mathcal{R}} \inb a \nabla^2 \eta(\bu) (U-\bu) \bigl(- \la_\be(\ru_{\be,-}) + \nabla A(\ru_\be) \bigr)\p_x \ru_\be \dx \, , \\
    &\mathcal{H}_3 (U) = - \sum_{\be \in \mathcal{NP}} \inb a \bigl(p(v)-p(\bv)\bigr) \p_x \nh_\be \dx - \sum_{\be \in \mathcal{NP}} \inb a (h-\bh)p'(\bv) \p_x \nv_\be \dx \, , \\
    &\mathcal{P} (U) = -\nu \inb \p_x \Bigl(a\bigl(p(v)-p(\bv)\bigr) \Bigr) \bigl(\mu(v) v^{\ga} p(v)_x - \sum_{\be \in \mathcal{S}}\mu(\sv_\be)(\sv_\be)^{\ga} p(\sv_\be)_x \bigr) \dx \, , \\
    &\mathcal{J}_{good} (U) = \sum_{\be \in \mathcal{S}}\inb \las_\be (a_\be)_x \eta(U|\bu) \dx \, .
\end{aligned}
\end{equation}

The hyperbolic terms $\mathcal{H}_1 $, $\mathcal{H}_2 $, and $\mathcal{H}_3 $ contain terms localized by the derivatives of the weights and shocks, by the derivatives of rarefactions, and by the derivatives of non-physical waves, respectively. The term $\mathcal{B} $ denotes part of the boundary contribution. These are the terms that need to be controlled. The estimate on $\mathcal{H}_2 $ is given in Section~\ref{sec rarefaction}, while the estimates on $\mathcal{H}_3 $ and $\mathcal{B} $ are given in Section~\ref{sec NP and boundary}. In Section~\ref{section shock}, we use the diffusion term obtained from $\mathcal{P} $ in Section~\ref{sec parabolic}, together with the terms $\sum_\be \dot X_\be Y_\be $ arising from the introduction of shifts on the shocks, to control $\mathcal{H}_1 $ through a Poincar\'{e}-type inequality. We summarize the result in the following proposition. The proof is given in Section~\ref{subsec proof of energy estimate between stopping times}. 

Finally, We note that the boundary term $-y_b'(t)a\eta(U|\bu)|_{x=y_b(t)}$ is used later in Section~\ref{sec global in time estimates} to control the increase in relative entropy caused by mergings.

\medskip
Sections~\ref{sec rarefaction}, \ref{sec NP and boundary}, \ref{sec parabolic}, and \ref{section shock} are devoted to deriving the following energy estimate between stopping times from \eqref{relative entropy method inequality}. The proof of the proposition, which summarizes the results obtained in these sections, is given in Section~\ref{subsec proof of energy estimate between stopping times}.

\begin{prop}\label{proposition energy estimate}
For all $t$ between stopping times, 
\begin{equation}\label{proposition energy estimate eq}
    \begin{aligned}
        \frac{d}{dt} \inb a \eta(U|\bu) \dx + G(U)(t) \leq C \inb a \eta(U|\bu) \dx + C \nu^{\cbet} \ln \rn^{-1} h(t) - l(t) \, , 
    \end{aligned}
\end{equation}
where 
\begin{equation}\label{def of D_v, G}
    \begin{aligned}
        G(U)(t) &= \de_* \sum_{\be \in \ms} \int |(\su_\be)_x| p(v|\bv) \dx + \sum_{\be \in \mr} \int |(\ru_\be)_x| p(v|\bv) \dx  + \frac{1}{2}|{y}_b'(t)| \eta(U|\bu)\big|_{x=y_b(t)} \\ 
        &\quad+ \nu \inb a \mu_1(v) v^{\ga} \Bigr|\bigl(p(v)-p(\bv)\bigr)_x \Bigr|^2 \dx + \nu \inb a \mu_2(v) v^{\ga} \bigl|p(v)_x\bigr|^2 \dx\, , \\
        h(t) &=  \nu \mathcal{D}_v(U) + 1 \, , \\
        l(t) &=  \frac{1}{4}|{y}_b'(t)| \eta(U|\bu)\big|_{x=y_b(t)} \, ,
    \end{aligned}
\end{equation}
with $\de_*$ chosen in Section \ref{subsec choice of e_0}, and $\cbet$ defined in \eqref{def of bold c}, and $\mathcal{D}_v$ is defined in \eqref{def of D_v}.
\end{prop}

\subsection{Useful estimates}\label{subsec useful estimates}
We establish several estimates that will be used in the subsequent analysis.
\subsubsection{Useful estimates for waves and interaction terms}
We state useful estimates on waves and their interaction terms.
\begin{lemm}\label{lemma estimate of waves}
There exist $C,\widetilde{C}>0$, depending only on $U_*$, such that for all $x$ and $t$ between stopping times, the following holds. 

We consider a wave $\be$ with strength $\s_\be$. Recall that $y_\be$ denotes its trajectory, $I_\be$, defined in \eqref{def of transition zone}, denotes its transition zone, and, if $\be$ is a shock, $\widetilde{I}_\be$, defined in \eqref{def of main transition zone S}, denotes its main transition zone. Let $U_{-}$ denote the left state and $U_{+}$ denote the right state of the wave.
\begin{itemize}
    \item[1)] For any $\be \in \mathcal{S}$, 
    \begin{align*}
        |\su_\be(x,t)-U_{-}| &\leq C |\s_\be| \, \textup{exp}\bigl( -\frac{\widetilde{C}|\s_\be||x-y_\be|}{\nu} \bigr) \, , \quad x<y_\be \, , \\
        |\su_\be(x,t)-U_{+}| &\leq C |\s_\be| \, \textup{exp}\bigl( -\frac{\widetilde{C}|\s_\be||x-y_\be|}{\nu} \bigr) \, , \quad x>y_\be \, , \\
        C^{-1}\frac{\s_\be^2}{\nu} \textup{exp}\bigl( -\frac{\widetilde{C}|\s_\be||x-y_\be|}{\nu} \bigr) &\leq |\p_x \su_\be (x,t)| \leq C\frac{\s_\be^2}{\nu} \textup{exp}\bigl( -\frac{\widetilde{C}|\s_\be||x-y_\be|}{\nu} \bigr) \, , \\
        |\p_x^2 \su_\be(x,t)| &\leq C \frac{|\s_\be|}{\nu} |\p_x \su_\be(x,t)| \, . 
    \end{align*}
    \item[2)] For any $\be \in \mathcal{R}$,
    \begin{align*}
        \p_x \ru_\be(x,t) &= \frac{\s_\be}{\rn^{\cor}} \phi'(\frac{x}{\rn^{\cor}}) \, \mathbf{r}_{k}\bigl(\ru_\be(x,t)\bigr) \, , \\
        |\p_x \ru_\be(x,t)| &\leq C \frac{\s_\be}{\rn^{\cor}} \mathds{1}_{|x-y_\be|\leq \rn^{\cor}}\, , \\
        |\p_{xx} \ru_\be(x,t)| &\leq C \frac{\s_\be}{\rn^{2\cor}} \mathds{1}_{|x-y_\be|\leq \rn^{\cor}} \, .
    \end{align*}
    \item[3)] For any $\be \in\mathcal{NP}$,
    \begin{align*}
        |\p_x \nnu_\be| &\leq C \frac{\s_\be}{\rn^{\cop}} \mathds{1}_{|x-y_\be|\leq \rn^{\cop}} \, , \\
        |\p_{xx} \nnu_\be| &\leq C \frac{\s_\be}{\rn^{2\cop}} \mathds{1}_{|x-y_\be|\leq \rn^{\cop}} \, .
    \end{align*}
\end{itemize}
For any $n \in \mathbb{N}_*$, there exists $C_n>0$, depending only on $U_*$ and $n$, such that the following is true.
\begin{itemize}
    \item[4)] For any $\be \in \mathcal{S}$,
    \begin{align*}
        |\su_\be - U_{-}| \leq C_n \nu^n& \, , \; \textup{ if } x \leq y_\be - \rn^{\cz}  \, ; \quad
        |\su_\be - U_{+}| \leq C_n \nu^n \, , \; \textup{ if } x \geq y_\be + \rn^{\cz} \, ; \\
        |\su_\be - U_{-}| \leq C_n \nu^n& \, , \; \textup{ if } x \leq y_\be - \rn^{\co}  \,; \quad
        |\su_\be - U_{+}| \leq C_n \nu^n \, , \; \textup{ if } x \geq y_\be + \rn^{\co} \, ;  \\
        &|\p_{x} \su_\be| \leq C_n \nu^n \, , \; \textup{ on } [y_{\be}-\rn^{\co},y_{\be}+\rn^{\co}] \, ,
    \end{align*}
    \item[5)] 
    For any $\be \in \mathcal{R} \cup \mathcal{NP}$,
    \begin{align*}
        \ru_\be = U_{-} \, , \; \textup{ if } x \leq y_\be - \rn^{\co}  \,; \quad
        \ru_\be = U_{+} \, , \; \textup{ if } x \geq y_\be + \rn^{\co} \, ;\\
        \nnu_\be = U_{-} \, , \; \textup{ if } x \leq y_\be - \rn^{\co}  \,; \quad
        \nnu_\be = U_{+} \, , \; \textup{ if } x \geq y_\be + \rn^{\co} \, ;
    \end{align*}
    \item[6)] We have
    \begin{align*}
        &\Bigl(\sum_{\be\in \mathcal{S}} |(\sv_\be)_x| \Bigr)^2 
        \leq \sum_{\be\in \mathcal{S}} |(\sv_\be)_x|^2 + C_n \nu^n \, , \\
        &\Bigl(\sum_{\be\in \mathcal{S}} |(\sv_\be)_x| + \sum_{\be\in \mathcal{R}} |(\rv_\be)_x| + \sum_{\be\in \mathcal{NP}} |(\nv_\be)_x| \Bigr)^2 \\
        &\leq \sum_{\be\in \mathcal{S}} |(\sv_\be)_x|^2  + \sum_{\be\in \mathcal{R}} |(\rv_\be)_x|^2 + \sum_{\be\in \mathcal{NP}} |(\nv_\be)_x|^2 + C_n \nu^n \, .
    \end{align*}
\end{itemize}
\end{lemm}
\begin{proof}
The definition of rarefaction in \eqref{def of rarefaction nu} gives 2), while 3) follows from the definition of non-physical waves in \eqref{def of NP nu}. The proof of 1) can be found in Section 2.1 of \cite{MR4195742}.

We prove 4). Let $\be \in \mathcal{S}$. First, we note that \eqref{threshold mft}, 1) in Corollary \ref{lemma strength of waves}, and \eqref{rt transition S} imply that
\begin{align*}
    C \frac{\s_\be^2}{\nu} \textup{exp}\bigl( -\frac{\widetilde{C}\rn^{1+\cz}}{\nu} \bigr) 
    \leq C \frac{\s_\be^2}{\nu} \textup{exp}\bigl( -\frac{\widetilde{C}\rn^{1+\co}}{\nu} \bigr) 
    \leq \frac{C}{\nu} \textup{exp}\bigl(-\widetilde{C}\nu^{(1+\cz)\cs-1}\bigr) 
    \leq C_n \nu^n \, .
\end{align*}
By 1), we complete the proof of 4).

Statement 5) follows immediately from the fact that the derivatives of rarefactions and non-physical waves vanish outside their transition zones.

Finally, we prove 6). Let $w$ and $\widetilde{w}$ be two distinct waves in $\mathcal{S}\cup\mathcal{R}\cup\mathcal{NP}$. Then (4) and (5) imply that
\begin{align*}
    |w_x||\widetilde{w}_x| \leq C_n \nu^n \, .
\end{align*}
Together with Lemma \ref{lemma number of waves and interactions}, we get
\begin{align*}
    \Bigl(\sum_{\be\in \mathcal{S}} |(\sv_\be)_x|\Bigr)^2 &\leq \sum_{\be\in \mathcal{S}} |(\sv_\be)_x|^2 + C_n \nu^n \rn^{-2\cpr} (\ln \rn^{-1})^2 \, ,
\end{align*}
and
\begin{align*}
    &\Bigl(\sum_{\be\in \mathcal{S}} |(\sv_\be)_x| + \sum_{\be\in \mathcal{R}} |(\rv_\be)_x| + \sum_{\be\in \mathcal{NP}} |(\nv_\be)_x| \Bigr)^2 \\
    &\leq \sum_{\be\in \mathcal{S}} |(\sv_\be)_x|^2  + \sum_{\be\in \mathcal{R}} |(\rv_\be)_x|^2 + \sum_{\be\in \mathcal{NP}} |(\nv_\be)_x|^2 + C_n \nu^n \rn^{-2\cpr} (\ln \rn^{-1})^2 \, .
\end{align*}
Since $C_n \nu^n \rn^{-2\cpr} (\ln \rn^{-1})^2\leq C_n \nu^n$, we finish the proof.

\end{proof}
Property~\ref{prop transition zone} implies that the transition zones of all waves are well separated. The derivatives of rarefactions and non-physical waves are compactly supported inside their transition zones, while 4) in Lemma~\ref{lemma estimate of waves} gives exponential decay of shocks outside their main transition zones $\widetilde{I}_\be$, and hence outside their transition zones $I_\be$. Therefore, we obtain the following lemma for arbitrary $n$.

\begin{lemm}\label{lemma estimates of interaction terms}
For any $n \in \mathbb{N}_*$, there exists $C_n>0$, depending only on $U_*$ and $n$, such that the following is true.
\begin{align*}
    \inb \Bigl(\sum_{\be \in\mathcal{S}} |\bv - \sv_\be| |(\sv_\be)_x|\Bigr)^2 \dx &\leq C_n \nu^n \, , \\
    \inb \Bigl(\sum_{\be \in\mathcal{R}} |\bv - \rv_\be| |(\rv_\be)_x|\Bigr)^2 \dx &\leq C_n \nu^n \, , \\
    \inb \sum_{\be \in\mathcal{S}} |\bv - \sv_\be| |(\sv_\be)_x| \dx &\leq C_n \nu^n \, , \\
    \inb \sum_{\be \in\mathcal{R}} |\bv - \rv_\be| |(\rv_\be)_x| \dx &\leq C_n \nu^n \, .
\end{align*}
In addition, 
\begin{align*}
    \int_{v>2v_*} \Bigl(\sum_{\be \in\mathcal{S}} |\bv - \sv_\be| |(\sv_\be)_x|\Bigr)^2 v^{\ga-\al} \dx \leq C_n \nu^{n} \sum_{\be \in\mathcal{S}} \inb |(\sv_\be)_x| p(v|\bv) \dx + C_n \nu^{n} \inb Q(v|\bv) \dx \, .
\end{align*}
\end{lemm}
\begin{proof}
By Lemma \ref{lemma number of waves and interactions}, we have
\begin{align*}
    &\inb \Bigl(\sum_{\be \in\mathcal{S}} |\bv - \sv_\be| |(\sv_\be)_x|\Bigr)^2 \dx \\
    &\leq C \frac{\ln \rn^{-1}}{\rn^{\cpr}} \sum_{\be \in\mathcal{S}} \inb |\bv - \sv_\be|^2 |(\sv_\be)_x|^2 \dx \\
    &\leq C \frac{\ln \rn^{-1}}{\rn^{\cpr}} \sum_{\be \in\mathcal{S}} \int_{I_\be} |\bv - \sv_\be|^2 |(\sv_\be)_x|^2 \dx + C \frac{\ln \rn^{-1}}{\rn^{\cpr}} \sum_{\be \in\mathcal{S}} \int_{I_\be^c} |\bv - \sv_\be|^2 |(\sv_\be)_x|^2 \dx \\ 
    &=:J_1 + J_2\, .
\end{align*}
Lemma \ref{lemma estimate of waves} and Lemma \ref{lemma number of waves and interactions} imply that on $I_\be$, 
\begin{align}\label{estimate inside Ii S}
    |\bv - \sv_\be| \leq \Bigl|\sum_{\widetilde{\be}\in \mathcal{S}\backslash \{\be\} } (\sv_\be - \underline{v}_{\widetilde{\be}})\Bigr| \leq C_n \nu^n \, ,
\end{align}
where $\underline{v}_{\widetilde{\be}}$ denotes the left state (respectively, right state) of $\widetilde{\be}$ if $\widetilde{\be}$ lies to the right (respectively, left) of $\be$. Together with Lemma \ref{lemma L+kQ}, we get
\begin{align*}
    J_1 \leq C_n \nu^n \, .
\end{align*}
By Lemma \ref{lemma estimate of waves}, Corollary \ref{lemma strength of waves}, Lemma \ref{lemma L+kQ}, and \eqref{rt transition S}, we have
\begin{align*}
    J_2 &\leq C \frac{\ln \rn^{-1}}{\rn^{\cpr}} \sum_{\be \in\mathcal{S}} \int_{I_\be^c} \frac{\s_\be^4}{\nu^2} \textup{exp}\bigl( -\frac{C|\s_\be||x-y_\be|}{\nu} \bigr) \dx \\
    &\leq C \frac{\ln \rn^{-1}}{\rn^{\cpr}} \sum_{\be \in\mathcal{S}} \frac{|\s_\be|^3}{\nu} \textup{exp}\bigl( -\frac{C\rn^{1+\co}}{\nu} \bigr) \\
    &\leq C \frac{\ln \rn^{-1}}{\rn^{\cpr}} \sum_{\be \in\mathcal{S}} \frac{|\s_\be|^3}{\nu} \textup{exp}\bigl(-\frac{C}{\nu^{1-(1+\co)\cs}}\bigr) \\
    &\leq C_n \nu^n \, .
\end{align*}
We prove the first inequality.

By Lemma \ref{lemma estimate of waves}, we have
\begin{align*}
    &\inb \Bigl(\sum_{\be \in\mathcal{R}} |\bv - \rv_\be| |(\rv_\be)_x|\Bigr)^2 \dx \\
    &\leq \sum_{\be \in\mathcal{R}} \int_{I_\be} |\bv - \rv_\be|^2 |(\rv_\be)_x|^2 \dx \\ 
    &=:J_3 \, .
\end{align*}
Lemma \ref{lemma estimate of waves} and Lemma \ref{lemma number of waves and interactions} show that on $I_\be$, 
\begin{align}\label{estimate inside Ii R}
    |\bv - \rv_\be| \leq \Bigl|\sum_{\widetilde{\be}\in \mathcal{S} } (\sv_\be - \underline{v}_{\widetilde{\be}})\Bigr| \leq C_n \nu^n \, ,
\end{align}
where $\underline{v}_{\widetilde{\be}}$ denotes the left state (respectively, right state) of $\widetilde{\be}$ if $\widetilde{\be}$ lies to the right (respectively, left) of $\be$. Together with Lemma \ref{lemma L+kQ}, we get
\begin{align*}
    J_3 \leq C_n \nu^n \, .
\end{align*}
We get the second inequality.

By Lemma \ref{lemma estimate of waves}, Lemma \ref{lemma L+kQ}, \eqref{estimate inside Ii S}, \eqref{rt transition S}, and \eqref{estimate inside Ii R}, we have
\begin{align*}
    &\inb \sum_{\be \in\mathcal{S}} |\bv - \sv_\be| |(\sv_\be)_x| \dx  \\
    &\leq \sum_{\be \in\mathcal{S}} \int_{I_\be} |\bv - \sv_\be| |(\sv_\be)_x| \dx  + \sum_{\be \in\mathcal{S}} \int_{I_\be^c} |\bv - \sv_\be| |(\sv_\be)_x| \dx \\
    &\leq C_n \nu^n L(0) + \sum_{\be \in\mathcal{S}} |\s_\be| \textup{exp}\bigl(-\frac{C}{\nu^{1-(1+\co)\cs}}\bigr) \\
    &\leq C_n \nu^n \, ,
\end{align*}
and
\begin{align*}
    &\inb \sum_{\be \in\mathcal{R}} |\bv - \rv_\be| |(\rv_\be)_x| \dx  \\
    &\leq \sum_{\be \in\mathcal{R}} \int_{I_\be} |\bv - \rv_\be| |(\rv_\be)_x| \dx  \\
    &\leq C_n \nu^n L(0) \\
    &\leq C_n \nu^n \, ,
\end{align*}

Finally, we prove the last inequality. Since $0\leq \ga-\al\leq 1$, we have
\begin{align*}
    v^{\ga-\al} \mathds{1}_{v>2v_*} \leq C Q (v|\bv) \, , \quad v^{\ga-\al} \mathds{1}_{v>2v_*} \leq C p (v|\bv) \, .
\end{align*}
Together with Lemma \ref{lemma number of waves and interactions}, \eqref{estimate inside Ii S}, Lemma \ref{lemma estimate of waves}, and \eqref{rt transition S}, we get
\begin{align*}
    &\int_{v>2v_*} \Bigl(\sum_{\be \in\mathcal{S}} |\bv - \sv_\be| |(\sv_\be)_x|\Bigr)^2 v^{\ga-\al} \dx \\
    &\leq C \frac{\ln \rn^{-1}}{\rn^{\cpr}} \sum_{\be \in\mathcal{S}} \int_{v>2v_*} |\bv - \sv_\be|^2 |(\sv_\be)_x|^2 v^{\ga-\al} \dx \\
    &\leq C \frac{\ln \rn^{-1}}{\rn^{\cpr}} \sum_{\be \in\mathcal{S}} \int_{I_\be} |\bv - \sv_\be|^2 |(\sv_\be)_x|^2 p(v|\bv) \dx + C \frac{\ln \rn^{-1}}{\rn^{\cpr}} \sum_{\be \in\mathcal{S}} \int_{I_\be^c} |\bv - \sv_\be|^2 |(\sv_\be)_x|^2 Q(v|\bv) \dx \\ 
    &\leq C_n \nu^n \sum_{\be \in\mathcal{S}} \inb  |(\sv_\be)_x| p(v|\bv) \dx + C \frac{\ln \rn^{-1}}{\rn^{\cpr}} \sum_{\be \in\mathcal{S}} \frac{\s_\be^4}{\nu^2} \textup{exp}\bigl(-\frac{C}{\nu^{1-(1+\co)\cs}}\bigr) \inb Q(v|\bv) \dx \\
    &\leq C_n \nu^n \sum_{\be \in\mathcal{S}} \inb  |(\sv_\be)_x| p(v|\bv) \dx + C_n \nu^n \inb Q(v|\bv) \dx \, .
\end{align*}
\end{proof}

\subsubsection{Quantitative subsonicity of the approximate solution}
The following lemma shows that when the boundary moves pathologically fast compared to the lower bound on the wave speeds that follows from the subsonicity of the approximate solution, the variable $v$ in the Navier-Stokes system must be sufficiently small compared to $v_*$, which in turn yields a positive lower bound for $Q(v|\bv)$ at the boundary.

\begin{lemm}\label{lemma estimates on relative functions 2}
Given $U_*=(v_*, u_*)\in \Omega_{sub}^-$ as defined in \eqref{def subsonic L}, there exist $\La_0,C_X,\e_0> 0$ small enough such that the following is true.

Assume that $\ml(0)\leq \e_0$, and let $C_X>0$ be the constant defining the shifts $X_\be$ associated with the shocks $\be$ through \eqref{ODE of Xi}.

Then the absolute value of the speed of every wave in the approximate solution is bounded below by $2\La_0$. Moreover, if $y_b'+\La_0$ is greater than or equal to the speed of some $2$-wave in the approximate solution $\bu=(\bv,\bh)$ defined in Section \ref{subsec mft} at time $t$, then
\begin{align*}
    Q\bigl(v(y_b(t),t)\mid \bv(y_b(t),t)\bigr)\geq C \, ,
\end{align*}
for some $C>0$ depending only on $U_*$ and $\La_0$, provided that $C_X$ and $\e_0$ are sufficiently small.
\end{lemm}
\begin{proof}
Since $U_* \in \Omega_{sub}^-$, we have
\begin{align*}
    0<\frac{-u_*}{v_*} < \ga^{\frac{1}{2}} v_*^{-\frac{\ga+1}{2}} \, .
\end{align*}
Lemma \ref{lemma L+kQ} and \eqref{speed of shock bound} imply that the absolute value of the speed of any wave is bounded below by
\begin{align*}
    \la_{lower} : = (1-C_X)(1-C\e_0) \ga^{\frac{1}{2}} v_*^{-\frac{\ga+1}{2}} \, .
\end{align*}
We take
\begin{align*}
    \La_0 = \frac{1}{3}\bigl(\ga^{\frac{1}{2}} v_*^{-\frac{\ga+1}{2}} - \frac{-u_*}{v_*}\bigr) \, .
\end{align*}
Then by taking $\e_0,C_X$ small enough, we have 
\begin{align*}
    \la_{lower} > \frac{-u_*}{v_*} + 2\La_0 > 2\La_0 \, .
\end{align*}
When $y_b'+\La_0$ exceeds the speed of some $2$-wave in the approximate solution, we have
\begin{align*}
    \frac{-u_*}{v(y_b(t),t)} + \La_0 \geq \la_{lower} > \frac{-u_*}{v_*} + 2\La_0 \, .
\end{align*}
Hence, there exists $c>0$ depending on $\La_0$ such that
\begin{align*}
    v(y_b(t),t) + c < v_* \, .
\end{align*}
By taking $\e_0$ small enough, Lemma \ref{lemma L+kQ} implies that
\begin{align*}
    Q\bigl(v(y_b(t),t)\mid \bv(y_b(t),t)\bigr)\geq C \, .
\end{align*}
\end{proof}

\section{Estimates for main terms localized by derivatives of rarefaction waves}\label{sec rarefaction}
We handle the following term $\mathcal{H}_2 $ defined in \eqref{relative entropy method terms}
\begin{align*}
    \mathcal{H}_2 (U) &= - \sum_{\be \in\mathcal{R}} \inb a \nabla^2 \eta(\bu) \p_x \ru_\be A(U|\bu) \dx - \sum_{\be \in \mathcal{R}} \inb a (h-\bh) \bigl( p'(\bv) - p'(\rv_\be) \bigr) (\rv_\be)_x \dx \\
    &\quad - \sum_{\be \in \mathcal{R}} \inb a \nabla^2 \eta(\bu) (U-\bu) \bigl(- \la_\be(\ru_{\be,-}) + \nabla A(\ru_\be) \bigr)\p_x \ru_\be \dx  \\
    &=: \mathcal{H}_{21} + \mathcal{H}_{22} + \mathcal{H}_{23} \, .
\end{align*}
Later in this section, we further split $\mathcal{H}_{23}$ into $\mathcal{H}_{231}$ and $\mathcal{H}_{232}$. The estimates for all these terms are similar to those in \cite{NStoE}, except for $\mathcal{H}_{232}$. Instead of using the dissipation, we directly use the $L^1$-norm of the relative entropy to control this term. This is possible here because the transition zones of the rarefactions are chosen to be not too small compared to their strengths, as indicated by \eqref{rt transition R}. In this way, we avoid introducing time-dependent artificial shifts for the rarefactions, as in \cite{NStoE}.

By \eqref{useful computation}, Lemma \ref{lemma estimate of waves}, Lemma \ref{lemma L+kQ}, and
\begin{align*}
    \mathbf{r}_{i}(U_*) \, \parallel
    \begin{pmatrix}
        (-1)^{i+1} \\ \sqrt{\ga v_*^{-\ga-1}} \, ,
    \end{pmatrix}
\end{align*}
we get
\begin{align}\label{estimate on H21}
    \mathcal{H}_{21} \leq -C \sum_{\be \in\mathcal{R}} G_\be^R \, ,
\end{align}
where
\begin{align}\label{def of G_i^R}
    G_\be^R =  \inb \bigl|(\ru_\be)_x\bigr| \, p(v|\bv) \dx \, .
\end{align}
H\"older's inequality and Lemma \ref{lemma estimates of interaction terms} give
\begin{equation}\label{estimate on H_22}
    \begin{aligned}
        |\mathcal{H}_{22}| &\leq \inb a |h-\bh|^2 \dx + C \inb \Bigl(\sum_{\be \in\mathcal{R}} |\bv-\rv_\be||(\rv_\be)_x| \Bigr)^2 \dx \\
        &\leq \inb a |h-\bh|^2 \dx + C_n \nu^n \, .
    \end{aligned}
\end{equation}
By \eqref{useful computation} and \eqref{def of rarefaction nu}, we have
\begin{align*}
    |\mathcal{H}_{23}| &= \Bigl| \sum_{\be \in \mathcal{R}} \inb a \nabla^2 \eta(\bu) (U-\bu) \bigl(- \la_\be(\ru_{\be,-}) + \la_\be (\ru_\be) \bigr)\p_x \ru_\be \dx \Bigr| \\
    &\leq C \sum_{\be \in\mathcal{R}} |\s_\be| \inb |v-\bv| |\p_x \rv_\be| \dx + C \sum_{\be \in\mathcal{R}} |\s_\be| \inb |h-\bh| |\p_x \rh_\be| \dx \\
    &=: \mathcal{H}_{231} + \mathcal{H}_{232} \, .
\end{align*}
We have
\begin{align*}
    |v-\bv| \mathds{1}_{v>2v_*} \leq C v \leq p(v|\bv) \, .
\end{align*}
By 2) in Corollary \ref{lemma strength of waves} and H\"older's inequality, we have
\begin{equation}\label{estimate on H231}
    \begin{aligned}
        \mathcal{H}_{231} &\leq C \rn^{\cpr} \sum_{\be \in\mathcal{R}} \inb |v-\bv| |\p_x \rv_\be| \mathds{1}_{v\leq 2\vs} \dx + C \rn^{\cpr} \sum_{\be \in\mathcal{R}} \inb |v-\bv| |\p_x \rv_\be| \mathds{1}_{v> 2\vs} \dx \\
        &\leq C \rn^{\cpr} \sum_{\be \in\mathcal{R}} \inb |v-\bv|^2 |\p_x \rv_\be| \mathds{1}_{v\leq 2\vs} \dx + C \rn^{\cpr} \sum_{\be \in\mathcal{R}} \inb |\p_x \rv_\be| \mathds{1}_{v\leq 2\vs} \dx \\
        &\quad \, + C \rn^{\cpr} \sum_{\be \in\mathcal{R}} \inb p(v|\bv) |\p_x \rv_\be| \mathds{1}_{v> 2\vs} \dx \\
        &\leq C\rn^{\cpr} \sum_{\be \in\mathcal{R}} G_\be^R + C\rn^{\cpr} L(0) \, .
    \end{aligned}
\end{equation}
By H\"older's inequality, Young's inequality, Property \ref{prop transition zone}, 3) in Corollary \ref{lemma strength of waves}, and Lemma \ref{lemma L+kQ}, we get
\begin{equation}\label{estimate on H232}
    \begin{aligned}
        |\mathcal{H}_{232}| &\leq C \sum_{\be \in\mathcal{R}} \frac{\s_\be^2}{\rn^{\cor}} \inb |h-\bh| \mathds{1}_{[y_\be-\rn^{\cor}, y_\be+\rn^{\cor}]} \dx \\
        &\leq C \sum_{\be \in\mathcal{R}} \frac{\s_\be^2}{\rn^{\cor/2}} \Bigl( \inb |h-\bh|^2 \mathds{1}_{[y_\be-\rn^{\cor}, y_\be+\rn^{\cor}]} \dx \Bigr)^{1/2} \\
        &\leq \inb |h-\bh|^2 \sum_{\be \in\mathcal{R}} \mathds{1}_{[y_\be-\rn^{\cor}, y_\be+\rn^{\cor}]} \dx +  C \sum_{\be \in\mathcal{R}} \frac{\s_\be^4}{\rn^{\cor}} \\
        &\leq \inb |h-\bh|^2 \dx + \frac{\rn^{3\cpr}}{\rn^{\cor}}L(0) \, .
    \end{aligned}
\end{equation}
By \eqref{estimate on H21}, \eqref{estimate on H_22}, \eqref{estimate on H231}, and \eqref{estimate on H232}, we get
\begin{equation}\label{estimate on H_2}
    \begin{aligned}
        \mathcal{H}_2 &\leq -C\sum_{\be \in\mathcal{R}} G_\be^R + C \inb a |h-\bh| \dx + C_n \nu^n + C\rn^{\cpr} + C \rn^{3\cpr-\cor}
        \, .
    \end{aligned}
\end{equation}
To ensure that this inequality gives the desired estimate, we impose the relation \eqref{rt transition R}.

\section{Estimates for main terms localized by derivatives of non-physical waves and for boundary terms}\label{sec NP and boundary}
The key estimate is a $\nu$-dependent $L^\infty$ bound for $v^{-\ga}$, obtained using the dissipation $\mathcal{D}_v$ defined in \eqref{def of D_v}. We first derive this estimate in Section~\ref{subsec NP and B v^-ga}, and then bound the terms localized by the derivatives of non-physical waves $\mathcal{H}_3 $ in Section~\ref{subsec NP bound} and the boundary terms $\mathcal{B} $ in Section~\ref{subsec B bound}.
\subsection{Estimates for $\|v^{-\ga}\|_{L^\infty}$}\label{subsec NP and B v^-ga}
We first note that $\|v^{-\ga}\|_{L^\infty}$ is bounded when $v$ is large, so it suffices to control it when $v$ is small. Recall Remark \ref{remark uniform bound on dissipation lagrangian}. We have
\begin{align*}
    \nu \int_0^T \mathcal{D}_v \, dt \leq C \, ,
\end{align*}
where $\mathcal{D}_v$ is defined in \eqref{def of D_v}. 
Accordingly, in view of the dissipation $\mathcal{D}_v$, we define the function
\begin{align}\label{def of g}
    g(v) = \bigl(v^{-\frac{\ga+\al}{2}} - (\frac{\vs}{2})^{-\frac{\ga+\al}{2}}\bigr) \mathds{1}_{v\leq \frac{\vs}{2}} \, ,
\end{align}
By the Gagliardo-Nirenberg inequality and \eqref{remark lemma uniform bound on integral of Q}, we get
\begin{align*}
    \|g(v)\|_{L^\infty} &\leq C \|g(v)_x\|_{L^2}^{2/3} \|g(v)\|_{L^1}^{1/3} \\ 
    &\leq C \mathcal{D}_v^{1/3} \bigl (\int |v^{-\ga+1}||v^{\frac{\ga-\al}{2}-1}| \mathds{1}_{v\leq \frac{\vs}{2}} \, dx \bigr)^{1/3} \\
    &\leq C \mathcal{D}_v^{1/3} \|v^{\frac{\ga-\al-2}{6}}\mathds{1}_{v\leq \frac{\vs}{2}}\|_{L^\infty}
    \, .
\end{align*}
Hence, there exist some constant $C_1,C_2>0$ such that 
\begin{align*}
    \|v^{-\frac{\ga+\al}{2}}\mathds{1}_{v\leq \frac{\vs}{2}}\|_{L^\infty} \leq C_1 + C_2 \mathcal{D}_v^{1/3} \|v^{\frac{\ga-\al-2}{6}}\mathds{1}_{v\leq \frac{\vs}{2}}\|_{L^\infty} \, .
\end{align*}
If $C_1\geq C_2 \mathcal{D}_v^{1/3} \|v^{\frac{\ga-\al-2}{6}}\mathds{1}_{v\leq \frac{\vs}{2}}\|_{L^\infty}$, then 
\begin{align}\label{bound on v^-gamma 1}
    \|v^{-\ga}\mathds{1}_{v\leq \frac{\vs}{2}}\|_{L^\infty} \leq C \, .
\end{align}
Assume $C_1 < C_2 \mathcal{D}_v^{1/3} \|v^{\frac{\ga-\al-2}{6}}\mathds{1}_{v\leq \frac{\vs}{2}}\|_{L^\infty}$. Since $\ga-1\leq \al \leq \ga$ and $\ga>1$, we have 
\begin{align*}
    \frac{\ga-\al-2}{6} < 0 \, , \quad 
    -\frac{\ga+\al}{2} - \frac{\ga-\al-2}{6} = -\frac{2\ga + \al - 1}{3} < 0 \, .
\end{align*} 
Therefore, we get
\begin{align*}
    \|v^{-\frac{2\ga+\al-1}{3}} \mathds{1}_{v\leq \frac{\vs}{2}}\|_{L^\infty} \leq C \mathcal{D}_v^{1/3} \, ,
\end{align*}
which implies
\begin{align}\label{bound on v^-gamma 2}
    \|v^{-\ga}\mathds{1}_{\{v\leq \frac{\vs}{2}\}} \|_{L^\infty} \leq C D_v^{-\frac{\ga}{2\ga+\al-1}} \, .
\end{align}

In all, \eqref{bound on v^-gamma 1} and \eqref{bound on v^-gamma 2} give
\begin{align}\label{bound on v^-gamma}
    \|v^{-\gamma}\|_{L^\infty} \leq C + C \mathcal{D}_v^{\cg} \, ,
\end{align}
where $\cg=\frac{\ga}{2\ga+\al-1}$ is defined in \eqref{def of c gamma}.

\subsection{Estimates for main terms localized by non-physical waves}\label{subsec NP bound}
We evaluate $\mathcal{H}_3 (U)$ defined in \eqref{relative entropy method terms}:
\begin{align*}
    \mathcal{H}_3 (U) &= - \sum_{\be \in \mathcal{NP}} \inb a \bigl(p(v)-p(\bv)\bigr) \p_x \nh_\be \dx - \sum_{\be \in \mathcal{NP}} \inb a (h-\bh)p'(\bv) \p_x \nv_\be \dx \\
    &=: \mathcal{H}_{31} + \mathcal{H}_{32} \, .
\end{align*}
By 3) in Corollary \ref{lemma strength of waves}, and Lemma \ref{lemma number of waves and interactions}, we have
\begin{align*}
    |\mathcal{H}_{31}| \leq C (1 + \|v^{-\ga}\|_{L^\infty}) 
    \sum_{\be \in \mathcal{NP}} |\s_\be|  \leq C (1 + \|v^{-\ga}\|_{L^\infty}) 
    \rn^{1-\cpr} \ln{\rn^{-1}}  \, .
\end{align*}
By \eqref{bound on v^-gamma}, \eqref{rt mft total strength NP}, \eqref{property of c gamma}, and Young's inequality, we have
\begin{equation}\label{evaluation of v^-gamma}
    \begin{aligned}
        &\rn^{1-\cpr} \|v^{-\ga}\|_{L^\infty}  \\
        &\leq C \rn^{1-\cpr} + C \nu^{(1-\cpr)\cs}\mathcal{D}_v^{\cg} \\
        &\leq C \rn^{1-\cpr} + C(\nu^{\cg + \frac{1-\cg}{4}} \mathcal{D}_v^{\cg}) \nu^{\frac{1-\cg}{4}} \\
        &\leq C \rn^{1-\cpr} + C\nu^{\frac{1-\cg}{4\cg}} \nu \mathcal{D}_v + C \nu^{\frac{1}{4}} \, .
    \end{aligned}
\end{equation}
Therefore, we get
\begin{align}\label{estimate H31}
    |\mathcal{H}_{31}| \leq C (\nu^{\frac{1-\cg}{4\cg}} \nu \mathcal{D}_v + \rn^{1-\cpr} + \nu^{\frac{1}{4}}) \ln{\rn^{-1}} \, .
\end{align}
By H\"older's inequality, we get
\begin{align*}
    |\mathcal{H}_{32}| &\leq \inb a |h-\bh|^2 \dx + C \inb \bigl( \sum_{\be \in \mathcal{NP}} \p_x \nv_\be \bigr)^2 \dx \, .
\end{align*}
Lemma \ref{lemma estimate of waves}, 3) in Corollary \ref{lemma strength of waves}, and Lemma \ref{lemma number of waves and interactions} give
\begin{align*}
    \inb \bigl( \sum_{\be \in \mathcal{NP}} \p_x \nv_\be \bigr)^2 \dx 
    = \sum_{\be \in \mathcal{NP}} \inb |\p_x \nv_\be|^2 \dx 
    \leq C\sum_{\be \in \mathcal{NP}} \frac{|\s_\be|^2}{\rn^{\cop}} \leq C \rn^{2-\cop-\cpr} \ln \rn^{-1}\, .
\end{align*}
Together with \eqref{estimate H31}, we get
\begin{align}\label{estimate on H3}
    |\mathcal{H}_3 (U)| \leq \inb a |h-\bh|^2 \dx + C\Bigl(\nu^{\frac{1-\cg}{4\cg}} \nu \mathcal{D}_v + \rn^{2-\cop-\cpr} + \rn^{1-\cpr} + \nu^{\frac{1}{4}} \Bigr) \ln \rn^{-1} \, .
\end{align}

\subsection{Estimates for boundary terms}\label{subsec B bound}
We estimate $\mathcal{B} (U)$ defined in \eqref{relative entropy method terms}. We note that the Navier-Stokes solution always satisfies the outflow boundary condition \eqref{outflow boundary condition}. However, the boundary condition of the approximate solution $\bar U$ constructed in Section~\ref{sec mft} may be violated when a merging or an overtaking causes non-physical waves to disappear. Moreover, since $\bu$ is constructed as an approximate solution to the Navier-Stokes system, an additional error is produced due to the non-localization of the viscous shocks. In all, Property \ref{prop transition zone}, Lemma \ref{lemma estimate of waves}, 3) in Corollary \ref{lemma strength of waves} imply that
\begin{align*}
    |u-\bar{u}| \Big|_{x=y_b(t)} 
    \leq \sum_{\be \in \mathcal{S}} |\tilde{u}_\be| \Big|_{x=y_b(t)}  + \textup{number of elements} \times C \rn
\end{align*}
By Lemma~\ref{lemma number of waves and interactions}, Lemma~\ref{lemma estimate of waves}, and \eqref{rt transition S}, we get
\begin{align*}
    |u-\bar{u}| \Big|_{x=y_b(t)} \leq C \rn^{-\cpr} \ln \rn^{-1} \bigl( \exp(-C\frac{\rn^{1+\co}}{\nu}) + \rn \bigr) \leq C \rn^{1-\cpr} \ln \rn^{-1} \, .
\end{align*}
Together with \eqref{evaluation of v^-gamma}, we get
\begin{align}\label{estimate on B}
    |\mathcal{B} (U)| \leq C \rn^{1-\cpr} \ln \rn^{-1} \bigl(\|v^{-\ga}\|_{L^\infty} + 1 \bigr)  \leq C \Bigl( \rn^{1-\cpr} + \nu^{\frac{1-\cg}{4\cg}} \nu \mathcal{D}_v + \nu^{\frac{1}{4}} \Bigr) \ln \rn^{-1} \, .
\end{align}

\section{Estimates for parabolic term}\label{sec parabolic}
We estimate the following term $\mcp $ defined in \eqref{relative entropy method terms}. We first note that Assumption \ref{property viscosity L} and Lemma \ref{lemma L+kQ} imply that for any $\be \in \mathcal{S}$, 
\begin{align*}
    \mu(\sv_{\be}) = \mu_1(\sv_{\be}) \, .
\end{align*}
Therefore, we can write
\begin{align*}
    \mcp (U) &= -\nu \inb \p_x \Bigl(a\bigl(p(v)-p(\bv)\bigr) \Bigr) \bigl(\mu_1(v) v^{\ga} p(v)_x - \sum_{\be \in \mathcal{S}}\mu_1(\sv_\be)(\sv_\be)^{\ga} p(\sv_\be)_x \bigr) \dx \\
    &\quad - \nu \inb \p_x \Bigl(a\bigl(p(v)-p(\bv)\bigr) \Bigr) \mu_2(v) v^{\ga} p(v)_x \dx \\
    &= \mcp_1 + \mcp_2 + \mcp_3 + \mcp_4 + \mcp_5 + \mcp_6 \, ,
\end{align*}
where
\begin{align*}
    &\mcp_1 = -\nu \inb a\bigl(p(v)-p(\bv)\bigr)_x \, \mu(v) v^{\ga} \bigl(p(v)-\sum_{\be \in \mathcal{S}} p(\sv_\be)\bigr)_x \dx \, , \\
    &\mcp_2 = -\nu \inb a\bigl(p(v)-p(\bv)\bigr)_x \sum_{\be \in \mathcal{S}} \bigl( \mu(v) v^{\ga}-\mu(\sv_\be)\sv_\be^{\ga}\bigr) p(\sv_\be)_x \dx \, , \\
    &\mcp_3 = -\nu \inb \Bigl(\sum_{\be \in \mathcal{S}} (a_\be)_x\Bigr)\bigl(p(v)-p(\bv)\bigr) \, \mu(v) v^{\ga} \bigl(p(v)-\sum_{\be \in \mathcal{S}} p(\sv_\be)\bigr)_x \dx \, , \\
    &\mcp_4 = -\nu \inb \Bigl(\sum_{\be \in \mathcal{S}} (a_\be)_x\Bigr) \bigl(p(v)-p(\bv)\bigr) \sum_{\be \in \mathcal{S}} \bigl( \mu(v) v^{\ga}-\mu(\sv_\be)\sv_\be^{\ga}\bigr) p(\sv_\be)_x \dx \, ,
\end{align*}
and
\begin{align*}
    &\mcp_5 = - \nu \inb a \bigl(p(v)-p(\bv)\bigr)_x \mu_2(v) v^\ga p(v)_x \dx \, , \\
    &\mcp_6 = - \nu \inb \Bigl(\sum_{\be \in \mathcal{S}} (a_\be)_x\Bigr) \bigl(p(v)-p(\bv)\bigr) \mu_2(v) v^\ga p(v)_x \dx \, . 
\end{align*}
The analysis essentially follows the treatment of $\mathcal{P}$ in \cite{NStoE}. Since the front tracking algorithm has been substantially modified, we carry out the computation in the present setting. 

\noindent
\textbf{Estimates for $\mcp_1$.}
We first analyze $\mcp_1$. 
\begin{align*}
    \mcp_1 = -\md - \mcp_{11} \, ,
\end{align*}
where $\md$ is the dissipation we aim to obtain
\begin{align}\label{def of D}
    \md =  \nu \inb a \mu_1(v) v^{\ga} \Bigr|\bigl(p(v)-p(\bv)\bigr)_x \Bigr|^2 \dx
    \, ,
\end{align}
and $\mcp_{11}$ is the error term 
\begin{align*}
    \mcp_{11} = - \nu \inb a\bigl(p(v)-p(\bv)\bigr)_x \, \mu_1(v) v^{\ga} \bigl(p(\bv)-\sum_{\be \in \mathcal{S}} p(\sv_\be)\bigr)_x \dx \, .
\end{align*}
We have
\begin{align*}
    \mcp_{11} =& - \nu \inb a\bigl(p(v)-p(\bv)\bigr)_x \, \mu_1(v) v^{\ga} p'(\bv) \Bigl(\sum_{\be \in \mathcal{R}} (\rv_\be)_x + \sum_{\be \in \mathcal{NP}} (\nv_\be)_x\Bigr) \dx \\
    & - \nu \inb a\bigl(p(v)-p(\bv)\bigr)_x \, \mu_1(v) v^{\ga} \sum_{\be \in \mathcal{S}} \bigl(p'(\bv) - p'(\sv_\be) \bigr) (\sv_\be)_x \dx \\
    =:& \, \mcp_{111} + \mcp_{112} \, .
\end{align*}
The transition zones of rarefactions and non-physical waves do not overlap. Hence, H\"older's inequality and Assumption~\ref{property viscosity L} imply
\begin{align*}
    \mcp_{111}\leq \de_* \md + \mcp_{1111} \, ,
\end{align*}
where
\begin{align}\label{def of P1111}
    \mcp_{1111} = \frac{C\nu}{\de_*} \inb \Bigl(\sum_{\be \in \mathcal{R}} (\rv_\be)_x^2 + \sum_{\be \in \mathcal{NP}} (\nv_\be)_x^2\Bigr) v^{\ga-\al} \dx \, .
\end{align}
Since $0\leq \ga-\al\leq 1$, we have
\begin{align*}
    v^{\ga-\al} \leq C + C p(v|\bv) \, , \quad v^{\ga-\al} \leq C + C Q(v|\bv)  \, .
\end{align*}
By Lemma \ref{lemma estimate of waves} and Corollary \ref{lemma strength of waves}, we have
\begin{align*}
    \mcp_{1111} \leq& \frac{C\nu}{\de_*} \sum_{\be \in \mathcal{R}} \inb (\rv_\be)_x^2 \dx + \frac{C\nu}{\de_*} \sum_{\be \in \mathcal{NP}} \inb (\nv_\be)_x^2 \dx \\
    & + \frac{C\nu}{\de_*} \sum_{\be \in \mathcal{R}} \inb (\rv_\be)_x^2 p(v|\bv)\dx + \frac{C\nu}{\de_*} \sum_{\be \in \mathcal{NP}} \inb (\nv_\be)_x^2 Q(v|\bv)\dx \\
    \leq & \frac{C\nu}{\de_*} (\rn^{\cpr - \cor} + \rn^{1-\cop}) L(0) + \frac{C\nu}{\de_*} \rn^{\cpr-\cor} \sum_{\be \in \mathcal{R}} G_\be^R + \frac{C\nu}{\de_*} \rn^{2(1-\cop)} \inb Q(v|\bv) \dx \, ,
\end{align*}
where $G_\be^R$ is defined in \eqref{def of G_i^R}. By \eqref{rt transition R} and \eqref{threshold mft}, we get
\begin{align}\label{control on P1111}
    \mcp_{1111} \leq C \nu^{1-2\cs\cpr} + C \nu^{1-2\cs\cpr} \sum_{\be \in \mathcal{R}} G_\be^R + C \nu \inb Q(v|\bv) \dx \, .
\end{align}
On the other hand, we have
\begin{align*}
    \mcp_{112} &\leq \nu \inb |\bigl(p(v)-p(\bv)\bigr)_x| \, \mu_1(v) v^{\ga} \sum_{\be \in \mathcal{S}} |\bv - \sv_\be | |(\sv_\be)_x| \dx \\
    &\leq \de_* \md + \mcp_{1121} \, ,
\end{align*}
where
\begin{align}\label{def of P1121}
    \mcp_{1121} = \frac{C\nu}{\de_*} \inb \Bigl(\sum_{\be \in \mathcal{S}} |\bv - \sv_\be | |(\sv_\be)_x|\Bigr)^2 v^{\ga-\al} \dx \, .
\end{align}
By Lemma \ref{lemma estimates of interaction terms}, we get
\begin{equation}\label{control on P1121}
    \begin{aligned}
        \mcp_{1121} 
        &\leq \frac{C\nu}{\de_*} \inb \Bigl(\sum_{\be \in \mathcal{S}} |\bv - \sv_\be | |(\sv_\be)_x|\Bigr)^2 (1+v^{\ga-\al} \mathds{1}_{v>2\vs})\\
        &\leq C_n \nu^n + C_n \nu^n \sum_{\be \in \mathcal{S}} \inb |(\sv_\be)_x| p(v|\bv) \dx + C_n \nu^n \inb Q(v|\bv) \dx \, .
    \end{aligned}
\end{equation}
In all, we have
\begin{equation}\label{estimate on P1}
    \begin{aligned}
        \mcp_{1} \leq& -(1-2\de_*) \md + C \nu^{1-2\cs\cpr}  \sum_{\be \in \mathcal{R}} G_\be^R + C \nu \sum_{\be \in \mathcal{S}} \inb |(\sv_\be)_x| p(v|\bv) \dx \\
        &+ C \inb Q(v|\bv) \dx + C \nu^{1-2\cs\cpr} \, .
    \end{aligned}
\end{equation}

\noindent
\textbf{Estimates for $\mcp_2$.}
We have
\begin{align*}
    \mcp_2 =& -\nu \inb a\bigl(p(v)-p(\bv)\bigr)_x \bigl( \mu_1(v) v^{\ga}- \mu_1(\bv) \bv^{\ga}\bigr) \sum_{\be \in \mathcal{S}} p(\sv_\be)_x \dx \\
    &-\nu \inb a\bigl(p(v)-p(\bv)\bigr)_x \sum_{\be \in \mathcal{S}} \bigl( \mu_1(\bv) \bv^{\ga}-\mu_1(\sv_\be)(\sv_\be)^{\ga}\bigr) p(\sv_\be)_x \dx \\
    =: & \, \mcp_{21} + \mcp_{22} \,. 
\end{align*}
Since $0\leq \ga-\al\leq 1$, we have
\begin{align*}
    \frac{|\mu_1(v)v^{\ga} - \mu_1(\bv)\bv^{\ga}|^2}{\mu_1(v)v^{\ga}} &\leq C v^{-\ga+\al} \mathds{1}_{v\leq \vs/2} + C |v - \bv|^2 \mathds{1}_{\vs/2< v< 2\vs} + v^{\ga-\al} \mathds{1}_{v\geq 2\vs} \\
    &\leq C p(v|\bv) \, .
\end{align*}
By Lemma \ref{lemma estimate of waves} and \eqref{weight relation between L and de*}, we have
\begin{align}\label{estimate of sum of shock L infinity}
    \|\sum_{\be \in\mathcal{S}}|(\sv_\be)_x|\|_{L^\infty} \leq C \sum_{\be \in\mathcal{S}} \frac{\s_\be^2}{\nu} \leq \frac{C}{\nu} L(0) \leq \frac{C\de_*^3}{\nu}\, .
\end{align}
By \eqref{estimate of sum of shock L infinity} and \eqref{weight relation between L and de*}, we get
\begin{align*}
    \mcp_{21} &\leq C\nu \sqrt{\frac{\md}{\nu}} \sqrt{\inb \Bigl(\sum_{\be \in\mathcal{S}}|(\sv_\be)_x|\Bigr)^2 \frac{|\mu_1(v)v^{\ga} - \mu_1(\bv)\bv^{\ga}|^2}{\mu_1(v)v^{\ga}}\dx }  \\
    &\leq \de_* \md + C \de_*^2 \sum_{\be \in\mathcal{S}} \inb |(\sv_\be)_x| p(v|\bv) \dx \, .
\end{align*}
Since $0\leq \ga-\al\leq 1$, we have
\begin{align*}
    v^{-\ga+\al} \leq C + C p(v|\bv) \, .
\end{align*}
By \eqref{estimate of sum of shock L infinity} and Lemma \ref{lemma estimates of interaction terms}, we get
\begin{align*}
    \mcp_{22} &\leq C\nu \inb \Bigl|\bigl(p(v)-p(\bv)\bigr)_x\Bigr| \sum_{\be \in \mathcal{S}} |\bv - \sv_\be| |(\sv_\be)_x| \dx \\
    &\leq C\nu \sqrt{\frac{\md}{\nu}} \sqrt{\inb \Bigl(\sum_{\be \in \mathcal{S}} |\bv - \sv_\be| |(\sv_\be)_x| \Bigr)^2 v^{-\ga+\al}\dx} \\ 
    &\leq \de_*\md + \frac{C\nu}{\de_*} \inb \Bigl(\sum_{\be \in \mathcal{S}} |\bv - \sv_\be| |(\sv_\be)_x| \Bigr)^2 (p(v|\bv) + 1) \dx \\
    &\leq \de_*\md + C\de_*^2 \sum_{\be \in \mathcal{S}} \inb |(\sv_\be)_x| (p(v|\bv) \dx + C_n \nu^n \, .
\end{align*}
In all, we get
\begin{align}\label{estimate on P2}
    \mcp_{2} \leq 2\de_* \md + C\de_* \sum_{\be \in \mathcal{S}} \inb |(\sv_\be)_x| (p(v|\bv) \dx + C \nu \, .
\end{align}

\noindent
\textbf{Estimates for $\mcp_3$.}
We see
\begin{align*}
    \mcp_3 =& -\nu \inb \Bigl(\sum_{\be \in \mathcal{S}} (a_\be)_x\Bigr)\bigl(p(v)-p(\bv)\bigr) \, \mu_1(v) v^{\ga} \bigl(p(v)-p(\bv))_x \dx \\
    & -\nu \inb \Bigl(\sum_{\be \in \mathcal{S}} (a_\be)_x\Bigr)\bigl(p(v)-p(\bv)\bigr) \, \mu_1(v) v^{\ga} p'(\bv) \Bigl(\sum_{\be \in \mathcal{R}} (\rv_\be)_x + \sum_{\be \in \mathcal{NP}} (\nv_\be)_x\Bigr) \dx \\
    & -\nu \inb \Bigl(\sum_{\be \in \mathcal{S}} (a_\be)_x\Bigr)\bigl(p(v)-p(\bv)\bigr) \, \mu_1(v) v^{\ga} \sum_{\be \in \mathcal{S}} \bigl(p'(\bv) - p'(\sv_\be) \bigr) (\sv_\be)_x \dx \\
    =:&\, \mcp_{31} + \mcp_{32} + \mcp_{33} \, .
\end{align*}
We have
\begin{align*}
    \mcp_{31} &\leq C \nu \sqrt{\frac{\md}{\nu}} \sqrt{\inb \Bigl(\sum_{\be \in \mathcal{S}} (a_\be)_x\Bigr)^2 \bigl|p(v)-p(\bv)\bigr|^2 v^{\ga-\al} \dx} \\
    &\leq \de_*\md + \frac{C\nu}{\de_*}\inb \Bigl(\sum_{\be \in \mathcal{S}} (a_\be)_x\Bigr)^2\bigl|p(v)-p(\bv)\bigr|^2 v^{\ga-\al}\dx \, ,
\end{align*}
where the last term will be controlled in Section \ref{section shock}.

Recall \eqref{def of P1111}. We get
\begin{align*}
    \mcp_{32} &\leq \nu \inb \Bigl|\sum_{\be \in \mathcal{S}} (a_\be)_x\Bigr| \bigl|p(v)-p(\bv)\bigr| v^{\ga-\al} \Bigl|\sum_{\be \in \mathcal{R}} (\rv_\be)_x + \sum_{\be \in \mathcal{NP}} (\nv_\be)_x\Bigr| \dx \\
    &\leq C\nu\inb \Bigl(\sum_{\be \in \mathcal{S}} (a_\be)_x\Bigr)^2 \bigl|p(v)-p(\bv)\bigr|^2 v^{\ga-\al} \dx + C \de_* \mcp_{1111} \, .
\end{align*}
Recall \eqref{def of P1121}. We have
\begin{align*}
    \mcp_{33} &\leq C \nu \inb \Bigl|\sum_{\be \in \mathcal{S}} (a_\be)_x\Bigr|\bigl|p(v)-p(\bv)\bigr| v^{\ga-\al} \sum_{\be \in \mathcal{S}} \bigl|\bv -\sv_\be \bigr| |(\sv_\be)_x| \dx \\
    &\leq C \nu \inb \Bigl(\sum_{\be \in \mathcal{S}} (a_\be)_x\Bigr)^2 \bigl|p(v)-p(\bv)\bigr|^2 v^{\ga-\al} \dx + C \de_* \mcp_{1121} \, .
\end{align*}
In all, we get
\begin{align}\label{estimate on P3}
    \mcp_3 \leq \de_* \md + \frac{C \nu}{\de_*} \inb \Bigl(\sum_{\be \in \mathcal{S}} (a_\be)_x\Bigr)^2 \bigl|p(v)-p(\bv)\bigr|^2 v^{\ga-\al} \dx + C \de_* \mcp_{1111} + C \de_* \mcp_{1121} \, .
\end{align}

\noindent
\textbf{Estimates for $\mcp_4$.}
We have
\begin{align*}
    \mcp_4 =& -\nu \inb \Bigl(\sum_{\be \in \mathcal{S}} (a_\be)_x\Bigr) \bigl(p(v)-p(\bv)\bigr) \bigl( \mu_1(v) v^{\ga}- \mu_1(\bv) \bv^{\ga}\bigr) \sum_{\be \in \mathcal{S}} p(\sv_\be)_x \dx \\
    &-\nu \inb \Bigl(\sum_{\be \in \mathcal{S}} (a_\be)_x\Bigr) \bigl(p(v)-p(\bv)\bigr) \sum_{\be \in \mathcal{S}} \bigl( \mu_1(\bv) \bv^{\ga}-\mu_1(\sv_\be)(\sv_\be)^{\ga}\bigr) p(\sv_\be)_x \dx \\
    =: & \, \mcp_{41} + \mcp_{42} \,. 
\end{align*}
Since $0\leq \ga-\al\leq 1$, we have
\begin{align*}
    &\bigl|p(v)-p(\bv)\bigr| \Bigl| \mu_1(v) v^{\ga}- \mu_1(\bv) \bv^{\ga}\Bigr| \\
    &\leq C v^{-\ga} \mathds{1}_{v\leq \vs/2} + C |v-\bv|^2 \mathds{1}_{\vs/2< v< 2\vs} + C v^{\ga-\al}\mathds{1}_{v\geq 2\vs} \\
    &\leq C p(v|\bv) \, .
\end{align*}
By \eqref{def of weight a_i} and \eqref{estimate of sum of shock L infinity}, we get
\begin{align*}
    \mcp_{41} &\leq \frac{C\nu}{\de_*} \inb \Bigl(\sum_{\be \in \mathcal{S}} |(\sv_\be)_x| \Bigr)^2 p(v|\bv) \dx \\
    &\leq C \de_*^2 \sum_{\be \in \mathcal{S}} \inb |(\sv_\be)_x|  p(v|\bv) \dx \\
    &\leq C \de_*^2 \sum_{\be \in \mathcal{S}} \inb |(\sv_\be)_x|  p(v|\bv) \dx \, .
\end{align*}
Since $0\leq \ga-\al\leq 1$, we have
\begin{align*}
    \bigl|p(v)-p(\bv)\bigr| \leq C v^{-\ga} \mathds{1}_{v\leq \vs/2} + C \mathds{1}_{v> \vs/2} \leq C p(v|\bv) + C \, .
\end{align*}
By \eqref{def of weight a_i}, \eqref{estimate of sum of shock L infinity}, and Lemma \ref{lemma estimates of interaction terms}, we get
\begin{align*}
    \mcp_{42} &\leq \frac{C\nu}{\de_*} \inb \Bigl(\sum_{\be \in \mathcal{S}} |(\sv_\be)_x| \Bigr)\Bigl(\sum_{\be \in \mathcal{S}} |\bv - \sv_\be| |(\sv_\be)_x| \Bigr) \bigl(p(v|\bv)+1\bigr) \dx \\
    &\leq C \de_*^2 \Bigl[ \sum_{\be \in \mathcal{S}} \inb |(\sv_\be)_x| p(v|\bv) \dx + \inb \sum_{\be \in \mathcal{S}} |\bv - \sv_\be| |(\sv_\be)_x | \dx \Bigr] \\
    &\leq C\de_*^2 \sum_{\be \in \mathcal{S}} \inb |(\sv_\be)_x| p(v|\bv) \dx + C_n \nu^n \, .
\end{align*}
In all, we get
\begin{align}\label{estimate on P4}
    \mcp_4 \leq C \de_* \sum_{\be \in \mathcal{S}} \inb |(\sv_\be)_x|  p(v|\bv) \dx +C \nu \, .
\end{align}

\noindent
\textbf{Estimates for $\mcp_5,\mcp_6$.}
Let
\begin{align}\label{def of D tilde}
    \widetilde{\mathcal{D}} = \nu \inb a \mu_2(v) v^\ga \bigl|p(v)_x\bigr|^2 \dx \, .
\end{align}
We have
\begin{align*}
    \mcp_5 \leq - \widetilde{\mathcal{D}} + \mcp_{51} \, ,
\end{align*}
where
\begin{align*}
    \mcp_{51} = \nu \inb a p(\bv)_x \mu_2(v) v^\ga p(v)_x \dx \, .
\end{align*}
We also have
\begin{align*}
    \mcp_6 = \mcp_{61} + \mcp_{62} \, , 
\end{align*}
where
\begin{align*}
    &\mcp_{61} = - \nu \inb \Bigl(\sum_{\be \in \mathcal{S}} (a_\be)_x\Bigr) p(v) \mu_2(v) v^\ga p(v)_x \dx \, , \\
    &\mcp_{62} = \nu \inb \Bigl(\sum_{\be \in \mathcal{S}} (a_\be)_x\Bigr) p(\bv) \mu_2(v) v^\ga p(v)_x \dx \, .
\end{align*}

We first estimate $\mcp_{51}$ and $\mcp_{62}$. As in \cite{NStoE}, we define
\begin{align*}
    \Phi(v) = \int_{2v_*}^v \frac{\ga \mu_2(s)}{s} \, ds \, .
\end{align*}
Since $\Phi'(v) = -\mu_2(v) v^\ga p'(v)$ and
\begin{align}\label{parabolic evaluation of phi}
    \Phi(v) = \Phi(v) \mathds{1}_{v > 2 v_*}  \leq C| v^{-\widetilde{\al}} - (2v_*)^{-\widetilde{\al}} | \mathds{1}_{v > 2 v_*} \leq C p(v|\bv) \, ,
\end{align}
we can write 
\begin{align*}
    \mcp_{51} &= - \nu \int a p'(\bv) \bv_x \Phi(v)_x \dx = \nu \int A_{51} \, p(v|\bv) \dx + B_{51} \, , \\
    \mcp_{62} &= - \nu \int \Bigl(\sum_{\be \in \mathcal{S}} (a_\be)_x\Bigr) p(\bv) \Phi(v)_x \dx = \nu \int A_{62} \, p(v|\bv) \dx + B_{62} \, ,
\end{align*}
where
\begin{align*}
    A_{51} &= a p''(\bv) \bv_x^2 + a p'(\bv) \bv_{xx} + \Bigl(\sum_{\be \in \mathcal{S}} (a_\be)_x\Bigr) p'(\bv) \bv_x \, , \\
    A_{62} &= \Bigl(\sum_{\be \in \mathcal{S}} (a_\be)_x\Bigr) p'(\bv) \bv_x +  \Bigl(\sum_{\be \in \mathcal{S}} (a_\be)_{xx}\Bigr) p(\bv) \, ,
\end{align*}
and
\begin{align*}
    B_{51} &= -\nu a p'(\bv) \bv_x \Phi(v)|_{x=y_b}^\infty \, , \\
    B_{62} &= -\nu \Bigl(\sum_{\be \in \mathcal{S}} (a_\be)_x\Bigr) p(\bv) \Phi(v) |_{x=y_b}^\infty \, .
\end{align*}
By \eqref{def of weight}, Lemma \ref{lemma estimate of waves}, Corollary \ref{lemma strength of waves}, Lemma \ref{lemma number of waves and interactions}, and \eqref{weight relation between L and de*}, we have
\begin{align*}
    &|A_{51}| + |A_{62}| \\
    &\leq \frac{C}{\de_*} \Bigl(\sum_{\be\in \mathcal{S}} |(\sv_\be)_x| + \sum_{\be\in \mathcal{R}} |(\rv_\be)_x| + \sum_{\be\in \mathcal{NP}} |(\nv_\be)_x| \Bigr)^2 \\
    &\quad +\frac{C}{\de_*} \Bigl(\sum_{\be\in \mathcal{S}} |(\sv_\be)_{xx}| + \sum_{\be\in \mathcal{R}} |(\rv_\be)_{xx}| + \sum_{\be\in \mathcal{NP}} |(\nv_\be)_{xx}| \Bigr) \\
    &\leq \frac{C}{\de_*} \Bigl(\sum_{\be\in \mathcal{S}} |(\sv_\be)_x|^2  + \sum_{\be\in \mathcal{R}} |(\rv_\be)_x|^2 + \sum_{\be\in \mathcal{NP}} |(\nv_\be)_x|^2 + C_n \nu^n \\
    &\quad + \sum_{\be\in \mathcal{S}} \frac{|\s_\be|}{\nu} |(\sv_\be)_x| +  \ml(0)\rn^{-2\co} \Bigr)\\
    &\leq \frac{C}{\de_*}\sum_{\be\in \mathcal{S}} \frac{|\s_\be|}{\nu} |(\sv_\be)_x| + C\de_*^2\rn^{-2\co} \\
    &\leq C \frac{\de_*^2}{\nu} \sum_{\be\in \mathcal{S}}  |(\sv_\be)_x| + C \rn^{-2\co} \, .
\end{align*}
By \eqref{def of weight}, Lemma \ref{lemma estimate of waves}, Lemma \ref{lemma number of waves and interactions}, and $\Phi(v)\leq C$, we have
\begin{align*}
    |B_{51}| + |B_{62}| \leq \frac{C}{\de_*} C_n \nu^n \rn^{-\cpr} \ln \rn^{-1} \leq C_n \nu^n \, .
\end{align*}
Since $p(v|\bv) \mathds{1}_{v>2v_*}\leq C Q(v|\bv) \mathds{1}_{v>2v_*}$, we get
\begin{equation}\label{estimate on P51 and P62}
    \begin{aligned}
        \mcp_{51} + \mcp_{62} &\leq C\de_*^2 \sum_{\be \in \mathcal{S}} \inb |(\sv_\be)_x| (p(v|\bv) \dx + C \nu^{1-2\cs\co} \inb Q(v|\bv) \dx + C_n \nu^n \, .
    \end{aligned}
\end{equation}

We estimate $\mcp_{61}$. Since $p(v)^2 \mu_2(v) v^\ga \leq C \mathds{1}_{v>2v_*} $,
\begin{align*}
    p(v)^2 \mu_2(v) v^\ga \leq C p(v|\bv) \, , \quad p(v)^2 \mu_2(v) v^\ga \leq C Q(v|\bv) \, .
\end{align*}
Together with \eqref{def of weight}, we have
\begin{align*}
    \mcp_{61} &\leq \frac{1}{2} \widetilde{\mathcal{D}} +  \frac{C \nu}{\de_*^2} \inb \Bigl(\sum_{\be\in \mathcal{S}} |(\sv_\be)_x|\Bigr)^2 p(v|\bv) \, dx  \, .
\end{align*}
By Lemma \ref{lemma estimate of waves} and \eqref{weight relation between L and de*}, we get
\begin{align*}
    \Bigl(\sum_{\be\in \mathcal{S}} |(\sv_\be)_x|\Bigr)^2 \leq \sum_{\be\in \mathcal{S}} |(\sv_\be)_x|^2 + C_n \nu^n  \leq C \frac{\de_*^6}{\nu}\sum_{\be\in \mathcal{S}} |(\sv_\be)_x| + C_n \nu^n \, .
\end{align*}
In all, we get
\begin{align}\label{estimate on P61}
    \mcp_{61}\leq \frac{1}{2} \widetilde{\mathcal{D}} + C\de_*^4 \sum_{\be \in \mathcal{S}} \inb |(\sv_\be)_x| (p(v|\bv) \dx + C_n \nu^n \inb Q(v|\bv) \dx \, .
\end{align}

\noindent
\textbf{Summary.}
By \eqref{estimate on P1}, \eqref{estimate on P2}, \eqref{estimate on P3}, \eqref{estimate on P4}, \eqref{estimate on P51 and P62}, \eqref{estimate on P61}, \eqref{rt transition R}, and \eqref{rt transition R NP}, we have
\begin{equation}\label{estimate on P}
    \begin{aligned}
        \mcp &\leq -(1-C\de_*)\md - \frac{1}{2} \widetilde{\mathcal{D}} + C\de_* \sum_{\be \in \mathcal{S}} \inb |(\sv_\be)_x| (p(v|\bv) \dx + C \nu^{1-2\cs\cpr}  \sum_{\be \in \mathcal{R}} G_\be^R \\
        &\quad + C \inb Q(v|\bv) \dx + C \nu^{1-2\cs\cpr} + \frac{C \nu}{\de_*} \inb \Bigl(\sum_{\be \in \mathcal{S}} (a_\be)_x\Bigr)^2 \bigl|p(v)-p(\bv)\bigr|^2 v^{\ga-\al} \dx \, .
    \end{aligned}
\end{equation}

\section{Estimates for main terms localized by derivatives of weights and shocks}\label{section shock}

In this section, we bound the terms localized by the derivatives of the weights and the shocks. We first summarize the results obtained in Section~\ref{sec rarefaction}, Section~\ref{sec NP and boundary}, and Section~\ref{sec parabolic}. We then maximize the terms with respect to $h-\bh$ when $p(v)$ is small compared to $p(\bv)$ in Section~\ref{subsec maximize h}, and define the shifts associated with the shocks in Section~\ref{subsec def of shift}. In Section~\ref{subsec proof of energy estimate between stopping times}, we state Proposition~\ref{proposition shock estimate}, which provides the main estimates in this section for the terms localized by the derivatives of the weights and the shocks, and prove Proposition~\ref{proposition energy estimate}, which gives the estimates between stopping times. The remainder of the section is devoted to proving Proposition~\ref{proposition shock estimate}. We truncate $v$ and estimate the region near $\bar v$ using a Poincar\'e-type inequality for each shock in Section~\ref{subsection estimate inside the truncation}, and the region outside the truncation using the fundamental theorem of calculus in Section~\ref{subsection estimate outside the truncation}.

This section closely follows \cite{NStoE}, and the overall structure of the analysis is similar. The main adaptation is that we work in a half space determined by the Navier-Stokes solution rather than on the whole space, which leads to a slightly different Poincar\'e inequality. Thanks to the construction of the approximate solution in Section~\ref{subsec mft}, in particular Property~\ref{prop transition zone}, the wave fronts remain sufficiently far from the boundary, allowing this Poincar\'e inequality to be reduced to the whole space case.

Since our front tracking construction substantially modifies that of \cite{NStoE}, we present the full argument carefully and include the details in the appendix. We also take this opportunity to reorganize the proof slightly by isolating several estimates that are used repeatedly into separate lemmas.

\medskip

Recall $\mathcal{H}_1$ defined in \eqref{relative entropy method terms}. Lemma \ref{lemma estimates of interaction terms} implies the last term can be bound as follows.
\begin{equation*}
    \begin{aligned}
        &\Bigl|\sum_{\be \in \mathcal{S}} \inb a (h-\bh) \bigl( p'(\bv) - p'({\sv}_\be) \bigr) ({\sv}_\be)_x \dx \Bigr| \\
        &\leq C \inb a |h-\bh| \sum_{\be \in \mathcal{S}} |\bv-{\sv}_\be| \bigl|({\sv}_\be)_x\bigr| \dx \\
        &\leq C \inb a \eta(U|\bu) \dx + C_n \nu^n \, .
    \end{aligned}
\end{equation*}
Together with \eqref{relative entropy method inequality}, \eqref{estimate on H_2}, \eqref{estimate on H3}, \eqref{estimate on B}, \eqref{estimate on P}, \eqref{rt mft total number}, and \eqref{weight a close to 1}, we have
\begin{equation}\label{relative entropy estimate general 1}
    \begin{aligned}
        \frac{d}{dt} \inb a \eta(U|\bu) \dx &\leq \mr(U) + C \inb a \eta(U|\bu) \dx - C_G \sum_{\be \in R}G_\be^R(U) - \frac{1}{2}\widetilde{\mathcal{D}}\\
        &\qquad + C \nu^{\cbet}\ln \rn^{-1} h - \frac{3}{4} |{y}_b'(t)| \eta(U|\bu)\big|_{x=y_b(t)} \, ,
    \end{aligned}
\end{equation}
where $C,C_G>0$ are constant depending only on $U_*$, $G_\be^R$ is defined in \eqref{def of G_i^R}, $\widetilde{\mathcal{D}}$ is defined in \eqref{def of D tilde}, $\cbet$ is defined in \eqref{def of bold c}, $h$ is defined in \eqref{def of D_v, G}, and
\begin{equation}\label{def of R}
    \begin{aligned}
        \mr(U) = \sum_{\be \in \ms} \dot{X}_\be Y_\be(U) + \overline{\mathcal{H}_1}(U) - \mj_{good}(U) - (1-C\de_*) \md(U) \\
        + C \de_* \sum_{\be \in \ms} \inb \bigl| ({\sv}_\be)_x \bigr| p(v|\bv) \dx + \frac{C \nu}{\de_*} \inb \Bigl(\sum_{\be \in \mathcal{S}} |(a_\be)_x|\Bigr)^2 \bigl|p(v)-p({\bv})\bigr|^2 v^{\ga-\al} \dx \, , 
    \end{aligned}
\end{equation}
where $Y_\be(U)$ is defined in \eqref{relative entropy method terms}, $\md$ is defined in \eqref{def of D}, and 
\begin{align*}
    \overline{\mathcal{H}_1}(U) = \sum_{\be \in \ms} \inb (a_\be)_x \bigl(p(v)-p(\bv)\bigr) (h-\bh) \dx + \sum_{\be \in \ms} \las_\be \inb a (\sv_\be)_x p(v|\bv) \dx \, .
\end{align*}
\subsection{Maximization in terms of $h-\bh$}\label{subsec maximize h}
By \eqref{def of R}, we can write
\begin{equation}\label{def of R(U), J}
    \begin{aligned}
        \mr(U) + \de_* \int |(\sv_\be)_x| p(v|\bv) \dx 
        = \sum_{\be \in \ms} \dot{X}_\be Y_\be + \mj_{bad}(U) - \mj_{good}(U) - (1-C \de_*) \md(U) \, ,
    \end{aligned}
\end{equation}
where
\begin{equation}\label{def of J_bad}
    \begin{aligned}
        \mj_{bad}(U) &= \overline{\mathcal{H}_1}(U) + C \de_* \sum_{\be \in \ms} \inb \bigl| ({\sv}_\be)_x \bigr| p(v|\bv) \dx \\
        &\qquad + \frac{C\nu}{\de_*} \inb \Bigl(\sum_{\be \in \mathcal{S}} |(a_\be)_x|\Bigr)^2 \bigl|p(v)-p(\bv)\bigr|^2 v^{\ga-\al} \dx \, .
    \end{aligned}
\end{equation}
We define
\begin{align}\label{def of Omega}
    \om = \{x : \bigl(p(v) - p(\bv)\bigr) (x) \leq \de_1 \} \, ,
\end{align}
where $\de_1$ is chosen in Section \ref{subsec choice of e_0}. We maximize $\mj_{bad}-\mj_{good}$ with respect to $h-\bh$ on $\om$ as follows
\begin{align*}
    \mj_{bad}(U) - \mj_{good}(U) = B_{\de_1}(U) - G_{\de_1}(U) \, ,
\end{align*}
where
\begin{equation}\label{def of B and G}
    \begin{aligned}
        B_{\de_1}(U) &= \sum_{\be \in \ms} \bigl( B_{1\be}(U) + B_{2\be}^+(U) + B_{2\be}^-(U) \bigr) + B_3(U) \, , \\
        G_{\de_1}(U) &= \sum_{i \in \ms} \bigl( G_{1\be}^+(U) + G_{1\be}^-(U) + G_{2\be}(U) \bigr)\, ,
    \end{aligned}
\end{equation}
where
\begin{align*}
    B_{1\be}(U) &= \int (|\las_\be|a + \bc\de_*) |(\sv_\be)_x| p(v|\bv) \dx \, , \\
    B_{2\be}^+(U) &= \frac{1}{2\las_\be} \int_\om (a_\be)_x \bigl|p(v)-p(\bv)\bigr|^2 \dx \, , \\
    B_{2\be}^-(U) &= \int_{\om^c} (a_\be)_x \bigl(p(v) - p(\bv)\bigr) (h-\bh) \dx \, , \\
    B_3(U) &= \frac{C\nu}{\de_*} \int \Bigl(\sum_{\be \in \ms} |(a_\be)_x| \Bigr)^2 v^{\ga-\al} \bigl|p(v) - p(\bv)\bigr|^2 \dx \, , \\
\end{align*}
for some constant 
\begin{align}\label{def of bc}
    \bc > 0 \, ,
\end{align}
depending only on $U_*$, and 
\begin{align*}
    G_{1\be}^+(U) &= \frac{\las_\be}{2} \int_{\om} (a_\be)_x \Bigl|h-\bh - \frac{\pd}{\las_\be}\Bigr|^2 \dx \, , \\
    G_{1\be}^-(U) &= \frac{\las_\be}{2} \int_{\om^c} (a_\be)_x |h-\bh|^2 \dx \, , \\
    G_{2\be}(U) &= \las_\be \int (a_\be)_x Q(v|\bv) \dx \, . 
\end{align*}
By \eqref{def of R(U), J}, we get
\begin{equation}\label{def of R(U), B, G}
    \begin{aligned}
        \mr(U) + \de_* \int |(\sv_\be)_x| p(v|\bv) \dx = \sum_{\be \in \ms} \dot{X}_\be Y_\be + B_{\de_1}(U) - G_{\de_1}(U) - (1-C \de_*) \md(U) \, .
    \end{aligned}
\end{equation}
We also maximize $Y_\be$ with respect to $h-\bh$ on $\om$ as follows.
\begin{equation}\label{def of Yig, Yib, Yil, Yis}
    \begin{aligned}
        Y_\be(U) = Y_\be^g(U) + Y_\be^b(U) + Y_\be^l(U) + Y_\be^s(U) \, ,
    \end{aligned}
\end{equation}
where
\begin{align*}
    Y_\be^g(U) &= - \frac{1}{2\las_\be^2} \int_\om (a_\be)_x \bigl|p(v)-p(\bv)\bigr|^2 - \int_\om (a_\be)_x Q(v|\bv) \dx \\ &\quad - \int_\om a (\sv_\be)_x p'(\bv)(v-\bv) \dx+ \frac{1}{\las_\be} \int_\om a (\sh_\be)_x \bigl(p(v) - p(\bv) \bigr) \dx \, , \\
    Y_\be^b(U) &= -\frac{1}{2} \int_\om (a_\be)_x \Bigl( h-\bh - \frac{p(v)-p(\bv)}{\las_\be}\Bigr)^2 \dx \\
    &\quad - \frac{1}{\las_\be} \int_\om (a_\be)_x \bigl(p(v)-p(\bv)\bigr)\Bigl(h-\bh - \frac{p(v)-p(\bv)}{\las_\be}\Bigr) \dx \, , \\
    Y_\be^l(U) &= \int_\om a (\sh_\be)_x \Bigl( h-\bh - \frac{p(v)-p(\bv)}{\las_\be}\Bigr) \dx \, , \\
    Y_\be^s(U) &= - \int_{\om^c} (a_\be)_x Q(v|\bv) \dx - \int_{\om^c} a (\sv_\be)_x p'(\bv)(v-\bv) \dx  \\
    &\quad - \int_{\om^c} (a_\be)_x \frac{|h-\bh|^2}{2} \dx + \int_{\om^c} a (\sh_\be)_x (h-\bh)  \dx \, .
\end{align*}

\subsection{Construction of shifts for shocks}\label{subsec def of shift}
We define the shift in the same way as in \cite{NStoE}. Let us first define a Lipschitz function $\Phi_\epsilon$ by
\begin{equation*}
    \Phi_\epsilon =
    \begin{cases}
        0 \, , \quad \textup{if } y \leq 0 \, , \\
        -\frac{1}{\epsilon^4} y \, , \quad \textup{if } 0\leq y \leq \epsilon^2 \, , \\
        -\frac{1}{\epsilon^2} \, , \quad \textup{if } y \geq \epsilon^2 \, .
    \end{cases}
\end{equation*}
For $t \in [t_j, t_{j+1})$, i.e., between the $j$-th and $(j+1)$-st stopping times, we define the shift function $X_\be$ for any $i$-shock $\be \in \ms$ by the following nonlinear ODEs.
\begin{equation}\label{ODE of Xi}
    \begin{cases}
        \dot{X}_\be(t) = (-1)^{i-1}\Phi_{\s_\be}\bigl((-1)^{i-1} Y_\be(U)\bigr) \bigl( 2|\mj_{bad}(U)| + 1 \bigr)  + \s_\be^2 C_X{\las_\be}\Phi_{\s_\be}\bigl((-1)^{i}Y_\be(U)\bigr) \, , \\
        X_\be(t_j) = 0 \, ,
    \end{cases}
\end{equation}
where $0<C_X<1$ is the constant chosen in Lemma \ref{lemma estimates on relative functions 2}, and $\mj_{bad}(U)$ is defined in \eqref{def of J_bad}. By Appendix C of \cite{twoshocks}, there exists a unique absolutely continuous solution $X_\be$ on $[t_j, t_{j+1})$.
By \eqref{ODE of Xi}, we get that for each $\be \in \ms$, 
\begin{equation}\label{X_i estimates}
    \dot{X}_\be = \begin{cases}
        (-1)^{i} \s_\be^{-2} \bigl( 2|\mj_{bad}(U)|  + 1 \bigr) \, , \quad \textup{if } (-1)^{i-1} Y_\be(U) \geq \s_\be^2 \, , \\
        -\s_\be^{-4} Y_\be(U) \bigl( 2|\mj_{bad}(U)|  + 1 \bigr) \, , \quad \textup{if } 0\leq (-1)^{i-1} Y_\be(U) \leq \s_\be^2 \, , \\
        (-1)^{i-1} C_X{\las_\be} \s_\be^{-2} Y_\be(U) \, , \quad \textup{if } -\s_\be^2 \leq (-1)^{i-1} Y_\be(U) \leq 0 \, , \\
        -C_X{\las_\be} \, , \quad \textup{if } (-1)^{i-1} Y_\be(U) \leq -\s_\be^2 \, .
    \end{cases}
\end{equation}
We see that
\begin{align*}
    \dot{X}_\be \leq - C_X{\las_\be} \, , \quad \textup{if } i = 1 \, , \quad \dot{X}_\be \geq - C_X{\las_\be} \, , \quad \textup{if } i = 2 \, .
\end{align*}
Hence, we have 
\begin{align}\label{speed of shock bound}
    \dot{X}_\be + \las_\be \leq (1-{C_X})\las_\be<0 \, , \quad \textup{if } i = 1 \, , \quad \dot{X}_\be + \las_\be \geq (1-C_X){\las_\be}>0 \, , \quad \textup{if } i = 2 \, ,
\end{align}
which ensures the separation of waves.

\subsection{Proof of Proposition \ref{proposition energy estimate}}\label{subsec proof of energy estimate between stopping times}
By \eqref{X_i estimates}, we have
\begin{equation}\label{X_i Y_i estimates}
    \dot{X}_\be Y_\be(U) = 
    \begin{cases}
        - 2|\mj_{bad}(U)| \, , \quad \textup{if } (-1)^{i-1} Y_\be(U) \geq \s_\be^2 \, , \\
        -\s_\be^{-4} |Y_\be(U)|^2 \, , \quad \textup{if } 0\leq (-1)^{i-1} Y_\be(U) \leq \s_\be^2 \, , \\
        (-1)^{i-1} C_X{\las_\be} \s_\be^{-2} |Y_\be(U)|^2 \, , \quad \textup{if } -\s_\be^2 \leq (-1)^{i-1} Y_\be(U) \leq 0 \, , \\
        -C_X{\las_\be} Y_\be(U) \, , \quad \textup{if } (-1)^{i-1} Y_\be(U) \leq -\s_\be^2 \, .
    \end{cases}
\end{equation}
Recall \eqref{def of R(U), J}. For any $U \in \cup_{\be \in \ms} \{U : (-1)^{i-1}Y_\be(U) > \s_\be^2\}$, we have
\begin{align}\label{R(U) estimate simple case}
    \mr(U) + \de_* \int |(\sv_\be)_x| p(v|\bv) \dx \leq - |\mj_{bad}(U)|  - \mj_{good}(U) - (1-C \de_*)\md(U) \leq 0 \, .
\end{align}
For $U \in \cup_{\be \in \ms} \{U : (-1)^{i-1}Y_\be(U) > \s_\be^2\}^c = \cap_{\be \in \ms} \{U : (-1)^{i-1}Y_\be(U) \leq \s_\be^2\}$, we apply the following proposition, which corresponds to Proposition 10.1 in \cite{NStoE}. 

The remainder of this section after Section~\ref{subsec proof of energy estimate between stopping times} is devoted to its proof. More precisely, the main idea in proving Proposition~\ref{proposition shock estimate} is to truncate $v$ near $\bv$. The proof is then divided into estimates within and outside the truncation region. The estimates within the truncation region are derived in Section~\ref{subsection estimate inside the truncation}, while those outside the truncation region are derived in Section~\ref{subsection estimate outside the truncation}. These estimates are combined in Appendix~\ref{section proof of prop shock estimate} to complete the proof of Proposition~\ref{proposition shock estimate}.
\begin{prop}\label{proposition shock estimate}
There exists a constant $\cde>0$, depending only on $U_*$ and $\de_1$, such that for any $U\in \cap_{\be \in \ms} \{U : (-1)^{i-1}Y_\be(U) \leq \s_\be^2\}$, the following holds.
\begin{align*}
    &\mr(U) + \de_* \int |(\sv_\be)_x| p(v|\bv) \dx + \frac{\de_1}{4} \md(U) \\
    &\leq \cde \de_*^{1/2} \sum_{\be \in \mr} \int |(\rv_\be)_x| p(v|\bv) \dx + \cde \de_*^{1/2} \sum_{\be \in \mnp} \int |(\nv_\be)_x| Q(v|\bv) \dx \\
    &\qquad+ C\nu^2 + \frac{C\nu^{1-\cs(2\co+\cpr)}\ln \rn^{-1}}{\de_*}  \int a \eta(U|\bu) \dx \, .
\end{align*}
\end{prop}

By \eqref{R(U) estimate simple case} and Proposition \ref{proposition shock estimate}, 3) in Lemma \ref{lemma estimate of waves}, 3) in Corollary \ref{lemma strength of waves}, we have
\begin{align*}
    &\mr(U) + \de_* \int |(\sv_\be)_x| p(v|\bv) \dx + C \md (U) \\
    &\leq \cde \de_*^{1/2} \sum_{\be \in \mr} G_\be^R(U) + \cde \de_*^{1/2} \nu^{\cs(1-\co-\cpr)} \ln\rn^{-1} \int Q(v|\bv) \dx \\
    &\qquad + C\nu^2 + \frac{C\nu^{1-\cs(2\co+\cpr)}\ln \rn^{-1}}{\de_*}  \int a \eta(U|\bu) \dx \, ,
\end{align*}
where $G_\be^R$ is defined in \eqref{def of G_i^R}. By \eqref{rt mft total number}, we get
\begin{align*}
    &\mr (U) + \de_* \int |(\sv_\be)_x| p(v|\bv) \dx + C \md(U) \\
    &\leq \widetilde{C}_G \de_*^{1/2} \sum_{\be \in \mr} G_\be^R(U) +  C\nu^2 + C \int a \eta(U|\bu) \dx \, ,
\end{align*}
for some constants $\widetilde{C}_G,C>0$ depending only on $U_*,\de_1$.

By \eqref{relative entropy estimate general 1}, we can take $\de_*$ small enough such that
\begin{align}\label{relation delta 5}
    \widetilde{C}_G \de_* < C_G \, ,
\end{align}
and get \eqref{proposition energy estimate eq}.

\subsection{Decomposition of the domain}\label{subsec domain decomposition}
We first decompose $\md$ with respect to the transition zone $I_\be$, defined in \eqref{def of transition zone}, of each shock $\be$.  More precisely, we define 
\begin{equation}\label{def of ji}
    \ji =
    \begin{cases}
        \bigl(y_b, y_\be + \frac{y^R}{4}] \, , \quad \textup{if } y_\be \, \textup{ is the leftmost wave,} \\
        (y_\be - \frac{y^L}{4} , \infty) \, , \quad \textup{if } y_\be \textup{ is the rightmost wave,} \\
        (y_\be - \frac{y^L}{4} , y_\be + \frac{y^R}{4}] \, , \quad \textup{otherwise.}
    \end{cases}
\end{equation}
where $y^R$ (respectively, $y^L$) denotes the absolute distance between $y_\be$ and the wave adjacent to it on the right (respectively, left).

We note that
\begin{align*}
    (y_b,\infty) = \cup_{\be \in \ms \cup \mr \cup \mnp} \ji \, .
\end{align*}
We observe
\begin{align}\label{scale of Ji}
    |\ji| \geq \frac{\rn^{\co}}{2} = \frac{\nu^{\co\cs}}{2}\, .
\end{align}
For each $\be \in \ms$, we introduce a cutoff function $\ts$
\begin{equation}\label{def of cutoff function}
    \ts = 
    \begin{cases}
        1 \, , \quad \textup{if } x\in \ji \, , \\
        0 \, , \quad \textup{if } x\in \overline{\ji}^c \, , \\
        \textup{linear,} \quad \textup{otherwise,}
    \end{cases}
\end{equation}
where
\begin{align*}
    \overline{\ji}(t) =
    \begin{cases}
        \bigl(y_b, y_\be + \frac{y^R}{2}] \, , \quad \textup{if } y_\be \, \textup{ is the leftmost wave,}  \\
        (y_\be - \frac{y^L}{2} , \infty) \, , \quad \textup{if } y_\be \textup{ is the rightmost wave,} \\
        (y_\be - \frac{y^L}{2} , y_\be + \frac{y^R}{2}] \, , \quad \textup{otherwise.}
    \end{cases}
\end{align*}
We note that 
\begin{align}\label{properties of ts}
    \sum_{\be \in \ms} \ts \leq 1 \, , \quad 0\leq \ts^2\leq \ts \leq 1 \, , \quad |(\ts)_x| \leq C\nu^{-\cs\co}
\end{align}

\subsection{Sharp estimates inside the truncation}\label{subsection estimate inside the truncation}
For any $\de>0$, we define 
\begin{align}
    \psi_{\de}(y) = \inf \bigl(\de, \sup(-\de, y)\bigr) \, ,
\end{align}
and define $\lv_\de$ by
\begin{align}\label{def of truncation general}
    p(\lv_{\de}) - p(\bv) = \psi_{\de}\bigl(p(v)-p(\bv)\bigr) \, .
\end{align}
In other words, we truncate $v$ around $\bv$ at level $\de$.

We consider the following functionals
\begin{equation}\label{terms associated with shocks}
    \begin{aligned}
        \my_\be^g(v) =& - \frac{1}{2\las_\be^2} \inb (a_\be)_x \ts^2 |\pd|^2 \dx - \inb (a_\be)_x \ts^2 Q(v|\bv) \dx \\
        & - \inb a \, \p_x p(\sv_\be) \ts (v-\bv) \dx + \frac{1}{\las_\be} \inb a (\sh_\be)_x \ts (\pd) \dx \, \\
        \mi_{1\be}(v) =& \, \las_\be \inb a (\sv_\be)_x \ts^2 p(v|\bv) \dx \, , \\
        \mi_{2\be}(v) =& \,  \frac{1}{2\las_\be} \inb (a_\be)_x \ts^2 |\pd|^2 \dx \, , \\
        \mg_{2\be}(v) =& \, \las_\be \inb (a_\be)_x \Bigl( \frac{1}{2\ga} p(\sv_\be)^{-\frac{1}{\ga}-1}\ts^2 \bigl(\pd\bigr)^2 \\
        &- \frac{1+\ga}{3\ga^2} p(\sv_\be)^{-\frac{1}{\ga}-2}\ts^3\bigl(\pd\bigr)^3\Bigr) \dx \, , \\
        \widetilde{\md}_\be(v) =& \nu\inb a \mu_1(v)v^\ga \Bigl| \bigl(\ts \pd\bigr)_x \Bigr|^2 \dx \, .
    \end{aligned}
\end{equation}
Recall \eqref{def of weight a_i}. We have
\begin{align*}
    \mi_{1\be}(v) \, , \; \mi_{2\be}(v) \geq 0 \, . 
\end{align*}

We first introduce a lemma containing estimates on relative functions.
\begin{lemm}\label{lemma relative evaluations}
Let $\ga>1$ and $v_*>0$. There exist sufficiently small constants $\de_0,C_0>0$ and a constant $C>0$ such that the following holds.

\noindent
Assume $\ml(0) \leq C_0 \de_0$ and $\bv$ is constructed as in Section \ref{sec mft}. Then 
\begin{align}\label{relative evaluation Q lower bounds}
    C Q(w|\bv) \geq  
    \begin{cases}
        |w-\bv|^2 \, , \quad \textup{if } 0<w<3v_* \, , \\
        |w-\bv| \, , \quad \textup{if } w\geq 3v_* \, .
    \end{cases}
\end{align}
We define $\lv=\lv_\de$, with $\de \leq \de_0$, as in \eqref{def of truncation general}. Then
\begin{align}\label{relative evaluation Q}
    Q(v|\bv) \geq Q(\lv|\bv) \, ,
\end{align}
and there exists some constant $C_{\de} > 0 $, depending on $\de$, such that
\begin{align}\label{relative evaluation Q and ||}
    Q(v|\bv) - Q(\lv|\bv) \geq C_{\de}|v-\lv| \, .
\end{align}
Finally, we have
\begin{align}\label{relative evaluation A.3}
    Q(\lv|\bv) \geq \frac{p(\bv)^{-\frac{1}{\ga}-1}}{2\ga} \bigl(p(\lv) - p(\bv) \bigr)^2 - \frac{1+\ga}{3\ga^2} p(\lv)^{-\frac{1}{\ga}-2}\bigl(p(\lv) - p(\bv)\bigr)^3 \, .
\end{align}

\end{lemm}

The following lemma, corresponding to Lemma 10.2 in \cite{NStoE}, shows that the relative entropy localized near a shock $\su_\be$ is small provided that the associated quantity $(-1)^{i-1}Y_\be$ is small. The proof follows the same argument. In particular, the presence of the boundary does not affect the proof. For completeness, we include a short proof in Appendix \ref{section proof lemma smallness of localized relative entropy}.

Recall that $\de_*$ is used to define the weight function in Section \ref{subsec construction of weights}, and that it bounds the weighted total variation $\ml$ as in \eqref{weight relation between L and de*}.

\begin{lemm}\label{lemma smallness of localized relative entropy 1}
There exist $C,\de_0>0$, depending only on $U_*$, such that the following holds. 

Assume $\de_*<\de_0$. Then whenever $(-1)^{i-1}Y_\be \leq \s_\be^2$, 
\begin{align}\label{lemma smallness eq 1}
    \inb |(a_\be)_x| |h-\bh|^2 \dx + \inb |(a_\be)_x| Q(v|\bv) \dx \leq C |\s_\be| \de_* \, ,
\end{align}
and 
\begin{align}\label{lemma smallness eq 2}
    |Y_\be(U)| \leq C |\s_\be| \de_* \, .
\end{align}
\end{lemm}
The following lemma, corresponding to Lemma 10.3 in \cite{NStoE}, shows that $\my_\be^g(\lv_{\de})$ is small. We note that the constant $C_1$ in this lemma depends on the constant $C$ in Lemma \ref{lemma smallness of localized relative entropy 1}.

\begin{lemm}\label{lemma smallnss of Y_i}
There exist $C_1, \de_0 > 0$, depending only on $U_*$, such that the following is true.

Assume $\de_*< \de_0$. Then whenever $(-1)^{i-1}Y_\be \leq \s_\be^2$,
\begin{align}\label{lemma smallness of Y_i eq}
    |\my_\be^g(\lv_{\de})| \leq C_1 |\s_\be| \de_* \, , \quad \textup{for any } \de \leq \de_0 \, .
\end{align}
\end{lemm}
\begin{proof}
By \eqref{relative evaluation Q lower bounds} and \eqref{lemma smallness eq 2}, we have
\begin{align*}
    |\my_\be^g(\lv_{\de})| &\leq C\inb |(a_\be)_x| |\lv_{\de}-\bv|^2 \dx + C\inb |(\sv_\be)_x| |\lv_{\de}-\bv| \dx \\ 
    & \leq C\inb |(a_\be)_x| Q(\lv_{\de}|\bv) \dx + C\sqrt{\inb |(\sv_\be)_x|\dx} \sqrt{\inb |(\sv_\be)_x| Q(\lv_{\de}|\bv) \dx} \\
    &\leq C\inb |(a_\be)_x| Q(v|\bv) \dx + C \sqrt{|\s_\be|}\sqrt{\inb |(\sv_\be)_x| Q(v|\bv) \dx} \\
    &\leq C|\s_\be| \de_* \, .
\end{align*}
\end{proof}
The following lemma, corresponding to Lemma 10.1 in \cite{NStoE}, provides sharp estimates for the terms localized by a shock profile in terms of the truncation of $v$ near $\bv$. The proof, in particular the change of variables and the derivation of the Poincar\'e inequality, is essentially the same. The presence of the boundary leads to a slightly different final form of the Poincar\'e inequality. However, since the transition zones of all shocks lie strictly inside the domain, as implied by Property~\ref{prop transition zone}, this inequality can be reduced to the one in \cite{NStoE}. For completeness, we include the detailed proof in Appendix~\ref{section proof lemma inside truncation}.

\begin{lemm}\label{lemma inside truncation}
For any constant $C_1>0$, there exists $\de_0>0$, depending only on $C_1$ and $U_*$, such that the following is true.

Assume $\de_* < \de_0$ and let $\be \in \mathcal{S}$. For any function $v$ such that $\widetilde{\md}_\be(v) + \mg_{2\be}(v)$ is finite, if 
\begin{align}\label{lemma inside truncation assumption}
    |\my_\be^g(v)| \leq C_1 |\s_\be| \de_* \, , \quad \|\pd\|_{L^\infty} \leq \de_0 \, ,
\end{align}
then
\begin{equation}\label{lemma inside truncation conclusion}
    \begin{aligned}
        \mathcal{R}_{\de_*}^\be(v) = & (1+\bc\de_0)|\mi_{1\be}(v)| + (1+\de_0\de_*)|\mi_{2\be}(v)| - (1-\de_0\de_*) \mg_{2\be}(v) \\
        &- (1-\de_0) \widetilde{\md}_\be(v) -\frac{1}{|\s_\be|\de_0} \my_\be^g(v)^2 \leq 0 \, ,
    \end{aligned}
\end{equation}
where $\bc$ are defined in \eqref{def of bc}, $\mi_{1\be}, \mi_{2\be}, \mg_{2\be}, \widetilde{\md}_\be, \my_\be^g$ are defined in \eqref{terms associated with shocks}, and both $\mi_{1\be}, \mi_{2\be}$ are nonnegative.
\end{lemm}

\subsection{Estimates outside truncation}\label{subsection estimate outside the truncation}
We define the truncation $\lv$ of $v$ relative to $\bv$ by
\begin{align}\label{def of truncation}
    \lv = \lv_{\de_1} \, \, \quad \lu = (\lv, h) \, ,
\end{align}
where $\bv_{\de_1}$ is defined in \eqref{def of truncation general} with $\de_1$ chosen in Section \ref{subsec choice of e_0}.

The following proposition corresponds to Proposition 10.2 in \cite{NStoE} and serves as the main estimate outside the truncation in the proof of Proposition~\ref{proposition shock estimate}. The proof follows the same idea. For completeness, we go through the argument equation by equation and provide details for the parts that differ from those in \cite{NStoE}. The main technical ingredient is given later in Lemma~\ref{lemma technical}, which corresponds to Lemma 10.4 in \cite{NStoE}.

\begin{prop}\label{prop difference}
There exists $C>0$, depending on $U_*,\de_1$, such that the following holds.
For each $\be \in \ms$,
\begin{equation}\label{prop difference eq 1}
    \begin{aligned}
        &|B_{1\be}(U)-B_{1\be}(\lu)| \leq \cde \de_* \bigl(G_{2\be}(U) - G_{2\be}(\lu)\bigr) + \cde \de_*^{3/2} (\md_\be+\tg_\be) + \cde \nu^2 |\s_\be| (\md(U) + \tg) \, , \\
        &|B_{2\be}^+(U)-B_{2\be}^+(\lu)| \leq \cde \de_*^{1/2} \md_\be + \cde \nu^2 |\s_\be| \md \, , \\ 
        &|B_{2\be}^-(U)| \leq C \de_*^{1/2}  (\md_\be+\tg_\be) +  C \nu^2 |\s_\be| (\md(U)+\tg) + (\de_1 + \cde \de_*^{1/4}) G_{1\be}^-(U) + C \nu^2 |\s_\be|\, , \\
        &|B_3(U)| \leq C \de_*^4 \sum_{\be \in \ms} G_{2\be}(U) + C \de_*^{9/2}(\md(U)+\tg) \, ,
    \end{aligned}
\end{equation}
and 
\begin{align}\label{prop difference eq 2}
    |B_{1\be}(\lu)| + |B_{2\be}^+(\lu)| \leq C \de_*|\s_\be| \, ,
\end{align}
where $\lu$ is defined in \eqref{def of truncation}, $B_{1\be}, B_{2\be}^+, B_{2\be}^-, B_3, G_{2\be}, G_{1\be}^-$ are defined in \eqref{def of B and G}, $\md$ is defined in \eqref{def of D}, and
\begin{align}
    \tg(t) = \int |(\bv)_x| Q(v|\bv) \dx \, , \quad \tg_\be(t) = \int_{\ji} |(\bv)_x| Q(v|\bv) \dx \, , \label{def of tg and tg_i} \\
    \md_\be(U) = \nu \int_{\ji} a \mu_1(v)v^{\ga} \Bigl| \bigl(\pd\bigr)_x \Bigr|^2 \dx \, . \label{def of Di}
\end{align}
In addition, we have 
\begin{equation}\label{prop difference eq 3}
    \begin{aligned}
        &|Y_\be^g(\lu)-\mathcal{Y}_\be^g(\lu)|^2 \leq \cde |\s_\be| \de_*^{3/2} (\md_\be+\tg_\be) + C \nu^4 |\s_\be|^2 (\md+\tg) + C \nu^4|\s_\be|^2 \, , \\
        &|B_{2\be}^+(\lu)-\mi_{2\be}(\lv)| \leq \cde \de_*^{1/2} (\md_\be+\tg_\be) + C\nu^2 |\s_\be| (\md+\tg) + C \nu^2|\s_\be| \, , \\
        &|B_{1\be}(\lu)-(1+\bc \de_*)\mi_{1\be}(\lv)| \leq C \nu^2 |\s_\be| \, , \\
        &-\bigl(G_{2\be}(\lu)-\mg_{2\be}(\lv)\bigr) \leq C \nu^2 |\s_\be|\, , \\
        &-\md(\lu) \leq -\frac{1}{1+\de_*} \sum_{\be \in \ms} \widetilde{\md}_\be(\lv) + \frac{C\nu^{1-\cs(2\co+\cpr)}\ln \rn^{-1}}{\de_*}  \int a \eta(U|\bu) \dx \, ,
    \end{aligned}
\end{equation}
where $Y_\be^g$ is defined in \eqref{def of Yig, Yib, Yil, Yis}, and $\my_\be^g, \mi_{2\be}, \mi_{1\be}, \mg_{2\be},\widetilde{\md}_\be$ are defined in \eqref{terms associated with shocks}.

Furthermore, for each $\be \in\ms$, if 
\begin{align}\label{prop difference separation}
    \de_*^{1/2} \md_\be + \, \nu |\s_\be| \md(U) \leq |\s_\be| \de_* \, ,
\end{align}
then we have
\begin{equation}\label{prop difference eq 4}
    \begin{aligned}
        &|Y_\be^g(U) - Y_\be^g(\lu)|^2 + |Y_\be^b(U)|^2 + |Y_\be^l(U)|^2 + |Y_\be^s(U)|^2\\
        &\leq \cde \de_* |\s_\be| \Bigl[\de_*^{1/2}(\md_\be+\tg_\be) + \nu^2|\s_\be|(\md(U)+\tg) + G_{2\be}(U) - G_{2\be}(\lu) \\
        &\qquad + G_{1\be}^-(U) + \de_*^{-1/4}G_{1\be}^+(U) + \de_*^{1/4}G_{2\be}(\lu) 
        \Bigr] \, ,
    \end{aligned}
\end{equation}
where $Y_\be^b, Y_\be^l, Y_\be^s$ are defined in \eqref{def of Yig, Yib, Yil, Yis}, and $G_{1\be}^+$ is defined in \eqref{def of B and G}.
\end{prop}
\begin{proof}
Note first that Corollary \ref{lemma strength of waves} indicates that 
\begin{align*}
    |\s_\be| \geq \rn = \nu^{\cs} \, .
\end{align*}
Following the same argument as in \cite{NStoE}, we can get \eqref{prop difference eq 1}, \eqref{prop difference eq 4}, and \eqref{prop difference eq 2} from Lemma \ref{lemma technical}.

\noindent
\textbf{Estimates for \eqref{prop difference eq 3}:}
Recall the definition of the cutoff function $\ts$ in \eqref{def of cutoff function}. By Lemma \ref{lemma technical}, we get
\begin{align*}
    &|Y_\be^g(\lu) - \my_\be^g(\lv)| \\
    &\leq \int_{\ji^c} |(a_\be)_x| \dx  + \cde |\s_\be|^{1/2} \de_*^{3/4} (\md_\be+\tg_\be)^{1/2} + C_n \nu^n (\md+\tg)^{1/2} \, , \\
    &|B_{2\be}^+(\lu)-\mi_{2\be}(\lv)| \leq \int_{\ji^c} |(a_\be)_x| \dx + \cde \de_*^{1/2} (\md_\be+\tg_\be) + C_n \nu^n (\md+\tg) \, , \\
    &|B_{1\be}(\lu)-(1+\bc \de_*)\mi_{1\be}(\lv)| \leq \int_{\ji^c} |(a_\be)_x| \dx \, .
\end{align*}
By \eqref{scale of Ji}, \eqref{def of weight a_i}, Lemma \ref{lemma estimate of waves}, 1) in Corollary \ref{lemma strength of waves}, and \eqref{rt transition S}, we get
\begin{align}\label{estimate of derivative of ai outside ji}
    \int_{\ji^c} |(a_\be)_x| \dx \leq \frac{C}{\de_*} \int_{|x-y_\be|\geq \frac{\rn^{\co}}{2}} \frac{\s_\be^2}{\nu}e^{-\frac{C|\s_\be||x-y_\be|}{\nu}} \dx \leq \frac{C|\s_\be|}{\de_*} e^{-C\nu^{-1+\cs(1+\co)}} \leq C_n \nu^n \, .
\end{align}
Hence, we get the first three inequalities.

By \eqref{relative evaluation A.3}, we have
\begin{align*}
    -\bigl(G_{2\be}(\lu)-\mg_{2\be}(\lv)\bigr) \leq C \int |(a_\be)_x| |\bv-\sv_\be| \dx + \int_{\ji^c} |(a_\be)_x | \, .
\end{align*}
By Lemma \ref{lemma estimates of interaction terms} and \eqref{estimate of derivative of ai outside ji}, we get the fourth inequality.

By \eqref{properties of ts}, we get
\begin{align*}
    \md(\lu) \geq \sum_{\be \in S} \nu\int a \ts^2 \mu_1(\lv) \lv^\ga \bigl(p(\lv)-p(\bv)\bigr)_x^2 \dx \, ,
\end{align*}
and
\begin{align*}
    \widetilde{\md}_\be(\lv) &\leq (1+\de_*) \nu \int a \mu_1(\lv) \lv^\ga \ts^2 \bigl(p(\lv)-p(\bv)\bigr)_x^2 \dx + \frac{C\nu}{\de_*} \int a \mu_1(\lv) \lv^\ga |(\ts)_x|^2 \bigl(p(\lv)-p(\bv)\bigr)^2 \\
    &\leq (1+\de_*) \nu\int a \mu_1(\lv) \lv^\ga \ts^2 \bigl(p(\lv)-p(\bv)\bigr)_x^2 \dx + \frac{C\nu^{1-2\cs\co}}{\de_*} \int a \mu_1(\lv) \lv^\ga \bigl(p(\lv)-p(\bv)\bigr)^2 \dx \, .
\end{align*}
By Lemma \ref{lemma number of waves and interactions}, we have
\begin{align*}
    -\md(\lu) &\leq -\frac{1}{1+\de_*} \sum_{\be \in \ms} \widetilde{\md}_\be(\lv) +   \frac{C\nu^{1-2\cs\co}}{\de_*} \sum_{\be \in \ms}\int a \eta(\lu|\bu) \dx \\
    &\leq -\frac{1}{1+\de_*} \sum_{\be \in \ms} \widetilde{\md}_\be(\lv) + \frac{C\nu^{1-2\cs\co-\cs\cpr}\ln \rn^{-1}}{\de_*} \int a \eta(\lu|\bu) \dx \\
    &\leq -\frac{1}{1+\de_*} \sum_{\be \in \ms} \widetilde{\md}_\be(\lv) + \frac{C\nu^{1-\cs(2\co+\cpr)}\ln \rn^{-1}}{\de_*} \int a \eta(U|\bu) \dx \, .
\end{align*}
\end{proof}

\subsubsection{Truncations for small and big values of $v$}
We truncate $v$ relative to $\bv$ and define the small value truncation $\lv_s$ and the large-value truncation $\lv_b$ by
\begin{equation}\label{def of vb vs}
    \begin{aligned}
        p(\lv_s) - p(\bv) &= \inf \bigl(\de_1, p(v) - p(\bv)\bigr) \, , \\
        p(\lv_b) - p(\bv) &= \sup \bigl(-\de_1, p(v) - p(\bv)\bigr) \, ,
    \end{aligned}
\end{equation}
where $\de_1$ is chosen in Section \ref{subsec choice of e_0}.

Recall the definition of $\om$ in \eqref{def of Omega}. We note that
\begin{equation}
    \lv = 
    \begin{cases}
        \lv_b \, , \quad \textup{on } \om \, , \\
        \lv_s \, , \quad \textup{on } \om^c \, .
    \end{cases}
\end{equation}
Moreover, we have
\begin{equation}\label{truncation estimate 1}
    \begin{aligned}
        p(v) - p(\lv_s) &= \bigl(p(v) - p(\bv) - \de_1\bigr)_+ \, , \\
        p(\lv_b) - p(v) &= \bigl(p(\bv) - p(v) - \de_1 \bigr)_+ \, , \\
        p(v) - p(\lv)  &= \bigl( |p(v) - p(\bv)| - \de_1 \bigr)_+ \, .
    \end{aligned}
\end{equation}

\subsubsection{Pointwise estimates}
We use the smallness condition \eqref{lemma smallness eq 1}, established in Lemma \ref{lemma smallness of localized relative entropy 1}, to obtain pointwise estimates for both $\pdb$ and $\pds$.

For any $\be \in \ms$, we define
\begin{align}
    \uji = [y_\be - \frac{\nu}{|\s_\be|},y_\be + \frac{\nu}{|\s_\be|}] \, .
\end{align}
We first notice that \eqref{rt transition S} implies
\begin{align*}
    \frac{\nu}{|\s_\be|} = \nu^{1-\cs} \ll \nu^{\co \cs}  \, .
\end{align*}
Recall \eqref{scale of Ji}. We have
\begin{align}\label{scale of underline ji compared with ji}
    \uji \subset \ji \, .
\end{align}
By \eqref{lemma smallness eq 1} in Lemma \ref{lemma smallness of localized relative entropy 1}, \eqref{weight equivalence}, and 2) in Lemma \ref{lemma estimate of waves}, we get
\begin{align*}
    |\s_\be| \int_{\uji} Q(v|\bv) \dx \leq C \de_*^2 \nu \, ,
\end{align*}
for some $C$ depending only on $U_*$. Hence, there exists some $\xoi \in \uji$ such that 
\begin{align*}
    \Bigl| \bigl(p(v) - p(\bv)\bigr) (\xoi) \Bigr| \leq \mathbf{C} \de_* \, ,
\end{align*}
for some $\mathbf{C}$ depending only on $U_*,\de_1$.
By taking
\begin{align}\label{relation delta 1}
    \mathbf{C}\de_* < \de_1 \, ,
\end{align}
we get
\begin{align}
    \bigl(p(v) - p(\lv)\bigr) (\xoi) = \bigl(p(v) - p(\lv_b)\bigr) (\xoi) = \bigl(p(v) - p(\lv_s)\bigr) (\xoi) = 0 \, .
\end{align}
Let $x \in \mj_{i'}$ for some $i'$. Following the same argument as for (4.53) and (4.63) in \cite{invent}, we obtain
\begin{equation}\label{truncation estimate vb}
    \Bigl|\bigl(\pdb)\bigr(x) \Bigr| \leq
    \begin{cases}
        C \sqrt{|x-\xoi|} \sqrt{\frac{\md_\be}{\nu}} \, , \quad \textup{if } i'=i \, , \\
        C \sqrt{|x-\xoi|} \sqrt{\frac{\md}{\nu}} \, , \quad \textup{if } i'\neq i \, ,
    \end{cases}
\end{equation}
and
\begin{equation}\label{truncation estimate vs}
    \Bigl|v^{\frac{\ga-\al}{2}} \bigl(\pds\bigr)\Bigr| \leq 
    \begin{cases}
        \cde \sqrt{|x-\xoi|} \bigl(\sqrt{\frac{\md_\be}{\nu}}+\sqrt{\frac{\tg_\be}{\nu}}\bigr) \, , \quad \textup{if } i'=i \, , \\
        \cde \sqrt{|x-\xoi|} (\sqrt{\frac{\md}{\nu}}+\sqrt{\frac{\tg}{\nu}}) \, , \quad \textup{if } i'\neq i \, ,
    \end{cases}
\end{equation}
where $C>0$ depends on $U_*,\de_1$, $\md$ is defined in \eqref{def of D}, $\tg, \tg_\be$ are defined in \eqref{def of tg and tg_i}, and $\md_\be$ is defined in \eqref{def of Di}. For completeness and to verify that the argument remains valid under our scaling of rarefactions and non-physical waves, we include the proof for \eqref{truncation estimate vb} and \eqref{truncation estimate vs} in Appendix \ref{section proof of ftoc}.

Finally, we notice that
\begin{align}\label{distance between x and x0i}
    |x-\xoi| \leq |x-y_\be| + |y_\be-\xoi| \leq |x-y_\be| + \frac{\nu}{|\s_\be|} \, .
\end{align}

\subsubsection{Useful estimates}
Before proving the technical lemma for Proposition \ref{prop difference}, we first establish several preliminary estimates.
\begin{lemm}\label{lemma integration of x^b e^-x}
Let $C>0$. Let $a\geq 1$ and $1/2\leq b\leq 2$. There exist $C_1,C_2 >0$, depending only on $C$, such that
\begin{align*}
    C_1 a^{b} e^{-Ca} \leq \int_a^\infty x^b e^{-Cx} \dx \leq C_2 a^b e^{-Ca} \, .
\end{align*}
\end{lemm}
\begin{proof}
We note first 
\begin{align*}
    \int_a^\infty x^b e^{-Cx}\dx \geq a^b \int_{a}^\infty e^{-Cx} \dx \geq \frac{a^b}{C} e^{-ca} \, .
\end{align*}
If $1/2\leq b \leq 1$, we have
\begin{align*}
    \int_a^\infty x^b e^{-cx} \dx \leq a^{b-1} \int_{a}^\infty x e^{-Cx} \dx = \frac{C+\frac{1}{a}}{C^2} a^{b} e^{-Ca} \leq \frac{C+1}{C^2} a^{b}e^{-Ca}\, . 
\end{align*}
If $1\leq b \leq 2$, we have
\begin{align*}
    \int_a^\infty x^b e^{-cx} \dx \leq a^{b-2} \int_{a}^\infty x^2 e^{-Cx} \dx = \frac{C^2+\frac{2C}{a} + \frac{2}{a^2}}{C^3} a^{b} e^{-Ca} \leq \frac{C+2C+2}{C^3} a^{b}e^{-Ca}\, . 
\end{align*}
\end{proof}

\begin{lemm}\label{lemma technical preliminary}
Let $1/2\leq q\leq 2$ and $0<p<1$.
There exist $C>0$, depending on $U_*, \de_1, q, p$, and $C_n>0$, depending on $U_*, n, q$, such that for any $\be \in\ms$,
\begin{align}
    \int_{\ji} |(a_\be)_x| |x-\xoi|^q \mathds{1}_{\{|\pd|\geq \de_1\}}\dx &\leq C\nu^q|\s_\be|^{1-q} \de_*^{p} \, , \label{lemma tech pre 1}\\
    \int_{\ji^c} |(a_\be)_x| |x-\xoi|^q \dx &\leq C_n \nu^n \, , \label{lemma tech pre 2}
\end{align}
where $\ji$ is defined in \eqref{def of ji}.
\end{lemm}
\begin{proof}
We prove \eqref{lemma tech pre 1} by further decomposing $\ji$. Let $0<\tc<1$ be a constant, independent of $\nu$, to be chosen later. Recall \eqref{scale of underline ji compared with ji}. Since the scale $\{|x-y_\be|\le \frac{\nu}{\tc|\s_\be|}\}$ is comparable to that of $\uji$, we split the integral as
\begin{align*}
    L&:=\int_{\ji} |(a_\be)_x| |x-\xoi|^q \mathds{1}_{\{|\pd|\geq \de_1\}} dx \\
    &= \int_{\{|x-y_\be|\leq \frac{\nu}{\tc |\s_\be|}\}} |(a_\be)_x| |x-\xoi|^q \mathds{1}_{\{|\pd|\geq \de_1\}} \dx \\
    &\qquad + 
    \int_{\{|x-y_\be|> \frac{\nu}{\tc |\s_\be|}\} \cap \ji} |(a_\be)_x| |x-\xoi|^q \mathds{1}_{\{|\pd|\geq \de_1\}} \dx \\
    &=: L_{1} + L_{2} \, .
\end{align*}
On $\{|\pd|\geq \de_1\}$, we have
\begin{align*}
    Q(v|\bv) \geq C \, .
\end{align*}
By \eqref{distance between x and x0i} and Lemma \ref{lemma smallness of localized relative entropy 1}, we take $\tc^q = \de_*^{1-p}$ and get
\begin{align}\label{inequality ki1}
    L_{1} \leq \frac{\cde \nu^q}{\tc^q |\s_\be|^q} \int |(a_\be)_x| Q(v|\bv) \dx 
    \leq \cde \nu^q { |\s_\be|^{1-q}}{\de_*^p }
    \, .
\end{align}
On the other hand, \eqref{distance between x and x0i}  gives
\begin{align*}
    L_{2} 
    &\leq C \int_{\{|x-y_\be|> \frac{\nu}{\tc |\s_\be|}\} } |(a_\be)_x| |x-y_\be|^q \mathds{1}_{\{|\pd|\geq \de_1\}} \dx \\
    &\qquad +  \frac{\cde\nu^q}{|\s_\be|^q} \int |(a_\be)_x| Q(v|\bv) \dx \\
    &=: L_{21} + L_{22} \, . 
\end{align*}
By Lemma \ref{lemma smallness of localized relative entropy 1}, we get
\begin{align*}
    L_{22} \leq \cde \nu^{q} |\s_\be|^{1-q} \de_*\, .
\end{align*}
By Lemma \ref{lemma estimate of waves} and Lemma \ref{lemma integration of x^b e^-x}, we get
\begin{align*}
    L_{21} &\leq \frac{C}{\de_*} \int_{\{|z|> \frac{\nu}{\tc |\s_\be|}\} } \frac{\s_\be^2}{\nu} e^{-\frac{C|\s_\be||z|}{\nu}} |z|^q \, dz \\
    &= C \frac{\nu^q|\s_\be|^{1-q}}{\de_*} \int_{\{|\tilde{z}|> \frac{1}{\tc}\} } e^{-C|\tilde{z}|} |\tilde{z}|^q \, d\tilde{z} \\
    &\leq C \frac{\nu^q|\s_\be|^{1-q}}{\tc^q \de_*} e^{-C \tc} \\
    &\leq C \nu^q|\s_\be|^{1-q}\de_*^p\, .
\end{align*}
In all, we conclude 
\begin{align*}
    L \leq \cde \nu^q |\s_\be|^{1-q} \de_*^p \, .
\end{align*}
On the other hand, \eqref{scale of Ji} and \eqref{distance between x and x0i} imply
\begin{align*}
    L^c : = \int_{\ji^c} |(a_\be)_x| |x-\xoi|^q \dx 
    \leq \frac{C}{\de_*} \int_{|x-y_\be|\geq \frac{\rn^{\co}}{2}} \frac{\s_\be^2}{\nu} e^{-\frac{C|\s_\be||x-y_\be|}{\nu}} |x-y_\be|^q \dx \, .
\end{align*}
By 1) in Corollary \ref{lemma strength of waves} and \eqref{rt transition S}, we have
\begin{align*}
    L^c \leq C_n \nu^n \, .
\end{align*}
\end{proof}

\subsubsection{Technical lemma for Proposition \ref{prop difference}}
The following lemma corresponds to Lemma 10.4 in \cite{NStoE} and is the main technical lemma used to prove Proposition~\ref{prop difference}. The proof follows the same idea. However, since we modify the scalings associated with rarefactions and non-physical waves, we go through the full argument in Appendix~\ref{section proof of lemma technical}. We also take this opportunity to isolate in Lemma~\ref{lemma technical preliminary} some estimates that are used repeatedly.
\begin{lemm}\label{lemma technical}
There exists $C>0$, depending on $U_*, \de_1$, and $C_n>0$, depending on $U_*, \de_1, n$, such that for any $\be \in \ms$,
\begin{equation}\label{lemma technical 3}
    \int |(a_\be)_x|\Bigl(\bigl|Q(v|\bv)-Q(\lv|\bv)\bigr| + |v-\lv|\Bigr) \dx \leq \cde \bigl(G_{2\be}(U) - G_{2\be}(\lu)\bigr)  \, ,
\end{equation}
\begin{equation}\label{lemma technical 1}
    \begin{aligned}
        &\int_\om |(a_\be)_x|\, |p(v)-p(\lv)| \dx \leq \cde \de_*^{1/2} \md_\be + \cde_n \nu^n \md \, , \\
        &\int_\om |(a_\be)_x|\, \Bigl| |p(v)-p(\bv)|^2 - |p(\lv)-p(\bv)|^2 \Bigr| \dx \leq \cde \de_*^{1/2} \md_\be + \cde_n \nu^n \md \, ,
    \end{aligned}
\end{equation}
\begin{equation}\label{lemma technical 2}
    \begin{aligned}
        &\int |(a_\be)_x| v^{\ga-\al} |p(v)-p(\lv)|^2 \dx \\
        &\leq \cde \bigl(G_{2\be}(U) - G_{2\be}(\lu)\bigr) + \cde \de_*^{1/2} (\md_\be+\tg_\be) + \cde_n \nu^n (\md + \tg) \, , \\
        &\int |(a_\be)_x|\, \Bigl| v^{\ga-\al}|p(v)-p(\bv)|^2 - \lv^{\ga-\al}|p(\lv)-p(\bv)|^2 \Bigr| \dx \\
        &\leq \cde \bigl(G_{2\be}(U) - G_{2\be}(\lu)\bigr) + \cde \de_*^{1/2} (\md_\be+\tg_\be) + \cde_n \nu^n (\md + \tg) \, ,
    \end{aligned}
\end{equation}
\begin{equation}\label{lemma technical 4}
    \begin{aligned}
        &\int |(a_\be)_x| \bigl|p(v|\bv)-p(\lv|\bv)\bigr| \dx \\
        &\leq \cde \bigl(G_{2\be}(U) - G_{2\be}(\lu)\bigr) + \cde \de_*^{1/2} (\md_\be+\tg_\be) + \cde_n \nu^n (\md + \tg) \, ,
    \end{aligned}
\end{equation}
\begin{equation}\label{lemma technical 5}
    \begin{aligned}
        &\int_{\om^c} |(a_\be)_x|\bigl|p(v)-p(\bv)\bigr||h-\bh| \dx \\
        &\leq C \de_*^{1/2}  (\md_\be+\tg_\be) +  C_n \nu^n (\md+\tg) + (\de_1 + \cde \de_*^{1/4}) G_{1\be}^-(U) + C_n \nu^n \, ,
    \end{aligned}
\end{equation}
\begin{equation}\label{lemma technical 6}
    \begin{aligned}
        \int_{\om^c} |(a_\be)_x| \Bigl(Q(\lv|\bv) + |\lv-\bv|\Bigr) \dx &\leq \cde \de_*^{1/2} (\md_\be+\tg_\be) + C_n \nu^n (\md+\tg)\, , \\
        \int_{\om^c} |(a_\be)_x| \Bigl(Q(\lv|\bv) + |\lv-\bv|\Bigr) \dx &\leq \cde |\s_\be|^{1/2} \de_*^{3/4} (\md_\be+\tg_\be)^{1/2} + C_n \nu^n (\md+\tg)^{1/2} \, ,
    \end{aligned}
\end{equation}
where $\md_\be$ is defined in \eqref{def of Di}, $\md$ is defined in \eqref{def of D}, $\tg, \tg_\be$ are defined in \eqref{def of tg and tg_i}, $G_{2\be},G_{1\be}^-$ are defined in \eqref{def of B and G}, and $\lu$ is defined in \eqref{def of truncation}.
\end{lemm}

\section{Estimate of the weighted norm at stopping times}\label{sec estimates at stopping time}
In this section, we evaluate the change in the weighted relative entropy across events in preparation for the global-in-time estimates in Section~\ref{subsec global-in-time estimates}. For interactions, passings, and partitions, the changes in the strengths of shocks and rarefactions are well controlled by Lemma~\ref{lemma L + q + omega}, thanks to the construction of the algorithm and the decay of the total variation, which yield a sharper estimate. On the other hand, merging and overtaking events do not enjoy such control and therefore require a different estimate.

For simplicity, when considering a given event, we use the subscripts $-$ and $+$ to denote quantities before and after the event, respectively. More precisely, we denote the approximate solution and the weight before and after the event by $\bu_\pm$ and $a_\pm$, respectively. We denote the sets of incoming waves, shocks, rarefactions, and non-physical waves by $\ma^-$, $\ms^-$, $\mr^-$, and $\mnp^-$, respectively, and the corresponding sets of outgoing waves by $\ma^+$, $\ms^+$, $\mr^+$, and $\mnp^+$. We study
\begin{align*}
    I = \int a_+ \eta(U|\bu_+) \dx - \int a_- \eta(U|\bu_-) \dx \, ,
\end{align*}

\subsection{Estimate at interactions, passings, and partitions}\label{subsec weighted norm stopping time with decay a}
Assume that an interaction, a passing event, or a partition occurs. 
\subsubsection{Control of wave-strength changes}\label{subsubsec decay of the weight}
The following lemma shows that the decay of $\ml$ at an interaction, a passing, or a partition compensates for the errors caused by the corresponding changes in the strengths of shocks and rarefactions. The estimates for passing and partition events follow directly from the construction, while the estimate for interactions follows the same spirit as the proof of Proposition~11.1 in \cite{NStoE}. The details are given in Appendix~\ref{sec decay of weight lemma proof}.

\begin{lemm}\label{lemma L + q + omega}
For any $U_*\in \om_{sub}^-$ as defined in \eqref{def subsonic L}, there exist constants $\e_0, C_w, C_0>0$ such that the following is true.

\noindent
Assume $\ml(0)\leq \e_0$. When an interaction, a passing, or a partition occurs, we have, on $(y_b,\infty)$,
\begin{align*}
    C_w \Delta \ml + \Delta g^S \leq C_0 \rn \, , \quad C_w \Delta \ml - \Delta g^R \leq C_0 \rn \, ,
\end{align*}
where $\Delta g^W = g^W_+ - g^W_-$, and
\begin{align*}
    g^S_\pm(x) = \int_\infty^x \sum_{\be \in \ms^\pm} \p_x \bigl(-p(\sv_\be)\bigr)  \, dy  \, dy\, , \quad
    g^R_\pm(x) = \int_\infty^x \sum_{\be \in \mr^\pm} \p_x \bigl(-p(\rv_\be)\bigr)  \, dy\, .
\end{align*}

\end{lemm}

\subsubsection{Estimate of the weighted relative entropy}
We have
\begin{align*}
    I &= \int a_+ \, \eta(U|\bu_+) \dx - \int a_- \, \eta(U|\bu_-) \dx \\
    &= \int (a_+ - a_-) \, \eta(U|\bu_-) \dx + \int a_+ \bigl( \eta(U|\bu_+) - \eta(U|\bu_-) \bigr)\dx \\
    &=: I_1 + I_2 \, .
\end{align*}
By \eqref{def of weight} and Lemma \ref{lemma L + q + omega}, we have
\begin{align}\label{decay of a global in time}
    a_+ - a_- = \frac{1}{\de_*} \bigl[ C_w \Delta \ml + \Delta g^S(x) \bigr] \leq \frac{1}{\de_*} f + C \rn \, , 
\end{align}
where 
\begin{align}\label{property of f}
    f = C_w \Delta \ml + \Delta g^S - C_0 \rn \leq 0 \, ,
\end{align}
with $C_w, C_0>0$ chosen in Lemma \ref{lemma L + q + omega}. Hence, we get
\begin{align*}
    I_1 \leq \frac{1}{\de_*}\int f \,  \eta(U|\bu_-) \dx + C\rn \int \eta(U|\bu_-) \dx \, .
\end{align*}
On the other hand, \eqref{useful computation} implies
\begin{equation}\label{computation diff of relative entropy}
    \begin{aligned}
        &\eta(U|\bu_+) - \eta(U|\bu_-) \\
        &= \eta(\bu_-) - \eta(\bu_+) - \nabla\eta(\bu_-) \bu_- + \nabla\eta(\bu_+) \bu_+ + \bigl( \nabla\eta(\bu_-) - \nabla\eta(\bu_+) \bigr) U \\
        &\leq C|\bu_- - \bu_+| + \bigl( -p(\bv_-) + p(\bv_+) \bigr) v + (\bh_- - \bh_+) h \, .
    \end{aligned}
\end{equation}
Together with \eqref{weight a close to 1}, we have
\begin{equation*}
    \begin{aligned}
        I_2 &\leq \int a_+ \bigl( -p(\bv_-) + p(\bv_+) \bigr) v \mathds{1}_{\{v\geq 2 v_*\}} \dx + \int a_+ (\bh_- - \bh_+) (h-\bh_-) \dx \\
        &\quad + \|\bu_+ - \bu_-\|_{L^1} \, .
    \end{aligned}
\end{equation*}
By \eqref{property of f}, we conclude
\begin{align}\label{eq of I global estimate}
    I \leq I_3 + I_4 + C\rn \int \eta(U|\bu_-) \dx + C \|\bu_+ - \bu_-\|_{L^1} \, ,
\end{align}
where
\begin{equation}\label{estimate on I3, I4, def, global estimate}
    \begin{aligned}
        I_3 &= \frac{1}{\de_*}\int f \,  Q(v|\bv_-) \dx  + \int a_+ \bigl( -p(\bv_-) + p(\bv_+) \bigr) v \mathds{1}_{\{v\geq 2 v_*\}} \dx \, , \\
        I_4 &= \int a_+ (\bh_- - \bh_+) ( h - \bh_- ) \dx \, .
    \end{aligned}
\end{equation}

\noindent
\textbf{Estimate of $-p(\bv_-) + p(\bv_+)$.}
We have
\begin{align*}
    p(\bv_\pm)(x) &= p(v_*) + \int_\infty^x \p_x \bigl(p(\bv_\pm)\bigr) \, dy \\
    &= p(v_*) + A_1^\pm + A_2^\pm \, ,
\end{align*}
where
\begin{align*}
    A_1^\pm &= \sum_{\be \in \ms^\pm} \int_\infty^x \p_x  \bigl(p(\sv_\be)\bigr) \, dy
    + \sum_{\be \in \mr^\pm} \int_\infty^x \p_x  \bigl(p(\rv_\be)\bigr) \, dy
    + \sum_{\be \in \mnp^\pm} \int_\infty^x \p_x  \bigl(p(\nv_\be)\bigr) \, dy \, , \\
    A_2^\pm &= \sum_{\be \in \ms^\pm} \int_\infty^x \bigl( p'(\bv_\pm) - p'(\sv_\be) \bigr) \p_x \sv_\be \, dy + \sum_{\be \in \mr^\pm} \int_\infty^x \bigl( p'(\bv_\pm) - p'(\rv_\be) \bigr) \p_x \rv_\be \, dy \\
    &\qquad + \sum_{\be \in \mnp^\pm} \int_\infty^x \bigl( p'(\bv_\pm) - p'(\nv_\be) \bigr) \p_x \nv_\be \, dy \, .
\end{align*}
We have
\begin{align*}
    \bigl(-p(\bv_-) + p(\bv_+)\bigr) (x) = (-A_1^- + A_1^+) + (-A_2^- + A_2^+) \, .
\end{align*}
By Corollary \ref{lemma strength of waves} and Lemma \ref{lemma number of waves and interactions}, we get
\begin{align*}
    -A_1^- + A_1^+ \leq -\Delta g^S - \Delta g^R + C \rn^{1-\cpr} \ln \rn^{-1}\, .
\end{align*}
Lemma \ref{lemma estimates of interaction terms}, Corollary \ref{lemma strength of waves}, and Lemma \ref{lemma number of waves and interactions} give
\begin{align*}
    -A_2^- + A_2^+ \leq C \rn^{1-\cpr} \ln \rn^{-1} \, .
\end{align*}
Together with Lemma \ref{lemma L + q + omega}, we get
\begin{equation}\label{estimate p(bar v) difference}
    \begin{aligned}
        \bigl(-p(\bv_-) + p(\bv_+)\bigr) (x) &\leq -\Delta g^S - \Delta g^R + C \rn^{1-\cpr} \ln \rn^{-1} \\
        &\leq C_w \Delta \ml - \Delta g^R - (C_w \Delta \ml + \Delta g^S) + C \rn^{1-\cpr} \ln \rn^{-1} \\ 
        &\leq - (C_w \Delta \ml + \Delta g^S) + C \rn^{1-\cpr} \ln \rn^{-1} \, .
    \end{aligned}
\end{equation}

\noindent
\textbf{Estimate of $I_3$.} There exists $\mathbf{C}_Q>0$, depending only on $U_*$, such that
\begin{equation}\label{property of Q global estimate}
    \begin{aligned}
        v \mathds{1}_{\{v\geq 2 v_*\}} \leq \mathbf{C}_Q Q(v|\bv_\pm)  \, .
    \end{aligned}
\end{equation}
By \eqref{property of f} and \eqref{estimate p(bar v) difference}, we have
\begin{align*}
    I_3 &\leq \int \frac{1}{\mathbf{C}_Q\de_*} f v \mathds{1}_{\{v\geq 2 v_*\}} \dx + \int a_+ \bigl( -p(\bv_-) + p(\bv_+) \bigr) v \mathds{1}_{\{v\geq 2 v_*\}} \dx \\
    &\leq \int A v \mathds{1}_{\{v\geq 2 v_*\}} \dx \, ,
\end{align*}
where
\begin{align*}
    A = \frac{1}{\mathbf{C}_Q\de_*} \bigl(C_w \Delta \ml + \Delta g^S - C_0 \rn \bigr)  - a_+ \bigl(C_w \Delta \ml + \Delta g^S  \bigr) + C\rn^{1-\cpr} \ln \rn^{-1} \, .
\end{align*}
By \eqref{weight a close to 1}, we can take
\begin{align}\label{relation delta 9}
    \de_* \leq \frac{1}{2\mathbf{C}_Q} \, ,
\end{align}
and get
\begin{align*}
    A \leq C\rn^{1-\cpr} \ln \rn^{-1} \, .
\end{align*}
By \eqref{property of Q global estimate}, we get
\begin{align}\label{eq of I3 global estimate}
    I_3 \leq C\rn^{1-\cpr} \ln \rn^{-1} \int a_- Q(v|\bv_-) \dx  \, .
\end{align}

\noindent
\textbf{Estimate of $I_4$.} 
Recall the definition of $I_4$ in \eqref{estimate on I3, I4, def, global estimate}. By \eqref{property of f} and \eqref{weight a close to 1}, Young's inequality gives
\begin{equation}\label{eq of I4 global estimate}
    \begin{aligned}
        I_4 &\leq \|\bu_+ - \bu_-\|_{L^2} \bigl(\int a_- (h-\bh_-)^2 \dx\bigr)^{1/2} \\
        &\leq \|\bu_+ - \bu_-\|_{L^2} \int a_- (h-\bh_-)^2 \dx + C \|\bu_+ - \bu_-\|_{L^2} \, .
    \end{aligned}
\end{equation}

\noindent
\textbf{Conclusion.} By \eqref{eq of I global estimate}, \eqref{eq of I3 global estimate}, and \eqref{eq of I4 global estimate}, we get
\begin{equation}\label{energy estimate with decay of weight intermediate}
    \begin{aligned}
        &\int a_+ \, \eta(U|\bu_+) \dx - \int a_- \, \eta(U|\bu_-) \dx \\
        &\leq  C \max \{\rn^{1-\cpr} \ln \rn^{-1},\,\|\bu_+ - \bu_-\|_{L^2} \}  \int a_- \, \eta(U|\bu_-) \dx \\
        &\qquad + C\|\bu_+ - \bu_-\|_{L^2} + C\|\bu_+ - \bu_-\|_{L^1} \, .
    \end{aligned}
\end{equation}

\subsubsection{Change of the approximate solution}
The construction in Section \ref{subsec mft} provides the following estimates.

\begin{lemm}\label{lemma bar U change at stopping time}
There exists $C>0$, depending on $U_*$, such that the following is true.

\noindent
If an interaction, a passing, or a partition occurs, then 
\begin{align*}
    \|\bu_\pm\|_{L^\infty} \leq C \, , \quad \|\bu_+ - \bu_-\|_{L^1} \leq C \rn^{\co} \Sigma \, , \quad \|\bu_+ - \bu_-\|_{L^2}^2 \leq C \rn^{\co} \Sigma^2 \, ,
\end{align*}
where $\Sigma$ denotes the sum of the absolute values of the strengths of all incoming waves.

\end{lemm}
\begin{proof}
The main change in the approximate solution occurs inside an interval $J$ of scale $\rn^{\co}$. Let $q=1,2$. We get
\begin{align*}
    \|\bu_+-\bu_-\|_{L^q}^q &\leq \|\bu_+-\bu_-\|_{L^q(J)}^q + \|\bu_+-\bu_-\|_{L^q(J^c)}^q \\
    &\leq C \rn^{\co} \Sigma^q + \|\bu_+-\bu_-\|_{L^q(J^c)}^q \, .
\end{align*}
Lemma \ref{lemma estimate of waves}, 1) in Corollary \ref{lemma strength of waves}, and \eqref{rt transition S} imply
\begin{align*}
    \|\bu_+-\bu_-\|_{L^q(J^c)}^q &\leq \sum_{\be \in \ms^+ \cup \ms^-} \int_{|x-y_\be|\geq \rn^{\co}} (\frac{\s_\be^2}{\nu})^q e^{-\frac{C|\s_\be||x-y_\be|}{\nu}} \dx \\
    &\leq C \frac{e^{-C\nu^{(1+\co)\cs-1}}}{\nu^{q-1}} \sum_{\be \in \ms^+ \cup \ms^-} |\s_\be|^{2q-1} \\
    &\leq C_n \nu^n \, .
\end{align*}
\end{proof}
Together with \eqref{energy estimate with decay of weight intermediate}, we get that, if an interaction, a passing event, or a partition event occurs, then
\begin{equation}\label{energy estimate with decay of weight}
    \begin{aligned}
        &\int a_+ \, \eta(U|\bu_+) \dx - \int a_- \, \eta(U|\bu_-) \dx \\
        &\leq  C \max \{\rn^{1-\cpr} \ln \rn^{-1}, \rn^{\co/2} \Sigma \}  \int a_- \, \eta(U|\bu_-) \dx  + C\rn^{\co/2} \Sigma \, .
    \end{aligned}
\end{equation}

\subsection{Estimate of the weighted norm when a merging or an overtaking occurs} 
Assume that a merging or an overtaking event occurs.

\subsubsection{Estimate of the weighted norm}
We have
\begin{align*}
    I &= \int a_+ \bigl( \eta(U|\bu_+) - \eta(U|\bu_-) \bigr) \dx + \int (a_+-a_-) \, \eta(U|\bu_-) \\
    &=: I_1 + I_2 \, .
\end{align*}
Recall \eqref{computation diff of relative entropy}. We get
\begin{align*}
    I_1 \leq C \|\bu_+ - \bu_-\|_{L^1} + \int a_+ \bigl( -p(\bv_-) + p(\bv_+) \bigr) v \dx + \int a_+ (\bh_- - \bh_+) \, h \dx \, .
\end{align*}
By \eqref{property of Q global estimate}, we have
\begin{align*}
    &\int a_+ \bigl( -p(\bv_-) + p(\bv_+) \bigr) v \dx \\
    &\leq C \int \bigl( -p(\bv_-) + p(\bv_+) \bigr) v \, \mathds{1}_{\{v\geq 2v_*\} } \dx + C \|\bu_+ - \bu_-\|_{L^1} \\
    &\leq C \|\bu_- - \bu_+\|_{L^\infty} \int \, Q(v|\bv_-) \dx + C \|\bu_+ - \bu_-\|_{L^1} \, .
\end{align*}
Similarly, since there exists $C>0$, depending only on $U_*$, such that
\begin{align*}
    h \, \bigl(\mathds{1}_{\{h < \bh_-/2\} } + \mathds{1}_{\{h > 2\bh_-\} } \bigr) \leq C |h - \bh_-|^2 \, ,
\end{align*}
we get
\begin{align*}
    &\int a_+ (\bh_- - \bh_+) \, h \dx \\
    &\leq C \int (\bh_- - \bh_+) \, h \, \bigl(\mathds{1}_{\{ h < \bh_-/2\} } + \mathds{1}_{\{h > 2\bh_-\} } \bigr) \dx + C \|\bu_+ - \bu_-\|_{L^1} \\
    &\leq C \|\bu_- - \bu_+\|_{L^\infty} \int |h - \bh_-|^2 \dx + C \|\bu_+ - \bu_-\|_{L^1} \, .
\end{align*}
By \eqref{decay of a global in time},
On the other hand, recall the definition of the weight $a$ in \eqref{def of weight}. If an overtaking occurs, then
\begin{align*}
    a_+ - a_- \leq 0 \, .
\end{align*}
If a merging occurs, the construction in Section \ref{subsub merging} implies
\begin{align*}
    a_+ - a_- \leq C \Sigma \, ,
\end{align*}
where $\Sigma$ denotes the sum of the absolute values of the strengths of all incoming waves.

Hence, if an overtaking occurs, then
\begin{equation}\label{energy estimate without weight decay intermediate o}
    \begin{aligned}
        &\int a_+ \, \eta(U|\bu_+) \dx - \int a_- \, \eta(U|\bu_-) \dx \\
        &\leq C \|\bu_- - \bu_+\|_{L^\infty} \int a_- \, \eta(U|\bu_-) \dx + C \|\bu_+ - \bu_-\|_{L^1} \, ,
    \end{aligned}
\end{equation}
and if a merging occurs, then
\begin{equation}\label{energy estimate without weight decay intermediate}
    \begin{aligned}
        &\int a_+ \, \eta(U|\bu_+) \dx - \int a_- \, \eta(U|\bu_-) \dx \\
        &\leq C \max \bigl\{\|\bu_- - \bu_+\|_{L^\infty}, \Sigma \bigr\}\int a_- \, \eta(U|\bu_-) \dx + C \|\bu_+ - \bu_-\|_{L^1} \, .
    \end{aligned}
\end{equation}

\subsubsection{Change of the approximate solution}

The proof follows the same strategy as that of Lemma~\ref{lemma bar U change at stopping time}. 

\begin{lemm}\label{lemma bar U change at merging}
There exist $C>0$, depending only on $U_*$, and $C_n>0$, depending only on $U_*,n$, such that the following is true.

\noindent
If an overtaking occurs, then
\begin{align*}
    \|\bu_\pm\|_{L^\infty} \leq C \, , \quad \|\bu_+ - \bu_-\|_{L^\infty} \leq C \rn \, , \quad
    \|\bu_+ - \bu_-\|_{L^1} \leq C \rn^{1+\co} \, .
\end{align*} 
If a merging occurs, then
\begin{align*}
    \|\bu_\pm\|_{L^\infty} \leq C \, , \quad \|\bu_+ - \bu_-\|_{L^\infty} \leq C \Sigma \, , \quad
    \|\bu_+ - \bu_-\|_{L^1} \leq C \mev \Sigma \, .
\end{align*}
where $\Sigma$ denotes the sum of the absolute values of the strengths of all incoming waves.

\end{lemm}

Together with \eqref{energy estimate without weight decay intermediate o} and \eqref{energy estimate without weight decay intermediate}, we get that, if an overtaking occurs, then
\begin{equation}\label{energy estimate without weight decay o}
    \begin{aligned}
        &\int a_+ \, \eta(U|\bu_+) \dx - \int a_- \, \eta(U|\bu_-) \dx \\
        &\leq C \rn \int a_- \, \eta(U|\bu_-) \dx + C \rn^{1+\co} \, ,
    \end{aligned}
\end{equation}
and if a merging occurs, then
\begin{equation}\label{energy estimate without weight decay m}
    \begin{aligned}
        &\int a_+ \, \eta(U|\bu_+) \dx - \int a_- \, \eta(U|\bu_-) \dx \\
        &\leq C \Sigma \int a_- \, \eta(U|\bu_-) \dx + C \mev \Sigma \, .
    \end{aligned}
\end{equation}

\section{Proof of Proposition \ref{prop E approximate sol}}\label{sec global in time estimates}
In this section, we derive the global-in-time estimates of the weighted relative entropy by combining the estimates between stopping times, in particular Proposition~\ref{proposition energy estimate}, with the estimates at stopping times obtained in Section~\ref{sec estimates at stopping time}. We first establish a technical lemma in Section~\ref{subsub uni technical lemma}.

To control the total increase, we carefully analyze the contribution of each type of event in Section~\ref{subsec global-in-time estimates}. In particular, for interactions, we divide the events into different classes according to the strengths of the incoming waves. Interactions involving small incoming waves result in smaller increases, while the total number of interactions involving large waves is smaller, as shown in Lemma~\ref{lemma number of big interactions}.

For mergings, we distinguish between those involving a previously merged wave and those that do not. The increase caused by the latter is easily controlled, since each element can be involved in such mergings at most twice. For the former, we show that the decrease produced by the boundary contribution between successive mergings involving a previously merged wave is sufficiently large to compensate for the increase at each such merging.

For passings and partitions, we first estimate the increase during each discretization process, and then bound the total increase caused throughout the evolution using the uniform control, given in Lemma~\ref{lemma total number of R introduced big fan}, on the total number of rarefactions created that trigger a discretization process.

In Section~\ref{sub sec proof of prop limit of approximate solution}, using the notion of elements developed in Section~\ref{sec mft} and the global-in-time estimates derived in Section~\ref{subsec global-in-time estimates}, we prove that any limit of the approximate solution $\bu_\nu$, after reparametrizing the domain $(y_b^\nu,\infty)\times\mathbb{R}_+$ onto $\mathbb{R}_+\times\mathbb{R}_+$ with respect to the boundary $y_b^\nu$, satisfies the outflow boundary condition \eqref{outflow boundary condition}.

\subsection{Technical lemma}\label{subsub uni technical lemma} 
We begin with a technical lemma establishing the Gronwall inequality that will be applied directly in Section~\ref{subsec global-in-time estimates}. As a brief preview, the special set of events consists of mergings involving a previously merged wave. The increase caused by these mergings is compensated by the boundary contribution, in particular the term $-l$, with $l$ defined in \eqref{def of D_v, G}, accumulated between successive such mergings. The total increase caused by all other events throughout the evolution is bounded by $M$. In Section~\ref{subsec global-in-time estimates}, we estimate the increase associated with each type of event separately.

\begin{lemm}\label{lemma global in time estimates}
Let $\mathbf{C}>0$ and $\widetilde{\mathbf{C}}>1$. Then there exists $\e_0>0$ such that the following is true.

Let $T,C_1,C_2,e,M>0$, and let $f,g,h,l:[0,T]\to\mathbb{R}_{\ge0}$ be locally integrable nonnegative functions, with $f$ in addition absolutely continuous, such that the following conditions hold.

Assume that, on $[0,T]$, finitely many events occur at stopping times, with possibly multiple events occurring at the same stopping time. We denote by $\mathcal{E}$ the set of all events occurring on $[0,T]$. Let $\mathcal{E}_s\subset\mathcal{E}$ be a special set of events such that at most one event in $\mathcal{E}_s$ occurs at each stopping time. We denote the stopping times at which an event in $\mathcal{E}_s$ occurs, together with the left endpoint $0$, by
\begin{align*}
    t_0 := 0 < t_1 < \cdots < t_N \leq T \, .
\end{align*}
We assume that in between stopping times, 
\begin{align*}
    f'(t) + g(t) \leq C_1 f(t) + e^2 h(t) - l(t)\, .
\end{align*}
Moreover, we assume
\begin{align*}
    \int_0^T h(t) \, dt \leq C_2 \, , \qquad
    \int_{t_{j-1}}^{t_{j}} l(t) \geq \mathbf{C} e \quad \textup{for any } 1 \leq j \leq N \, .
\end{align*}
For any event $E\in \mathcal{E}$, we assume that there exists a constant $m_E>0$, depending on the event, such that
\begin{align*}
    f_+ - f_- \leq m_E f_- \, 
\end{align*}
where $f_-$ and $f_+$ denote the values of $f$ immediately before and after the event, respectively. We assume 
\begin{align*}
    m_E \leq \e_0 \, , \; \textup{if } E \in \mathcal{E}_s \, , \qquad 
    \sum_{E\in \mathcal{E}\backslash \mathcal{E}_s} m_E \leq M \, .
\end{align*}
We also assume
\begin{align*}
    e\leq f|_{[0,T)} \leq \widetilde{\mathbf{C}} e \, , \qquad  \lim_{t\rightarrow 0} f(t) = f(0) \, .
\end{align*}
Then, for all $t\in [0,T)$,
\begin{align*}
    f(t) &\leq f(0) e^{C_1 T + C_2 e + M} \, , \\
    \int_0^T g(t) \, dt &\leq f(0) \bigl( C_1 T +1 + M)e^{C_1 T + C_2 e + M} + C_2 e^2 \, .
\end{align*}

\end{lemm}

\begin{proof}
Let 
\begin{align*}
    J(t) = \ln f(t) \, .
\end{align*}
For any time $t$ between stopping times, we have
\begin{align*}
    J'(t) \leq \frac{C_1 f(t) + e^2 h(t) - l(t)}{f(t)} \leq C_1 + e  h(t) - \frac{l(t)}{\widetilde{\mathbf{C}} e}\, .
\end{align*}
At an event $E$, we have
\begin{align*}
    J_+ - J_- = \ln(\frac{f_+}{f_-}) 
    \leq \ln(1+m_E) \leq m_E \, .
\end{align*}
We let $\mathcal{E}(t)$ denote the set of events occurring on $[0,t]$.
Since we have
\begin{align*}
    J(t+) - J(0) \leq \int_0^t J'(s)\, ds + \sum_{\mathcal{E}(t)} J_+ - J_-\, ,
\end{align*}
we get
\begin{align*}
    J(t+) - J(0) &\leq C_1 T + e \int_0^T h(t) \, dt - \tilde{N}(t) \frac{\mathbf{C}}{\widetilde{\mathbf{C}}} + M + \tilde{N}(t) \e_0  \, ,
\end{align*}
where 
\begin{align*}
    \tilde{N}(t) = \max \{j: t_{j} \leq t\} \, .
\end{align*}
We choose $\e_0$ so that
\begin{align}\label{proof of lemma uni in time e star}
    \frac{\mathbf{C}}{\widetilde{\mathbf{C}}} \geq \e_0 \, , 
\end{align}
and get
\begin{align*}
    J(t+) - J(0) \leq C_1 T + C_2 e + M \, ,
\end{align*}
which implies
\begin{align*}
    f|_{[0,T]} \leq f(0) e^{C_1 T + C_2 e + M} =: \mathbf{C}_f\, .
\end{align*}

\smallskip

We denote the stopping times on $(0,T)$, together with the endpoints $0$ and $T$, by
\begin{align*}
    \widetilde{t}_0:= 0 < \widetilde{t}_1 < ... < \widetilde{t}_{\widetilde{N}-1} < T = : \widetilde{t}_{\widetilde{N}} \, .
\end{align*}
We also notice that
\begin{align*}
    &\int_0^T g(t) \, dt \\
    &= \sum_{1\leq j\leq \widetilde{N}} \int_{\widetilde{t}_{j-1}}^{\widetilde{t}_{j}} g(t) \, dt \\
    &=  \sum_{1\leq j\leq \widetilde{N}} \int_{\widetilde{t}_{j-1}}^{\widetilde{t}_{j}} g(t) + f'(t) \, dt  - \sum_{1\leq j\leq \widetilde{N}} f(\widetilde{t}_{j}-) - f(\widetilde{t}_{j-1}+) \\
    &\leq C_1 \int_0^T f(t) \, dt + e^2 \int_0^T h(t) \, dt - N {\mathbf{C} e} + f(0) - f(T-) + \sum_{\mathcal{E}} f_+ - f_- \\
    &\leq C_1 \int_0^T f(t) \, dt + e^2 \int_0^T h(t) \, dt - N {\mathbf{C} e} + f(0) + \sum_{E\in\mathcal{E}} m_E f_- \\
    &\leq \mathbf{C}_f C_1 T + C_2 e^2 - N e {\mathbf{C}} + f(0) + \mathbf{C}_f M + N e \widetilde{\mathbf{C}} \e_0   \, .
\end{align*}
By \eqref{proof of lemma uni in time e star}, we get
\begin{align*}
    \int_0^T g(t) \, dt \leq \mathbf{C}_f \bigl( C_1 T + 1 + M \bigr) + C_2 e^2 \, .
\end{align*}

\end{proof}

\subsection{Global-in-time uniform estimates}\label{subsec global-in-time estimates}
We define
\begin{align}
    f(t) = \int a \, \eta(U|\bu) \dx + \mev \, , \quad g(t) = G(U) \, ,
\end{align}
where $G(U)$ is defined in \eqref{def of D_v, G}. Recall Property \ref{prop at most one merging each stopping time}. We let $\mathcal{E}_s$ denote the set of mergings involving a previously merged wave, and denote by
\begin{align*}
    t_0:= 0 < t_1 < ... < t_N < T 
\end{align*}
the left endpoint $0$ together with the stopping times at which such mergings occur.

\subsubsection{Increase between $t_{j-1}$ and $t_j$}
Proposition \ref{proposition energy estimate} and \eqref{rt uni in time between stopping times} imply that in between stopping times, we have
\begin{align}\label{uni estimate between stopping times}
    f'(t) + g(t) \leq C_1 f(t) + \mev^2 h(t) - l(t)\, ,
\end{align}
where $C_1>0$ is a constant depending only on $U_*, \de_1$, and $h,l\geq 0$ are defined in \eqref{def of D_v, G}. Note that \eqref{lemma uniform bounds D v} implies that there exists $C_2>0$ such that 
\begin{align*}
    \int_0^T h(t) \, dt \leq C_2 \, .
\end{align*}
In addition, Lemma \ref{lemma estimates on relative functions 2} and \eqref{merging lower bound of y_b} imply that
\begin{equation*}
    \begin{aligned}
        \int_{t_{j-1}}^{t_j} l(t) \, dt 
        &\geq C\int_{t_{j-1}}^{t_j} y_b'(t) Q(v| \bv)\bigl|_{y_b(t)} \, dt \\
        &\geq C\int_{t_{j-1}}^{t_j} y_b'(t) \mathds{1}_{\{y_b'(t) \geq \la_{\inf}^{(2)}(t)\}} \, dt \\
        &\geq \mathbf{C} \mev \, ,
    \end{aligned}
\end{equation*}
where $\la_{\inf}^{(2)}(t)$, defined in \eqref{def of the minimal speed of the physical 2-wave}, denotes the smallest speed among all $2$-waves in the approximate solution at time $t$. Here $\mathbf{C}$ depends only on $U_*$ and on $\La_0$, which is chosen in Lemma~\ref{lemma estimates on relative functions 2}.

\medskip
We restrict our attention to the time interval during which the integration of the weighted relative entropy is bounded by $\mev$ up to a constant. More precisely, we restrict to the interval $[0,\widetilde T]$, where
\begin{align}\label{def of tilde T}
    \widetilde T := \sup\{0\leq t\leq T : \int a \eta(U|\bu) \leq 2 ( 2\mathbf{C}_0 + 1) e^{C_1 T + 1} \mev\} \, ,
\end{align}
with $\mathbf{C}_0$ is defined in \eqref{decay rate of difference between initial and app initial}, and $C_1$ is defined in \eqref{uni estimate between stopping times}. Note that in our case, we apply Lemma~\ref{lemma technical} with
\begin{align*}
    \tilde{\mathbf{C}} = 2 (2 \mathbf{C}_0 + 1) e^{C_1 T + 1} \, .
\end{align*}
Here $\tilde{\mathbf{C}}$ depends on $U_*$ and $\de_1$.

We then estimate the increase at events throughout the evolution. Following the notation $M$ in Lemma~\ref{lemma technical}, we denote by $M_I$, $M_P$, $M_M$, and $M_O$ the sums of $m_E$ over all events $E$ corresponding to interactions, passings and partitions, mergings involving no previously merged wave, and overtakings, respectively. Note that we have
\begin{align}\label{global in time estimates M def}
    M = M_I + M_P + M_M + M_O \, .
\end{align}

\subsubsection{Increase due to interactions}
Let $0<\kappa<1$. We classify waves into three types, big, medium, and small, according to their strength. We call a wave big if its strength is greater than or equal to $\rn^{\kappa\cpr}$. We call a wave medium if its strength is greater than or equal to $\rn^{\frac{1+\kappa}{2}\cpr}$ but smaller than $\rn^{\kappa\cpr}$. We call a wave small otherwise.

We call an interaction big if at least one of the incoming waves is big. We call an interaction medium if neither incoming wave is big, but at least one incoming wave is medium. Finally, we call an interaction small otherwise. We denote the total numbers of big, medium, and small interactions by $N_b$, $N_m$, and $N_s$, respectively. Lemma~\ref{lemma number of big interactions} and Lemma~\ref{lemma number of waves and interactions} imply that
\begin{align*}
    N_b \leq C \rn^{-(1+\kappa) \cpr} \ln \rn^{-1} \, , \quad N_m \leq C \rn^{-\frac{3+\kappa}{2} {\cpr}} \ln \rn^{-1} \, , \quad N_s \leq \rn^{-2\cpr} ( \ln \rn^{-1} )^2\, .
\end{align*}
At the same time, the sum $\Sigma$ of the absolute values of the strengths of all incoming waves is bounded by
\begin{align*}
    \Sigma = 
    \begin{cases}
        C \ml(0)  \, , \quad \textup{if it is a big interaction,} \\
        C \rn^{\kappa \cpr} \, , \quad \textup{if it is a medium interaction,} \\
        C \rn^{\frac{1+\kappa}{2} \cpr} \, , \quad \textup{if it is a small interaction.}
    \end{cases}
\end{align*}
Together with  the last inequality in \eqref{rt uni in time at stopping times}, for any interaction, \eqref{energy estimate with decay of weight} becomes
\begin{align*}
    \int a_+ \, \eta(U|\bu_+) \dx - \int a_- \, \eta(U|\bu_-) \dx
    \leq  C \rn^{\co/2} \Sigma \int a_- \, \eta(U|\bu_-) \dx  + C\rn^{\co/2} \Sigma \, .
\end{align*}
Therefore, when an interaction occurs,
\begin{align*}
    f_+ - f_- \leq C m f_- \, ,
\end{align*}
where
\begin{equation*}
    m = 
    \begin{cases}
        \rn^{\frac{\co}{2}-\cfi}  \, , \quad \textup{if it is a big interaction,} \\
        \rn^{\frac{\co}{2}+\kappa \cpr-\cfi} \, , \quad \textup{if it is a medium interaction,} \\
        \rn^{\frac{\co}{2}+\frac{1+\kappa}{2}{\cpr}-\cfi} \, , \quad \textup{if it is a small interaction.}
    \end{cases}
\end{equation*}
We denote by $m_b$, $m_m$, and $m_s$ the values of $m$ corresponding to big, medium, and small interactions, respectively.

We have
\begin{align*}
    M_{I} &\leq N_b m_b + N_m m_m + N_s m_s  \\
    &\leq C {\rn}^{\frac{\co}{2}-\cfi-(1+\kappa)\cpr} + C {\rn}^{\frac{\co}{2}-\cfi + \frac{\kappa-3}{2}\cpr}  \, ,
\end{align*}
Taking $\kappa=\frac{1}{3}$, the first inequality in \eqref{rt uni in time at stopping times} gives
\begin{align}\label{global in time estimates interaction}
    M_I \leq C \rn^{\frac{\co}{2} - \cfi - \frac{4}{3} \cpr} ( \ln \rn^{-1} )^2
    \rightarrow 0 \, .
\end{align}

\subsubsection{Increase due to partition and passing events}
Assume that an interaction between two shocks of the same family generates a rarefaction $\be''$ of the other family whose associated ordered set contains $p>1$ elements. We consider a partition or passing event occurring during the subsequent discretization process. By Property~\ref{prop strength R discretization} and the last inequality in \eqref{rt uni in time at stopping times}, \eqref{energy estimate with decay of weight} becomes
\begin{align*}
    \int a_+ \, \eta(U|\bu_+) \dx - \int a_- \, \eta(U|\bu_-) \dx
    \leq  C p\rn^{\frac{\co}{2}+\cpr} \int a_- \, \eta(U|\bu_-) \dx  + C p \rn^{\frac{\co}{2}+\cpr} \, ,
\end{align*}
which implies
\begin{align*}
    f_+ - f_- \leq C p \rn^{\frac{\co}{2}+\cpr - \cfi} f_-\, .
\end{align*}
Also, during the subsequent discretization process, both the number of partition events and the number of passing events are bounded by the total number of elements. By Lemma~\ref{lemma number of waves and interactions}, the sum of $m_E$ over these events is bounded above by
\begin{align*}
    C p \rn^{\frac{\co}{2} - \cfi} \ln \rn^{-1} \, .
\end{align*}
Lemma \ref{lemma total number of R introduced big fan} implies that the sum of all $p$ throughout the evolution is bounded above by $\rn^{-\cpr}$. Hence, we have
\begin{align*}
    M_P \leq C \rn^{\frac{\co}{2}-\cfi-\cpr} \ln \rn^{-1} \, .
\end{align*}
By the first inequality in \eqref{rt uni in time at stopping times}, it follows that, as $\nu \rightarrow 0$,
\begin{align}\label{global in time estimates partition and passing}
    M_P \rightarrow 0 \, .
\end{align}

\subsubsection{Increase due to mergings}
We estimate separately the increase caused by mergings involving a previously merged wave and the increase caused by mergings involving no previously merged wave.

Recall \eqref{energy estimate without weight decay m} and Lemma~\ref{lemma L+kQ}. When a merging that merges a previously merged wave occurs, we have
\begin{align*}
    f_+ - f_- \leq C\ml(0) f_- \, .
\end{align*}
On the other hand, Property \ref{prop merging any element merged at most twice} and Lemma \ref{lemma L+kQ} imply that
\begin{align}\label{global in time estimates merging}
    M_{M} \leq C \ml_{total} \leq C \ml(0) \, . 
\end{align}

\subsubsection{Increase due to overtakings}
By \eqref{energy estimate without weight decay o}, we have
\begin{align*}
    f_+ - f_- \leq C \rn^{1-\cfi} f_- \, .
\end{align*}
Together with Lemma \ref{lemma number of waves and interactions},
\begin{align*}
    M_O \leq \textup{number of overtakings} \times C \rn^{1-\cfi} \leq C \rn^{1-\cfi-\cpr} \ln \rn^{-1} \, .
\end{align*}
By \eqref{rt uni in time at stopping times}, as $\nu \rightarrow 0$,
\begin{align}\label{global in time estimates overtaking}
    M_O \rightarrow 0 \, .
\end{align}

\medskip

Applying Lemma \ref{lemma global in time estimates}, we get that on $[0, \widetilde T]$,
\begin{align*}
    f(t) \leq f(0) e^{C_1 T + C_2 \mev + M} \, , \quad \int_0^{\widetilde T} g(t) \, dt \leq  f(0) ( C_1 T + 1 + M) e^{C_1 T + C_2 \mev + M} + C_2 \mev^2 \, ,
\end{align*}
where $M$ is defined in \eqref{global in time estimates M def}. By \eqref{global in time estimates interaction}, \eqref{global in time estimates partition and passing}, \eqref{global in time estimates merging}, \eqref{global in time estimates overtaking}, \eqref{decay rate of difference between initial and app initial}, \eqref{decay rate of difference between initial and app initial}, and \eqref{weight a close to 1}, we get that on $[0, \widetilde T]$, 
\begin{align*}
    f(t) \leq f(0) e^{C_1 T + 1} \leq (2\mathbf{C}_0 + 1) e^{C_1 T + 1} \mev \, .
\end{align*}
Recall \eqref{def of tilde T}. We obtain that $\widetilde T = T$. In all, we get
\begin{align}\label{convergence of F and G}
    \int a \eta(U|\bu) \dx \leq C \mev \rightarrow 0 \, , \; \textup{for any }t \in [0,T] \, , \quad \int_0^T G(U) \, dt \leq C \mev \rightarrow 0 \, .
\end{align}

\subsection{Properties of the approximate solutions}\label{sub sec proof of prop limit of approximate solution}
Lemma \ref{lemma L+kQ} ensures \eqref{thm lagrangian small bv}. It remains to show that 
\begin{align*}
    \bar{U}^\nu(\ef,t) = \bu_\nu(\ef+y_b^\nu,t)
\end{align*}
converges, up to a subsequence, in $L_{\mathrm{loc}}^1(\mathbb{R}_+\times \mathbb{R}_+)$, and that any subsequential limit satisfies the outflow boundary condition \eqref{outflow boundary condition}.

\subsubsection{Construction of the translated approximate solution $\bar{V}^\nu$}
We first note that $\bar{U}^\nu$ is the reparametrization of the front tracking approximate solution $\bu_\nu$ with respect to the boundary. In particular, the definition of $\bu_\nu$ in \eqref{def approximate sol at time t_j} gives
\begin{align*}
    \bar{U}^\nu(\ef,t) =  \ust - \int_{\ef + y_b^\nu(t)}^\infty \sum_{\sse} \p_x {\usn}(y,t) + \sum_{\rse} \p_x \urn(y,t) + \sum_{\npse} \p_x \unn(y) \, dy \, ,
\end{align*}
where $\sse$, $\rse$, and $\npse$ denote the sets of shocks, rarefactions, and non-physical waves, respectively, in $\bu_\nu$.

Since we do not have $\nu$-uniform $L^\infty$ control on the Navier-Stokes solutions, the speed of a $2$-wave may be smaller than the speed of the boundary. In other words, after reparametrization, $2$-waves in $\bar{U}^\nu$ may travel to the left. We construct a translated approximate solution $\bar V^\nu$ that omits the non-physical waves, preserves the same configuration of physical waves as $\bar U^\nu$, and retains the property that $2$-waves travel to the right and $1$-waves travel to the left. More precisely, we define
\begin{align*}
    \bar{V}^\nu(\ef,t) =  \ust - \int_{\ef+\eft(t)}^\infty \sum_{\sse} \p_x {\usn}(y,t) + \sum_{\rse} \p_x \urn(y,t) \, dy \, ,
\end{align*}
where $\sse$ and $\rse$ are the sets of shocks, rarefactions, and non-physical waves, respectively, in $\bu_\nu$, and
\begin{align}\label{proof of boundary condition translation}
    \eft(t) = \int_0^t (y_b^\nu)'(s) + \min \{\la_{\inf}^{(2)}(s) - (y_b^\nu)'(s) - \La_0, 0\} \, ds \, .
\end{align}
Here $\la_{\inf}^{(2)}$ denotes the minimum speed among the $2$-waves in $\bu_\nu$, as defined in \eqref{def of the minimal speed of the physical 2-wave}, and $\La_0>0$ is the constant chosen in Lemma \ref{lemma estimates on relative functions 2}. 

In other words, we artificially decrease the speed of the boundary so that it is always at least $\La_0$ smaller than the speed of any $2$-wave. Whenever such a decrease is applied, the resulting change in the speed of the boundary is
\begin{align}\label{velocity difference between shifted and unshifted}
    \la_{\inf}^{(2)}(s)-(y_b^\nu)'(s)-\La_0 \leq 0 \, .
\end{align}
Hence, the speed of a $1$-wave $\bar V^\nu(\ef,t)$ is
\begin{align*}
    \la_1-(y_b^\nu)'-\bigl(\la_{\inf}^{(2)}(s)-(y_b^\nu)'(s)-\La_0\bigr) \mathds{1}_{\{\la_{\inf}^{(2)}(s)-(y_b^\nu)'(s) - \La_0 \leq 0\} } \, ,
\end{align*}
where $\la_1$ denotes the speed of the corresponding $1$-wave in $\bu_\nu$.

Lemma~\ref{lemma estimates on relative functions 2} shows that the absolute value of the speed of any wave in $\bu_\nu$ is bounded below by $2\La_0$. Hence, if $\la_{\inf}^{(2)}(s) - (y_b^\nu)'(s) - \La_0 > 0$, then the speed of the $1$-wave is bounded by 
\begin{align*}
    \la_1-(y_b^\nu)' \leq \la_1 \leq -2\La_0 \, ,
\end{align*}
and if $\la_{\inf}^{(2)}(s) - (y_b^\nu)'(s) - \La_0 < 0$, then the speed of the $1$-wave is bounded by 
\begin{align*}
    \la_1-(y_b^\nu)'(s)-\bigl(\la_{\inf}^{(2)}(s)-(y_b^\nu)'(s)-\La_0\bigr)
    =\la_1 -\la_{\inf}^{(2)} + \La_0
    \leq -3\La_0 \, .
\end{align*}
We conclude that in $\bar V^\nu(\ef,t)$, the speed of $1$-waves are bounded above by $-2\La_0$.

We fix $x_*\in \mathbb{R}_+$. We note that in $\bar V^\nu(\ef,t)$, $1$-waves travel to the left, $2$-waves travel to the right, and the absolute values of the speeds of all waves are bounded below by $\La_0$. Together with Lemma~\ref{lemma L+kQ}, we get 
\begin{align*}
    \|\bar{V}^\nu(x_*, t)\|_{BV([0,T])} &\leq 4\sum_{b} |\s_b(t_0^b)| \leq C \ml(0) \, , \\
    \|\partial_x \bar{V}^\nu(x_*, t)\|_{\mathcal{M}([0,T])} &\leq \frac{4}{\La_0}\sum_{b} |\s_b(t_0^b)| \leq \frac{C}{\La_0}\ml(0)\, ,
\end{align*}
The Aubin-Lions lemma then implies that for any $M>0$, up to a subsequence, $\bar{V}^\nu(\ef,t)$ converges in $C_x([0,M];L_t^1([0,T]))$.

Since $\eft\leq y_b^\nu$, the distance of each physical wave from the boundary in $\bar V^\nu$ is greater than or equal to the corresponding distance in $\bu^\nu$. The boundary condition may be violated due to the omission of non-physical waves, overtakings, and the viscous effects of shocks. Hence, at any time $t$, Corollary~\ref{lemma strength of waves}, Lemma~\ref{lemma estimate of waves}, and Lemma~\ref{lemma number of waves and interactions} imply that the deviation of the boundary velocity of the translated approximate solution $\bar V^\nu$ from the prescribed value $u_B$ is bounded by
\begin{align*}
    \bigl(C\rn+C_n\nu^n\bigr)\rn^{-\cpr}\ln\rn^{-1}
    \leq C\rn^{1-\cpr}\ln\rn^{-1}.
\end{align*}
By \eqref{threshold mft}, the upper bound converges to $0$. Hence, any subsequential limit of the translated approximate solutions satisfies the outflow boundary condition \eqref{outflow boundary condition}.

\subsubsection{Difference between the approximate solution and the translated approximate solution}\label{subsub difference betwen translated and original approximate solution}
The spatial translation of $\bar{V}^\nu$ from $\bar{U}^\nu$ at time $t$ is
\begin{align*}
    y_b^\nu(t) - \eft(t) = \int_0^t -\bigl( \la_{\inf}^{(2)}(s) - (y_b^\nu)'(s) - \La_0 \bigr)  \mathds{1}_{\{\la_{\inf}^{(2)}(s)-(y_b^\nu)'(s) - \La_0 \leq 0\} } \, dt \, .
\end{align*}
By Lemma \ref{lemma estimates on relative functions 2} and \eqref{convergence of F and G}, we get that on $[0,T]$,
\begin{align*}
    y_b^\nu(t) - \eft(t) \leq \int_0^T  y_b^{\nu\,\prime}(t) \mathds{1}_{\{y_b^{\nu\,\prime}(t) + \La_0 \geq \la_{\inf}^{(2)}(t)\}} \, dt \leq C \int_0^T y_b^{\nu\,\prime} Q(v_b^\nu|\bv_\nu) \, dt \leq C \mev \, .
\end{align*}
Together with Lemma \ref{lemma L+kQ}, Corollary \ref{lemma strength of waves}, and Lemma \ref{lemma number of waves and interactions}, we have that for any $0\leq t\leq T$,
\begin{align}\label{differnce between approximate solutions V and U}
    \|\bar{V}^\nu(\ef, t) - \bu^\nu(\ef, t)\|_{L^1(0,M)} \leq  C\mev \ml(0) + C M \rn^{1-\cpr} \ln \rn^{-1} \rightarrow 0 \, .
\end{align}

\section{Proof of Proposition \ref{prop E approximate sol}}\label{sec from Lagrange to Euler}

We recall the definition of the Lagrangian mass variable $y^\nu$ from \eqref{def of y^nu}.
As in Section~12 of \cite{NStoE}, we define $U^0(y^\nu,t)$ from $U^0(x,t)$ by
\begin{align*}
    v^0(y^\nu,t) = \frac{1}{\rho^0(x,t)} \, , \quad h^0(y^\nu,t) = u^0(x,t) \, , 
\end{align*}
obtain 
\begin{align*}
    \bu_\nu(y^\nu,t)=\bigl(\bvn(y^\nu, t),\bhn(y^\nu, t)\bigr)
\end{align*}
from Theorem \ref{thm v,h}, and define $\bu_\nu(x,t)$ through the relations
\begin{align*}
    \brn(x,t) = \frac{1}{\bvn(y^\nu(x,t), t)} \, , \quad \bun(x,t) = \bhn(y^\nu(x,t), t) \, .
\end{align*}
Following the same argument as in Section 12 of \cite{NStoE}, Theorem~\ref{thm v,h} proves all the statements in Proposition~\ref{prop E approximate sol}, except that any subsequential limit of $\bu_\nu(x,t)$ satisfies the outflow boundary condition \eqref{outflow boundary condition}. 

In particular, Proposition~\ref{prop E compensated compactness} and \eqref{thm lagrangian small bv}, we have that, up to a subsequence, there exists $U^E=(\rho^E, \rho^E u^E)$ such that
\begin{align*}
    U^\nu(x,t) \to U^E(x,t) \text{ in } L^1_{\mathrm{loc}}(\mathbb{R}_+ \times \mathbb{R}_+)\, , \quad
    \bar U_\nu (x,t) \overset{*}{\rightharpoonup} U^E (x,t) \text{ in } L^\infty(\mathbb{R}_+ \times \mathbb{R}_+) \, ,
\end{align*}
and
\begin{align}\label{from e to l prelim convergence}
    \bun(x,t) \overset{*}{\rightharpoonup} u^E (x,t) \text{ in } L^\infty(\mathbb{R}_+ \times \mathbb{R}_+) \, .
\end{align}

We recall that the boundary in Lagrangian coordinates is denoted by $y_b^\nu$, as defined in \eqref{def of y_b^nu}. Theorem~\ref{thm v,h} implies that there exists a further subsequence that
\begin{align*}
    \bar{u}^\nu(\ef,t):=\bhn(\ef+y_b^\nu, t) \to u^L(\ef,t)
    \quad\text{in }L_{\mathrm{loc}}^1(\mathbb{R}_+\times\mathbb{R}_+) \, ,
\end{align*}
to some function $u^L(\ef,t)$ satisfying the outflow boundary condition \eqref{outflow boundary condition}.

We note that $\bar{u}^\nu(\ef,t)$ and $\bun(x,t)$ are different reparametrizations of $\bhn(y^\nu, t)$ from $(y_b^\nu,\infty)\times\mathbb R_+$ to $\mathbb R_+\times\mathbb R_+$. Therefore, the fact that $u^L(\ef,t)$ satisfies the outflow boundary condition \eqref{outflow boundary condition} does not directly imply that $u^E(x,t)$ satisfies the same boundary condition. In the following, we study the map between $\ef$ and $x$.

Let $T>0$. By \eqref{thm lagrangian decomposition}, we get that
\begin{align*}
    \int_0^T (y_b^\nu)'(t) Q\Bigl(\vn\bigl(y_b^\nu(t),t\bigr)|\bv_\nu\bigl(y_b^\nu(t),t\bigr)\Bigr) \, dt \rightarrow 0 \, .
\end{align*}
It is easy to see that there exists $C>0$, depending only on $U_*$, such that
\begin{align*}
    &(y_b^\nu)'(t) Q\Bigl(\vn\bigl(y_b^\nu(t),t\bigr)|\bv_\nu\bigl(y_b^\nu(t),t\bigr)\Bigr) \\
    &\geq C \Bigl( |\frac{1}{\vn\bigl(y_b^\nu(t),t\bigr)} - \frac{1}{\bv_\nu\bigl(y_b^\nu(t),t\bigr)}| \mathds{1}_{\{|\vn(y_b^\nu(t),t) - v_*| > {v_*}\}} \\
    &\qquad + |\frac{1}{\vn\bigl(y_b^\nu(t),t\bigr)} - \frac{1}{\bv_\nu\bigl(y_b^\nu(t),t\bigr)}|^2 \mathds{1}_{\{|\vn(y_b^\nu(t),t) - v_*| \leq {v_*}\}} \Bigr) \\
    &\geq C\Bigl( |(y_b^\nu)'(t)-\frac{u_B}{\bvn\bigl(y_b^\nu(t),t\bigr)}| \mathds{1}_{\{|\vn(y_b^\nu(t),t) - v_*| > {v_*}\}} \\
    &\qquad + |(y_b^\nu)'(t)-\frac{u_B}{\bvn\bigl(y_b^\nu(t),t\bigr)}|^2 \mathds{1}_{\{|\vn(y_b^\nu(t),t) - v_*| \leq {v_*}\}} \Bigr) \, .
\end{align*}
Hence, we get that
\begin{align*}
    \int_0^T |(y_b^\nu)'(t)-\frac{u_B}{\bvn\bigl(y_b^\nu(t),t\bigr)}| \, dt \rightarrow 0 \, .
\end{align*}
On the other hand, we see that, up to a subsequence,
\begin{align*}
    \frac{u_*}{\bar v^\nu\bigl(y_b^\nu(t),t\bigr)} \overset{*}{\rightharpoonup} g \, , 
    \qquad \textup{in } L^\infty(0,T) \, .
\end{align*}
Hence, up to a subsequence, 
\begin{align}\label{from e to l eq 1}
    \|y_b^\nu(t) -  y_b(t)\|_{L^\infty(0,T)} \rightarrow 0 \, .
\end{align}
where
\begin{align*}
    y_b(t) = \int_0^t g(s) \,  ds \, .
\end{align*}
Let 
\begin{align}\label{from e to l def of tilde y}
    \tilde{y}(x,t) = \int_0^x \rho^E(z,t) \, dz + y_b(t) \, .
\end{align}
We have
\begin{align*}
    y^\nu(x,t) - \tilde{y}(x,t) &= \int_0^x  \rhn(z,t) - \rho^E(z,t) \, dz - 
    \int_0^t \rhn (0,s) \un(0,s)  \, ds - y_b(t) \\
    &= \int_0^x  \rhn(z,t) - \rho^E(z,t) \, dz + \bigl( (y_b^\nu)(t) - y_b(t) \bigr) \, .
\end{align*}
Let $M>0$. For a fixed $0\leq t\leq T$, we have
\begin{align}\label{from lagrangian to euler convergence of y}
    \bigl|y^\nu(\cdot,t) - \tilde{y}(\cdot,t)\bigr|_{[0,M]} \leq \int_0^M \bigl|\rhn(z,t) - \rho^E(z,t)\bigr| \, dz +\bigl| (y_b^\nu)(t) - y_b(t) \bigr| =: C_M(t) \, .
\end{align}
We reparametrize $\bar{u}^\nu(\ef,t)$ and $u^L(\ef,t)$ by setting
\begin{align*}
    \ef(x,t)=\tilde y(x,t)-y_b(t) \, .
\end{align*}
In other words, we define
\begin{align*}
    \bar{\textbf{u}}^\nu(x,t) = \bar{u}^\nu\bigl(\ef(x,t),t\bigr) \, , \quad
    \textbf{u}^L(x,t) = u^L\bigl(\ef(x,t),t\bigr) \, .
\end{align*}
Since $U^E$ is a small BV solution, there exist constants $0<C_1<C_2$ such that, for each fixed $t \ge 0$, the map $x \mapsto \ef(x,t)$ is increasing and satisfies
\begin{align*}
    C_1 \leq \frac{d}{dx} \ef(x,t) \leq C_2 \, .
\end{align*}
Hence, up to a subsequence,
\begin{align}\label{from l to e eq 4}
    \bar{\textbf{u}}^\nu(x,t) \rightarrow \textbf{u}^L(x,t) \, , \quad \textup{ in } L^1_\mathrm{loc}(\mathbb{R}_+ \times \mathbb{R}_+) \, ,
\end{align}
and the trace of $\textbf{u}^L(x,t)$ at $x=0$ coincides with the trace of $u^L(\ef, t)$ at $\ef=0$. Since $u^L(\ef, t)$ satisfies the outflow boundary condition \eqref{outflow boundary condition}, we conclude that $\textbf{u}^L(x,t)$ also satisfies \eqref{outflow boundary condition}.

Finally, we show that $\textbf{u}^L(x,t)=u^E(x,t)$. To this end, \eqref{from e to l prelim convergence} and \eqref{from l to e eq 4} imply that it suffices to compare $\bun(x,t)$ and $\bar{\textbf{u}}^\nu(x,t)$ in $L_{\mathrm{loc}}^1(\mathbb{R}_+\times \mathbb{R})$. By \eqref{from lagrangian to euler convergence of y}, \eqref{from e to l def of tilde y}, and \eqref{thm lagrangian small bv}, we have, up to a subsequence,
\begin{align*}
    I&=\int_0^T \int_0^{M}|\bun(x,t) - \bar{\textbf{u}}^\nu(x,t)| \dx \, dt 
    \\
    & \leq \int_0^T \int_0^{M} |\bhn(y^\nu(x,t),t) - \bhn(\tilde{y}(x,t)-y_b(t)+y_b^\nu(t),t)| \dx \, dt \\
    & \leq \int_0^T \int_{y_b^\nu(t)}^{\infty} |\p_z \bhn(z,t)| |\{x\in [0,M]:|z-\tilde{y}(x,t)|\leq 2C_M(t)\}|\, dz \, dt \\
    & \leq C \int_0^T C_M(t) \|\bu_\nu(\cdot, t)\|_{\textup{BV}(y_b^\nu,\infty)} \, dt \\
    &\rightarrow 0 \, .
\end{align*}

\appendix
\section{Proof of Lemma \ref{lemma L+kQ}}\label{sec proof of lemma L+kQ}
We fix $C_g>0$ sufficiently large so that 
\begin{align}\label{proof of well-definedness of mft: constant}
   C_g>11C_0 \, , 
\end{align}
where $C_0$ denotes the maximum of the constants $C$ appearing in Lemma~\ref{lemma tracked strength at an event}, Lemma~\ref{lemma interaction wave strength estimates} and Lemma \ref{lemma interaction wave strength estimates NP}. 

By Properties~\ref{prop introduction of new elements} and~\ref{prop non-increasing passing, partition, merging}, $\ml$ is non-increasing across all events except interactions. It remains to show that $\ml$ is also non-increasing across interactions.

\smallskip
We first focus on the time interval $[0,T_*)$, where 
\begin{equation}\label{proof of well-definedness T*}
    \begin{aligned}
        T_* = \sup\{t: \textup{on } s\in[0,t) \, , & \;|\s_b(s)|\leq 2 |\s_b(t_0^b)| \textup{ for every } b \, , \\ 
        &\ml (s) \leq 2 \ml(0) = 2\e\, , \;
        \textup{and } \widetilde{\ml}_{total}(s) \leq 3\ml(0) = 3\e \} \, .
    \end{aligned}
\end{equation}
Here $\widetilde{\ml}_{total}(t)$ denotes the total initial strength of all elements generated before time $t$. More precisely, 
\begin{align*}
    \widetilde{\ml}_{total}(t) = \sum_{\{b: 0\leq t_0^b<t\}} |\s_b(t_0^b)| \, .
\end{align*} 
The goal is to prove that $T_*=\infty$.

By Remark \ref{remark interaction at most two times}, on $t\in[0,T_*)$, by taking $\e_0$ small enough, we have
\begin{align}\label{proof of well-definedness of mft: equivalence g}
    |\s_b(t)| \geq g_b(t) \geq |\s_b(t)| e^{-2C_g\widetilde{\ml}_{total}(t)} \geq |\s_b(t)| e^{-6C_g\e}\geq \frac{|\s_b(t)|}{2} \, ,
\end{align}
which implies 
\begin{align}\label{proof of well-definedness of mft: equivalence}
    \sum |\s| \leq 2\ml(t) \leq 4 \ml(0) = 4 \e \, ,
\end{align}
where the sum is taken over the strengths of all waves that exist at time $t$.

\medskip

Assume that an $i$-wave $\be=\{b_1,\ldots,b_m\}$ of strength $\s$ interacts with a $j$-wave $\widetilde{\be}=\{b_{m+1},\ldots,b_n\}$ of strength $\widetilde{\s}$ at some time $t\in[0,T_*]$. 

Suppose $b_l$ is in the associated ordered set of an incoming wave. If $b_l$ is also in the associated ordered set of an outgoing wave, then the definition of $T_*$ in \eqref{proof of well-definedness T*}, Lemma \ref{lemma tracked strength at an event}, and the inequality $1+x\leq e^x$ imply that 
\begin{equation}\label{proof of well-definedness change of g}
    \begin{aligned}
        g_{b_l}(t-) - g_{b_l}(t+) &\geq e^{-2C_g\widetilde{\ml}_{total}(t+)} \bigl( |\s_{b_l}(t-)| e^{C_g|{\bar{\s}}|}-|\s_{b_l}(t+)|\bigr) \\
        &\geq  e^{-6C_g\e} |\s_{b_l}(t-)| \Bigl(e^{C_g|\bar{\s}|}-\bigl(1+C_0|\bar{\s}|\bigr)\Bigr) \\
        &\geq e^{-6C_g\e} |\s_{b_l}(t-)| |\bar{\s}| (C_g-C_0) \, ,
    \end{aligned}
\end{equation}
where $\bar{\s}$ is the strength of the wave that interacts with the wave whose associated ordered set contains $b_l$. In particular,
\begin{align*}
    \bar{\s} = 
    \begin{cases}
        \widetilde{\s} \, , \quad \textup{if } 1\leq l\leq m \, , \\
        \s \, , \quad \textup{if } m+1 \leq l \leq n \, .
    \end{cases}
\end{align*}
If $b_l$ is not in the associated ordered set of any outgoing wave, then the definition of $T_*$ in \eqref{proof of well-definedness T*} implies that
\begin{equation}\label{proof of well-definedness change of g simpler}
    g_{b_l}(t-) \geq e^{-6C_g\e} |\s_{b_l}(t-)|  \, .
\end{equation}

\smallskip
\noindent
\textbf{Case 1: $i\neq j$.} By Property \ref{prop generation order non-increasing} and \eqref{proof of well-definedness change of g}, we can take $\e_0$ small enough so that \eqref{proof of well-definedness of mft: constant} implies
\begin{align}\label{proof of well-definedness delta L 1}
    \Delta {\ml}(t) \leq -2e^{-6C_g\e}|\s||\widetilde{\s}|(C_g-C_0) \leq -10 C_0|\s||\widetilde{\s}| \, .
\end{align}

\noindent
\textbf{Case 2: $i=j$.} We denote the outgoing $i$-wave by $\be'$ and the outgoing wave in the other family $i'$ by $\be''$. We denote their strengths by $\s'$ and $\s''$, respectively. 

Lemma \ref{lemma interaction wave strength estimates}, \eqref{proof of well-definedness of mft: equivalence}, and \eqref{proof of well-definedness T*} imply that
\begin{align}\label{proof of well-definedness same class different type}
    |\s''|\leq C_0 |\s||\widetilde{\s}|\bigl(|\s|+|\widetilde{\s}|\bigr) \leq 2C_0 |\s||\widetilde{\s}| {\ml}(t-) \leq 4 C_0 \e |\s||\widetilde{\s}| \, .
\end{align}

\noindent
\textit{Case 2.1: $i=j$, and the incoming waves are of different types.} Recall Section~\ref{subsub interaction two physical in same family}. Without loss of generality, we assume that $\be'$ is of the same type as $\be$. Let $\widetilde{k}$ denote the minimum generation order among the elements in $\widetilde{\be}$. Then the generation order of the element in the associated ordered set of $\be''$ is $\widetilde{k}+1$.

By \eqref{proof of well-definedness change of g}, Property \ref{prop generation order non-increasing}, \eqref{proof of well-definedness change of g simpler}, \eqref{proof of well-definedness same class different type}, and \eqref{proof of well-definedness of mft: equivalence}, we can take $\e_0 $ small enough such that
\begin{equation}\label{proof of well-definedness delta L 2}
    \begin{aligned}
        \Delta {\ml}(t) &\leq - \e^{-\widetilde{k} + 1} e^{-6C_g\e} |\widetilde{\s}| + \e^{-\widetilde{k}}  4 C_0 \e |\s||\widetilde{\s}| \\
        &\leq \e^{-\widetilde{k} + 1} |\widetilde{\s}| (-e^{-6C_g\e} + 4C_0 |\s|) \\
        &\leq \e^{-\widetilde{k} + 1} |\widetilde{\s}| (-e^{-6C_g\e} + 16C_0 \e) \\
        &\leq -\frac{1}{2}|\widetilde{\s}| \, .
    \end{aligned}
\end{equation}
By \eqref{proof of well-definedness same class different type} and \eqref{proof of well-definedness of mft: equivalence}, we can take $\e_0 $ small enough such that
\begin{align}\label{proof of well-definedness of mft: ineq 1}
    |\s''| \leq 16 C_0 \e^2 |\widetilde{\s}| \leq |\Delta {\ml}(t)| \, .
\end{align}

\noindent
\textit{Case 2.2: $i=j$, and the incoming waves are both shocks.} 
Recall Section~\ref{subsub interaction two physical in same family}. We know $\be'$ is an $i$-shock and  $\be'=\{b_1,...,b_n\}$. Let $k$ and $\widetilde{k}$ denote the generation orders of the elements in $\be$ and $\widetilde{\be}$, respectively. Without loss of generality, we assume that $k\leq \widetilde{k}$. The generation orders of the elements in $\be'$ and $\be''$ are $k$ and $\widetilde{k}+1$, respectively.

By \eqref{proof of well-definedness change of g}, \eqref{proof of well-definedness same class different type}, \eqref{proof of well-definedness of mft: equivalence} and \eqref{proof of well-definedness of mft: constant}, we can take $\e_0$ small enough such that
\begin{equation}\label{proof of well-definedness delta L 3}
    \begin{aligned}
        \Delta {\ml}(t) &\leq - \e^{-\widetilde{k} + 1} e^{-6C_g\e} |\s||\widetilde{\s}| (C_g-C_0) + \e^{-\widetilde{k}}  4 C_0 \e |\s||\widetilde{\s}| \\
        &\leq \e^{-\widetilde{k} + 1} |\s| |\widetilde{\s}| (-e^{-6C_g\e} (C_g-C_0) + 4C_0) \\
        &\leq \e^{-\widetilde{k} + 1} |\s| |\widetilde{\s}| (-\frac{C_g-C_0}{2} + 4C_0) \\
        &\leq -C_0|\s||\widetilde{\s}| \, .
    \end{aligned}
\end{equation}
By \eqref{proof of well-definedness same class different type}, we can take $\e_0 $ small enough such that
\begin{align}\label{proof of well-definedness of mft: ineq 2}
    |\s''| \leq |\Delta {\ml}(t)| \, .
\end{align}

\medskip
In all, on $[0,T_*)$, \eqref{proof of well-definedness delta L 1}, \eqref{proof of well-definedness delta L 2}, and \eqref{proof of well-definedness delta L 3} imply that ${\ml}$ is non-increasing, and at each interaction, we have
\begin{align}\label{proof of well-definedness delta L}
    \Delta {\ml}(t) \leq -C|\s||\widetilde{\s}| \, .
\end{align}
Property \ref{prop introduction of new elements}, \eqref{proof of well-definedness of mft: ineq 1}, \eqref{proof of well-definedness of mft: ineq 2} imply that on $[0,T_*]$,
\begin{align*}
    \widetilde{\ml}_{total} &\leq \sum_{\{b: 0\leq t_0^b = 0\}} |\s_b(t_0^b)| + \sum_{\{b: 0< t_0^b \leq T_* \}} |\s_b(t_0^b)| \\
    &\leq {\ml}(0) + \sum_{\{0\leq t\leq T_*: t \textup{ interaction times}\}} - \Delta {\ml}(t) \\
    &\leq 2\e \, .
\end{align*}
Moreover, for any element $b$ generated on $[0,T_*)$, $g_b$ is non-increasing, which gives
\begin{align*}
    |\s_b(t)| \leq e^{2C_g\widetilde{\ml}_{total}(t)} |\s_b(t_0^b)| \, .
\end{align*}
Together with Remark \ref{remark interaction at most two times}, we can take $\e_0$ small enough so that on $[0,T_*]$,
\begin{align*}
    |\s_b(t)| \leq e^{8C_g\e} |\s_b(t_0^b)| \leq \frac{3}{2} |\s_b(t_0^b)| \, .
\end{align*}
Hence, we conclude that $T_*=\infty$. Property~\ref{prop introduction of new elements} implies that only interactions between two physical waves can introduce new elements. Moreover, each such interaction introduces only finitely many new elements. Since the strength of each physical wave is bounded from below by $\rn$, as defined in \eqref{def of threshold P and NP}, together with \eqref{proof of well-definedness delta L}, the number of elements is finite. Hence, the number of waves is finite, and so is the number of events. We conclude that the approximate solution is globally well-defined.

\section{Proof of Lemma \ref{lemma smallness of localized relative entropy 1}}\label{section proof lemma smallness of localized relative entropy}
By \eqref{relative evaluation Q lower bounds}, there exist $C_1, C_2>0$ depending only on $U_*$ such that
\begin{equation}\label{proof of smallness of localized relative entropy 1}
    \begin{aligned}
        &\inb |(a_\be)_x| \eta(U|\bu) \dx \\
        &\geq \inb |(a_\be)_x| \frac{|h-\bh|^2}{2} + 1/2\inb |(a_\be)_x| Q(v|\bv) \dx \\
        &\quad + C_1 \inb \mathds{1}_{\{v\leq 3v_*\}} |(a_\be)_x| |v-\bv|^2 \dx + C_2 \inb \mathds{1}_{\{v> 3v_*\}} |(a_\be)_x| |v-\bv| \dx \, .
    \end{aligned}
\end{equation}
On the other hand, we recall the definition of $Y_\be$ in \eqref{relative entropy method terms}. We have
\begin{align}\label{proof of smallness of localized relative entropy 2}
    -\inb (a_\be)_x \eta(U|\bu) \dx = Y_\be - \inb a \nabla^2 \eta(\bu)(U-\bu) \p_x \su_\be \dx \, .
\end{align}
We multiply the equation by the sign of $-(a_\be)_x$. Recall that $a_\be$ decreases if $i=1$ and increases if $i=2$. By \eqref{useful computation}, \eqref{weight equivalence}, and the assumption $(-1)^{k_\be-1}Y_\be \leq \s_\be^2$, we get
\begin{align*}
    \inb |(a_\be)_x| \eta(U|\bu) \dx &\leq C|\s_\be|\de_* + \Bigl|\inb a \nabla^2 \eta(\bu)(U-\bu) \p_x \su_\be \dx \Bigr| \, ,
\end{align*}
where
\begin{equation}\label{proof of smallness of localized relative entropy 3}
    \begin{aligned}
        &\Bigl|\inb a \nabla^2 \eta(\bu)(U-\bu) \p_x \su_\be \dx \Bigr| \\
        &\leq C \de_* \inb |(a_\be)_x||v-\bv| \mathds{1}_{\{v> 3v_*\}} \dx  +  C \de_* \inb |(a_\be)_x||v-\bv| \mathds{1}_{\{v\leq 3v_*\}} \dx \\
        &\quad + C \de_* \Bigl( \inb |(a_\be)_x| |h-\bh|^2 \dx \Bigr)^{1/2} \Bigl( \inb |(a_\be)_x| \dx \Bigr)^{1/2} \\
        &\leq C \de_* \inb |(a_\be)_x||v-\bv| \mathds{1}_{\{v> 3v_*\}} \dx + C_1 \inb \mathds{1}_{\{v\leq 3v_*\}} |(a_\be)_x| |v-\bv|^2 \dx \\
        &\quad + \inb |(a_\be)_x| \frac{|h-\bh|^2}{4} + C|\s_\be|\de_* \, .
    \end{aligned}
\end{equation}
Note that \eqref{proof of smallness of localized relative entropy 1}, \eqref{proof of smallness of localized relative entropy 2}, and \eqref{proof of smallness of localized relative entropy 3} give the first inequality \eqref{lemma smallness eq 1}.

Also, \eqref{proof of smallness of localized relative entropy 3} gives
\begin{align*}
    |\inb a \nabla^2 \eta(\bu)(U-\bu) \p_x \su_\be \dx| \leq \inb |(a_\be)_x| \eta(U|\bu) \dx + C|\s_\be| \de_* \, ,
\end{align*}
which implies the second inequality \eqref{lemma smallness eq 2}.

\section{Proof of Lemma \ref{lemma inside truncation}}\label{section proof lemma inside truncation}
We prove the case of a $2$-shock; the case of a $1$-shock follows analogously.

We rewrite the functionals $\mi_{1\be}, \mi_{2\be}, \mg_{2\be}, \md_\be, \my_\be^g$ in terms of the following variables
\begin{align*}
    W = \frac{\pd}{\de_*} \ts \, , \quad y = \frac{p(\sv_\be)-p(v_r)}{\bsi} \, ,
\end{align*}
where $v_l$ denotes the left state and $v_r$ denotes the right state of $\sv_\be$, and
\begin{align*}
    \bsi = p(v_l)-p(v_r)>0 \, .
\end{align*}
Recall \eqref{def of weight a_i}, \eqref{weight relation between L and de*}, and \eqref{def of weight}. We have
\begin{align}\label{proof of lemma inside truncation a property}
    (a_\be)_x = - \frac{\p_x p(\sv_\be)}{\de_*} \, , \quad \|a-1\|_{L^\infty} \leq C \de_*^2 \, .
\end{align}
The range of $y$ is $[0,c_b(t)]$, where $0<c_b<1$.
Recall that the transition zone of the $2$-shock does not intersect the boundary, as shown in \eqref{dj infinity}. By Corollary \ref{lemma strength of waves}, Lemma \ref{lemma estimate of waves} and \eqref{rt transition S}, we have 
\begin{equation}\label{property of c_b}
    \begin{aligned}
        1-c_b(t) = \frac{p(v_l) - p(\sv_\be)|_{x=y_b(t)}}{\bsi} \leq C \frac{\sv_\be|_{x=y_b(t)}-v_l}{\bsi} \leq C \exp \bigl( -\frac{C\rn^{1+\co}}{\nu} \bigr) \leq C \nu \, .
    \end{aligned}
\end{equation}
We also rewrite the shock-related constants in terms of $U_*$. In particular, since $\las_\be = \sqrt{-\frac{p(v_r)-p(v_l)}{v_r-v_l}}$, we define 
\begin{align*}
    \la_* = \sqrt{-p'(v_*)} \, .
\end{align*}
By \eqref{weight relation between L and de*}, we have
\begin{align*}
    |\las_\be - \la_*| \leq C L(0) \leq C \de_*^3 \, .
\end{align*}
Let us first rewrite $\mi_{1\be}$.
By Taylor expansion, we have
\begin{align*}
    \Bigl| p(v|\bv) - 1/2 \frac{p''(\vs)}{p'(\vs)^2} \bigl( \pd \bigr)^2 \Bigr| \leq C \de_0  \bigl( \pd \bigr)^2 \, .
\end{align*}
Hence, we get
\begin{equation}\label{lemma inside truncation eq I1i}
    \begin{aligned}
        \mi_{1\be} &\leq \frac{\la_*}{p'(v_*)} 1/2 \frac{p''(\vs)}{p'(\vs)^2} (1 + C \de_0) \inb \p_x p(\sv_\be) \bigl(\pd\bigr)^2 \ts^2 \dx \\
        & = \bsi \de_*^2 (1 + C \de_0) \frac{\al_*}{2} \int_0^{c_b} W^2 \,dy \, ,
    \end{aligned}
\end{equation}
where 
\begin{align*}
    \al_* = \frac{\ga+1}{\ga^{\frac{3}{2}}} \vs^{\frac{3\ga + 1}{2}} \, .
\end{align*}
Similarly, by Taylor expansion, we have
\begin{align}\label{lemma inside truncation eq I2i}
    \mi_{2\be} = \bsi \de_* (1 + C \de_*^3) \frac{\al_{**}}{2} \int_0^{c_b} W^2 \,dy \, ,
\end{align}
and
\begin{equation}\label{lemma inside truncation eq G2i}
    \begin{aligned}
        -\mg_{2\be} \leq& -\bsi \de_* (1 - C\de_*^3) \frac{\al_{**}}{2} \int_0^{c_b} W^2 \,dy + \bsi \de_*^2 \frac{\al_*}{3} \int_0^{c_b} W^3 \,dy \\
        & + C \bsi \de_*^5 \int_0^{c_b} |W|^3 \,dy \, ,
    \end{aligned}
\end{equation}
where
\begin{align*}
    \al_{**} = \frac{\vs^{\frac{\ga+1}{2}}}{\ga^{1/2}} \, .
\end{align*}
We analyze $\mathcal{D}_\be$. First, we observe
\begin{align*}
    \md_\be = -\de_*^2 \int_0^{c_b} a |\p_y W|^2 \nu \mu_1(v) v^\ga \frac{\p y}{\p x} \, dy \, .
\end{align*}
Similarly to the proof of Proposition~6.1 in Appendix B of \cite{twoshocks}, we have
\begin{align*}
    \Bigl| -\nu\frac{\mu_1(v)v^\ga}{y(1-y)} \frac{\p y}{\p x} - \frac{\bsi \al_*}{2}\Bigr| \leq C \bsi \de_*^3 \, .
\end{align*}
Together with \eqref{proof of lemma inside truncation a property}, we get
\begin{equation}\label{lemma inside truncation eq Di}
    \begin{aligned}
        -\md_\be \leq - \bsi \de_*^2 (1-C\de_*^2)  \frac{\al_*}{2} \int_0^{c_b} |\p_y W|^2 y(1-y) \, dy \, .
    \end{aligned}
\end{equation}
Finally, we evaluate $\my_\be^g$. By Taylor expansion, we have
\begin{align*}
    &\Bigl| \my_\be^g + \bsi \de_* \al_{**}^2 \int_0^{c_b} W^2 \, dy 
    + \bsi \de_* 2\al_{**}^2 \int_0^{c_b} W \, dy \Bigr| \\
    &\leq C \bsi \de_* \de_0 \Bigl( \int_0^{c_b} W^2 \, dy + \int_0^{c_b} |W| \, dy \Bigr) \, .
\end{align*}
By \eqref{lemma inside truncation assumption} and H\"older's inequality, we have
\begin{align*}
    \int_0^{c_b} W^2 \, dy \leq C  \, .
\end{align*}
Hence, we get
\begin{equation}\label{lemma inside truncation eq Yi}
    \begin{aligned}
        -\frac{1}{\bsi \de_0} |\my_\be^g|^2 \leq & - \frac{\bsi \de_*^2}{\de_0}\al_{**}^2 \Bigl( \int_0^{c_b} W^2 \, dy + 2 \int_0^{c_b} W \, dy \Bigr)^2 \\ 
        &+ C \bsi \de_*^2 \de_0 \Bigl( \int_0^{c_b} W^2 \, dy + \int_0^{c_b} |W| \, dy \Bigr)^2 \\
        \leq & - \frac{\bsi \de_*^2}{\de_0}\al_{**}^2 \Bigl( \int_0^{c_b} W^2 \, dy + 2 \int_0^{c_b} W \, dy \Bigr)^2 + C \bsi \de_*^2 \de_0 \int_0^{c_b} W^2 \, dy \, .
    \end{aligned}
\end{equation}
By \eqref{lemma inside truncation eq I1i}, \eqref{lemma inside truncation eq I2i}, \eqref{lemma inside truncation eq G2i}, \eqref{lemma inside truncation eq Di}, and \eqref{lemma inside truncation eq Yi}, we get
\begin{align*}
    R_{\de_*}^\be \leq&\, \bsi \de_*^2 \frac{\al_*}{2} \tilde{R}_{\de_*}^\be \, ,
\end{align*}
where
\begin{align*}
    \tilde{R}_{\de_*}^\be =& \, (1+\bfc_1 \de_0) \int_0^{c_b} W^2 \, dy + \frac{2}{3} \int_0^{c_b} W^3 \, dy + \bfc_2 \de_0 \int_0^{c_b} |W|^3 \, dy \\
    &- (1-\bfc_3 \de_0)\int_0^{c_b} |\p_y W|^2 y(1-y) \, dy - \frac{\bfc_4}{\de_0} \bigl( \int_0^{c_b} W^2 \, dy + 2 \int_0^{c_b} W \, dy \bigr)^2 \, ,
\end{align*}
where $\bfc_1, \bfc_2, \bfc_3, \bfc_4>0$ are constants that depend only on $U_*,C_2, \bc$.

We set $\ty=\frac{y}{c_b}$ and change variables from $y$ to $\ty$. By \eqref{property of c_b}, we have
\begin{align*}
    \frac{\tilde{R}_{\de_*}^\be}{c_b} =& \, (1+\bfc_1 \de_0) \int_0^{1} W^2 \, d\ty + \frac{2}{3} \int_0^{1} W^3 \, d\ty + \bfc_2 \de_0 \int_0^{1} |W|^3 \, d\ty \\
    &- (1-\bfc_3 \de_0)\int_0^{1} |\p_{\ty} W|^2 \ty(\frac{1}{c_b}-\ty) \, d\ty - \frac{\bfc_4}{\de_0} c_b\bigl( \int_0^{1} W^2 \, d\ty + 2 \int_0^{1} W \, d\ty \bigr)^2  \\
    \leq& \, (1+\bfc_1 \de_0) \int_0^{1} W^2 \, d\ty + \frac{2}{3} \int_0^{1} W^3 \, d\ty + \bfc_2 \de_0 \int_0^{1} |W|^3 \, d\ty \\
    &- (1-\bfc_3 \de_0)\int_0^{1} |\p_{\ty} W|^2 \ty(1-\ty) \, d\ty - \frac{\bfc_4}{2\de_0} \bigl( \int_0^{1} W^2 \, d\ty + 2 \int_0^{1} W \, d\ty \bigr)^2 \, .
\end{align*}
By Proposition B.1 in \cite{twoshocks}, we can take $\de_0$ sufficiently small and conclude the proof.

\section{Proof of \eqref{truncation estimate vb} and \eqref{truncation estimate vs}}\label{section proof of ftoc}
By \eqref{truncation estimate 1} and the fact that $0\leq \ga-\al\leq 1$, we have
\begin{align*}
    \Bigl|\bigl(\pdb)\bigr(x) \Bigr| &\leq \int_{\xoi}^x \Bigl| \bigl(\pd\bigr)_x \Bigr| \mathds{1}_{\pdl} \dx\\
    &\leq C \int_{\xoi}^x \Bigl| \bigl(\pd\bigr)_x \Bigr| v^{\frac{\ga-\al}{2}} \dx \, .
\end{align*}
Hence, we get \eqref{truncation estimate vb}.

On the other hand, we notice that $\lv_s$ differs from $v$ when $v$ is small. We evaluate instead 
\begin{align*}
    v^{\frac{\ga-\al}{2}} \bigl(p(v)-p(\lv_s)\bigr) \, .
\end{align*}
By \eqref{truncation estimate 1}, we get
\begin{align*}
    \Bigl|v^{\frac{\ga-\al}{2}} \bigl(\pds\bigr)\Bigr| &= \int_{\xoi}^x |A| \dx \, ,
\end{align*}
where
\begin{align*}
    A = \Bigl(v^{\frac{\ga-\al}{2}} \bigl(\pd\bigr)\Bigr)_x \mathds{1}_{\pdu} \, .
\end{align*}
We have
\begin{align*}
    A 
    &= \Bigl[v^{\frac{\ga-\al}{2}} \bigl(\pd\bigr)_x - \frac{\ga-\al}{2\ga}\frac{\pd}{p(v)} v^{\frac{\ga-\al}{2}} p(v)_x \Bigr] \mathds{1}_{\pdu} \\
    &= \Bigl[\bigl(1-\frac{\ga-\al}{2\ga}\frac{\pd}{p(v)}\bigr) v^{\frac{\ga-\al}{2}} \bigl(\pd\bigr)_x \\
    &\qquad - \frac{\ga-\al}{2\ga}\frac{\pd}{p(v)} v^{\frac{\ga-\al}{2}} p(\bv)_x \Bigr] \mathds{1}_{\pdu} \, .
\end{align*}
On $\pdu$, we get
\begin{align*}
    0\leq \frac{\pd}{p(v)} \leq 1 \, .
\end{align*}
Hence, we have 
\begin{align*}
    \Bigl|v^{\frac{\ga-\al}{2}} \bigl(\pds\bigr)\Bigr| &\leq \int_{\xoi}^x v^{\frac{\ga-\al}{2}} \Bigl|\bigl(\pd\bigr)_x \Bigr| \dx + \int_{\xoi}^x v^{\frac{\ga-\al}{2}} |(\bv)_x| \dx \, .
\end{align*}
By Lemma \ref{lemma estimate of waves} and Corollary \ref{lemma strength of waves}, \eqref{rt mft total number}, we get
\begin{align*}
    |(\bv)_x| &\leq C \max\{\frac{1}{\nu}, \rn^{\cpr-\co}\} + C_n \nu^n \\
    &\leq \frac{C}{\nu} + C \nu^{(\cpr-\co)\cs} \\
    &\leq \frac{C}{\nu} \, .
\end{align*}
On $\pdu$, we have
\begin{align*}
    \frac{v^{\ga-\al}}{|v-\bv|^2} \leq \cde \, , \quad |v-\bv|^2\leq C Q(v|\bv) 
    \, .
\end{align*}
Hence, we get \eqref{truncation estimate vs}.

\section{Proof of Lemma \ref{lemma technical}}\label{section proof of lemma technical}

\noindent
\textbf{Estimate for \eqref{lemma technical 3}:} 
By \eqref{relative evaluation Q and ||}, we have
\begin{equation}\label{lemma tech pre 3}
    \begin{aligned}
        &\int |(a_\be)_x| |v-\lv| \dx + \int |(a_\be)_x| \bigl|Q(v|\bv)-Q(\lv|\bv)\bigr| \dx \\
        &=\int_{\{|\pd|>\de_1\}} |(a_\be)_x|\Bigl[ |v-\lv| + \bigl|Q(v|\bv)-Q(\lv|\bv)\bigr|\Bigr]\dx\, \\
        &\leq \cde \inb |(a_\be)_x| \bigl(Q(v|\bv)-Q(\lv|\bv)\bigr) \dx \\
        &\leq \cde \bigl(G_{2\be}(U) - G_{2\be}(\lu) \bigr) \, .
    \end{aligned}
\end{equation}

\noindent
\textbf{Estimate for \eqref{lemma technical 1}:} 
Recall \eqref{def of Omega}. We have
\begin{align*}
    L_1:= \int_\om |(a_\be)_x|\, |p(v)-p(\lv)|^2 \dx = \int_{\{\pd\leq-\de_1\}} |(a_\be)_x|\, |p(v)-p(\lv_b)|^2 \dx \, .
\end{align*}
By \eqref{truncation estimate vb} and Lemma \ref{lemma technical preliminary}, we get
\begin{equation}\label{lemma technical 11}
    \begin{aligned}
        L_1 &\leq C \frac{\md_\be}{\nu} \int_{\ji} |(a_\be)_x|\, |x-\xoi| \mathds{1}_{\{\pd\leq -\de_1\}}\dx + C \frac{\md}{\nu} \int_{\ji^c} |(a_\be)_x|\, |x-\xoi| \dx \\
        &\leq C \de_*^{1/2} \md_\be + C_n \nu^n \md \, .
    \end{aligned}
\end{equation}
On the other hand, we have that for any $y\in \mathbb{R}$,
\begin{align}\label{lemma technical trick exponent}
    \bigl(|y| - \de_1\bigr)_+ \leq \bigl(|y| - \frac{\de_1}{2} \bigr)_+ \mathds{1}_{\{|y|\geq \de_1\}} \leq \bigl(|y| - \frac{\de_1}{2} \bigr)_+ \frac{\bigl(|y| - \frac{\de_1}{2} \bigr)_+}{\frac{\de_1}{2}} \leq \frac{2}{\de_1} \bigl(|y| - \frac{\de_1}{2} \bigr)_+^2 \, .
\end{align}
Hence, we get
\begin{equation}\label{lemma technical 12}
    \begin{aligned}
        \int_\om |(a_\be)_x|\, |p(v)-p(\lv)| \dx &\leq \frac{2}{\de_1}\int_\om |(a_\be)_x|\, |p(v)-p(\lv_{\de_1/2})|^2 \dx \\
        &\leq \frac{2}{\de_1}\int_{ \{\pd \leq -\de_1/2\} } |(a_\be)_x|\, |p(v)-p(\lv_{\de_1/2})|^2 \dx\, ,
    \end{aligned}
\end{equation}
where $\lv_{\de_1/2}$ is defined in \eqref{def of truncation general}. We can get a similar estimate \eqref{truncation estimate vs} for $v_{\de_1/2}$ as for $v_{\de_1}$, so \eqref{lemma technical 12} and \eqref{lemma technical 11} imply 
\begin{align}\label{lemma technical 13}
    \int_\om |(a_\be)_x|\, |p(v)-p(\lv)| \dx \leq \cde\de_*^{1/2} \md_\be + \cde_n \nu^n \md \, .
\end{align}
For the second inequality, we observe that
\begin{align*}
    L_2&:=\int_\om |(a_\be)_x|\, \Bigl| |p(v)-p(\bv)|^2 - |p(\lv)-p(\bv)|^2 \Bigr| \dx \\
    &\leq \int_\om |(a_\be)_x|\, |p(v)-p(\lv)| \Bigl( |p(v)-p(\lv)| + 2\de_1 \Bigr) \dx \, .
\end{align*}
By \eqref{lemma technical 11} and \eqref{lemma technical 13}, we get
\begin{align*}
    L_2 \leq \cde\de_*^{1/2} \md_\be + \cde_n \nu^n \md \, .
\end{align*}
\noindent
\textbf{Estimate for \eqref{lemma technical 2}:} We have 
\begin{align*}
    L_1 := \int |(a_\be)_x| v^{\ga-\al} |p(v)-p(\lv)|^2 \dx = L_{11} + L_{21} \, , 
\end{align*}
where
\begin{align*}
    L_{11} &= \int_{\{\pd<-\de_1\}} |(a_\be)_x| v^{\ga-\al} |p(v)-p(\lv_b)|^2 \dx \, , \\
    L_{12} &= \int_{\{\pd>\de_1\}} |(a_\be)_x| v^{\ga-\al} |p(v)-p(\lv_s)|^2 \dx \, .
\end{align*}
Recall that $0\leq \ga-\al\leq 1$.
On $\{\pd<-\de_1\}$, we can write $v = c\lv$ for some $c>1$ and get
\begin{align*}
    \frac{v^{\ga-\al} |p(v)-p(\lv_b)|^2 }{|v-\lv|} = \frac{v^{\ga-\al} |p(v)-p(\lv)|^2 }{|v-\lv|} = \lv^{-\ga-\al-1} \frac{c^{\ga-\al}(c^{-\ga}-1)^2}{c-1} \leq C \, .
\end{align*}
Together with \eqref{lemma tech pre 3}, we get
\begin{equation}\label{lemma technical 21}
    \begin{aligned}
        L_{11} \leq \cde \bigl(G_{2\be}(U) - G_{2\be}(\lu) \bigr) \, .
    \end{aligned}
\end{equation}
By \eqref{truncation estimate vs} and Lemma \ref{lemma technical preliminary}, we get
\begin{equation}\label{lemma technical 22}
    \begin{aligned}
        L_{12} &\leq \cde \frac{\md_\be+\tg_\be}{\nu} \int_{\ji} |(a_\be)_x|\, |x-\xoi| \mathds{1}_{\{\pd>\de_1\}}\dx \\
        &\qquad + \cde \frac{\md+\tg}{\nu} \int_{\ji^c} |(a_\be)_x|\, |x-\xoi| \dx \\
        &\leq \cde \de_*^{1/2} (\md_\be+\tg_\be) + C_n \nu^n (\md+\tg) \, .
    \end{aligned}
\end{equation}
From \eqref{lemma technical 21} and \eqref{lemma technical 22}, we obtain 
\begin{equation}\label{lemma technical 23}
    L_1 \leq \cde \bigl(G_{2\be}(U) - G_{2\be}(\lu) \bigr) + \cde \de_*^{1/2} (\md_\be+\tg_\be) + C_n \nu^n (\md+\tg) \, .
\end{equation}
For the second inequality, we have that
\begin{align*}
    L_2 &:=  \int |(a_\be)_x|\, v^{\ga-\al} \Bigl| |p(v)-p(\bv)|^2 - |p(\lv)-p(\bv)|^2 \Bigr| \dx \\
    & \qquad + \int |(a_\be)_x| |v^{\ga-\al} - \lv^{\ga-\al} | \bigl|p(\lv)-p(\bv)\bigr|^2 \dx \\
    &=: L_{21} + L_{22} \, .
\end{align*}
We observe that
\begin{align*}
    L_{21} \leq \int |(a_\be)_x|\, v^{\ga-\al} |p(v)-p(\lv)| \Bigl( |p(v)-p(\lv)| + 2\de_1 \Bigr) \dx \, .
\end{align*}
By \eqref{lemma technical trick exponent}, we have 
\begin{align}\label{lemma technical 24}
    \int |(a_\be)_x|\, v^{\ga-\al} |p(v)-p(\lv)| \dx \leq \frac{2}{\de_1}\int |(a_\be)_x|\, v^{\ga-\al} |p(v)-p(\lv_{\de_1/2})|^2 \dx \, ,
\end{align}
where $\lv_{\de_1/2}$ is defined in \eqref{def of truncation general}. We can get similar estimates \eqref{truncation estimate vb} and \eqref{truncation estimate vs} for $v_{\de_1/2}$ as for $v_{\de_1}$, so \eqref{lemma technical 24} and \eqref{lemma technical 23} imply 
\begin{align}\label{lemma technical 2 L21}
    L_{21} \leq \cde \bigl(G_{2\be}(U) - G_{2\be}(\lu) \bigr) + \cde \de_*^{1/2} (\md_\be+\tg_\be) + C_n \nu^n (\md+\tg) \, .
\end{align}
On $\{|\pd|>\de_1\}$, we can write $v=c\lv$ for some $0<c\neq 1$ and get
\begin{align*}
    \frac{|v^{\ga-\al}-\lv^{\ga-\al}|}{|v-\lv|} = \lv^{\ga-\al-1}\frac{c^{\ga-\al}-1}{c-1} \leq C \, .
\end{align*}
Together with \eqref{lemma tech pre 3}, we get
\begin{equation}\label{lemma technical 2 L22}
    \begin{aligned}
        L_{22} \leq \cde \bigl(G_{2\be}(U) - G_{2\be}(\lu) \bigr) \, .
    \end{aligned}
\end{equation}
By \eqref{lemma technical 2 L21} and \eqref{lemma technical 2 L22}, we have
\begin{align*}
    L_2 \leq \cde \bigl(G_{2\be}(U) - G_{2\be}(\lu) \bigr) + \cde \de_*^{1/2} (\md_\be+\tg_\be) + C_n \nu^n (\md+\tg) \, .
\end{align*}

\noindent
\textbf{Estimate for \eqref{lemma technical 4}:} By \eqref{lemma tech pre 3}, we have
\begin{align*}
    L &:= \int |(a_\be)_x| \bigl|p(v|\bv)-p(\lv|\bv)\bigr| \dx  \\
    &= \int |(a_\be)_x| |p(v)-p(\lv)| \dx + \int |(a_\be)_x| |v-\lv| \dx \\
    &= L_{1} + L_{2} + L_{3} + \cde \bigl(G_{2\be}(U) - G_{2\be}(\lu) \bigr) \, ,
\end{align*}
where 
\begin{align*}
    L_{1} &= \int_{\om} |(a_\be)_x| |p(v)-p(\lv)| \dx \, , \\
    L_{2} &= \int_{\{\de_1< \pd\leq 2\de_1\}} |(a_\be)_x| |p(v)-p(\lv)| \dx \, , \\
    L_{3} &= \int_{\{\pd > 2\de_1\}} |(a_\be)_x| |p(v)-p(\lv_s)| \dx \, .
\end{align*}
By \eqref{lemma technical 1}, we have
\begin{align}\label{lemma technical 4 L11}
    L_{1} \leq \cde \de_*^{1/2} \md_\be + \cde_n \nu^n \md \, .
\end{align}
By \eqref{lemma technical 3}, we get
\begin{align}\label{lemma technical 4 L12}
    L_{2} \leq \cde \int |(a_\be)_x||v-\lv|\dx \leq \cde \bigl(G_{2\be}(U) - G_{2\be}(\lu) \bigr) \, .
\end{align}
On $\{\pd> 2\de_1\}$, we can write $v=c\lv_s$ for some $c<1-C\de_1$ and get
\begin{align*}
    \frac{|p(v)-p(\lv_s)|}{v^{\ga-\al}\bigl|p(v)-p(\lv_s)\bigr|^2} = \lv^{\al}\frac{c^{-\ga}-1}{c^{\ga-\al}(c^{-\ga}-1)^2} = \frac{\lv^{\al}}{c^{-\al}(1-c^{\ga})} \leq \cde \, .
\end{align*}
Together with \eqref{truncation estimate vs} and Lemma \ref{lemma technical preliminary}, we get
\begin{equation}\label{lemma technical 4 L13}
    \begin{aligned}
        L_{3} &\leq \cde \int_{\{\pd > 2\de_1\}} |(a_\be)_x| v^{\ga-\al}\bigl|p(v)-p(\lv_s)\bigr|^2 \dx \\
        &\leq \cde \frac{\md_\be+\tg_\be}{\nu} \int_{\ji} |(a_\be)_x|\, |x-\xoi|\mathds{1}_{\{\pd>2\de_1\}}\dx \\
        &\qquad + \cde \frac{\md+\tg}{\nu} \int_{\ji^c} |(a_\be)_x|\, |x-\xoi|\dx \\
        &\leq \cde \de_*^{1/2} (\md_\be+\tg_\be) + C_n \nu^n (\md+\tg) \, .
    \end{aligned}
\end{equation}
By \eqref{lemma technical 4 L11}, \eqref{lemma technical 4 L12}, and \eqref{lemma technical 4 L13}, we get
\begin{align*}
    L \leq \cde \bigl(G_{2\be}(U) - G_{2\be}(\lu) \bigr) + \cde \sqrt{\de_*} \md_\be + \cde_n \nu^n \md \, .
\end{align*}

\noindent
\textbf{Estimate for \eqref{lemma technical 5}:} We have
\begin{align*}
    L&:=\int_{\om^c} |(a_\be)_x|\, \bigl|p(v)-p(\bv)\bigr|\, |h-\bh| \dx \\
    &\leq \int_{\{\pd>\de_1\}} |(a_\be)_x|\, \bigl|p(v)-p(\lv_s)\bigr|\, |h-\bh| \dx \\
    &\qquad + C\de_1 \int_{\{\pd>\de_1\}}|(a_\be)_x|\, |h-\bh| \dx \\
    &\leq \int_{\{\pd>2\de_1\}} |(a_\be)_x|\, \bigl|p(v)-p(\lv_s)\bigr|\, |h-\bh| \dx \\
    &\qquad + C\de_1 \int_{\{\pd>\de_1\}}|(a_\be)_x|\, |h-\bh| \dx \\
    &=: L_{1} + L_{2} \, .
\end{align*}
Let 
\begin{align*}
    1\leq q = \frac{2\ga}{\ga+\al} < 2.
\end{align*}
On $\{\pd> 2\de_1\}$, we can write $v=c\lv_s$ for some $c<1-C\de_1$ and get
\begin{align}\label{lemma technical trick exponent 2}
    \frac{|p(v)-p(\lv_s)|}{\Bigl|v^{\ga-\al}\bigl|p(v)-p(\lv_s)\bigr|^2 \Bigr| ^{q/2}} = \frac{c^{-\ga}-1}{c^\frac{\ga(\ga-\al)}{\ga+\al}(c^{-\ga}-1)^\frac{2\ga}{\ga+\al}} = \frac{1}{(1-c^{\ga})^{\frac{\ga-\al}{\ga+\al}}} \leq \cde \, .
\end{align}
Together with \eqref{truncation estimate vs} and Lemma \ref{lemma technical preliminary}, we get
\begin{align*}
    L_{1} \leq C L_{11}^{\frac{1}{2}} G_{1\be}^-(U)^{\frac{1}{2}} \, ,
\end{align*}
where
\begin{equation}\label{lemma technical 5 L111}
    \begin{aligned}
        L_{11} &= \int_{\{\pd>2\de_1\}} |(a_\be)_x|\, \bigl|p(v)-p(\lv_s)\bigr|^2 \dx \\
        &\leq \cde \int_{\{\pd > 2\de_1\}} |(a_\be)_x| \Bigl|v^{\ga-\al}\bigl|p(v)-p(\lv_s)\bigr|^2 \Bigr|^q \dx \\
        &\leq \cde \frac{(\md_\be+\tg_\be)^q}{\nu^q} \int_{\ji} |(a_\be)_x|\, |x-\xoi|^q \mathds{1}_{\{\pd>2\de_1\}}\dx \\
        &\qquad + \frac{(\md+\tg)^q}{\nu^q} \int_{\ji^c} |(a_\be)_x|\, |x-\xoi|^q \dx \\
        &\leq \cde |\s_\be|^{1-q} \de_*^{p} (\md_\be+\tg_\be)^q  + C_n \nu^n (\md+\tg)^{q} \, ,
    \end{aligned}
\end{equation}
for some $0<p<1$ is chosen later.
By Young's inequality and Lemma \ref{lemma smallness of localized relative entropy 1}, we have
\begin{align*}
    A_1&:= \cde \Bigl(|\s_\be|^{1-q} \de_*^{p} (\md_\be+\tg_\be)^{q} \Bigl)^{1/2}G_{1\be}^-(U)^{1/2} \\
    &\leq \cde \de_*^{q/4} (\md_\be+\tg_\be)^{q/2} |\s_\be|^{\frac{1-q}{2}} \de_*^{\frac{p}{2}-\frac{q}{4}}  G_{1\be}^-(U)^{1/2} \\
    &\leq \de_*^{1/2}  (\md_\be+\tg_\be) + \cde |\s_\be|^{\frac{1-q}{2-q}} \de_*^{\frac{2p-q}{2(2-q)}} G_{1\be}^-(U)^{\frac{1}{2-q}} \\
    &\leq \de_*^{1/2}  (\md_\be+\tg_\be) + \cde \de_*^{\frac{2p+q-2}{2(2-q)}}   G_{1\be}^-(U) \, ,
\end{align*}
and
\begin{align*}
    &C_n \nu^n (\md+\tg)^{q/2} (|\s_\be|\de_*)^{1/2} \\
    &\leq C_n \nu^n (\md+\tg) + C_n \nu^n \, .
\end{align*}
We let $p = 1-\frac{q}{4}$ and get 
\begin{align*}
    \frac{2p+q-2}{2(2-q)} = \frac{q}{4(2-q)} \geq \frac{1}{4} \, .
\end{align*}
In all, we have
\begin{align}\label{lemma technical 5 L11}
    L_{1}\leq  \de_*^{1/2}  (\md_\be+\tg_\be) +  C_n \nu^n (\md+\tg) + \cde \de_*^{1/4} G_{1\be}^-(U) + C_n \nu^n \, .
\end{align}
On the other hand, H\"older's inequality gives
\begin{align*}
    L_{2} \leq \de_1 G_{1\be}^-(U) + \cde L_{121} \, , 
\end{align*}
where 
\begin{align*}
    L_{21} &= \int_{\{\pd>\de_1\}} |(a_\be)_x| \dx \, .
\end{align*}
On $\{\pd>\de_1\}$, we have that
\begin{align}\label{lemma technical trick exponent 3}
    |p(v)-p(\lv_{\de_1/2})| = \bigl(p(v)-p(\bv) - \frac{\de_1}{2}\bigr)_+ \geq 
    \frac{\de_1}{2} \, ,
\end{align}
where $\lv_{\de_1/2}$ is defined in \eqref{def of truncation general}. We can get similar estimates \eqref{lemma technical trick exponent 2} and \eqref{truncation estimate vs} for $\lv_{\de_1/2}$ as for $\lv_{\de_1}$, so \eqref{lemma technical 4 L13} implies
\begin{equation}\label{lemma technical 5 L21}
    \begin{aligned}
        L_{21} 
        &\leq \cde \int_{\{\pd>\de_1\}} |(a_\be)_x|\, \bigl|p(v)-p(\lv_{\de_1/2})\bigr|^{2/q} \dx \\
        &\leq \cde \int_{\{\pd>\de_1\}} |(a_\be)_x| v^{\ga-\al}\bigl|p(v)-p(\lv_{\de_1/2})\bigr|^2  \dx \\
        &\leq \cde \de_*^{1/2} (\md_\be+\tg_\be) + C_n \nu^n (\md+\tg) \, .
    \end{aligned}
\end{equation}
We get 
\begin{align}\label{lemma technical 5 L12}
    L_{2} \leq \cde \de_*^{1/2} (\md_\be+\tg_\be) + C_n \nu^n (\md+\tg) + \de_1 G_{1\be}^-(U) \, .
\end{align}
By \eqref{lemma technical 5 L11} and \eqref{lemma technical 5 L12}, we get
\begin{align*}
    L\leq \de_*^{1/2}  (\md_\be+\tg_\be) +  C_n \nu^n (\md+\tg) + (\de_1 + \cde \de_*^{1/4}) G_{1\be}^-(U) + C_n \nu^n \, .
\end{align*}

\noindent
\textbf{Estimate for \eqref{lemma technical 6}:} 
We have 
\begin{align*}
    L &:= \int_{\om^c} |(a_\be)_x| \Bigl(Q(\lv|\bv) + |\lv-\bv|\Bigr) \dx \\
    &\leq C\int_{\{\pd>\de_1\}} |(a_\be)_x| \dx \, .
\end{align*}
By \eqref{lemma technical 5 L21}, we get the first inequality.

On the other hand, \eqref{lemma technical trick exponent 3} implies
\begin{align*}
    L \leq  \cde \int_{\{\pd>\de_1\}} |(a_\be)_x|\, \bigl|p(v)-p(\lv_{\de_1/2})\bigr|^{1/q} \dx \, ,
\end{align*}
where $\lv_{\de_1/2}$ is defined in \eqref{def of truncation general}. We can get similar estimates \eqref{lemma technical trick exponent 2} and \eqref{truncation estimate vs} for $\lv_{\de_1/2}$ as for $\lv_{\de_1}$, so Lemma \ref{lemma technical preliminary} implies
\begin{align*}
    L &\leq \cde \int_{\{\pd>\de_1\}} |(a_\be)_x|\, v^{\frac{\ga-\al}{2}}\bigl|p(v)-p(\lv_{\de_1/2})\bigr| \dx \\
    &\leq \cde  \sqrt{\frac{\md_\be+\tg_\be}{\nu}} \int_{\ji} |(a_\be)_x|\, |x-\xoi|^{1/2}\mathds{1}_{\{\pd>\de_1\}}\dx \\
    &\qquad + \cde \sqrt{\frac{\md+\tg}{\nu}} \int_{\ji^c} |(a_\be)_x|\, |x-\xoi|^{1/2} \dx \\
    &\leq \cde |\s_\be|^{1/2} \de_*^{3/4} (\md_\be+\tg_\be)^{1/2} + C_n \nu^n (\md+\tg)^{1/2} \, .
\end{align*}

\section{Proof of Proposition \ref{proposition shock estimate}}\label{section proof of prop shock estimate}
By \eqref{def of R(U), B, G}, \eqref{X_i Y_i estimates}, and \eqref{def of B and G}, we have
\begin{align*}
    &\mr(U) + \de_* \int |(\sv_\be)_x| p(v|\bv) \dx \\
    &= \sum_{\be\in \ms} \dot{X}_\be Y_\be(U) + B_{\de_1}(U) - G_{\de_1}(U) - (1-C\de_*) \md(U) \\
    &\leq \sum_{\be \in \ms} G_{Y_\be}(U) + \sum_{\be \in \ms} \bigl( |B_{1\be}(U)| + |B_{2\be}^+(U)| + |B_{2\be}^-(U) | \bigl) + |B_3(U)| \\
    &\qquad - \sum_{\be \in \ms} \bigl(G_{1\be}^+(U) + G_{1\be}^-(U) + G_{2\be}(U) \bigr) - (1-C\de_*) \md(U) \, ,
\end{align*}
where 
\begin{align*}
    G_{Y_\be}(U) &= -\frac{1}{\s_\be^4} |Y_\be(U)|^2 \mathds{1}_{\{0\leq (-1)^{k_\be-1}Y_\be(U)\leq \s_\be^2\}} - C_X\frac{|\las_\be|}{\s_\be^2} |Y_\be(U)|^2 \mathds{1}_{\{-\s_\be^2\leq (-1)^{k_\be-1}Y_\be(U)\leq 0\}} 
    \\&\qquad - C_X{|\las_\be|} |Y_\be(U)| \mathds{1}_{\{(-1)^{k_\be-1}Y_\be(U)\leq -\s_\be^2\}} \, .
\end{align*}
Note that $G_{Y_\be}\leq 0$.

To separate the shocks using the property \eqref{prop difference separation} stated in Proposition \ref{prop difference}, we define
\begin{align}\label{def of I_*}
    I_* = \{\be \in \ms \, | \, \de_*^{1/2}\md_\be + \nu|\s_\be|\md(U) \leq |\s_\be|\de_* \} \, .
\end{align}
We get
\begin{align}
    \mr(U) + \de_* \int |(\sv_\be)_x| p(v|\bv) \dx \leq \mr_*(U) + \mr_*^c(U) \, ,
\end{align}
where
\begin{align*}
    \mr_*(U) &\leq \sum_{\be \in I_*} G_{Y_\be}(U) + \sum_{\be \in I_*} \bigl( |B_{1\be}(U)| + |B_{2\be}^+(U)| + |B_{2\be}^-(U) | \bigl) + |B_3(U)| \\
    &\qquad - \sum_{\be \in I_*} \bigl(G_{1\be}^+(U) + G_{1\be}^-(U) + G_{2\be}(U) \bigr) - (1-C\de_*) \md(U) \, , \\
    \mr_*^c(U) &\leq \sum_{\be \in I_*^c} \bigl( |B_{1\be}(U)| + |B_{2\be}^+(U)| + |B_{2\be}^-(U) | \bigl) + |B_3(U)| \\
    &\qquad - \sum_{\be \in I_*^c} \bigl(G_{1\be}^+(U) + G_{1\be}^-(U) + G_{2\be}(U) \bigr) \, .
\end{align*}

\noindent
\textbf{Estimate for $\mr_*^c$.} We have
\begin{align*}
    \mr_*^c(U) \leq I_B + \bar{I}_B - \sum_{\be \in I_*^c} \bigl(G_{1\be}^+(U) + G_{1\be}^-(U) + G_{2\be}(U) \bigr) - (1-C\de_*) \md(U) \, ,
\end{align*}
where
\begin{align*}
    I_B &= \sum_{\be \in I_*^c} \bigl( |B_{1\be}(U) - B_{1\be}(\lu)| + |B_{2\be}^+(U) - B_{2\be}^+(\lu)| + |B_{2\be}^-(U)| \bigr) + |B_3(U)| \, , \\
    \bar{I}_B &= \sum_{\be \in I_*^c} \bigl( B_{1\be}(\lu) + B_{2\be}^+(\lu) \bigr) \, .
\end{align*}
Recall the definitions of $\md_\be$ in \eqref{def of Di}, $\md$ in \eqref{def of D}, and $\tg_\be,\tg$ in \eqref{def of tg and tg_i}. We have
\begin{align}\label{relation D, D_i, tg, tg_i}
    \sum_{\be \in \ms} \md_\be \leq C \md\, , \quad \sum_{\be \in \ms} \tg_\be \leq C \tg \, .
\end{align}
By \eqref{prop difference eq 1} and \eqref{prop difference eq 2} in Proposition \ref{prop difference} and \eqref{def of I_*}, we have
\begin{align*}
    I_B &\leq \mathbf{C} \de_* \sum_{i \in I_*^c} \bigl( G_{2\be}(U) - G_{2\be}(\lu) \bigr) + \mathbf{C}\de_*^4 \sum_{i \in \ms} G_{2\be}(U) \\
    &\qquad + C \de_*^{1/2}(\md + \tg) + (\de_1 + \mathbf{C}\de_*^{1/4}) \sum_{i \in I_*^c} G_{1\be}^-(U) + C\nu^2 \, , \\
    \bar{I}_B &\leq C \de_*^{1/2} \md \, ,
\end{align*}
for some $\mathbf{C},C>0$ depending only on $U_*, \de_1$.

We take $\de_*$ small enough such that
\begin{align}\label{relation delta 2}
    \mathbf{C}(\de_*+\de_*^4) < \frac{1}{2} \, , \quad \de_1 + \mathbf{C} \de_*^{1/4} < 1 \, .
\end{align}
Then \eqref{relative evaluation Q} gives
\begin{align}\label{estimate R_*^c}
    \mr_*^c(U) \leq C \de_*^4 \sum_{\be \in I_*} G_{2\be}(U) - \frac{1}{2} \sum_{i \in I_*^c} G_{2\be}(U) + C \de_*^{1/2} \tg + C \de_*^{1/2} \md + C\nu^2 \, .
\end{align}

\noindent
\textbf{Estimate for $\mr_*$.}
We have
\begin{align*}
    -C_X{|\las_\be|} |Y_\be(U)| \mathds{1}_{\{(-1)^{k_\be-1}Y_\be(U)\leq -\s_\be^2\}} &\leq - C_X{|\las_\be|} |Y_\be(U)| \mathds{1}_{\{-|\s_\be|\de_*\leq (-1)^{k_\be-1}Y_\be(U)\leq -\s_\be^2\}} \\
    &\leq - C_X\frac{|\las_\be|}{|\s_\be|\de_*} |Y_\be(U)|^2 \, .
\end{align*}
By Lemma \ref{lemma L+kQ}, we can take $\de_*$ small enough such that
\begin{align}\label{relation delta 3}
    \frac{6}{C_X|\las_\be|}\de_* \leq \frac{C}{C_X}\de_* \leq \de_1 \, ,
\end{align}
and get
\begin{align*}
    G_{Y_\be}(U)\leq -\frac{6}{|\s_\be|\de_1}|Y_\be(U)|^2 \, .
\end{align*}
Recall \eqref{def of Yig, Yib, Yil, Yis}. We have
\begin{align*}
    Y_\be(U) = \my_\be^g(\lv) + \bigl( Y_\be^g(U)-Y_\be^g(\lu) \bigr) + \bigl( Y_\be^g(\lu) - \my_\be^g(\lv) \bigr) + Y_\be^b(U) + Y_\be^l(U) + Y_\be^s(U) \, ,
\end{align*}
which implies
\begin{align*}
    |\my_\be^g(\lv)|^2 &\leq 6 \Bigl(|Y_\be(U)|^2 + \bigl| Y_\be^g(U)-Y_\be^g(\lu) \bigr|^2 + \bigl| Y_\be^g(\lu) - \my_\be^g(\lv) \bigr| \\
    &\qquad+ |Y_\be^b(U)|^2 + |Y_\be^l(U)|^2 +  |Y_\be^s(U)|^2 \Bigl) \, .
\end{align*}
Hence, we get
\begin{align*}
    G_{Y_\be}(U) &\leq - \frac{|\my_\be^g(\lv)|^2}{|\s_\be|\de_1} + \frac{6}{|\s_\be|\de_1} \Bigl(\bigl| Y_\be^g(U)-Y_\be^g(\lu) \bigr|^2 + \bigl| Y_\be^g(\lu) - \my_\be^g(\lv) \bigr|^2 \\
    &\qquad+ |Y_\be^b(U)|^2 + |Y_\be^l(U)|^2 +  |Y_\be^s(U)|^2 \Bigl) \, .
\end{align*}
Recall the definition of $\md$ in \eqref{def of D}. We have
\begin{align*}
    \md(U) &= \nu\int a \mu_1(v) v^{\ga} \Bigr|\bigl(p(v)-p(\bv)\bigr)_x \Bigr|^2 \dx \\
    &\geq \nu\int_{\{|\pd|\leq\de_1\}} a \mu_1(v) v^{\ga} \Bigr|\bigl(p(v)-p(\bv)\bigr)_x \Bigr|^2 \dx \\
    &= \md(\lu) \, .
\end{align*}
Together \eqref{prop difference eq 3} in Proposition \ref{prop difference}, we can take $\de_* > 0$ small enough such that
\begin{align}\label{relation delta 4}
    \frac{1-\frac{\de_1}{2}}{1+\de_*} \geq 1-\de_1 \, , \quad \textup{ that is } \quad \de_* \leq \frac{\de_1}{2(1-\de_1)} \, ,
\end{align}
and get
\begin{align*}
    -(1-\cde \de_*) \md(U) &\leq -(1-\de_1) \sum_{\be \in \ms} \widetilde{\md}_\be(\lv) + \frac{C\nu^{1-\cs(2\co+\cpr)}\ln \rn^{-1}}{\de_*}  \int a \eta(U|\bu) \dx \\
    &\quad - (\frac{\de_1}{2}-C\de_*) \md(U)\, .
\end{align*}
Also, we have 
\begin{align*}
    -G_{2\be}(U) &= - (1-\de_1\de_*) \mg_{2\be}(\lv) - (1-\de_1\de_*) \bigl(G_{2\be}(\lu) - \mg_{2\be}(\lv) \bigr) \\
    &\qquad - \de_1\de_* G_{2\be}(\lu) - \bigl(G_{2\be}(U) - G_{2\be}(\lu)\bigr) \, .
\end{align*}
In all, we get
\begin{align*}
    \mr_*(U) &\leq \sum_{\be \in I_*} \mr_{\de_*}^\be(\lv) + J_Y + J_B - \sum_{\be \in I_*} \Bigl( G_{1\be}^+(U) + G_{1\be}^-(U) + \bigl(G_{2\be}(U) - G_{2\be}(\bu)\bigr)\Bigr)
    \\
    &\qquad - \de_1\de_*\sum_{\be \in I_*} G_{2\be}(\lu) - (\frac{\de_1}{2} - \cde \de_*) \md(U) \, ,
\end{align*}
where 
\begin{align*}
    \mr_{\de_*}(\lv) &= (1+\bc\de_1) |\mi_{1\be}(\lv)| + |\mi_{2\be}(\lv)| - (1-\de_1\de_*)\mg_{2\be}(\lv) \\ 
    &\qquad - (1-\de_1)\widetilde{\md}_\be(\lv) - \frac{|\my_\be^g(\lv)|^2}{|\s_\be|\de_1} \, , \\
    J_Y &= \sum_{\be \in I_*} \frac{6}{|\s_\be|\de_1} \Bigl(\bigl| Y_\be^g(U)-Y_\be^g(\lu) \bigr|^2 + \bigl| Y_\be^g(\lu) - \my_\be^g(\lv) \bigr|^2 \\
    &\qquad+ |Y_\be^b(U)|^2 + |Y_\be^l(U)|^2 +  |Y_\be^s(U)|^2 \Bigl) \, , \\
    J_B &= \sum_{\be \in I_*} \Bigl(\bigl|B_{1\be}(U) - B_{1\be}(\lu)\bigr| + \bigl|B_{1\be}(\lu) - (1+\bc\de_*)\mi_{1\be}(\lv)\bigr| \\
    &\qquad + \bigl|B_{2\be}^+(U)-B_{2\be}^+(\lu)\bigr| + \bigl|B_{2\be}^+(\lu)-\mi_{2\be}(\lv)\bigr| + |B_{2\be}^-(U)|\Bigr) + |B_3(U)| \\
    &\qquad - (1-\de_1\de_*)\sum_{\be \in I_*}\bigl(G_{2\be}(\lu) - \mg_{2\be}(\lv) \bigr) + \frac{C\nu^{1-\cs(2\co+\cpr)}\ln \rn^{-1}}{\de_*} \int a \eta(U|\bu) \dx \, .
\end{align*}
Lemma \ref{lemma inside truncation} implies
\begin{align*}
    \mr_{\de_*}(\lv) \leq 0 \, .
\end{align*}
By \eqref{prop difference eq 3} and \eqref{prop difference eq 4} in Proposition \ref{prop difference} and \eqref{relation D, D_i, tg, tg_i}, we have
\begin{align*}
    J_Y &\leq \cde \de_* \sum_{\be \in I_*} \Bigl( \bigl( G_{2\be}(U) - G_{2\be}(\lu) \bigr) + \de_*^{-1/4}G_{1\be}^+(U) + G_{1\be}^-(U) + \de_*^{1/4}G_{2\be}(\lu) \Bigr) \\
    &\qquad + C \de_*^{3/2} (\md + \tg) + C \nu^4 \, .
\end{align*}
By \eqref{prop difference eq 1} and \eqref{prop difference eq 3} in Proposition \ref{prop difference} and \eqref{relation D, D_i, tg, tg_i}, we get
\begin{align*}
    J_B &\leq C \de_* \sum_{i \in I_*} \bigl(G_{2\be}(U) - G_{2\be}(\lu)\bigr) + C\de_*^4 \sum_{i \in \ms} G_{2\be}(U) \\
    &\qquad + C \de_*^{1/2}(\md + \tg) + (\de_1 + C\de_*^{1/4}) \sum_{i \in I_*} G_{1\be}^-(U) + C\nu^2 + \frac{C\nu^{1-\cs(2\co+\cpr)}\ln \rn^{-1}}{\de_*}  \int a \eta(U|\bu) \dx \, .
\end{align*}
In all, we get
\begin{equation*}
    \begin{aligned}
        \mr_*(U) &\leq -(1-C_G \de_*) \sum_{\be \in I_*} \bigl( G_{2\be}(U) - G_{2\be}(\lu) \bigr) - \de_*(\de_1 - C_G \de_*^{1/4}) \sum_{\be \in I_*} G_{2\be}(\lu) \\
        &\qquad + C_G \de_*^4 \sum_{\be \in I_*^c} G_{2\be}(U) - (\frac{\de_1}{2} - \cde \de_*^{1/2}) \md(U) + C \de_*^{1/2}\tg \\
        &\qquad + C \nu^2 + \frac{C\nu^{1-\cs(2\co+\cpr)}\ln \rn^{-1}}{\de_*}  \int a \eta(U|\bu) \dx \\
        &\qquad - (1-C_G\de_*^{3/4})\sum_{\be \in I_*} G_{1\be}^+(U) - \bigl(1-(\de_1+C_G\de_*^{1/4})\bigr) \sum_{\be \in I_*} G_{1\be}^-(U) \, ,
    \end{aligned}
\end{equation*}
for some $C_G,C>0$ depending on $U_*,\de_1$.

By taking $\de_*$ small enough such that
\begin{align}\label{relation delta 6}
    0<\de_*(\de_1-C_G\de_*^{1/4}) < 1 - C_G \de_* \, , \quad
    C_G\de_*^{3/4} < 1 \, , \quad \de_1+C_G\de_*^{1/4} < 1 \, ,
\end{align}
we get
\begin{equation}\label{estimate R_*}
    \begin{aligned}
        \mr_*(U) &\leq -\de_*(\de_1-C \de_*^{1/4}) \sum_{\be \in I_*} G_{2\be}(U) 
        + C \de_*^4 \sum_{\be \in I_*^c} G_{2\be}(U) - (\frac{\de_1}{2} - \cde \de_*^{1/2}) \md(U)\\
        &\qquad + C \de_*^{1/2}\tg + C \nu^2 + \frac{C\nu^{1-\cs(2\co+\cpr)}\ln \rn^{-1}}{\de_*}  \int a \eta(U|\bu) \dx  \, .
    \end{aligned}
\end{equation}

\noindent
\textbf{Conclusion.} 
Since there exists $C>0$, depending only on $U_*$, such that for any $v>0$,
\begin{align*}
    Q(v|\bv) \leq C p(v|\bv) \, ,
\end{align*}
we have
\begin{align*}
    \tg \leq C \de_* \sum_{\be \in \ms} G_{2\be}(U) + C \sum_{\be \in \mr} \int |(\rv_\be)_x| p(v|\bv) \dx + C \sum_{\be \in \mnp} \int |(\nv_\be)_x| Q(v|\bv) \dx \, ,
\end{align*}
Together with \eqref{estimate R_*^c} and \eqref{estimate R_*}, we get
\begin{align*}
    &\mr(U) + \de_* \int |(\sv_\be)_x| p(v|\bv) \dx \\
    &\leq - \de_*(\de_1 - C \de_*^{1/4}) \sum_{\be \in I_*} G_{2\be}(U) - (\frac{1}{2}-C \de_*^4) \sum_{\be \in I_*^c} G_{2\be}(U) - (\frac{\de_1}{2}-C\de_*^{1/2}) \md(U)\\
    &\qquad + C \de_*^{1/2} \tg + C\nu^2 + \frac{C\nu^{1-\cs(2\co+\cpr)}\ln \rn^{-1}}{\de_*}  \int a \eta(U|\bu) \dx \, , \\
    &\leq C \de_*^{1/2} \sum_{\be \in \mr} \int |(\rv_\be)_x| p(v|\bv) \dx + C \de_*^{1/2} \sum_{\be \in \mnp} \int |(\nv_\be)_x| Q(v|\bv) \dx \\
    &\qquad + C\nu^2 + \frac{C\nu^{1-\cs(2\co+\cpr)}\ln \rn^{-1}}{\de_*}  \int a \eta(U|\bu) \dx \\
    &\qquad - \de_*(\de_1 - \widetilde{C}_G \de_*^{1/4}) \sum_{\be \in I_*} G_{2\be}(U) - (\frac{1}{2}-\widetilde{C}_G \de_*^{3/2}) \sum_{\be \in I_*^c} G_{2\be}(U) - (\frac{\de_1}{2}-C_D\de_*^{1/2}) \md(U)
\end{align*}
for some constants $\widetilde{C}_G,C_D,C>0$ depending on $U_*, \de_1$.

By taking $\de_*$ small enough such that
\begin{align}\label{relation delta 7}
    \widetilde{C}_G \de_*^{1/4} < \de_1 \, , \quad \widetilde{C}_G \de_*^4 < \frac{1}{2} \, , \quad C_D\de_*^{1/2} < \frac{\de_1}{4} \, ,
\end{align}
we complete the proof.

\section{Proof of Lemma \ref{lemma L + q + omega}}\label{sec decay of weight lemma proof}
We first note that $p$ increases across $1$-shocks and $2$-rarefactions, and decreases across $2$-shocks and $1$-rarefactions. 

Suppose that, at time $t$, a passing or partition occurs. By \eqref{def of rarefaction nu} and the constructions of passing in Section~\ref{subsub passing} and partition in Section~\ref{subsub partition}, we have $\Delta g^S(t) = 0$ and
\begin{align*}
    -\Delta g^R(t) \leq 
    \begin{cases}
        0 \, , \quad \textup{if none of the outgoing waves is a non-physical wave,} \\
        C\rn \, , \quad \textup{otherwise.}
    \end{cases}
\end{align*}
Lemma \ref{lemma L+kQ} implies that
\begin{align*}
    \Delta \ml(t) \leq 0 \, .
\end{align*}
Therefore, we show the inequalities for passings and partitions.

We assume that at time $t$, an interaction occurs. We denote the strength of the incoming waves as $\s$ and $\widetilde{\s}$. The proof follows the same spirit as that of Proposition~11.1 in \cite{NStoE}. For a shock $\be$ located at $y_\be$, the main change of its profile occurs inside $[y_\be - \rn^{\cz}, y_\be + \rn^{\cz}]$. Hence, we define
\begin{equation}\label{def of approximate sol to g i}
    \om^{\be} = 
    \begin{cases}
        \int_\infty^{-\infty} \p_x \bigl(-p(\td v_\be)\bigr) \dx \, , \quad \textup{on } (y_b, y_\be - \rn^{\cz}) \, , \\
        \int^x_{y_\be + \rn^{\cz}} \p_x \bigl(-p(\td v_\be)\bigr) \dx \, , \quad \textup{on } [y_\be-\rn^{\cz},y_\be + \rn^{\cz}] \, , \\
        0 \, , \quad \textup{on } (y_\be + \rn^{\cz}, \infty) \, .
    \end{cases}
\end{equation}
We let

and approximate $g_\pm^S$ by 
\begin{align}\label{def of approximate sol to g}
    \om_\pm^S = \sum_{\be \in \ms^\pm \cap \ma^\pm}  \om^{\be} \, .
\end{align}
Lemma \ref{lemma estimate of waves}, 1) in Corollary \ref{lemma strength of waves}, Lemma \ref{lemma L+kQ}, and \eqref{rt transition S} imply that for any $n \in \mathbb{N}_*$, there exists a constant $C_n>0$, depending only on $U_*$ and $n$, such that
\begin{equation}\label{difference between g and omega}
    \begin{aligned}
        |\om_\pm^S - g_\pm^S| &\leq C \sum_{\be\in \ms^\pm \cap \ma^\pm} \int_{|x-y_\be|\geq \rn^{\cz}} \frac{\s_\be^2}{\nu} e^{-\frac{C|\s_\be||x-y_\be|}{\nu}} \dx \\
        &\leq C e^{-C \nu^{-1+\cs(1+\cz)}}  \sum_{\be \in \ms^\pm \cap \ma^\pm} |\s_\be|\\ 
        &\leq C_n \nu^n \, .
    \end{aligned}
\end{equation}
By the constructions in \eqref{placement of outgoing waves interaction i=1} and \eqref{placement of outgoing waves interaction i=2}, Lemma \ref{lemma interaction wave strength estimates} and Lemma \ref{lemma interaction wave strength estimates NP} imply that
\begin{align*}
    \om_+^S - \om_-^S \leq 
    \begin{cases}
        C|\s||\widetilde{\s}| \, , \quad \textup{if none of the outgoing waves is a non-physical wave,} \\
        C|\s||\widetilde{\s}| + C\rn \, , \quad \textup{otherwise,}
    \end{cases}
\end{align*}
and 
\begin{align*}
    -\Delta g^R \leq 
    \begin{cases}
        C|\s||\widetilde{\s}| \, , \quad \textup{if none of the outgoing waves is a non-physical wave,} \\
        C|\s||\widetilde{\s}| + C\rn \, , \quad \textup{otherwise.}
    \end{cases}
\end{align*}
On the other hand, Lemma \ref{lemma L+kQ} implies
\begin{align*}
    \Delta \ml(t) \leq - C|\s||\widetilde{\s}| \, .
\end{align*}
In all, we prove the inequalities for interactions.

\bibliographystyle{abbrv}
\bibliography{refs}

@misc{NStoE,
  author = {Chen, Geng and Kang, Moon-Jin and Vasseur, Alexis F.},
  title = {From {Navier--Stokes} to {BV} solutions of the barotropic {Euler} equations},
  year = {2024},
  howpublished = {arXiv preprint arXiv:2401.09305},
  url = {https://arxiv.org/abs/2401.09305},
  eprint = {2401.09305},
  archivePrefix = {arXiv},
  primaryClass = {math.AP},
}

@book{bressanbook,
  author = {Bressan, Alberto},
  title = {Hyperbolic systems of conservation laws},
  series = {Oxford Lecture Series in Mathematics and its Applications},
  volume = {20},
  publisher = {Oxford University Press, Oxford},
  year = {2000},
  pages = {xii+250},
  note = {The one-dimensional Cauchy problem},
  doi = {10.1093/oso/9780198507000.001.0001},
  url = {https://doi.org/10.1093/oso/9780198507000.001.0001},
  isbn = {0-19-850700-3},
  mrclass = {35-02 (35B35 35L65)},
  mrnumber = {1816648},
  mrreviewer = {Denis\ Serre},
}

@article{MR4195742,
  author = {Kang, Moon-Jin and Vasseur, Alexis F.},
  title = {Contraction property for large perturbations of shocks of the barotropic {N}avier-{S}tokes system},
  journal = {J. Eur. Math. Soc. (JEMS)},
  fjournal = {Journal of the European Mathematical Society (JEMS)},
  volume = {23},
  number = {2},
  year = {2021},
  pages = {585--638},
  doi = {10.4171/jems/1018},
  url = {https://doi.org/10.4171/jems/1018},
  issn = {1435-9855,1435-9863},
  mrclass = {76N15 (35Q30)},
  mrnumber = {4195742},
  mrreviewer = {Ming\ Mei},
}

@article{twoshocks,
  author = {Kang, Moon-Jin and Vasseur, Alexis F.},
  title = {Well-posedness of the {R}iemann problem with two shocks for the isentropic {E}uler system in a class of vanishing physical viscosity limits},
  journal = {J. Differential Equations},
  fjournal = {Journal of Differential Equations},
  volume = {338},
  year = {2022},
  pages = {128--226},
  doi = {10.1016/j.jde.2022.07.034},
  url = {https://doi.org/10.1016/j.jde.2022.07.034},
  issn = {0022-0396,1090-2732},
  mrclass = {76N15 (35B35 35Q30)},
  mrnumber = {4469142},
}

@article{invent,
  author = {Kang, Moon-Jin and Vasseur, Alexis F.},
  title = {Uniqueness and stability of entropy shocks to the isentropic {E}uler system in a class of inviscid limits from a large family of {N}avier-{S}tokes systems},
  journal = {Invent. Math.},
  fjournal = {Inventiones Mathematicae},
  volume = {224},
  number = {1},
  year = {2021},
  pages = {55--146},
  doi = {10.1007/s00222-020-01004-2},
  url = {https://doi.org/10.1007/s00222-020-01004-2},
  issn = {0020-9910,1432-1297},
  mrclass = {76N15 (35B35 35Q30)},
  mrnumber = {4228501},
  mrreviewer = {Ivan\ Dra\v zi\'c},
}

@article{compensated,
  author = {Chen, Gui-Qiang and Perepelitsa, Mikhail},
  title = {Vanishing viscosity limit of the {N}avier-{S}tokes equations to the {E}uler equations for compressible fluid flow},
  journal = {Comm. Pure Appl. Math.},
  fjournal = {Communications on Pure and Applied Mathematics},
  volume = {63},
  number = {11},
  year = {2010},
  pages = {1469--1504},
  doi = {10.1002/cpa.20332},
  url = {https://doi.org/10.1002/cpa.20332},
  issn = {0010-3640,1097-0312},
  mrclass = {35Q35 (35B25 35L65 76N10)},
  mrnumber = {2683391},
  mrreviewer = {Makram\ Hamouda},
}

@article{vanishingnondegenerate,
  author = {Bianchini, Stefano and Bressan, Alberto},
  title = {Vanishing viscosity solutions of nonlinear hyperbolic systems},
  journal = {Ann. of Math. (2)},
  fjournal = {Annals of Mathematics. Second Series},
  volume = {161},
  number = {1},
  year = {2005},
  pages = {223--342},
  doi = {10.4007/annals.2005.161.223},
  url = {https://doi.org/10.4007/annals.2005.161.223},
  issn = {0003-486X,1939-8980},
  mrclass = {35L65 (35B25 35L45)},
  mrnumber = {2150387},
  mrreviewer = {Marta\ Lewicka},
}

@article{BDentropy,
  author = {Bresch, Didier and Desjardins, Beno\^it},
  title = {Existence of global weak solutions for a 2{D} viscous shallow water equations and convergence to the quasi-geostrophic model},
  journal = {Comm. Math. Phys.},
  fjournal = {Communications in Mathematical Physics},
  volume = {238},
  number = {1-2},
  year = {2003},
  pages = {211--223},
  doi = {10.1007/s00220-003-0859-8},
  url = {https://doi.org/10.1007/s00220-003-0859-8},
  issn = {0010-3616,1432-0916},
  mrclass = {76D03 (35D05 35Q35 76U05 86A05)},
  mrnumber = {1989675},
  mrreviewer = {Nader\ Masmoudi},
}

@article{BJ,
  author = {Bresch, Didier and Jabin, Pierre-Emmanuel},
  title = {Global existence of weak solutions for compressible {N}avier-{S}tokes equations: thermodynamically unstable pressure and anisotropic viscous stress tensor},
  journal = {Ann. of Math. (2)},
  fjournal = {Annals of Mathematics. Second Series},
  volume = {188},
  number = {2},
  year = {2018},
  pages = {577--684},
  doi = {10.4007/annals.2018.188.2.4},
  url = {https://doi.org/10.4007/annals.2018.188.2.4},
  issn = {0003-486X,1939-8980},
  mrclass = {35Q35 (35D30 35Q86 42B37 54D30 92C35)},
  mrnumber = {3862947},
  mrreviewer = {Gabriela\ Planas},
}

@article{BC,
  author = {Bressan, Alberto and Colombo, Rinaldo M.},
  title = {Unique solutions of {$2\times 2$} conservation laws with large data},
  journal = {Indiana Univ. Math. J.},
  fjournal = {Indiana University Mathematics Journal},
  volume = {44},
  number = {3},
  year = {1995},
  pages = {677--725},
  doi = {10.1512/iumj.1995.44.2004},
  url = {https://doi.org/10.1512/iumj.1995.44.2004},
  issn = {0022-2518,1943-5258},
  mrclass = {35L65},
  mrnumber = {1375345},
  mrreviewer = {Alexander\ Doktor},
}

@article{BCP,
  author = {Bressan, Alberto and Crasta, Graziano and Piccoli, Benedetto},
  title = {Well-posedness of the {C}auchy problem for {$n\times n$} systems of conservation laws},
  journal = {Mem. Amer. Math. Soc.},
  fjournal = {Memoirs of the American Mathematical Society},
  volume = {146},
  number = {694},
  year = {2000},
  pages = {viii+134},
  doi = {10.1090/memo/0694},
  url = {https://doi.org/10.1090/memo/0694},
  issn = {0065-9266,1947-6221},
  mrclass = {35L65 (47H20)},
  mrnumber = {1686652},
  mrreviewer = {Hermano\ Frid},
}

@article{BG,
  author = {Bressan, Alberto and Goatin, Paola},
  title = {Oleinik type estimates and uniqueness for {$n\times n$} conservation laws},
  journal = {J. Differential Equations},
  fjournal = {Journal of Differential Equations},
  volume = {156},
  number = {1},
  year = {1999},
  pages = {26--49},
  doi = {10.1006/jdeq.1998.3606},
  url = {https://doi.org/10.1006/jdeq.1998.3606},
  issn = {0022-0396,1090-2732},
  mrclass = {35L65},
  mrnumber = {1701818},
}

@article{BGu,
  author = {Bressan, Alberto and Guerra, Graziano},
  title = {Unique solutions to hyperbolic conservation laws with a strictly convex entropy},
  journal = {J. Differential Equations},
  fjournal = {Journal of Differential Equations},
  volume = {387},
  year = {2024},
  pages = {432--447},
  doi = {10.1016/j.jde.2024.01.005},
  url = {https://doi.org/10.1016/j.jde.2024.01.005},
  issn = {0022-0396,1090-2732},
  mrclass = {35L65 (35A02)},
  mrnumber = {4690554},
}

@article{BL,
  author = {Bressan, Alberto and LeFloch, Philippe},
  title = {Uniqueness of weak solutions to systems of conservation laws},
  journal = {Arch. Ration. Mech. Anal.},
  fjournal = {Archive for Rational Mechanics and Analysis},
  volume = {140},
  number = {4},
  year = {1997},
  pages = {301--317},
  doi = {10.1007/s002050050068},
  url = {https://doi.org/10.1007/s002050050068},
  issn = {0003-9527},
  mrclass = {35L65 (35D05)},
  mrnumber = {1489317},
  mrreviewer = {Vyacheslav\ V.\ Chistyakov},
}

@article{BDL,
  author = {Bressan, Alberto and De Lellis, Camillo},
  title = {A remark on the uniqueness of solutions to hyperbolic conservation laws},
  journal = {Arch. Ration. Mech. Anal.},
  fjournal = {Archive for Rational Mechanics and Analysis},
  volume = {247},
  number = {6},
  year = {2023},
  pages = {Paper No. 106, 12},
  doi = {10.1007/s00205-023-01936-y},
  url = {https://doi.org/10.1007/s00205-023-01936-y},
  issn = {0003-9527,1432-0673},
  mrclass = {35L65},
  mrnumber = {4661213},
  mrreviewer = {Qin\ Zhao},
}

@article{BLe,
  author = {Bressan, Alberto and Lewicka, Marta},
  title = {A uniqueness condition for hyperbolic systems of conservation laws},
  journal = {Discrete Contin. Dyn. Syst.},
  fjournal = {Discrete and Continuous Dynamical Systems},
  volume = {6},
  number = {3},
  year = {2000},
  pages = {673--682},
  doi = {10.3934/dcds.2000.6.673},
  url = {https://doi.org/10.3934/dcds.2000.6.673},
  issn = {1078-0947,1553-5231},
  mrclass = {35L65 (35L45)},
  mrnumber = {1757395},
  mrreviewer = {Yu.\ G.\ Rykov},
}

@article{BLY,
  author = {Bressan, Alberto and Liu, Tai-Ping and Yang, Tong},
  title = {{$L^1$} stability estimates for {$n\times n$} conservation laws},
  journal = {Arch. Ration. Mech. Anal.},
  fjournal = {Archive for Rational Mechanics and Analysis},
  volume = {149},
  number = {1},
  year = {1999},
  pages = {1--22},
  doi = {10.1007/s002050050165},
  url = {https://doi.org/10.1007/s002050050165},
  issn = {0003-9527,1432-0673},
  mrclass = {35L65 (35B35 35B45)},
  mrnumber = {1723032},
}

@article{BV,
  author = {Buckmaster, Tristan and Vicol, Vlad},
  title = {Nonuniqueness of weak solutions to the {N}avier-{S}tokes equation},
  journal = {Ann. of Math. (2)},
  fjournal = {Annals of Mathematics. Second Series},
  volume = {189},
  number = {1},
  year = {2019},
  pages = {101--144},
  doi = {10.4007/annals.2019.189.1.3},
  url = {https://doi.org/10.4007/annals.2019.189.1.3},
  issn = {0003-486X,1939-8980},
  mrclass = {35Q30 (35Q31 35Q35 76D05 76F02)},
  mrnumber = {3898708},
  mrreviewer = {Isabelle\ Gruais},
}

@article{CKV,
  author = {Chen, Geng and Krupa, Sam G. and Vasseur, Alexis F.},
  title = {Uniqueness and weak-{BV} stability for {$2\times 2$} conservation laws},
  journal = {Arch. Ration. Mech. Anal.},
  fjournal = {Archive for Rational Mechanics and Analysis},
  volume = {246},
  number = {1},
  year = {2022},
  pages = {299--332},
  doi = {10.1007/s00205-022-01813-0},
  url = {https://doi.org/10.1007/s00205-022-01813-0},
  issn = {0003-9527,1432-0673},
  mrclass = {35L65},
  mrnumber = {4487515},
}

@incollection{Dafermos,
  author = {Dafermos, C. M.},
  title = {Entropy and the stability of classical solutions of hyperbolic systems of conservation laws},
  booktitle = {Recent mathematical methods in nonlinear wave propagation ({M}ontecatini {T}erme, 1994)},
  series = {Lecture Notes in Math.},
  volume = {1640},
  publisher = {Springer, Berlin},
  year = {1996},
  pages = {48--69},
  doi = {10.1007/BFb0093706},
  url = {https://doi.org/10.1007/BFb0093706},
  isbn = {3-540-61907-0},
  mrclass = {35L65 (35B35)},
  mrnumber = {1600904},
  mrreviewer = {Hermano\ Frid},
}

@book{Dafermos2,
  author = {Dafermos, Constantine M.},
  title = {Hyperbolic {C}onservation {L}aws in {C}ontinuum {P}hysics},
  series = {Grundlehren der mathematischen Wissenschaften [Fundamental Principles of Mathematical Sciences]},
  volume = {325},
  edition = {Fifth},
  publisher = {Springer, Berlin},
  year = {2026},
  pages = {xlii+1045},
  doi = {10.1007/978-3-662-72445-3},
  url = {https://doi.org/10.1007/978-3-662-72445-3},
  isbn = {978-3-662-72444-6; 978-3-662-72445-3},
  mrclass = {99-06},
  mrnumber = {5070576},
}

@article{DiPerna,
  author = {DiPerna, Ronald J.},
  title = {Uniqueness of solutions to hyperbolic conservation laws},
  journal = {Indiana Univ. Math. J.},
  fjournal = {Indiana University Mathematics Journal},
  volume = {28},
  number = {1},
  year = {1979},
  pages = {137--188},
  doi = {10.1512/iumj.1979.28.28011},
  url = {https://doi.org/10.1512/iumj.1979.28.28011},
  issn = {0022-2518,1943-5258},
  mrclass = {35L65},
  mrnumber = {523630},
  mrreviewer = {V.\ \v Durikovi\v c},
}

@article{FNP,
  author = {Feireisl, Eduard and Novotn\'y, Anton\'in and Petzeltov\'a, Hana},
  title = {On the existence of globally defined weak solutions to the {N}avier-{S}tokes equations},
  journal = {J. Math. Fluid Mech.},
  fjournal = {Journal of Mathematical Fluid Mechanics},
  volume = {3},
  number = {4},
  year = {2001},
  pages = {358--392},
  doi = {10.1007/PL00000976},
  url = {https://doi.org/10.1007/PL00000976},
  issn = {1422-6928,1422-6952},
  mrclass = {35Q30 (76N10)},
  mrnumber = {1867887},
  mrreviewer = {Cun-Zheng\ Wang},
}

@article{Glimm,
  author = {Glimm, James},
  title = {Solutions in the large for nonlinear hyperbolic systems of equations},
  journal = {Comm. Pure Appl. Math.},
  fjournal = {Communications on Pure and Applied Mathematics},
  volume = {18},
  number = {4},
  year = {1965},
  pages = {697--715},
  doi = {10.1002/cpa.3160180408},
  url = {https://doi.org/10.1002/cpa.3160180408},
  issn = {0010-3640,1097-0312},
  mrclass = {35.57},
  mrnumber = {194770},
  mrreviewer = {A.\ K.\ Aziz},
}

@book{HR,
  author = {Holden, Helge and Risebro, Nils Henrik},
  title = {Front tracking for hyperbolic conservation laws},
  series = {Applied Mathematical Sciences},
  volume = {152},
  edition = {Second},
  publisher = {Springer, Heidelberg},
  year = {2015},
  pages = {xiv+515},
  doi = {10.1007/978-3-662-47507-2},
  url = {https://doi.org/10.1007/978-3-662-47507-2},
  isbn = {978-3-662-47506-5; 978-3-662-47507-2},
  mrclass = {35-02 (35B30 35L45 35L65 35Q53 65-02)},
  mrnumber = {3443431},
}

@article{weightaintroduction,
  author = {Kang, Moon-Jin and Vasseur, Alexis F.},
  title = {Criteria on contractions for entropic discontinuities of systems of conservation laws},
  journal = {Arch. Ration. Mech. Anal.},
  fjournal = {Archive for Rational Mechanics and Analysis},
  volume = {222},
  number = {1},
  year = {2016},
  pages = {343--391},
  doi = {10.1007/s00205-016-1003-1},
  url = {https://doi.org/10.1007/s00205-016-1003-1},
  issn = {0003-9527,1432-0673},
  mrclass = {35L65 (35L45)},
  mrnumber = {3519973},
  mrreviewer = {Ramon\ G.\ Plaza},
}

@article{KVW,
  author = {Kang, Moon-Jin and Vasseur, Alexis F. and Wang, Yi},
  title = {Uniqueness of a planar contact discontinuity for 3{D} compressible {E}uler system in a class of zero dissipation limits from {N}avier-{S}tokes-{F}ourier system},
  journal = {Comm. Math. Phys.},
  fjournal = {Communications in Mathematical Physics},
  volume = {384},
  number = {3},
  year = {2021},
  pages = {1751--1782},
  doi = {10.1007/s00220-021-04100-3},
  url = {https://doi.org/10.1007/s00220-021-04100-3},
  issn = {0010-3616,1432-0916},
  mrclass = {35Q31},
  mrnumber = {4268832},
}

@article{KVWtime,
  author = {Kang, Moon-Jin and Vasseur, Alexis F. and Wang, Yi},
  title = {Time-asymptotic stability of composite waves of viscous shock and rarefaction for barotropic {N}avier-{S}tokes equations},
  journal = {Adv. Math.},
  fjournal = {Advances in Mathematics},
  volume = {419},
  year = {2023},
  pages = {Paper No. 108963, 66},
  doi = {10.1016/j.aim.2023.108963},
  url = {https://doi.org/10.1016/j.aim.2023.108963},
  issn = {0001-8708,1090-2082},
  mrclass = {35Q30 (76N06 76P05)},
  mrnumber = {4560999},
  mrreviewer = {Ond\v rej\ Kreml},
}

@article{Leger,
  author = {Leger, Nicholas},
  title = {{$L^2$} stability estimates for shock solutions of scalar conservation laws using the relative entropy method},
  journal = {Arch. Ration. Mech. Anal.},
  fjournal = {Archive for Rational Mechanics and Analysis},
  volume = {199},
  number = {3},
  year = {2011},
  pages = {761--778},
  doi = {10.1007/s00205-010-0341-7},
  url = {https://doi.org/10.1007/s00205-010-0341-7},
  issn = {0003-9527,1432-0673},
  mrclass = {35L65 (35B35 35L67 74J40 76L05)},
  mrnumber = {2771666},
  mrreviewer = {Marta\ Lewicka},
}

@book{Lions,
  author = {Lions, Pierre-Louis},
  title = {Mathematical topics in fluid mechanics. {V}ol. 2},
  series = {Oxford Lecture Series in Mathematics and its Applications},
  volume = {10},
  publisher = {The Clarendon Press, Oxford University Press, New York},
  year = {1998},
  pages = {xiv+348},
  note = {Compressible models, Oxford Science Publications},
  isbn = {0-19-851488-3},
  mrclass = {76-02 (35-02 35Q30 76N10)},
  mrnumber = {1637634},
  mrreviewer = {Denis\ Serre},
}

@article{LY,
  author = {Liu, Tai-Ping and Yang, Tong},
  title = {{$L_1$} stability for {$2\times 2$} systems of hyperbolic conservation laws},
  journal = {J. Amer. Math. Soc.},
  fjournal = {Journal of the American Mathematical Society},
  volume = {12},
  number = {3},
  year = {1999},
  pages = {729--774},
  doi = {10.1090/S0894-0347-99-00292-1},
  url = {https://doi.org/10.1090/S0894-0347-99-00292-1},
  issn = {0894-0347,1088-6834},
  mrclass = {35L65 (35B35)},
  mrnumber = {1646841},
  mrreviewer = {Yu.\ G.\ Rykov},
}

@article{LYu,
  author = {Liu, Tai-Ping and Yu, Shih-Hsien},
  title = {Navier-{S}tokes equations in gas dynamics: {G}reen's function, singularity, and well-posedness},
  journal = {Comm. Pure Appl. Math.},
  fjournal = {Communications on Pure and Applied Mathematics},
  volume = {75},
  number = {2},
  year = {2022},
  pages = {223--348},
  doi = {10.1002/cpa.22030},
  url = {https://doi.org/10.1002/cpa.22030},
  issn = {0010-3640,1097-0312},
  mrclass = {76N06},
  mrnumber = {4373171},
  mrreviewer = {Jianlin\ Zhang},
}

@article{SV,
  author = {Serre, Denis and Vasseur, Alexis F.},
  title = {{$L^2$}-type contraction for systems of conservation laws},
  journal = {J. \'Ec. polytech. Math.},
  fjournal = {Journal de l'\'Ecole polytechnique. Math\'ematiques},
  volume = {1},
  year = {2014},
  pages = {1--28},
  doi = {10.5802/jep.1},
  url = {https://doi.org/10.5802/jep.1},
  issn = {2429-7100,2270-518X},
  mrclass = {35K65 (35L45 35L67)},
  mrnumber = {3322780},
  mrreviewer = {Rodica\ Luca},
}

@article{SVintermediate,
  author = {Serre, Denis and Vasseur, Alexis F.},
  title = {The relative entropy method for the stability of intermediate shock waves; the rich case},
  journal = {Discrete Contin. Dyn. Syst.},
  fjournal = {Discrete and Continuous Dynamical Systems},
  volume = {36},
  number = {8},
  year = {2016},
  pages = {4569--4577},
  doi = {10.3934/dcds.2016.36.4569},
  url = {https://doi.org/10.3934/dcds.2016.36.4569},
  issn = {1078-0947,1553-5231},
  mrclass = {35L65 (35B35 35L67)},
  mrnumber = {3479527},
  mrreviewer = {Srboljub\ S.\ Simi\'c},
}

@incollection{Tartar,
  author = {Tartar, Luc},
  title = {Some remarks on separately convex functions},
  booktitle = {Microstructure and phase transition},
  series = {IMA Vol. Math. Appl.},
  volume = {54},
  publisher = {Springer, New York},
  year = {1993},
  pages = {191--204},
  doi = {10.1007/978-1-4613-8360-4_12},
  url = {https://doi.org/10.1007/978-1-4613-8360-4_12},
  isbn = {0-387-94112-6},
  mrclass = {49N60 (49J10 73C50 73V25)},
  mrnumber = {1320538},
  mrreviewer = {John\ M.\ Ball},
}

@misc{CKF,
  author = {Chen, Geng and Faile, Cooper and Krupa, Sam G.},
  title = {Solutions to conservation laws are {H\"older}-stable in {$L^2$} in the weak-{BV} setting},
  year = {2025},
  howpublished = {arXiv preprint arXiv:2507.23645},
  url = {https://arxiv.org/abs/2507.23645},
  eprint = {2507.23645},
  archivePrefix = {arXiv},
  primaryClass = {math.AP},
}

@article{asymptoticNSinflow,
  author = {Huang, Xushan and Kang, Moon-Jin and Kim, Jeongho and Lee, Hobin},
  title = {Asymptotic behavior toward viscous shock for impermeable wall and inflow problems of barotropic {Navier--Stokes} equations},
  journal = {J. Math. Anal. Appl.},
  fjournal = {Journal of Mathematical Analysis and Applications},
  volume = {552},
  number = {2},
  year = {2025},
  pages = {129803},
  doi = {10.1016/j.jmaa.2025.129803},
  url = {https://doi.org/10.1016/j.jmaa.2025.129803},
  eprint = {2405.03214},
  archivePrefix = {arXiv},
  primaryClass = {math.AP},
}

@article{asymptoticNSoutflow,
  author = {Kang, Moon-Jin and Oh, HyeonSeop and Wang, Yi},
  title = {Asymptotic behavior toward viscous shock for the outflow problem of barotropic {Navier--Stokes} equations},
  journal = {Nonlinearity},
  fjournal = {Nonlinearity},
  volume = {39},
  number = {7},
  year = {2026},
  pages = {075006},
  doi = {10.1088/1361-6544/ae7ec2},
  url = {https://doi.org/10.1088/1361-6544/ae7ec2},
  eprint = {2505.08171},
  archivePrefix = {arXiv},
  primaryClass = {math.AP},
}

@article{asymptoticNSboundarylayer,
  author = {Han, Sungho and Kang, Moon-Jin and Kim, Jeongho and Kim, Nayeon and Oh, HyeonSeop},
  title = {Convergence to superposition of boundary layer, rarefaction and shock for the inflow problem of the {1D} {Navier--Stokes} equations},
  journal = {Comm. Math. Phys.},
  fjournal = {Communications in Mathematical Physics},
  volume = {407},
  number = {4},
  year = {2026},
  pages = {79},
  doi = {10.1007/s00220-026-05572-x},
  url = {https://doi.org/10.1007/s00220-026-05572-x},
  eprint = {2502.08126},
  archivePrefix = {arXiv},
  primaryClass = {math.AP},
}

@misc{meng,
  author = {Meng, Jiayun},
  title = {Time-asymptotic stability of composite weak planar waves for a general {$n\times n$} multi-{D} viscous system},
  year = {2025},
  howpublished = {arXiv preprint arXiv:2501.14188},
  url = {https://arxiv.org/abs/2501.14188},
  eprint = {2501.14188},
  archivePrefix = {arXiv},
  primaryClass = {math.AP},
}

@article{Euler,
  author = {Euler, Leonhard},
  title = {Principes g{\'e}n{\'e}raux du mouvement des fluides},
  journal = {M{\'e}moires de l'Acad{\'e}mie Royale des Sciences et des Belles-Lettres de Berlin},
  volume = {11},
  year = {1757},
  pages = {274--315},
  url = {https://scholarlycommons.pacific.edu/euler-works/226/},
}

@article{JLe99,
  author = {Joseph, K. T. and LeFloch, Philippe G.},
  title = {Boundary layers in weak solutions of hyperbolic conservation laws},
  journal = {Arch. Ration. Mech. Anal.},
  fjournal = {Archive for Rational Mechanics and Analysis},
  volume = {147},
  number = {1},
  year = {1999},
  pages = {47--88},
  doi = {10.1007/s002050050145},
  url = {https://doi.org/10.1007/s002050050145},
}

@misc{AMS24,
  author = {Ancona, Fabio and Marson, Andrea and Spinolo, Laura V.},
  title = {Existence of vanishing physical viscosity solutions of characteristic initial-boundary value problems for systems of conservation laws},
  year = {2024},
  howpublished = {arXiv preprint arXiv:2401.14865},
  url = {https://arxiv.org/abs/2401.14865},
  eprint = {2401.14865},
  archivePrefix = {arXiv},
  primaryClass = {math.AP},
}

@misc{prepweakbv,
      title={Weak-{BV} stability of inflow and outflow problems for the isentropic {E}uler system}, 
      author={Moon-Jin Kang and Jiayun Meng and HyeonSeop Oh and Alexis F. Vasseur},
      year={2026},
      eprint={2610.03236},
      archivePrefix={arXiv},
      primaryClass={math.AP},
      url={https://arxiv.org/abs/2610.03236}, 
}

@article{nspresse,
  author = {Mellet, A. and Vasseur, A.},
  title = {Existence and uniqueness of global strong solutions for one-dimensional compressible {N}avier-{S}tokes equations},
  journal = {SIAM J. Math. Anal.},
  fjournal = {SIAM Journal on Mathematical Analysis},
  volume = {39},
  number = {4},
  year = {2008},
  pages = {1344--1365},
  doi = {10.1137/060658199},
  url = {https://doi.org/10.1137/060658199},
  issn = {0036-1410,1095-7154},
  mrclass = {76N10 (35Q30)},
  mrnumber = {2368905},
  mrreviewer = {Vladimir\ V.\ Shelukhin},
}

@article{existence,
  author = {Kang, Moon-Jin and Vasseur, Alexis F.},
  title = {Global smooth solutions for 1{D} barotropic {N}avier-{S}tokes equations with a large class of degenerate viscosities},
  journal = {J. Nonlinear Sci.},
  fjournal = {Journal of Nonlinear Science},
  volume = {30},
  number = {4},
  year = {2020},
  pages = {1703--1721},
  doi = {10.1007/s00332-020-09622-z},
  url = {https://doi.org/10.1007/s00332-020-09622-z},
  issn = {0938-8974,1432-1467},
  mrclass = {35Q35 (76N10)},
  mrnumber = {4113341},
  mrreviewer = {Luc\ Paquet},
}

@article{impermeable,
  author = {Meng, Jiayun and Vasseur, Alexis F.},
  title = {Weak-{BV} stability for the impermeable wall problem of the isentropic {E}uler and its inviscid limit from {N}avier-{S}tokes},
  journal = {Discrete Contin. Dyn. Syst.},
  fjournal = {Discrete and Continuous Dynamical Systems},
  volume = {59},
  year = {2027},
  pages = {29--44},
  doi = {10.3934/dcds.2026159},
  url = {https://doi.org/10.3934/dcds.2026159},
  issn = {1078-0947,1553-5231},
  mrclass = {35 (76)},
  mrnumber = {5130618},
}

@book{smoller,
    AUTHOR = {Smoller, Joel},
     TITLE = {Shock waves and reaction-diffusion equations},
    SERIES = {Grundlehren der mathematischen Wissenschaften [Fundamental
              Principles of Mathematical Sciences]},
    VOLUME = {258},
   EDITION = {Second},
 PUBLISHER = {Springer-Verlag, New York},
      YEAR = {1994},
     PAGES = {xxiv+632},
      ISBN = {0-387-94259-9},
   MRCLASS = {35-02 (35K57 35L67 58E07)},
  MRNUMBER = {1301779},
       DOI = {10.1007/978-1-4612-0873-0},
       URL = {https://doi-org.ezproxy.lib.utexas.edu/10.1007/978-1-4612-0873-0},
}

\end{document}